\documentclass[11pt, reqno]{amsart}

\usepackage{amscd, amsmath, amssymb,verbatim,graphicx,chngcntr,xcolor}
\usepackage[colorlinks]{hyperref}

\newcommand{\N}{\mathbb{N}}

\newcommand{\R}{\mathbb{R}}
\newcommand{\C}{\mathbb{C}}
\newcommand{\D}{\mathbb{D}}
\newcommand{\A}{\mathbf{A}}
\newcommand{\de}{\partial}
\newcommand{\db}{\overline{\partial}}
\newcommand{\ddbar}{i\partial\overline\partial}
\newcommand{\ov}[1]{\overline{#1}}

\newcommand{\tr}[2]{\mathrm{tr}^{#1}{#2}}
\newcommand{\ti}[1]{\tilde{#1}}
\newcommand{\vp}{\varphi}
\newcommand{\ve}{\varepsilon}

\renewcommand{\P}{\mathbb{P}}
\renewcommand{\leq}{\leqslant}
\renewcommand{\geq}{\geqslant}

\renewcommand{\epsilon}{\varepsilon}
\renewcommand{\diamond}{\diamondsuit}

\newtheorem{theorem}{Theorem}[section]
\newtheorem{lemma}[theorem]{Lemma}

\newtheorem{proposition}[theorem]{Proposition}

\newcounter{mtheorem}
\newtheorem{mtheorem}[mtheorem]{Theorem}

\theoremstyle{definition}
\newtheorem{definition}[theorem]{Definition}
\newtheorem{rk}[theorem]{Remark}

\numberwithin{equation}{section}
\counterwithin{figure}{section}

\begin{document}

\title{Smooth asymptotics for collapsing Calabi-Yau metrics}

\author{Hans-Joachim Hein}
\address{Mathematisches Institut, WWU M\"unster, 48149 M\"unster, Germany}
\email{hhein@wwu.de}

\author{Valentino Tosatti}
\address{Courant Institute of Mathematical Sciences, New York University, New York, NY 10012, USA}
\email{tosatti@cims.nyu.edu}

\date{\today}

\begin{abstract}
We prove that Calabi-Yau metrics on compact Calabi-Yau manifolds whose K\"ahler classes shrink the fibers of a holomorphic fibration have a priori estimates of all orders away from the singular fibers.  To this end we prove an asymptotic expansion of these metrics in terms of powers of the fiber diameter,  with $k$-th order remainders that satisfy uniform $C^k$-estimates with respect to a collapsing family of background metrics.  The constants in these estimates are uniform not only in the sense that they are independent of the fiber diameter,  but also in the sense that they only depend on the constant in the estimate for $k = 0$ known from previous work of the second-named author.  For $k > 0$ the new estimates are proved by blowup and contradiction,  and each additional term of the expansion arises as the obstruction to proving a uniform bound on one additional derivative of the remainder.
\end{abstract}

\maketitle

\markboth{Smooth asymptotics for collapsing Calabi-Yau metrics}{Hans-Joachim Hein and Valentino Tosatti}

\setcounter{tocdepth}{3}
\tableofcontents

\section{Introduction}

Yau's proof of the Calabi conjecture \cite{Ya} shows that compact K\"ahler manifolds with vanishing real first Chern class (Calabi-Yau manifolds) admit Ricci-flat K\"ahler metrics, a unique one in each K\"ahler cohomology class. The study of these metrics has since been a central topic in complex geometry, and one particularly interesting question is to understand the behavior of these metrics when the K\"ahler class degenerates. The case when the metrics are volume noncollapsed is by now well-understood (see \cite{To4} and references therein), but the collapsing case presents a much harder challenge. This problem was first studied in the work of Gross-Wilson \cite{GW} on elliptically fibered $K3$ surfaces with $I_1$ singular fibers, where, by constructing the degenerating Ricci-flat metrics via gluing, they showed in particular that these metrics collapse locally smoothly away from the singular fibers to a canonical K\"ahler metric with nonnegative Ricci curvature on the base. Our goal in this paper is to prove that this conclusion holds in complete generality in all dimensions.  This will be a corollary of the existence of a complete asymptotic expansion of the collapsing Ricci-flat metrics locally uniformly away from the singular fibers.  The key new insight of this paper, which led us to discover this expansion but actually yields a stronger statement, is that the existence of such an expansion can be proved in tandem with Schauder-type regularity estimates. In particular,  each term of the expansion is bounded in all $C^k$ norms by a constant times a fixed power of the fiber diameter, where the constant is not only independent of the fiber diameter (the usual meaning of an asymptotic expansion) but is in fact almost independent of the given family of Ricci-flat metrics itself. Indeed, we will prove that all of these constants depend only on the constant in a $C^0$ estimate,  \eqref{unifequiv}, of these metrics due to the second-named author \cite{To}.

\subsection{Setup}\label{sectsetup}

Let $X$ be a compact Calabi-Yau manifold (compact K\"ahler with $c_1(X)=0$ in $H^2(X,\mathbb{R})$) of dimension $m+n$ which admits a surjective holomorphic map $f:X\to B$ with connected fibers onto a compact K\"ahler reduced and irreducible analytic space of dimension $m$. Let $S\subset X$ be the preimage of the singular locus of $B$ together with the critical values of $f$ on the regular part of $B$, so $S$ and $f(S)$ are closed proper analytic subvarieties and $f:X\setminus S\to B\setminus f(S)$ is a proper holomorphic submersion with $n$-dimensional Calabi-Yau fibers $X_b=f^{-1}(b)$. We will implicitly assume that $m,n>0$, so that the discussion is nontrivial. The set $S$ will be referred to as the union of the singular fibers of $f$.

Such fiber spaces arise naturally when $X$ is a projective Calabi-Yau manifold and $L$ is a semiample line bundle on $X$ with $(L^{m+n})=0$, by taking $f$ to be the morphism given by $|\ell L|$ for $\ell$ sufficiently large and divisible (conjecturally the semiampleness assumption can be relaxed to nefness, up to replacing $L$ by another numerically equivalent line bundle). This gives a wealth of examples, including the elliptic fibrations of $K3$ surfaces mentioned above and Lefschetz pencils on Calabi-Yau $3$-folds as in \cite{Li}.

It is important to note that, by Ehresmann's lemma, all fibers $X_b$, $b\in B\setminus f(S)$, are diffeomorphic to a fixed Calabi-Yau $n$-fold $Y$, and $f|_{X\setminus B}$ is a locally trivial $C^\infty$ fiber bundle. On the other hand, $f|_{X\setminus B}$ is a locally trivial holomorphic fiber bundle if and only if the smooth fibers $X_b$ are all pairwise biholomorphic, by the Fischer-Grauert theorem \cite{FG} (in this case $f$ is called isotrivial). Furthermore, for a Calabi-Yau fiber space $f:X\to B$ as above, it is proved in \cite{TZ} and \cite[Thm 3.3]{TZ2} that $S=\emptyset$ (i.e., $B$ is smooth and $f$ is a submersion) if and only if $f$ is a holomorphic fiber bundle.

Given now a fiber space $f:X\to B$ as above, fix K\"ahler metrics $\omega_X$ and $\omega_B$, with $\omega_X$ Ricci-flat. For all $t\geq 0$ we let $\omega^\bullet_t$ be the unique Ricci-flat K\"ahler metric on $X$ cohomologous to $f^*\omega_B+e^{-t}\omega_X$. We can write $\omega^\bullet_t=f^*\omega_B+e^{-t}\omega_X+\ddbar\psi_t$, where $\psi_t$ are smooth functions solving the degenerating family of complex Monge-Amp\`ere equations
\begin{equation}\label{ma_initial}
(\omega^\bullet_t)^{m+n}=(f^*\omega_B+e^{-t}\omega_X+\ddbar\psi_t)^{m+n}=c_t e^{-nt}\omega_X^{m+n},\quad \sup\nolimits_X\psi_t=0.
\end{equation}
Here $c_t>0$ is defined by integrating \eqref{ma_initial} over $X$, and as $t \to \infty$,
\begin{equation}c_t \to \binom{m+n}{n}\frac{\int_X f^*\omega_B^m\wedge\omega_X^{n}}{\int_X\omega_X^{m+n}}>0.
\end{equation} In particular, the total volume $\mathrm{Vol}(X,\omega^\bullet_t)$ of $X$ as well as the total volume ${\rm Vol}(X_b, \omega_t^\bullet|_{X_b})$ of each fiber is comparable to $e^{-nt}$, so we have volume collapse as $t\to\infty$.

To describe the leading order behavior of $\omega_t^\bullet$ as $t \to \infty$, one then produces a K\"ahler metric $\omega_{\rm can}=\omega_B+\ddbar \psi_\infty$ on $B\setminus f(S)$ by solving the complex Monge-Amp\`ere equation \cite{ST,To}
\begin{equation}\label{ma_limit}
\omega_{\rm can}^m=(\omega_B+\ddbar \psi_\infty)^m=c_\infty f_*(\omega_X^{m+n}),\quad \sup\nolimits_B\psi_t=0,
\end{equation}
where $\psi_t\in C^0(B)\cap C^\infty(B\setminus f(S))$ and $c_\infty=\int_B\omega_B^m\int_{X_b}\omega_X^n/\int_X\omega_X^{m+n}$. Then $\mathrm{Ric}(\omega_{\rm can})=\omega_{\rm WP}\geq 0$, where the Weil-Petersson form $\omega_{\rm WP}$ measures the variation of the complex structures of the fibers.

\subsection{Smooth collapsing}

Our first main theorem is the following. This confirms a conjecture of the second-named author from roughly 10 years ago, see for instance the surveys \cite{To2,To3,To4}.

\begin{mtheorem}\label{mthmA}
In the above setting, as $t \to \infty$, locally uniformly away from the singular fibers of $f$, the Ricci-flat metrics $\omega_t^\bullet$ converge to $f^*\omega_{\rm can}$ in the standard $C^\infty$ topology of tensor fields.
\end{mtheorem}

Equivalently, this means that $\psi_t\to f^*\psi_\infty$ in $C^k_{\rm loc}(X\setminus S,g_X)$ for all $k\geq 0$, so, in particular, that we have derivative estimates to all orders for the solution $\psi_t$ of the Monge-Amp\`ere equation \eqref{ma_initial}, away from the singular fibers of $f$ but uniformly in $t$. As observed in \cite[Rmk 1.7]{HT2}, the analogous purely local statement for the real or complex Monge-Amp\`ere equation with ellipticity degenerating along a foliation is false, and, while our proof is local on the base, it crucially relies on the leaves being closed manifolds. Importantly, the constants in our higher order $C^k$ estimates depend only on the constant in the $C^0$ estimate \eqref{unifequiv} due to \cite{To}. In this sense, our estimates are true \emph{a priori} estimates.

\subsection{Previous work}

There have been a number of partial results in the direction of Theorem \ref{mthmA}. As mentioned earlier, this statement follows from \cite{GW} in the case of elliptically fibered $K3$ surfaces with $I_1$ fibers, where a gluing construction for $\omega^\bullet_t$ is given (see also \cite{JS2} for suitable higher order estimates at the singular fibers). This was recently extended to more general singular fibers in \cite{CVZ}, and to Lefschetz fibrations $f:X^3 \to\P^1$ with $K3$ fibers in \cite{Li}. The gluing approach relies on knowing the precise structure of the singular fibers and on being able to use this knowledge to construct an ansatz for $\omega^\bullet_t$ near them, so it would be unrealistic to hope to prove Theorem \ref{mthmA} in this way in general.

In the special situation where there are no singular fibers (so in fact $f$ is a holomorphic fiber bundle by \cite{TZ,TZ2}), Theorem \ref{mthmA} follows from \cite{Fi0,Fi}, where $\omega^\bullet_t$ is constructed by considering a {\em semi-Ricci-flat} form $\omega_F$ on $X$ (to be discussed below) and deforming it using the implicit function theorem. For this approach the absence of singular fibers is crucial, as discussed in \cite[\S 1.3]{HT2}.

Progress towards Theorem \ref{mthmA} in general became possible thanks to \cite{DP,EGZ}, who by using Ko\l odziej's pluripotential theory techniques \cite{Kol} proved that $\sup_X|\psi_t|\leq C$ uniformly in $t$. Using this estimate, the second-named author \cite{To} extended Yau's second-order estimates \cite{Ya} to our setting (see also \cite{ST0} for the case when $m=n=1$) and showed that $\psi_t\to f^*\psi_\infty$ in $C^{1,\alpha}_{\rm loc}(X\setminus S,g_X)$ for all $0<\alpha<1$, and also established that for any $K\Subset X\setminus S$ there is a $C$ such that
\begin{equation}\label{unifequiv}
C^{-1}(f^*\omega_B+e^{-t}\omega_X)\leq \omega^\bullet_t\leq C(f^*\omega_B+e^{-t}\omega_X)
\end{equation}
holds on $K$ for all $t$.

If the smooth fibers $X_b$ are tori (or finite free quotients of tori), it is possible to ``unravel'' the fibers by passing to the universal cover of the preimage of a ball in the base. Applying a vertical stretching on this cover, estimate \eqref{unifequiv} with some work implies that the stretched Ricci-flat metrics are uniformly Euclidean, hence satisfy higher-order estimates by standard theory. In this way, Theorem \ref{mthmA} was proved in \cite{GTZ} when the fibers $X_b$ are tori, under a projectivity assumption which was removed in \cite{HT}, and then in \cite{TZ3} when the fibers are finite free quotients of tori. This approach crucially uses the fact that tori are covered by Euclidean space, and does not extend to more general smooth fibers.

The next general result towards Theorem \ref{mthmA} was obtained in \cite{TWY}, where it is shown that $\omega_t^\bullet\to f^*\omega_{\rm can}$ in $C^0_{\rm loc}(X\setminus S,g_X)$. This regularity appears to be a natural barrier for methods based on modifications of Yau's estimates \cite{Ya}. However, by introducing new methods, we reproved this result in \cite{HT2} and improved the convergence to $C^{0,\alpha}_{\rm loc}(X\setminus S,g_X)$ in general ($0<\alpha<1$), and also proved Theorem \ref{mthmA} if $f$ is isotrivial (i.e., the smooth fibers are all biholomorphic but $S$ can be nonempty). In that work, which is purely local on the base, we introduced a new nested blowup method and the idea of working with non-K\"ahler Riemannian product reference metrics, deriving a contradiction to various Liouville type theorems on cylinders. This method is also the foundation of the current paper.
\subsection{Smooth asymptotics}

Theorem \ref{mthmA} is in fact a simple corollary of a much stronger statement, which gives an asymptotic expansion for $\omega^\bullet_t$ with strong uniform estimates of the remainders. We will now describe this statement.  Broadly speaking, our expansion is a hybrid of an expansion in the usual sense and an interior regularity estimate, where each new term in the expansion is constructed as the obstruction to proving a uniform estimate of the previous remainder in a stronger H\"older-type norm.  By construction, the resulting terms then satisfy uniform interior  {\em a priori} estimates (depending only on background data and on the uniform $C^0$ bound \eqref{unifequiv} established in \cite{To}) in these norms.

The precise theorem that we prove is Theorem \ref{shutupandsuffer} below, whose statement we now describe in rough terms. First, thanks to \eqref{unifequiv}, one might naively expect that $\omega^\bullet_t$ satisfies higher order $C^k_{\rm loc}$ estimates even with respect to the shrinking reference metrics $f^*\omega_B+e^{-t}\omega_X$. However, a moment's thought reveals that such estimates very much depend on the fiberwise restrictions of these reference metrics for $t = 0$, as different choices give rise to $C^k$ norms that do not remain uniformly equivalent as $t\to\infty$ for $k \geq 1$. Furthermore, it was shown in \cite{TWY,TZ3} that $e^t\omega^\bullet_t|_{X_b}$ converges smoothly to the unique Ricci-flat K\"ahler metric $\omega_{F,b}$ on $X_b$ cohomologous to $\omega_X|_{X_b}$. Thus, if higher order estimates with respect to shrinking reference metrics hold at all, this forces the reference metrics to restrict to $e^{-t}\omega_{F,b}$ on each $X_b$. We thus define a $(1,1)$-form $\omega_F$ on $X\setminus S$ by writing $\omega_{F,b}=\omega_X|_{X_b}+\ddbar\rho_b$ where $\int_{X_b}\rho_b \omega_X^n=0$,
observing that $\rho_b$ varies smoothly in $b\in B\setminus f(S)$ and so defines a smooth function $\rho$ on $X\setminus S$, and defining
\begin{equation}\label{pstwlrllk}
\omega_F=\omega_X+\ddbar\rho, \quad \omega^\natural_t=f^*\omega_{\rm can}+e^{-t}\omega_F
\end{equation}
on $X\setminus S$. Then the $(1,1)$-form $\omega_F$ is semi-Ricci-flat in the sense that it restricts to a Ricci-flat metric on every smooth fiber, but it is not in general semipositive definite on $X\setminus S$; see \cite{CGP} for a counterexample. Nevertheless, given any $K\Subset X\setminus S$ there is a $t_K$ such that $\omega^\natural_t$ is a K\"ahler metric on $K$ for all $t\geq t_K$, uniformly equivalent to $f^*\omega_B+e^{-t}\omega_X$, and hence to $\omega^\bullet_t$ by \eqref{unifequiv}. Observe also that every closed real $(1,1)$-form on $X\setminus S$ which is $i\de\db$-cohomologous to $\omega_X$ and restricts to $\omega_{F,b}$ on each $X_b$ must be equal to $\omega_F+\ddbar f^*u$ for some smooth function $u$ on $B\setminus f(S)$. The semi-Ricci-flat form $\omega_F$ has also played a crucial role in previous work on this problem, see e.g. \cite{GSVY,GTZ,GW,HT,HT2,ST,To,TWY}.

From now on we work locally on the base, and simply let $B$ denote a Euclidean coordinate ball in the base away from $f(S)$ over which $f$ is $C^\infty$ trivial. We are thus working on $B\times Y$ with a complex structure $J$ which is not in general the product one, but for which ${\rm pr}_B$ is $(J,J_{\C^m})$-holomorphic. Up to replacing $\psi_t$ by $\psi_t-\psi_\infty-e^{-t}\rho$, we can write $\omega^\bullet_t=\omega^\natural_t+\ddbar\psi_t,$ where from the above-mentioned results we know that $\psi_t\to 0$ in $C^{2,\alpha}(B\times Y)$ (although for the purposes of the present paper it is sufficient to know that $i\partial\overline\partial\psi_t \to 0$ in the weak sense of currents, which was already proved in \cite{To}).

If $f$ is isotrivial, we can pick the $C^\infty$ trivialization to be holomorphic, so $J$ is a product, the metrics $\omega_{F,b}$ are independent of $b$ and $\omega^\natural_t$ is a product metric. In  \cite{HT2} we then proved that one does indeed have uniform $C^k_{\rm loc}(X\setminus S,g^\natural_t)$ estimates for $\omega^\bullet_t$. This was also proved in \cite{GTZ,HT} in the case of torus fibers if $J$ is not a product, but then the fiberwise Ricci-flat metrics $\omega_{F,b}$ are parallel with respect to each other in a suitable $C^\infty$ trivialization of $f$, which greatly simplifies the estimates. Outside of these two special cases, the metrics $\omega^\natural_t$ are not products and $\omega_{F,b}$ is not parallel with respect to $\omega_{F,b'}$ for any $b \neq b'$ with respect to any $C^\infty$ identification $X_b \cong X_{b'}$. As in our previous paper \cite{HT2}, this forces us to introduce a family of collapsing non-K\"ahler Riemannian product metrics $g_{b,t} = g_{\C^m} + e^{-t}g_{F,b}$.

The first innovation of the present paper is to introduce a certain connection $\D$ and parallel transport operator $\P$ adapted to the family $g_{b,t}$ such that, with a suitable definition of H\"older norms in terms of these objects, a sharp interpolation inequality can be proved whose constants are independent of $t$. Using this setup we are then also able prove the crucial new technical Theorem \ref{prop55}.

At this point one might still hope to prove that $\omega^\bullet_t$ has uniform $C^{k,\alpha}(\D,\P,\{g_{b,t}\}_{b\in B})$ bounds for all $k$ and $\alpha$, up to passing to a smaller ball in the base if necessary. This hope turns out to be justified for $k=0$. The proof of this amounts to a streamlined and improved version of our previous paper \cite{HT2}, and a similar statement with a similar proof also holds for $k=1$.

However, the second innovation of the present paper is to realize that no such statement can be true for $k = 2$ in general, and the obstructions roughly speaking come from the fact that the semi-Ricci-flat metric $\omega^\natural_t$ is itself not uniformly bounded in $C^{2,\alpha}(\D,\P,\{g_{b,t}\}_{b \in B})$ unless $f$ is isotrivial or the fibers are flat. With substantial work, for any given $k\geq 2$, we then identify explicit obstruction functions, and decompose $\omega^\bullet_t-\omega^\natural_t=\ddbar\psi_t$ into a sum of finitely many terms (constructed using these obstruction functions together with $\omega_F$) and a remainder, such that each of the finitely many terms is unbounded in $C^{k,\alpha}(\D,\P,\{g_{b,t}\}_{b \in B})$ but is almost explicit, while the remainder is not explicit at all but is indeed uniformly bounded in $C^{k,\alpha}(\D,\P,\{g_{b,t}\}_{b \in B})$. This is in a nutshell the content of Theorem \ref{shutupandsuffer}.

To deduce Theorem \ref{mthmA} from this, it suffices to check that those terms in the above decomposition of $\omega_t^\bullet-\omega_t^\natural$ that come from the obstruction functions, while not uniformly bounded in $C^{k,\alpha}(\D,\P,\{g_{b,t}\}_{b \in B})$, are at least uniformly bounded in the $C^{k,\alpha}$ norm of a $t$-independent metric.

From Theorem \ref{shutupandsuffer} it is possible to extract the explicit form of the first nontrivial term in the above expansion of $\omega^\bullet_t$ by plugging this expansion back into the Monge-Amp\`ere equation \eqref{ma_initial}. This is exactly the term which indirectly made us realize that starting at $k = 2$ a new strategy is required if $f$ is not isotrivial and the fibers are not flat. In fact, from the explicit form of this term one can see directly why these two special cases play a distinguished role. To state a simplified form of the result, let $z_1,\ldots,z_m$ denote the standard coordinates on $B \subset \C^m$, and for any given $b\in B$ and $1\leq\mu\leq m$ let $A_\mu$ denote the unique $T^{1,0}X_b$-valued $(0,1)$-form on $X_b$ which is harmonic with respect to the fiberwise Ricci-flat metric $\omega_{F,b}$ and represents the Kodaira-Spencer class of the variation of the complex structure of $X_b$ in direction $\partial_{z_\mu}$. Let also $\langle\cdot,\cdot\rangle$ denote the fiberwise Ricci-flat inner product, $\Delta_{Y}^{-1}$ the fiberwise inverse of the fiberwise Ricci-flat Laplacian, and $\underline{\psi}$ the fiberwise average of a function $\psi$ on the total space with respect to the fiberwise Calabi-Yau volume form $\omega_F^n|_{X_b}$. Then we have:

\begin{mtheorem}\label{mthmB}
The expansion mentioned above begins with
\begin{equation}\label{e:stronzo1}
\omega_t^\bullet = f^*\omega_{\rm can} + e^{-t}\omega_F - e^{-2t}i\partial\overline\partial\Delta_Y^{-1} \Delta_Y^{-1}(g_{\rm can}^{\mu\bar\nu}(\langle A_\mu, A_{\bar{\nu}}\rangle - \underline{\langle A_{\mu}, A_{\bar\nu}\rangle} )) + f^*\mathrm{error}_1 + \mathrm{error}_2.
\end{equation}
Here ${\rm error}_1$ is an $i\partial\overline\partial$-exact $(1,1)$-form on $B$ that goes to zero in $C^2(B)$ whereas ${\rm error}_2$ is an $i\partial\overline\partial$-exact $(1,1)$-form on $B \times Y$ that satisfies the following estimates. Fix any local $C^\infty$ product coordinate system on $B \times Y$. For any contravariant tensor $T$, let $T\{\ell\}$ denote the component of $T$ with $\ell$ fiber indices in terms of these local coordinates. Then for all $0 \leq j \leq 2$ we have that
\begin{equation}\label{simplicio}
|(\de^j\mathrm{error}_2)\{\ell\}| = \begin{cases}O(e^{-(2+\ell-j)\frac{t}{2}})&\text{if}\;\,0 \leq \ell < j+2,\\
o(e^{-2t})&\text{if}\;\,\ell = j+2.
\end{cases}
\end{equation}
\end{mtheorem}
Thus, for $0 \leq j \leq 2$, the components of $\partial^j {\rm error}_2$ with at least one base index might still swamp the corresponding components of $\partial^j$ of the explicit $e^{-2t}$ term in \eqref{e:stronzo1}. However, the ``all fiber'' component of $\partial^j{\rm error}_2$ decays strictly faster than the ``all fiber'' component of $\partial^j$ of that term.  Unfortunately, the estimates \eqref{simplicio} are not completely effective in terms of the constant $C$ of \eqref{unifequiv} because they depend on the rate at which $\text{error}_1$ goes to zero in $C^2(B)$,  and this convergence comes from an ineffective global argument in \cite{To} (subsequential limits solving a Monge-Amp\`ere equation globally on the base, which implies uniqueness of subsequential limits and hence convergence, but at an unknown rate).

Recall that $\underline{\langle A_\nu,A_{\bar\nu}\rangle}=\omega_{{\rm WP},\mu\bar\nu}$ and $\mathrm{Ric}(\omega_{\rm can})=\omega_{\rm WP}$. The quantity $\Delta_Y^{-1}(\langle A_\mu, A_{\bar{\nu}}\rangle - \underline{\langle A_{\mu}, A_{\bar\nu}\rangle})$ arises also in the work of Schumacher \cite{Schu}, and is related to the ``geodesic curvature'' appearing in the work of Semmes \cite{Se}. While the corresponding term in \eqref{e:stronzo1} obviously vanishes if $f$ is isotrivial or the fibers are flat, it does not vanish in general. Indeed, consider the case when $n=2$, $m=1$, and the fibers are $K3$ surfaces. Thanks to the smoothness of the moduli space of $K3$ surfaces, $A_1$ can be identified with an arbitrary closed, anti-self-dual, complex-valued $2$-form, and our term vanishes if and only if $A_1$ has constant length with respect to the Ricci-flat metric. But this is false in the asymptotically cylindrical gluing limit of $K3$ surfaces \cite[\S 5]{CC} because then there is a $16$-dimensional space of such forms that are exponentially small on the neck, see e.g.~\cite[\S 4]{He0}, hence cannot possibly have constant length.

As mentioned previously, in the case of flat or pairwise isomorphic fibers, Theorem \ref{mthmA} was already known \cite{GTZ,HT,HT2,TZ3}. It turns out that these papers also imply a strong version of Theorem \ref{mthmB} in these two special cases. Indeed, it was pointed out in the appendix of \cite{DJZ} that it follows from the estimates of \cite{GTZ,HT,TZ3} that if the fibers are flat, then $|\partial^j{\rm error}_2| = O_{j,K}(e^{-Kt})$ for all $j,K \in \mathbb{N}_0$. In the isotrivial case, the same conclusion can be extracted from \cite{HT2}. Thus, in these two cases, the expansion of $\omega_t^\bullet$ actually contains no terms of finite order in $e^{-t}$ beyond $f^*\omega_{\rm can} + e^{-t}\omega_F +  f^*{\rm error}_1$. However, the best known result towards Theorem \ref{mthmB} in general, which also follows from \cite{HT2}, only states that
\begin{equation}
\omega_t^\bullet = f^*\omega_{\rm can} + e^{-t}\omega_F + f^*{\rm error}'_1 + {\rm error}'_2\;\,{\rm with}\;\, |{\rm error}_2'\{\ell\}| = O(e^{-(\ell+\alpha)\frac{t}{2}})
\end{equation}
for all $0 \leq \ell \leq 2$ and $\alpha \in (0,1)$, where ${\rm error}'_1\to 0$ uniformly on $B$.

In the case of Calabi-Yau $3$-folds fibered by Lefschetz pencils of $K3$ surfaces, it might be interesting to compare Theorem \ref{mthmB} to the estimates resulting from the global gluing construction of \cite{Li}.

\subsection{Overview of the proofs}

Having already explained how Theorem \ref{mthmA} follows from the complete form of Theorem \ref{mthmB}, i.e., Theorem \ref{shutupandsuffer}, we focus here on the latter. To construct the desired expansion of the Ricci-flat metrics, we fix $k$ and for $0 \leq j \leq k$ aim to construct the first $j$ terms of the expansion (which, we emphasize, depend on the choice of $k$) by induction on $j$. For this, we first need to identify the obstruction functions mentioned before Theorem \ref{mthmB}. These are smooth functions $G_{i,p,k}$ on $B \times Y$ with fiberwise average zero, which are indexed over $0\leq i\leq j$ and $1\leq p\leq N_{i,k}$ and implicitly depend on $k$. They vanish identically for $j=0,1$ but for $j\geq 2$ have to be identified by a delicate procedure. In essence, they come from an expansion of the RHS of a Monge-Amp\`ere equation of the form
\begin{equation}\label{mak}
\frac{(\omega^\natural_t+\ddbar\psi_t)^{m+n}}{(\omega^\sharp_t)^{m+n}}=c_t e^{-nt}\frac{\omega_X^{m+n}}{(\omega^\sharp_t)^{m+n}},
\end{equation}
where $\omega^\sharp_t$ equals $\omega^\natural_t$ plus small corrections that involve the obstruction functions themselves. Breaking this apparent logical cycle requires an iterative procedure with a uniform gain at each step.  We will also require the obstruction functions to be fiberwise orthonormal in $L^2$, which is not immediate from Gram-Schmidt because the dimension of their fiberwise span may not be constant on the base. All of these problems will be solved in Section \ref{sectref}, specifically in Theorem \ref{ghost}.

Having chosen the obstruction functions, we decompose $\omega^\bullet_t-\omega^\natural_t$ inductively via
\begin{equation}
\gamma_{t,0}:=\ddbar\underline{\psi_t},\quad \gamma_{t,i,k}:=\sum_{p=1}^{N_{i,k}} \ddbar\mathfrak{G}_{t,k}(A_{t,i,p,k},G_{i,p,k})\;\,(2\leq i\leq j),\label{fahfee}
\end{equation}
\begin{equation}
\eta_{t,i,k}:=\omega^\bullet_t-\omega^\natural_t-\gamma_{t,0}-\sum_{\iota=2}^i \gamma_{t,\iota,k}.\label{fofum}
\end{equation}
Here $A_{t,i,p,k}$ is a function on $B$, approximately equal to the fiberwise $L^2$ component of $\mathrm{tr}^{\omega^\natural_t}\eta_{t,i-1,k}$ onto $G_{i,p,k}$, and $\mathfrak{G}_{t,k}$ is an approximate Green operator such that in a very rough sense, after scaling and stretching so that the fibers have unit size,
\begin{equation}\label{totalfuckinglie}
\Delta^{\omega^\natural_t}\mathfrak{G}_{t,k}(A_{t,i,p,k},G_{i,p,k}) \approx A_{t,i,p,k}G_{i,p,k}.
\end{equation}
Since $\mathfrak{G}_{t,k}$ acts as a differential rather than integral operator in the base directions (of order $2k$), this statement cannot be meaningful unless $A_{t,i,p,k}$ behaves like a polynomial of sufficiently low degree. For us, this assumption is justified because we are working with functions with bounded H\"older norms.

Our main Theorem \ref{shutupandsuffer} will then prove uniform $C^{j,\alpha}$ estimates for $\gamma_{t,0}$ and uniform $C^{j+2,\alpha}$ estimates for $A_{t,i,p,k}$ ($2\leq i\leq j$) with respect to the usual H\"older norms on $B$, uniform $C^{j,\alpha}$ estimates for $\eta_{t,j,k}$ with respect to the collapsing H\"older norms on $B \times Y$ defined in Section \ref{s:adia}, and auxiliary estimates for derivatives of $A_{t,i,p,k}$ of order $j+3$ to $2k+j+2$. While these auxiliary estimates are not uniform in $t$, they do imply a uniform $C^{j,\alpha}(g_X)$ bound for the pieces $\gamma_{t,i,k}$ $(2\leq i\leq j)$. All of these estimates are proved by a contradiction and blowup argument in the spirit of our paper \cite{HT2}, which essentially dealt with the case $j = k = 0$. It is worth noting that this case, which corresponds to the Evans-Krylov step in the classical regularity theory of the complex Monge-Amp\`ere equation, is the only one that requires a genuinely nonlinear ingredient (specifically, the Liouville theorem of \cite{He,LLZ}). However, the extension to $k > 0$ in this paper, while linear in nature, is vastly more complicated at a technical level.

If the desired estimates fail, we obtain a sequence of solutions violating these estimates more and more strongly, and we need to study $3$ cases according to how the blowup rate compares to $e^{t}$. In the fast-forming and regular-forming cases, after stretching and scaling, our background geometries limit to $\C^{m+n}$ and $\C^m\times Y$ respectively, and a contradiction can easily be derived from known regularity and Liouville theorems for the Monge-Amp\`ere equation. The majority of the work goes into understanding the slow-forming case, where the background geometries collapse to $\C^m$. Here we are quickly reduced to showing that the two points where the relevant H\"older seminorm achieves its maximum cannot collide, and one can think of this as an improvement of regularity of a more linear nature.

To achieve this improvement, we first split all our objects into their jets at a blowup basepoint and the corresponding Taylor remainders, and derive precise estimates on all these pieces, which ultimately allow us to expand and linearize the Monge-Amp\`ere equation. By construction, the Taylor remainders will have excellent convergence and growth properties, whereas the jets satisfy much worse bounds but are by definition polynomials. On the LHS of the Monge-Amp\`ere equation, we then move the jets into the reference metric $\omega_t^\natural$, thus creating the new reference metric $\omega^\sharp_t$ that appears in \eqref{mak}, and expand the LHS of \eqref{mak} as its
linearization at $\omega^\sharp_t$, plus nonlinearities, applied to the Taylor remainders of the various components of the original solution. The nonlinearities turn out to be negligible, as one would expect at this stage of the proof. The RHS is a mixture of background objects and jets of the original solution components, without any helpful growth bounds but with an almost explicit structure.

To actually prove that the blowup points do not collide, we again argue by contradiction and need to consider $3$ subcases, where after further stretching and scaling the background geometries converge to $\mathbb{C}^{m+n}, \mathbb{C}^m\times Y$ and $\mathbb{C}^m$, respectively. We deal with the fast-forming case by using Schauder estimates for the linearized PDE on balls in $\C^{m+n}$. While this is the easiest case conceptually, the proof is much longer than in the other two cases because of the large number of pieces appearing in the decomposition of the solution in this case and the complexity of the quantitative estimates that they satisfy.

In the regular-forming case, we argue that the LHS of the PDE has a limit thanks to the excellent growth of the Taylor remainders, so the RHS must have a limit as well, necessarily of the form
\begin{equation}\label{ratta}
K_0(z)+\sum_{q=1}^N K_q(z)H_q(y).
\end{equation}
Here $K_0,K_q$ are polynomials on $\C^m$ of degree at most $j$, and $H_q$ are smooth functions on $Y$ which by our choice of the obstruction functions lie in the linear span of the restrictions $G_{i,p,k}|_{\{z_\infty\}\times Y}$ ($2\leq i\leq j$), where $z_\infty \in B$ is the limit of the projections to $B$ of the blowup basepoints. On the other hand, the decomposition \eqref{fahfee}--\eqref{fofum} and the approximate Green property \eqref{totalfuckinglie} were made so that in the limit, the LHS of the PDE consists of one part coming from the Taylor remainders of $\gamma_{t,2,k}, \ldots, \gamma_{t,j,k}$, which lies in the fiberwise linear span of the functions $G_{i,p,k}|_{\{z_\infty\}\times Y}$, and another part coming from the Taylor remainders of $\gamma_{t,0}$ and $\eta_{t,j,k}$, which is fiberwise $L^2$ orthogonal to this span. Setting this equal to \eqref{ratta}, we obtain that the Taylor remainders of order $\geq j+1$ of $\gamma_{t,2,k}, \ldots, \gamma_{t,j,k}$ are polynomials of degree $\leq j$ in the limit (an immediate contradiction), and the Taylor remainders of order $\geq j+1$ of $\gamma_{t,0}$ and $\eta_{t,j,k}$ have trace equal to a polynomial of degree $\leq j$ in the limit (contradicting Liouville's theorem).

In the slow-forming case, Theorem \ref{prop55} tells us directly that $\eta_{t,j,k}$ is too small to contribute to the assumed failure of Theorem \ref{shutupandsuffer}, and it follows from the linearized Monge-Amp\`ere equation itself that $\gamma_{t,2,k}, \ldots, \gamma_{t,j,k}$ are too small to contribute as well. Thus, the only remaining contribution comes from $\gamma_{t,0}$, which lives on the base, and this eventually contradicts Liouville's theorem on $\C^m$.

Lastly, while the expansion of Theorem \ref{shutupandsuffer} is not quite precise enough to identify the first nontrivial term in Theorem \ref{mthmB} (which ultimately comes from the term $\gamma_{t,2,2}$ in Theorem \ref{shutupandsuffer}) explicitly, this can be done with some extra work by plugging this expansion back into the Monge-Amp\`ere equation

\subsection{Further applications}  The new techniques developed in this paper are very robust and can apply in a number of other contexts where a similar geometric collapsing picture appears. In particular, it can be applied to the study of long-time solutions of the K\"ahler-Ricci flow on compact K\"ahler manifolds with semiample canonical bundle and intermediate Kodaira dimension \cite{ST}, where the resulting analogous expansion would not only give smooth collapsing away from the singular fibers but also show that the evolving metrics have locally uniformly bounded Ricci curvature there, settling two long-standing conjectures of Song and Tian (see e.g. \cite{Ti}). This will be addressed in our forthcoming work with Lee \cite{HLT}. Our techniques also apply to the ``continuity method'' of \cite{LT}, which has been much studied recently.

\subsection{Organization of the paper}

In Section \ref{s:adia} we define new covariant derivatives, parallel transport operators and H\"older norms adapted to our setting, and prove the key interpolation inequality and the fundamental Theorem \ref{prop55}. Section \ref{sectref} contains our main technical constructions of the approximate Green operator and the selection of the obstruction functions in Theorem \ref{ghost}. Section \ref{s:mainest} is the heart of the paper where Theorem \ref{shutupandsuffer} is proved. Theorems \ref{mthmA} and \ref{mthmB} are deduced from this in Section \ref{sectmain}.

\subsection{Acknowledgments}

We thank L.~Lempert and S.~Zelditch for discussions, J.-M.~Bismut and C.~Margerin for encouragement, and M.-C.~Lee and a referee for many useful comments. Part of this work was carried out during the second-named author's visits to the Department of Mathematics and the Center for Mathematical Sciences and Applications at Harvard University, which he would like to thank for the hospitality. The authors were partially supported by NSF grants DMS-1745517 (H.-J.H.) and DMS-1610278, DMS-1903147, DMS-2231783 (V.T.). The first-named author was also partially funded by the German Research Foundation (DFG) under Germany's Excellence Strategy EXC 2044-390685587 ``Mathematics M\"unster:~Dynamics-Geometry-Structure" and by the CRC 1442 ``Geometry:~Deformations and Rigidity'' of the DFG.

\section{H\"older norms and interpolation inequalities}\label{s:adia}

\subsection{$\D$-derivatives and $\P$-transport}

Let $f:X\to B$ be a proper surjective holomorphic submersion with $n$-dimensional Calabi-Yau fibers over the unit ball $B\subset\C^m$, as in the introduction, equipped with a K\"ahler form $\omega_X$ on the total space, and let $Y=f^{-1}(0)$. For all $z\in B$ let $\omega_{F,z}$ be the unique Calabi-Yau metric on the fiber $f^{-1}(z)$ cohomologous to the restriction of $\omega_X$. Fixing an arbitrary $C^\infty$ trivialization $\Phi:B\times Y\to X$, denote by $J$ the complex structure on $B\times Y$ pulled back from $X$, let $g_{Y,z}$ be the Ricci-flat Riemannian metric on $Y$ associated to $\Phi^*\omega_{F,z}$ and $J$, extended trivially to $B\times Y$, and define a family of Riemannian product metrics on $B\times Y$ by
\begin{equation}\label{produkt}
g_{z,t}:={\rm pr}_B^*(g_{\mathbb{C}^m}) + e^{-t} {\rm pr}_Y^*(g_{Y,z})\;\,{\rm for\;all}\;z \in B.
\end{equation}
In our previous work \cite[\S 5]{HT2}, we defined a H\"older-type seminorm on contravariant tensors that uses $g_{z,t}$ and its parallel transport on the fiber over $z$. As it turns out, in order to generalize this construction to higher order derivatives and preserve all the good properties that we had in \cite{HT2}, we need to modify this construction by introducing a different notion of parallel transport adapted to this setting.

\begin{definition}
For $z \in B$ write $\nabla^z$ to denote the Levi-Civita connection of the product metric $g_{z,t}$, which is independent of $t$. Then let $\D$ denote the connection on the tangent bundle of $B \times Y$ and on all of its tensor bundles defined by setting
\begin{equation}\label{d:singlecov}
(\D \tau)(x) :=
(\nabla^{{\rm pr}_B(x)}\tau)(x)
\end{equation}
for all tensors $\tau$ on $B \times Y$ and all $x \in B \times Y$.
Write $\D^k$ for the iterated covariant derivatives of $\mathbb{D}$. For all curves $\gamma: [t_0,t_1] \to B \times Y$ and $a,b \in [t_0,t_1]$ let $\P^\gamma_{a,b}$ denote $\mathbb{D}$-parallel transport from $\gamma(a)$ to $\gamma(b)$.
\end{definition}

It is clear that $\D$ satisfies the definition of a connection on a vector bundle, which is torsion-free on the tangent bundle but in general not metric with respect to any fixed Riemannian metric. As usual, for all curves $\gamma: [t_0,t_1] \to B \times Y$ and all tensor fields $\tau$ along $\gamma$ we then have the useful formula
\begin{equation}\label{thebs}
\P^\gamma_{t_1,t_0}\frac{d}{dt}\bigg|_{t=t_0}\P^\gamma_{t,t_1}(\tau(t)) = \frac{\D}{dt}\biggr|_{t=t_0} \tau(t) =\frac{\nabla^{{\rm pr}_B(\gamma(t_0))}}{dt}\biggr|_{t=t_0}\tau(t).
\end{equation}
Also note that we obtain a natural definition of \emph{$\P$-geodesics}, but in this paper we will only be using the two obvious families of $\P$-geodesics:
vertical paths of the form $(z_0, \gamma(t))$, where $\gamma(t)$ is a $g_{Y,z_0}$-geodesic in the fiber $\{z_0\}\times Y$ (we will call such a geodesic \emph{minimal} if $\gamma$ is minimal with respect to $g_{Y,z_0}$), and horizontal paths of the form $(z(t),y_0)$, where $z(t)$ is an affine segment in $\C^m$. Every two points on $B\times Y$ can be joined by concatenating two of these $\P$-geodesics where the vertical one is minimal.

We now record two technical properties of $\P$ and $\D$ which will be very useful for us later on.

\subsubsection{The operator norm of $\P$-transport}\label{s:tridiagonal}

Let $x,x' \in B \times Y$ and let $\gamma: [t_0,t_1] \to B \times Y$ be a $\P$-geodesic from $x$ to $x'$, which we assume is either horizontal or minimal vertical.
Let $u$ be a tangent vector at $x$, which we again assume is either horizontal or vertical. Note that $\D(d{\rm pr}_B) = 0$ because all of the connections $\nabla^z$ are product connections on $B \times Y$. This implies that if $u$ is horizontal, then $\P_{a,b}^\gamma u = u$ for all $a,b \in [t_0,t_1]$, and if $u$ is vertical, then $\P_{a,b}^\gamma u$ is again vertical.

Thanks to this discussion, $\P^\gamma_{a,b}$ is block diagonal. Thus, in particular, the operator norm of $\P^\gamma_{t_0,t_1}$ with respect to the shrinking product metrics $g_{z,t}$ is fixed independent of $t$. Later, we will often be considering this picture after applying a family of stretching diffeomorphisms
\begin{equation}\label{gammat}
\Sigma_t:B_{e^{t/2}}\times Y\to B\times Y, \quad (z,y)=\Sigma_t(\check{z},\check{y})=(e^{-t/2}\check{z},\check{y}),
\end{equation}
where it is more natural to use the scaled product metrics
\begin{equation}
\check{g}_{\check{z},t}:=e^t\Sigma_t^*g_{z,t}={\rm pr}_B^*(g_{\mathbb{C}^m})+{\rm pr}_Y^*(g_{Y,e^{-t/2}\check{z}}),
\end{equation} which converge smoothly on compact sets to the fixed product metric
\begin{equation}g:={\rm pr}_B^*(g_{\mathbb{C}^m})+{\rm pr}_Y^*(g_{Y,0}).
\end{equation}
Then, since $\P$ commutes with this stretching and scaling, we conclude that the operator norm of $\P^{\check\gamma}$ from $\check{x}_t$ to $\check{x}'_t$ (points in $B_{e^{t/2}}\times Y$) still has uniform bounds independent of $t$ as long as their images $x_t=\Sigma_t(\check{x}_t)$,  $x'_t=\Sigma_t(\check{x}'_t)$ lie in a fixed relatively compact subset of $B\times Y$ (which will usually be $B'\times Y$ where $B'\subset B$ is a smaller ball, not necessarily concentric).

\subsubsection{Commutators of $\D$-derivatives}

If the complex structure on $B \times Y$ is a product, then $g_{Y,z} = g_Y$ for all $z \in B$, and $\D$ is simply the product connection on $B \times Y$ given by the Euclidean derivative in the $B$ directions and the Levi-Civita connection of $g_Y$ in the $Y$ directions. In particular, base and fiber derivatives with respect to $\D$ commute. In our general setting, we still have the following property.

\begin{lemma}\label{totalbullshit}
Let  $\ell \in \{1,\ldots,2m\}$. Let $\mathbf{h}^\ell$ denote the $\ell$-th standard basis vector field on $B$ trivially extended to $B\times Y$. Let $\tau$ be a tensor on $B \times Y$.  For any given $z \in B$ define
\begin{equation}\label{e:nefarious}
\A^z_{\ell}\tau := \frac{\partial}{\partial \tilde{z}^\ell}\biggr|_{\tilde{z} = z} \nabla^{\tilde{z}}\tau.
\end{equation}
Then it holds for all $x = (z,y) \in B \times Y$ and all $\mathbf{v} \in T_yY$ that
\begin{equation}\label{e:almostcommute}
(\D^2_{\mathbf{h}^\ell,\mathbf{v}}\tau)(x) = (\D^2_{\mathbf{v},\mathbf{h}^\ell}\tau)(x) + \mathbf{v} \,\lrcorner\,(\A^z_\ell\tau)(x).
\end{equation}
\end{lemma}

\begin{proof}
Fix the point $x = (z,y) \in B \times Y$. It is true by definition that
\begin{equation}(\D^2_{\mathbf{h}^\ell, \mathbf{v}}\tau)(x) = \mathbf{v} \,\lrcorner\, (\D_{\mathbf{h}^\ell}(\D\tau))(x).\end{equation}
Now observe that the operator $\D_{\mathbf{h}^\ell}$ is just the ordinary $\ell$-th partial derivative in $\R^{2m}$. Thus,
\begin{align}\label{herrfaltings}
(\D_{\mathbf{h}^\ell} (\D\tau))(x) &= \frac{\partial}{\partial \tilde{z}^\ell}\biggr|_{\tilde{z}=z} (\nabla ^{\tilde{z}}\tau)(\tilde{z},y) =  (\A^z_\ell\tau)(x) + (\nabla^z_{\mathbf{h}^\ell}(\nabla^z\tau))(x).
\end{align}
This then implies that
\begin{equation}\label{burkhard}(\D^2_{\mathbf{h}^\ell, \mathbf{v}}\tau)(x) = \mathbf{v} \,\lrcorner\, (\A^z_\ell\tau)(x) + ((\nabla^z)^2_{\mathbf{h}^\ell,\mathbf{v}}\tau)(x).\end{equation}
It remains to observe that $\nabla^z$ is a product connection, so that horizontal and vertical derivatives with respect to $\nabla^z$ commute, and that
\begin{equation}((\nabla^z)^2_{\mathbf{v},\mathbf{h}^\ell}\tau)(x) = \mathbf{h}^\ell \,\lrcorner\, (\nabla^z_{\mathbf{v}}(\nabla^z \tau))(x) = (\D^2_{\mathbf{v},\mathbf{h}^\ell}\tau)(x) \end{equation}
because $\nabla^z\tau = \D\tau$ along the entire fiber through $x$.
\end{proof}

Going back to the definition \eqref{e:nefarious}, we can also trivially define $\mathbf{A}^z_{\mathbf{v}}=0$ for any vertical vector $\mathbf{v}$ and in this way obtain a $(2,1)$ tensor $\mathbf{A}$ on $B\times Y$ that vanishes whenever the first lower index is tangent to $Y$. Then \eqref{burkhard} holds trivially by definition if we replace $\mathbf{h}^\ell$ by $\mathbf{v}$, i.e., we can compactly write
\begin{equation}\label{komm}\D^2\tau=\nabla^{z,2}\tau+\mathbf{A}\circledast\tau,\end{equation}
where here and in the rest of the paper $\circledast$ denotes some tensorial contraction.

We will need a generalization of this schematic formula to higher order derivatives, relating $\D^j$ with $\nabla^{z,j}$. To this end we introduce a dummy variable $\ti{z}\in B$ and write $\Gamma^{\ti{z}}(z,y)$ for the Christoffel symbols of $\nabla^{\ti{z}}$ at the point $(z,y)$. We can then write symbolically at the point $(z,y)$:
\begin{equation}\label{alien}\D=(\de_z+\de_y + \Gamma^{\ti{z}}(z,y)+ \de_{\ti{z}})|_{\tilde{z}=z},\end{equation}
which is applied to a tensor field $\tau$ on $B\times Y$ that may also depend on $\ti{z}$, for example when we take iterated $\D$ derivatives. Formally squaring \eqref{alien} clearly shows \eqref{komm}, where $\mathbf{A}=\de_{\ti{z}}\Gamma^{\ti{z}}|_{\tilde{z}=z}$.

\begin{lemma}\label{sgurz}
For each $j\geq 2$ we have that
\begin{equation}\label{domineddio}
\D^j\tau-\nabla^{z,j}\tau=\sum_{p=0}^{j-2}\nabla^{z,p}\tau\circledast \nabla^{j-2-p}_{z,y,\ti{z}}\mathbf{A},
\end{equation}
where $\nabla^{j-2-p}_{z,y,\ti{z}}$ means that there are $j-2-p$ derivatives, each of which is either a $\nabla^{\ti{z}}$ acting in any direction in the $(z,y)$ variables or a $\de_{\ti{z}}$ acting on some $\Gamma^{\tilde{z}}(z,y)$, with final evaluation at $\tilde{z}=z$.
\end{lemma}

\begin{proof}
To prove \eqref{domineddio}, we write the LHS as
\begin{equation}(\de_z+\de_y + \Gamma^{\ti{z}}(z,y)+ \de_{\ti{z}})^j(\tau(z,y))-(\de_z+\de_y + \Gamma^{\ti{z}}(z,y))^j(\tau(z,y)),\end{equation}
expand this (recalling that $\de_{\ti{z}}$ only acts on the Christoffel symbols $\Gamma^{\ti{z}}(z,y)$), and finally evaluate at $\ti{z}=z$. The fact that the derivatives go only up to $j-2$ is because the error term has to contain at least one $\de_{\ti{z}}$,  and thus also at least one $\Gamma^{\ti{z}}(z,y)$ (otherwise there is nothing for the $\de_{\ti{z}}$ to differentiate), which generates an $\mathbf{A}$, and the number of remaining derivatives is then at most $j-2$.
\end{proof}

In all applications the schematic formula \eqref{domineddio} will suffice without any need to make the error term on the RHS more explicit. \eqref{domineddio} shows that we can think of $\D^j\tau$ as a way of packaging together $\nabla^{z,j}\tau$ with lower derivatives of $\tau$ (of order up to $j-2$) contracted with some fixed background tensors.

\subsection{Definition of the H\"older norms}\label{holderiano}

Assume that $B$ is a ball of {\em any radius} centered at $0 \in \mathbb{C}^m$. In our applications this will be a rescaling of the unit ball that we have been using so far. Assume that we have a smooth family of Riemannian metrics $g_{Y,z}$ on $\{z\}\times Y$ for $z\in B$ such that there exists a constant $C_0 \geq 1$ such that for all $z\in B$ we have that
\begin{align}
C_0^{-1/2} \leq {\rm inj}(Y, g_{Y,z}) \leq {\rm diam}(Y,g_{Y,z}) \leq C_0^{1/2},\label{e:tilt0}
\end{align}
\begin{align}
C_0^{-1} g_{Y,0} \leq g_{Y,z} \leq C_0 g_{Y,0}.\label{e:tilt1}
\end{align}
In our later applications, this family will be simply our original family $g_{Y,z}$ where the variable $z$ has been rescaled, and such a constant $C_0$ will clearly exist after shrinking our original unit ball slightly. Define a product Riemannian metric
\begin{align}\label{mrn}
g := {\rm pr}_B^*(g_{\mathbb{C}^m}) + {\rm pr}_Y^*(g_{Y,0}).
\end{align}
For ease of notation, in this section we shall write
\begin{equation}
\mathbb{B}^g(p,R):=B^{g_{\C^m}}(z_*,R)\times B^{g_{Y,0}}(y_*,R)
\end{equation}
for all $p=(z_*,y_*)\in B\times Y$ and $R>0$.

\begin{definition}\label{def:holderdef}
For all $0 < \alpha < 1$ and $R > 0$ and for all $p \in B \times Y$  we define
\begin{align}\label{e:holderdef}
[\tau]_{C^{\alpha}(\mathbb{B}^g(p,R),g)} := \sup \left\{\frac{|\tau(x_0) - \P^\gamma_{t_1,t_0}(\tau(x_1))|_g}{d^g(x_0,x_1)^\alpha} \;\Bigg|\;
\parbox{64mm}{$\gamma: [t_0,t_1] \to B \times Y$ $\P$-geodesic, either \\ horizontal or minimal vertical, and\\
$x_i := \gamma(t_i) \in \mathbb{B}^g(p,R)$ for $i = 0,1$ } \right\}
\end{align}
for all smooth tensor fields $\tau$ on $B \times Y$.
\end{definition}

Observe that in this definition the family of metrics $g_{Y,z}$ is used to define (minimal) $\P$-geodesics and $\P$-transport, while the fixed metric $g$ is used to measure the distance $d^g(x_0,x_1)$ and to define $\mathbb{B}^g(p,R)$. We will often replace $g$ in \eqref{e:holderdef} by a family of product metrics $g_t$ with variable fiber size,
\begin{equation}
g_t := {\rm pr}_B^*(g_{\mathbb{C}^m}) + \gamma_t^2{\rm pr}_Y^*(g_{Y,0}),
\end{equation}
where $\gamma_t>0$ are arbitrary. We denote the corresponding H\"older seminorms by
\begin{equation}
[\tau]_{C^{\alpha}(\mathbb{B}^{g_t}(p,R),g_t)},
\end{equation}
where
\begin{equation}\label{babbubba}
\mathbb{B}^{g_t}(p,R):=B^{g_{\C^m}}(z_*,R)\times B^{\gamma_t^2g_{Y,0}}(y_*,R).
\end{equation}
In the case when $\gamma_t\to 0$ we will refer to these generically as shrinking or collapsing H\"older seminorms. Notice that in this case we have $\mathbb{B}^{g_t}(p,R)=B^{g_{\C^m}}(z_*,R)\times Y$ for all $t \geq t_R$ sufficiently large. This will be the setting where these seminorms are applied in almost all of the paper.

Let us also observe here that these seminorms satisfy a Leibniz-type formula
\begin{equation}
[\tau\circledast\sigma]_{C^\alpha(B_R\times Y,g)}\leq C[\tau]_{C^\alpha(B_R\times Y,g)}\|\sigma\|_{L^\infty(B_R\times Y,g)}+C[\sigma]_{C^\alpha(B_R\times Y,g)}\|\tau\|_{L^\infty(B_R\times Y,g)}
\end{equation}
for all $R > {\rm diam}(Y,g_{Y,0})$, where $B_R=B^{g_{\C^m}}(z_*,R)$ for brevity, $\tau,\sigma$ are any two tensor fields and $C$ is a uniform constant. This will be used many times in the sequel.

\begin{lemma}\label{kathoum}
Suppose that $z_* \in B$ and $R>0$ are such that $B^{g_{\C^m}}(z_*,R) \subset B$, where the bounds \eqref{e:tilt0}, \eqref{e:tilt1} hold. Then for every tensor $\tau$ we have that
\begin{equation}\label{qqq}
[\tau]_{C^\alpha(\mathbb{B}^g(p,R),g)}\leq C R^{1-\alpha}\|\D\tau\|_{L^\infty(\mathbb{B}^g(p,2C_0 R),g)}.
\end{equation}
Also, if $R<C_0^{-1}$, then we can replace $\mathbb{B}^g(p,2C_0 R)$ with $\mathbb{B}^g(p,R)$.
\end{lemma}

\begin{proof}
For any given $x_0,x_1\in \mathbb{B}^g(p,R)$, let $\gamma:[t_0,t_1]\to B\times Y$ be a $\P$-geodesic joining them which is either horizontal or minimal vertical, and use \eqref{thebs} to bound
\begin{equation}\begin{split}
|\tau(x_0)-\P_{t_1,t_0}^\gamma(\tau(x_1))|_g &= \left|\int_{0}^{t_1-t_0} \frac{d}{dt} \P^\gamma_{t_1-t,t_0}(\tau(\gamma(t_1-t))) \, dt\right|_g\\
&= \left|\int_0^{t_1-t_0} \P^\gamma_{t_1-t,t_0}\biggl[\frac{\nabla^{z(\gamma(t_1-t))}}{dt}\tau(\gamma(t_1-t))\biggr]\,dt\right|_g\\
&\leq Cd^g(x_0,x_1)\|\D\tau\|_{L^\infty(\mathbb{B}^g(p,2C_0R),g)}.
\end{split}\end{equation}
Here we have also applied the estimate for the operator norm of $\P$ from Section \ref{s:tridiagonal} and the bounds \eqref{e:tilt0}, \eqref{e:tilt1} on the geometry of $g_{Y,z}$ for all $z \in B^{g_{\C^m}}(z_*,R)$. If $R < C_0^{-1}$, then $\gamma \subset \mathbb{B}^g(p,R)$.
\end{proof}

\begin{lemma}\label{arschloch}
There is a constant $C$ depending only on $R$ such that if
\begin{equation}
\sum_{j=0}^k\|\D^j\tau\|_{L^\infty(\mathbb{B}^g(p,R),g)}+[\D^k\tau]_{C^{\alpha}(\mathbb{B}^g(p,R),g)}\leq A,
\end{equation}
then the standard $C^{k,\alpha}$ Hölder norm for the metric $g$ satisfies
\begin{equation}\label{krist}
\|\tau\|_{C^{k,\alpha}_{\rm std}(\mathbb{B}^g(p,R),g)}\leq CA.
\end{equation}
\end{lemma}

\begin{proof}
Using \eqref{domineddio} to convert $\D^j$ to $\nabla^{z,j}$ with error terms up to order $j-2$,  we get
\begin{equation}\label{kaka}
\sum_{j=0}^k\|\nabla^{z,j}\tau\|_{L^\infty(\mathbb{B}^g(p,R),g)}\leq CA.
\end{equation}
We can further change $\nabla^{z,j}$ to $\nabla^{g,j}$ since all the metrics $g_{Y,z}$ for $z \in B^{g_{\C^m}}(z_*,R)$ are smoothly bounded. When converting $[\D^k\tau]$ to $[\nabla^{g,k}\tau]$, the lower order error terms that involve H\"older seminorms can also be bounded by $CA$ since thanks to \eqref{qqq} they can be bounded in terms of \eqref{kaka} and the radius $R$. So we are left with the main term $[\nabla^{g,k}\tau]_{C^{\alpha}(\mathbb{B}^g(p,R),g)}$ in the sense of \eqref{e:holderdef}, which can be converted to the standard H\"older $g$-seminorm of $\nabla^{g,k}\tau$ up to an acceptable error bounded by \eqref{kaka}.
This can be seen as follows. First, for any $(z,y),(z',y')\in B\times Y$ we have that
\begin{equation}\label{excellente}
\max(|z-z'|,d^{g_{Y,0}}(y,y'))\leq d^g((z,y),(z',y'))\leq |z-z'|+d^{g_{Y,0}}(y,y').
\end{equation}
This implies that if, in the standard definition of the $g$-H\"older seminorm, we only consider horizontal or vertical rather than arbitrary minimal $g$-geodesics, then these two seminorms are still comparable by a factor of $2$. Next, we can compare this new seminorm (horizontal or vertical minimal $g$-geodesics, $g$-parallel transport) to the one defined in \eqref{e:holderdef} exactly along the lines of \cite[Lemma 3.6]{HT2}, using also that the operator norm of $\P$ is uniformly controlled thanks to the discussion in Section \ref{s:tridiagonal}.
\end{proof}

\begin{rk}\label{archiloco}
(1) Clearly, if, in our original setup, the metrics $g_{Y,z}$ are in fact independent of $z$, then the $\P$-transport, the H\"older seminorm \eqref{e:holderdef} and the $\D$-derivatives are the same as the usual parallel transport, H\"older seminorm and covariant derivatives of the product metric $g$ in \eqref{mrn}.

(2) If we pull back our setup by the diffeomorphisms $\Sigma_t$ in \eqref{gammat}, then the pulled-back fiber metrics $g_{Y,e^{-t/2}\check{z}}$ converge locally smoothly to the fixed metric $g_{Y,0}$. From this it follows easily \cite[Remark 3.7]{HT2} that our H\"older seminorm \eqref{e:holderdef} converges to the standard one for the product metric $g$ for any fixed $p$ in the pullback space and any fixed $R$. This also clearly remains true if the metric $g$ used in \eqref{e:holderdef} is replaced by a family of product metrics that converge locally smoothly to $g$.
\end{rk}

\subsection{An interpolation inequality}\label{s:interpol}

Let again $B$ denote the \emph{unit} ball in $\C^m$ with a family of fiberwise Riemannian metrics $g_{Y,z}$ on $\{z\} \times Y$ for all $z \in B$ satisfying \eqref{e:tilt0}, \eqref{e:tilt1}. Suppose we have numbers $0<\delta_t\leq 1$ for all $t \geq 0$
and consider the stretching diffeomorphisms
\begin{equation}\label{gammat2}
\Gamma_t:\check{B}_t\times Y\to B\times Y, \quad \check{B}_t = B_{\delta_t^{-1}}, \quad (z,y)=\Gamma_t(\check{z},\check{y})=(\delta_t\check{z},\check{y}).
\end{equation}
The  product metrics
\begin{equation}\label{shrpr}
g_t:={\rm pr}_B^*(g_{\mathbb{C}^m}) + \delta_t^2{\rm pr}_Y^*(g_{Y,0})
\end{equation}
on $B\times Y$ thus satisfy
\begin{align}
\check{g}_t := \delta_t^{-2}\Gamma_t^*g_t=g.
\end{align}
The formalism of Hölder norms developed in the previous section applies equally to $(B \times Y, g_t)$ with the given family $g_{Y,z}$ of fiberwise metrics defining a connection $\mathbb{D}$ and to $(\check{B}_t \times Y, \check{g}_t) = (\check{B}_t \times Y, g)$ with the pullback family $\check{g}_{Y,\check{z}} = g_{Y,z}$ defining a connection $\check{\mathbb{D}}_t = \Gamma_t^*\mathbb{D}$.  The main result of this section is the following interpolation inequality,  which we state in both of these (equivalent) settings for later reference. In the product case ($g_{Y,z}$ is independent of $z$), this was already proved in \cite[Lemma 3.5]{HT2}.  The proof in general is almost the same but we nevertheless give it in full because we have made a lot of small changes to the definitions to accommodate the general non-product case.

\begin{proposition}\label{l:higher-interpol}
For all $k \in \N_{\geq 1}$ and $0<\alpha<1$ there exists a constant $C_k = C_k(\alpha,C_0)$ such that the following holds for all $t \geq 0$.  Let $\tau$ be a smooth contravariant $p$-tensor $\tau$ and define
\begin{align}
\check{\tau}_t:=\delta_t^{-p}\Gamma_t^*\tau.
\end{align}
Then for all $\check{p} \in  \check{B}_t \times Y$ and all $0 < \rho < R<\frac{1}{2}d^g(\check{p},\de \check{B}_t \times Y)$ we have that
\begin{align}\label{hohoho}
\sum_{j=1}^k (R-\rho)^j \|\check{\mathbb{D}}_t^{j}\check{\tau}_t\|_{L^\infty(\mathbb{B}^g(\check{p},\rho),g)} \leq C_k ((R-\rho)^{k+\alpha}[\check{\mathbb{D}}_t^{k}\check{\tau}_t]_{C^{\alpha}(\mathbb{B}^g(\check{p},R),g)} + \|\check{\tau}_t\|_{L^\infty(\mathbb{B}^g(\check{p},R),g)}).
\end{align}
For all $j\in\mathbb{N}$ and $0<\beta<1$ with $j+\beta<k+\alpha$ we also have that
\begin{align}\label{hohohoho}
(R-\rho)^{j+\beta} [\check{\mathbb{D}}_t^{j}\check{\tau}_t]_{C^\beta(\mathbb{B}^g(\check{p},\rho),g)} \leq C_k ((R-\rho)^{k+\alpha}[\check{\mathbb{D}}_t^{k}\check{\tau}_t]_{C^{\alpha}(\mathbb{B}^g(\check{p},R),g)} + \|\check{\tau}_t\|_{L^\infty(\mathbb{B}^g(\check{p},R),g)}).
\end{align}
Furthermore, the above two inequalities are equivalent to the following:
\begin{equation}\label{sanctus}
\sum_{j=1}^k (R-\rho)^j \|\D^{j}\tau\|_{L^\infty(\mathbb{B}^{g_t}(p,\rho),g_t)} \leq C_k ((R-\rho)^{k+\alpha}[\D^{k}\tau]_{C^{\alpha}(\mathbb{B}^{g_t}(p,R),g_t)} + \|\tau\|_{L^\infty(\mathbb{B}^{g_t}(p,R),g_t)}),
\end{equation}
\begin{align}\label{benedictus}
(R-\rho)^{j+\beta} [\D^{j}\tau]_{C^\beta(\mathbb{B}^{g_t}(p,\rho),g_t)} \leq C_k ((R-\rho)^{k+\alpha}[\D^{k}\tau]_{C^{\alpha}(\mathbb{B}^{g_t}(p,R),g_t)} + \|\tau\|_{L^\infty(\mathbb{B}^{g_t}(p,R),g_t)}),
\end{align}
for all $p\in B\times Y$ and all $0<\rho<R<\frac{1}{2}d^{g_t}(p,\de B\times Y)$.
\end{proposition}

Let us remark that whenever this interpolation is applied to a tensor field which is pulled back from $B$, then it simply reduces to standard interpolation in Euclidean space.

In the proof and also later in this paper we will frequently use the following two lemmas from our previous paper \cite{HT2}, which we state here for convenience.  The first is \cite[Lemma 3.4]{HT2},  a slight variation of a classical iteration lemma from elliptic PDEs, see for instance \cite[Lemma 8.18]{GM2ndEd}.

\begin{lemma}\label{HT2-l:iterate}
For all $0 \leq \epsilon < 1$, $\beta_1 < \ldots < \beta_k$ and $\gamma_1 < \ldots < \gamma_m$, there exists a constant $C$ such that the following holds. Let $f_1, \ldots, f_k: [0,T] \to \R$ be bounded nonnegative functions such that
\begin{align}\label{e:iterate_hyp}
\sum_{j=1}^k (R-\rho)^{\beta_j}f_j(\rho) \leq \epsilon \sum_{j=1}^k (R-\rho)^{\beta_j}f_j(R) + \sum_{\ell = 1}^{m} A_\ell (R-\rho)^{\gamma_{\ell}}
\end{align}
for some $A_1, \ldots, A_m\geq 0$ and for all $0 \leq \rho < R \leq T$. Then for all $0 \leq \rho < R \leq T$,
\begin{align}\label{e:iterate_concl}
\sum_{j=1}^k(R-\rho)^{\beta_j}f_j(\rho) \leq C \sum_{\ell=1}^m A_\ell (R-\rho)^{\gamma_\ell}.
\end{align}
\end{lemma}

The second crucial lemma that we wish to recall from \cite{HT2} is \cite[Lemma 3.3]{HT2}.

\begin{lemma}\label{l:HT2:braindead}
Let $Y$ be a compact Riemannian manifold without boundary. Let $E$ be a metric vector bundle over $Y$ with a metric connection $\nabla$.
Then for all $k \in \N_{\geq 1}$,  $0 < \alpha < 1$, there exists a constant $C_k = C_k(Y,E,\alpha)$ such that for all $\sigma\in C^{k,\alpha}(Y,E)$,
\begin{align}\label{e:idioticestimate}
\|\nabla\sigma\|_{L^\infty(Y)} \leq C_k[\nabla^k\sigma]_{C^{\alpha}(Y)}.
\end{align}
\end{lemma}

\begin{proof}[Proof of Proposition \ref{l:higher-interpol}]
It suffices to prove \eqref{hohoho} and \eqref{hohohoho}, since \eqref{sanctus} and \eqref{benedictus} are then obtained from them by stretching and scaling. For ease of notation we will denote $\check{\mathbb{D}}_t, \check{\tau}_t,\check{p}$ simply by $\mathbb{D},\tau,p$.

We first prove \eqref{hohoho}. Aiming to apply Lemma \ref{HT2-l:iterate}, for $j \in \{1,\ldots,k\}$ define $\beta_j =  j$ and
$f_j(\rho) =\|\D^j\tau\|_{L^\infty(\mathbb{B}^g(p,\rho),g)}$.
In order to prove an inequality of the form
\begin{equation}\label{e:iterate_hyp2}
\sum_{j=1}^k (R-\rho)^{j}f_j(\rho) \leq \epsilon \sum_{j=1}^k (R-\rho)^{j}f_j(R) +C_\epsilon (R-\rho)^{k+\alpha}[\D^{k}\tau]_{C^{\alpha}(\mathbb{B}^g(p,R),g)} + C_\epsilon \|\tau\|_{L^\infty(\mathbb{B}^g(p,R),g)},
\end{equation}
consider the following three cases. The constant $C$ will always be a generic constant $\geq 1$.\medskip\

\noindent \textit{Case 1}: $R - \rho < C_0^{-1}$. Fix any $j \in \{1,\ldots,k\}$ and write $\sigma = \D^{j-1}\tau$. Fix any $x_0 = (z_0,y_0) \in \mathbb{B}^g(p,\rho)$. Let $v$ run over a
$g$-orthonormal basis of tangent vectors to $B \times Y$ at $x_0$ which are all either horizontal or vertical. Let $\gamma(t)$ be the unique $\P$-geodesic with $\gamma(0) = x_0$ and $\dot\gamma(0) = v$, with $t\in(0,R-\rho)$. If $v$ is vertical, then $\gamma(t)$ is a $g_{Y,z_0}$-geodesic in $\{z_0\} \times Y$, well within the injectivity radius (its speed is uniformly comparable to $1$ thanks to \eqref{e:tilt1}). If $v$ is horizontal, then $\gamma(t)$ is an affine line in $B\times\{y_0\}$. Either way, all subsegments of $\gamma$ are (minimal) $\P$-geodesics.

Let $x_1 = \gamma(t_1)$ for $t_1 = \epsilon(R-\rho)$ and $\epsilon \in (0,1)$. Writing $\sigma(t)$ to denote the tensor field along $\gamma(t)$ induced by $\sigma$, writing $z(t) = {\rm pr}_B(\gamma(t))$, and using \eqref{thebs}, we get that
\begin{align}\label{antivaxx}
\sigma(x_0) - \P_{t_1,0}^\gamma(\sigma(x_1)) = \int_0^{t_1} \frac{d}{dt} \P^\gamma_{t_1-t,0}(\sigma(t_1-t)) \, dt
= -\int_0^{t_1} \P^\gamma_{t_1-t,0}\biggl[\frac{\nabla^{z(t_1-t)}}{dt}\sigma(t_1-t)\biggr]\,dt.
\end{align}
We can rewrite the last integrand as $(\nabla^{z_0}_{v}\sigma)(x_0)+ \psi(t_1-t)$, where for all $t \in [0,t_1]$,
\begin{equation}\psi(t) :=\P^\gamma_{t,0}\biggl[(\nabla^{z(t)}_{\dot\gamma(t)}\sigma)(\gamma(t))\biggr]-(\nabla^{z_0}_v\sigma)(x_0).\end{equation}
Using the definition of the $C^\alpha$ seminorm and the fact that $\gamma$ is a $\P$-geodesic, so that we have the crucial property $\P^\gamma_{t,0}(\dot\gamma(t)) = v$, we can estimate
\begin{equation}|\psi(t)|_{g(x_0)} \leq
[\D\sigma]_{C^{\alpha}(\mathbb{B}^g(x_0,t_1),g)} d^g(x_0, \gamma(t))^\alpha \leq C [\D\sigma]_{C^{\alpha}(\mathbb{B}^g(x_0,t_1),g)}t^\alpha. \end{equation}
If $j<k$, then we alternatively also have that
\begin{equation}\begin{split}
\psi(t) &=\int_0^{t}\frac{d}{ds}\P^\gamma_{s,0}
\biggl[(\nabla^{z(s)}_{\dot{\gamma}(s)}\sigma)(\gamma(s))\biggr] \, ds = \int_0^t \P^\gamma_{s,0}\biggl[(\D^{2}_{\dot\gamma,\dot\gamma}\sigma)(\gamma(s)) \biggr] \, ds,
\end{split}\end{equation}
using again that $\dot{\gamma}$ is $\P$-parallel along $\gamma$. Since the operator norm of $\P$ is under control thanks to the discussion in Section \ref{s:tridiagonal}, we conclude that
\begin{equation}|\psi(t)|_{g(x_0)} \leq C\|\D^{2}\sigma\|_{L^\infty(\mathbb{B}^g(x_0,t_1),g)}t.\end{equation}
Summarizing, for all $t \in [0,t_1]$,
\begin{equation}|\psi(t)|_{g(x_0)} \leq
\begin{cases}
C[\D\sigma]_{C^{\alpha}(\mathbb{B}^g(x_0,t_1),g)} t^\alpha &{\rm for\; all}\;\,j,\\
C\|\D^{2}\sigma\|_{L^\infty(\mathbb{B} ^g(x_0,t_1),g)}t &{\rm for \;all}\;\,j < k.
\end{cases}\end{equation}
Together with \eqref{antivaxx}, this leads to
\begin{equation}|(\nabla^{z_0}_v\sigma)(x_0)|_{g(x_0)}t_1 \leq |\sigma(x_0)|_{g(x_0)} +|\P^\gamma_{t_1,0}(\sigma(x_1))|_{g(x_0)}+
\begin{cases}
C [\D\sigma]_{C^{\alpha}(\mathbb{B}^g(p,\rho+t_1),g)} t_1^{1+\alpha}& {\rm for \; all}\,\;j,\\
C\|\D^{2}\sigma\|_{L^\infty(\mathbb{B}^g(p,\rho+t_1),g)} t_1^2&{\rm for \;all}\,\;j < k.
\end{cases}\end{equation}
Recall that $v$ is an arbitrary element of a $g$-orthonormal basis at $x_0$ and that $\sigma = \D^{j-1}\tau$. Taking the sup over all $x_0 \in \mathbb{B}^g(p,\rho)$, we deduce that
\begin{align}\label{e:gazillionth_random_inequality}
 f_j(\rho) t_1 &\leq Cf_{j-1}(\rho + t_1) + \begin{cases}
C[\D^{j}\tau]_{C^{\alpha}(\mathbb{B}^g(p,\rho+t_1),g)}t_1^{1+\alpha}& {\rm for \; all}\;j,\\
C f_{j+1}(\rho+t_1)t_1^2 &{\rm for\;all}\;j < k.
\end{cases}
\end{align}
Now recall that $t_1 = \epsilon(R-\rho)$, where $\epsilon \in (0, 1)$ is arbitrary. Working backwards from $j = k$ to $j = 1$, decreasing and renaming $\epsilon$  in each step, we deduce that
\begin{align}\label{e:gazillion-and-one}
\sum_{j=1}^k (R-\rho)^{j }f_j(\rho) \leq  \epsilon \sum_{j=1}^k (R-\rho)^j f_j(R)  + \epsilon (R-\rho)^{k+\alpha}[\D^{k}\tau]_{C^{\alpha}(\mathbb{B}^g(p,R),g)} + C_\epsilon \|\tau\|_{L^\infty(\mathbb{B}^g(p,R),g)}.
\end{align}
This is the desired inequality of type \eqref{e:iterate_hyp2}.
\medskip\

\noindent \textit{Case 2}: $R-\rho \in [C_0^{-1}, C_0]$. The proof of \eqref{e:gazillion-and-one} in this case can be  reduced to Case 1. Let $R' = \rho + \frac{1}{2C_0}$ and apply Case 1 to the pair of radii $(\rho, R')$ instead of $(\rho,R)$. In \eqref{e:gazillion-and-one} with $R$ replaced with $R'$, notice that trivially $\mathbb{B}^g(x,R') \subset \mathbb{B}^g(x,R)$ and $(R'-\rho)^{k+\alpha} \leq (R-\rho)^{k+\alpha}$, so in order to obtain \eqref{e:gazillion-and-one} for the pair $(\rho,R)$ we only need to observe that $(R'-\rho)^j \geq (2C_0^2)^{-j} (R - \rho)^j$ for $j = 1, \ldots,k$.     \medskip\

\noindent \textit{Case 3}: $R-\rho > C_0$. Using the same idea as in Case 1, we can prove that \eqref{e:gazillionth_random_inequality} still holds if $f_j(\rho) t_1$ is replaced by $ \|\D_{\mathbf{b}}(\D^{j-1}\tau)\|_{L^\infty(\mathbb{B}^g(p,\rho),g)}t_1$ on the left-hand side. (Here and below, subscripts ${\mathbf{b}}$ and $\mathbf{f}$ denote covariant derivatives in the horizontal and fiber directions, respectively.) This is because we are free to take $\gamma$ to be an affine line in $B\times\{y_0\}$ for any $y_0\in Y$, and all such paths are minimal $\P$-geodesics (whereas a $g_{Y,z_0}$-geodesic in $\{z_0\} \times Y$ will never be minimal on a time interval of size $[0,R-\rho]$). On the other hand, for all $x = (z,y)\in \mathbb{B}^g(p,\rho)$,
\begin{equation}\label{filotto}|(\D_{\mathbf{f}}(\D^{j-1}\tau))(x)|_{g(x)} \leq C[\D_{\mathbf{f}\cdots\mathbf{f}}^{k-j+1}(\D^{j-1}\tau)]_{C^{\alpha}(\{z\}\times Y,g_{Y,z})} \leq C[\D^{k}\tau]_{C^{\alpha}(\mathbb{B}^g(p,R),g)}.\end{equation}
The first inequality holds because $(\D_{\mathbf{f}}\sigma)(x)$ equals the usual $g_{Y,z}$-covariant derivative of $\sigma$ along $\{z\} \times Y$ for all tensors $\sigma$, and because we are then in a position to apply Lemma \ref{l:HT2:braindead}. The second inequality holds because $\{z\} \times Y \subset \mathbb{B}^g(p,R)$ (because $R - \rho > C_0$) and again because $\P$-transport along curves contained in $\{z\}\times Y$ coincides with the standard notion of $g_{Y,z}$-parallel transport.

Proceeding as in Case 1 (working backwards from $j = k$), we get
\begin{align}\
\sum_{j=1}^k (R-\rho)^{j }f_j(\rho) \leq  \epsilon \sum_{j=1}^k (R-\rho)^j f_j(R)  + C_\epsilon (R-\rho)^{k+\alpha}[\D^{k}\tau]_{C^{\alpha}(\mathbb{B}^g(p,R),g)} + C_\epsilon \|\tau\|_{L^\infty(\mathbb{B}^g(p,R),g)}.
\end{align}
The only difference is that, instead of an $\varepsilon$, we now get a large factor of $C_\epsilon$ in front of $[\D^k\tau]_{C^\alpha}$.

Since \eqref{e:iterate_hyp2} has been proved in all three cases,  \eqref{hohoho} now follows using Lemma \ref{HT2-l:iterate}.

To conclude the proof of Proposition \ref{l:higher-interpol}, we show that \eqref{hohohoho} can be formally deduced from \eqref{hohoho}. Indeed, pick $x,x'\in \mathbb{B}^g(p,\rho)$ joined by a horizontal or minimal vertical $\P$-geodesic and let $d=d^g(x,x')$. If $d\geq (R-\rho)/(4C_0)$, we bound the $C^\beta$ difference quotient for $\D^j{\tau}$ at $x,x'$ using the triangle inequality and the boundedness of the operator norm of $\P$ from Section \ref{s:tridiagonal}. This yields the bound
\begin{equation}
Cd^{-\beta}\|\D^{j}\tau\|_{L^\infty(\mathbb{B}^g(p,\rho),g)}\leq C(R-\rho)^{-\beta}\|\D^{j}\tau\|_{L^\infty(\mathbb{B}^g(p,\rho),g)},
\end{equation}
which can then be bounded using \eqref{hohoho}. If $d<(R-\rho)/(4C_0)$ and $j=k$, we bound the $C^\beta$ difference quotient for $\D^k\tau$ at $x,x'$ trivially by
\begin{equation}
d^{\alpha-\beta}[\D^{k}\tau]_{C^\alpha(\mathbb{B}^g(p,R),g)}\leq (R-\rho)^{\alpha-\beta}[\D^{k}\tau]_{C^\alpha(\mathbb{B}^g(p,R),g)}.
\end{equation}
And if $d<(R-\rho)/(4C_0)$ and $j<k$, we apply Lemma \ref{kathoum} on $\mathbb{B}^g(x,d)$ and bound the $C^\beta$ difference quotient for $\D^j\tau$ at $x,x'$ by
\begin{equation}
Cd^{1-\beta}\|\D^{j+1}\tau\|_{L^\infty(\mathbb{B}^g(x,2C_0d),g)}\leq C(R-\rho)^{1-\beta}\|\D^{j+1}\tau\|_{L^\infty(\mathbb{B}^g(p,\rho + \frac{1}{2}(R-\rho)),g)},
\end{equation}
which can again be bounded using \eqref{hohoho}. This establishes \eqref{hohohoho}, and completes the proof.
\end{proof}

\subsection{From H\"older seminorm bounds to decay}\label{s:BIC}

We will now prove Theorem \ref{prop55},  the main result of Section \ref{s:adia}. This will be a cornerstone of all of our later arguments.

Assume we are in the same setting as in Section \ref{s:interpol}.
At each point $x=(z,y)\in B\times Y$ we have a splitting $T_x(B\times Y)=T_zB\oplus T_yY$, and we will denote by $\mathbf{f}$ any tangent vector in the second summand and by $\mathbf{b}$ any tangent vector in the first summand. If $Z$ is any vector in $\mathbb{C}^m$ (either real or complex), then we will denote by $\mathbf{Z}$ its trivial extension to a vector field (real or complex) on $B\times Y$.

One more piece of notation. First, we may assume without loss that $\int_{\{z\}\times Y}{\omega}_{F,z}^n=1$ for all $z\in B$. For a function $f$ on $B\times Y$ we will denote by $\underline{f}$ the function on $B$ given by
\begin{equation}
\underline{f}(z)=\int_{\{z\}\times Y}f\omega_{F,z}^n,
\end{equation}
i.e., the fiber average of $f$ with respect to the varying Calabi-Yau volume forms $\omega_{F,z}$ on $\{z\}\times Y$.

\begin{theorem}\label{prop55}
Let $0<\delta_t\leq 1$ be given for all $t \geq 0$. Define product metrics $g_t$ on $B\times Y$ by \eqref{shrpr}.
Given $k\in\mathbb{N}$, $0<\alpha<1$ and radii $0 <\rho<R<1$ with $\rho\geq C_0\sup_{t\geq 0} \delta_t$, there is a $C = C(k,\alpha,\rho,R)$ such that for all $(1,1)$-forms $\eta = i\partial\overline\partial\vp$ with $\underline{\vp}=0$ on $B \times Y$ and for all $0\leq j\leq k$ and $t\geq 0$,
\begin{equation}\label{estima}
\|\D^j\eta\|_{L^\infty(B_\rho(0) \times Y,g_t)}\leq C \delta_t^{k+\alpha-j}[\D^k\eta]_{C^\alpha(B_R(0)\times Y,g_t)}.
\end{equation}
Also, for all $0<\beta<1$ such that $j+\beta<k+\alpha$ we have that
\begin{equation}\label{estima2}
[\D^j\eta]_{C^\beta(B_\rho(0)\times Y,g_t)}\leq C \delta_t^{k+\alpha-j-\beta}[\D^k\eta]_{C^\alpha(B_R(0)\times Y,g_t)}.
\end{equation}
\end{theorem}

Here, $B_\rho(0), B_R(0) \subset B$ simply denote standard Euclidean balls in $\mathbb{C}^m$ and the Hölder seminorms in \eqref{estima}--\eqref{estima2} are understood in the sense of Definition \ref{def:holderdef} and the paragraph after it.  In particular,  we are using here that $B_\rho(0) \times Y = \mathbb{B}^{g_t}(p,\rho)$ and $B_R(0) \times Y= \mathbb{B}^{g_t}(p,R)$ for all $p \in \{0\}\times Y$.

Theorem \ref{prop55} is a quantitative and much improved version of \cite[Prop 5.5]{HT2}. Indeed, the proof also relies on our crucial Lemma \ref{l:HT2:braindead} and its variants in Lemmas \ref{l:braindead2}--\ref{l:braindead3} below, and on a commutation estimate from the proof of \cite[Prop 5.5]{HT2} that we include and improve here in Lemma \ref{l:commutators}.  (We defer these lemmas until after the proof of the theorem because they are a bit technical and the proofs, while similar to the ones in \cite{HT2}, need to be given in full.) The proof also uses Proposition \ref{l:higher-interpol}.  The idea that uniform bounds on the highest derivatives in a shrinking H\"older seminorm should imply quantitative decay for the lower order derivatives came to us after reading the appendix of \cite{DJZ} and \cite{JS1}, where weaker and more special results in this direction are proved with different techniques.

\begin{proof}
For ease of notation we will write $B_r=B_r(0) \subset \C^m$. First suppose we have shown that for all $k,\alpha$ and $0<\rho<R$ with $\rho\geq C_0\sup_{t \geq 0} \delta_t$ there is a $C$ such that for all $\eta$ as above and all $t \geq 0$,
\begin{equation}\label{estimatio}
\|\eta\|_{L^\infty(B_\rho\times Y,g_t)}\leq C \delta_t^{k+\alpha}[\D^k\eta]_{C^\alpha(B_R\times Y,g_t)}.
\end{equation}
Then we claim that \eqref{estima} and \eqref{estima2} follow. Indeed, \eqref{estimatio} clearly implies \eqref{estima} for $j=0$. For $j>0$, note first that by assumption we have $B_r\times Y=\mathbb{B}^{g_t}(p,r)$ for any $p\in \{0\}\times Y$ and any $r\geq \rho$. We then use our interpolation inequality, Proposition \ref{l:higher-interpol}, in the form \eqref{sanctus}. We apply \eqref{sanctus} with $\rho$ replaced by $\ti{\rho}=R-\min(\delta_t, R-\rho)$, so that $R-\ti{\rho}\leq \delta_t$ and $\rho\leq\ti{\rho}$. Thus, using also \eqref{estimatio},
\begin{equation}\begin{split}
(R-\ti{\rho})^j\|\D^j\eta\|_{L^\infty(B_{\ti{\rho}}\times Y,g_t)}&\leq C (R-\ti{\rho})^{k+\alpha}[\D^k\eta]_{C^\alpha(B_R\times Y,g_t)}+C\|\eta\|_{L^\infty(B_R\times Y, g_t)}\\
&\leq C\delta_t^{k+\alpha}[\D^k\eta]_{C^\alpha(B_R\times Y,g_t)}.
\end{split}\end{equation}
This implies \eqref{estima} by considering the two cases $\delta_t\leq R-\rho$ and $\delta_t > R-\rho$, since the constant $C$ in \eqref{estima} is allowed to depend on $R-\rho$. As for \eqref{estima2},  the interpolation inequality \eqref{benedictus} yields
\begin{equation}\begin{split}
(R-\ti{\rho})^{j+\beta}[\D^j\eta]_{C^\beta(B_{\ti{\rho}}\times Y,g_t)}&\leq C (R-\ti{\rho})^{k+\alpha}[\D^k\eta]_{C^\alpha(B_R\times Y,g_t)}+C\|\eta\|_{L^\infty(B_R\times Y, g_t)}\\
&\leq C\delta_t^{k+\alpha}[\D^k\eta]_{C^\alpha(B_R\times Y,g_t)},
\end{split}\end{equation}
which proves \eqref{estima2} by the same reasoning as above.

We now establish \eqref{estimatio}.  To do so, we begin by decomposing $\eta = \eta_{\mathbf{bb}} + \eta_{\mathbf{bf}} + \eta_{\mathbf{ff}}$ according to
\begin{equation}\label{porra}
\Lambda^2(B \times Y) = {\rm pr}_B^*\Lambda^2B \oplus ({\rm pr}_B^*\Lambda^1 B \otimes {\rm pr}_Y^*\Lambda^1 Y) \oplus {\rm pr}_Y^*\Lambda^2 Y.
\end{equation}
Thus,  $\eta_{\mathbf{bb}}, \eta_{\mathbf{bf}},\eta_{\mathbf{ff}}$ are globally defined sections of the corresponding vector bundles over $B \times Y$,  each of which inherits a fiber metric from any given Riemannian metric on $B \times Y$ and inherits a connection from $\mathbb{D}$. Moreover,  \eqref{estimatio} is equivalent to the following for all $z\in B_\rho$ (recall $g = g_{\mathbb{C}^m} + g_{Y,0}$):
\begin{eqnarray}
\label{bbfriend}
\sup_{\{z\}\times Y}|\eta_{\mathbf{bb}}|_{g} \leq C\delta_t^{k+\alpha}[\D^k\eta]_{C^\alpha(B_R\times Y,g_t)},\\
\label{bffriend}
\sup_{\{z\}\times Y}|\eta_{\mathbf{bf}}|_{g} \leq C\delta_t^{k+1+\alpha}[\D^k\eta]_{C^\alpha(B_R\times Y,g_t)},\\
\label{fffriend}
\sup_{\{z\}\times Y}|\eta_{\mathbf{ff}}|_{g} \leq C\delta_t^{k+2+\alpha}[\D^k\eta]_{C^\alpha(B_R\times Y,g_t)}.
\end{eqnarray}

We first establish \eqref{fffriend} by applying Lemma \ref{l:braindead2}.  Indeed, by the lemma,
\begin{equation}\label{sagredo}\begin{split}
\sup_{\{z\}\times Y}|\eta_{\mathbf{ff}}|_{g}
&\leq C  \sup_{\{z\}\times Y} |\eta|_{\{z\}\times Y}|_{g_{Y,z}}\\
&\leq C [(\nabla^{g_{Y,z}})^k(\eta|_{\{z\}\times Y})]_{C^\alpha(\{z\}\times Y, g_{Y,z})}\\
&= C\delta_t^{k+2+\alpha}  [(\nabla^{g_{Y,z}})^k(\eta|_{\{z\}\times Y})]_{C^\alpha(\{z\}\times Y, \delta_t^2 g_{Y,z})}.
\end{split}
\end{equation}
We now use the fact that, by definition,
\begin{equation}(\nabla^{g_{Y,z}})^k (\eta|_{\{z\}\times Y}) = \nabla^{z,k}_{\mathbf{f\cdots f}}(\eta_{\mathbf{ff}})=\D^k_{\mathbf{f\cdots f}}(\eta_{\mathbf{ff}}),
\end{equation}
and that $g_{Y,z}$-parallel transport on $T^\bullet Y$ is equal to $\mathbb{P}$-transport on ${\rm pr}_Y^*T^\bullet Y$ in the vertical directions.  Up to a uniform constant we can also replace the fiber metric induced by $\delta_t^2 g_{Y,z}$ on $T^\bullet Y$ by the one induced by $g_t = g_{\mathbb{C}^m} + \delta_t^2 g_{Y,0}$ on ${\rm pr}_Y^*T^\bullet Y$.  This allows us to further estimate
\begin{equation}
\begin{split}
\sup_{\{z\}\times Y}|\eta_{\mathbf{ff}}|_{g}
&\leq
C\delta_t^{k+2+\alpha}[\D^k_{\mathbf{f}\cdots\mathbf{f}}(\eta_{\mathbf{ff}})]_{C^\alpha(B_\rho\times Y,g_t)}\\
&\leq
C\delta_t^{k+2+\alpha}[\D^k_{\mathbf{f}\cdots\mathbf{f}}\eta]_{C^\alpha(B_\rho\times Y,g_t)}\\
&\leq
C\delta_t^{k+2+\alpha}[\D^k\eta]_{C^\alpha(B_\rho\times Y,g_t)},
\end{split}\end{equation}
which implies the claim. The second-to-last inequality is true because $\mathbb{D}$ and $\P$ preserve the three sub-bundles of \eqref{porra} and because these subbundles are orthogonal with respect to $g_t$.  The last inequality is true because restriction of contravariant tensors to subspaces is norm nonincreasing.

To establish \eqref{bffriend}, we begin by noting that any fixed vector $Z \in (\mathbb{C}^m)^{1,0}$ may trivially be viewed as a section $\mathbf{Z}$ of $T(B \times Y) \otimes_{\mathbb{R}} \mathbb{C}$, which is however not usually of type $(1,0)$, let alone holomorphic, with respect to the fibered complex structure $J$.  Because $\eta$ is real, we may then estimate
\begin{equation}
\begin{split}
\sup_{\{z\}\times Y} |\eta_{\mathbf{bf}}|_{g} &\leq C \sup_{\{z\}\times Y} \sup_{Z \in (\mathbb{C}^m)^{1,0}:|Z|=1} |(\mathbf{Z} \,\lrcorner\,\eta)|_{\{z\}\times Y}|_{g_{Y,0}}\\
&\leq C \sup_{\{z\}\times Y} \sup_{Z \in (\mathbb{C}^m)^{1,0}:|Z|=1} |(\mathbf{Z}\,\lrcorner\,\eta)|_{\{z\}\times Y}|_{g_{Y,z}}.
\end{split}
\end{equation}
If the complex structure $J$ on $B \times Y$ was a product,  then $\mathbf{Z}$ would indeed be a holomorphic $(1,0)$-vector field with respect to $J$ and we would have $(\mathbf{Z} \,\lrcorner\,\eta)|_{\{z\} \times Y} = \db(\mathbf{Z}(\vp)|_{\{z\}\times Y})$. This would in turn allow us to apply Lemma \ref{l:braindead3} to complete the proof of \eqref{bffriend}, in analogy with how Lemma \ref{l:braindead2} was used to complete the proof of \eqref{fffriend} above.  In general, this motivates considering the difference
\begin{equation}\label{splitso}\begin{split}
\text{Diff} :=&\; (\mathbf{Z}\,\lrcorner\,\eta)|_{{\{z\}\times Y}}-\db(\mathbf{Z}(\vp)|_{\{z\}\times Y}) \\
=&\; L_1 \circledast D(\vp|_{\{z\}\times Y}) + L_2 \circledast D^2(\vp|_{\{z\}\times Y})
\end{split}
\end{equation}
according to Lemma \ref{l:commutators}.  Using Lemma \ref{l:braindead3}, we obtain that
\begin{equation}\label{lkjh}\begin{split}
\sup_{\{z\}\times Y}|\mathbf{Z}\,\lrcorner\,\eta|_{g_{Y,z}}&\leq \sup_{\{z\}\times Y} |{\rm Diff}|_{g_{Y,z}}+ \sup_{\{z\}\times Y}|\db(\mathbf{Z}(\vp)|_{\{z\}\times Y})|_{g_{Y,z}}\\
&\leq  \sup_{\{z\}\times Y} |{\rm Diff}|_{g_{Y,z}}+C [(\nabla^{g_{Y,z}})^k(\db(\mathbf{Z}(\vp)|_{\{z\}\times Y}))]_{C^\alpha(\{z\}\times Y,g_{Y,z})}\\
&\leq  \|{\rm Diff}\|_{C^{k,\alpha}(\{z\}\times Y,g_{Y,z})}+C [(\nabla^{g_{Y,z}})^k((\mathbf{Z}\,\lrcorner\,\eta)|_{{\{z\}\times Y}})]_{C^\alpha(\{z\}\times Y,g_{Y,z})}\\
&\leq \|{\rm Diff}\|_{C^{k,\alpha}(\{z\}\times Y,g_{Y,z})}+C\delta_t^{k+1+\alpha}[(\nabla^{g_{Y,z}})^k ((\mathbf{Z}\,\lrcorner\,\eta)|_{{\{z\}\times Y}})]_{C^\alpha(\{z\}\times Y,\delta_t^2 g_{Y,z})}.
\end{split}\end{equation}
By applying the same logic as after \eqref{sagredo} to the second term, we obtain that
\begin{equation}\begin{split}
\sup_{\{z\}\times Y}|\mathbf{Z}\,\lrcorner\,\eta|_{g_{Y,z}}&\leq  \|{\rm Diff}\|_{C^{k,\alpha}(\{z\}\times Y,g_{Y,z})}+C\delta_t^{k+1+\alpha}[\D^k\eta]_{C^\alpha(B_\rho\times Y,g_t)}.
\end{split}\end{equation}
Applying Lemma \ref{l:commutators}, i.e., the second line of \eqref{splitso}, and Lemma \ref{l:braindead2} to the first term yields
\begin{equation}\begin{split}
\sup_{\{z\}\times Y}|\mathbf{Z}\,\lrcorner\,\eta|_{g_{Y,z}}&\leq C[(\nabla^{g_{Y,z}})^k ( i\partial\overline\partial(\vp|_{\{z\}\times Y}))]_{C^\alpha(\{z\}\times Y, g_{Y,z})} + C\delta_t^{k+1+\alpha}[\D^k\eta]_{C^\alpha(B_\rho \times Y,g_t)}\\
&\leq C\delta_t^{k+2+\alpha}[\D^k\eta]_{C^\alpha(B_\rho \times Y,g_t)}+C\delta_t^{k+1+\alpha}[\D^k\eta]_{C^\alpha(B_\rho \times Y,g_t)}\\
&\leq C\delta_t^{k+1+\alpha}[\D^k\eta]_{C^\alpha(B_\rho\times Y,g_t)}.
\end{split}\end{equation}

It remains to prove \eqref{bbfriend}. For any two vectors $Z,W\in (\mathbb{C}^m)^{1,0}$ of length $1$,  let $\mathbf{Z},\mathbf{W}$ again denote their trivial extensions as complexified vector fields on $B \times Y$. Because $\eta$ is real, it suffices to bound $\eta(\mathbf{Z},\mathbf{W})$ and $\eta(\mathbf{Z},\overline{\mathbf{W}})$.  Since $\eta$ is of $J$-type $(1,1)$, it suffices to bound
$\eta(\mathbf{Z}^{1,0}, \mathbf{W}^{0,1})$ and $\eta(\mathbf{Z}^{1,0}, {\overline{\mathbf{W}}}^{0,1})$, where the superscripts indicate $J$-type.  From \eqref{e:kohelet5} in the proof of Lemma \ref{l:commutators}, schematically,
\begin{equation}\label{DL/AL}
\mathbf{Z}^{1,0} = \frac{\partial}{\partial z} + F \frac{\partial}{\partial y},  \quad \mathbf{W}^{0,1} = G\frac{\partial}{\partial\overline{y}} ,\quad {\overline{\mathbf{W}}}^{0,1} = \frac{\partial}{\partial\overline{z}} + H \frac{\partial}{\partial\overline{y}},
\end{equation}
in any local $J$-holomorphic product coordinate system $(z,y)$, where $F,G,H$ are smooth functions. We will spell out explicitly the bound for $\eta(\mathbf{Z}^{1,0}, {\overline{\mathbf{W}}}^{0,1})$, since it is the more complicated one, and the bound for $\eta(\mathbf{Z}^{1,0}, \mathbf{W}^{0,1})$ then follows along similar lines.

The first main point will be to use the assumption $\underline{\vp}=0$ to show that
\begin{equation}\label{fiberavg}
\sup_{B_R}|\underline{\eta(\mathbf{Z}^{1,0}, {\overline{\mathbf{W}}}^{0,1})}|\leq C\delta_t^{k+1+\alpha}[\D^k\ddbar\vp]_{C^\alpha(B_R\times Y,g_t)}.
\end{equation}
To prove this, we first claim that
\begin{equation}\label{drano_brano}
|(\ddbar\vp)(\mathbf{Z}^{1,0},{\overline{\mathbf{W}}}^{0,1}) - \mathbf{Z}^{1,0}({\overline{\mathbf{W}}}^{0,1}(\vp))| \leq C\delta_t^{k+1+\alpha}[\D^k\ddbar\vp]_{C^\alpha(B_R\times Y,g_t)},\end{equation}
To see this, we insert \eqref{DL/AL} into the quantity on the LHS of \eqref{drano_brano}; the main (base-base) terms of the difference obviously cancel out as expected, leaving us with every possible kind of fiber, fiber-fiber, and base-fiber derivative of $\vp$ as the error terms (i.e., any possible combination of $(1,0)$ indices and $(0,1)$ indices, or equivalently, any possible combination of real indices). The fiber and fiber-fiber ones are bounded by $C\delta_t^{k+2+\alpha}[\D^k\ddbar\vp]_{C^\alpha(B_R\times Y,g_t)}$ by using Lemma \ref{l:braindead2} as in \eqref{lkjh}. The base-fiber ones are dealt with by first applying Lemma \ref{l:braindead3} to $\psi = \mathbf{Z}(\vp)$
to get
\begin{equation}\label{plethysm}
\sup_{\{z\}\times Y}|\nabla_{\mathbf{f}}\left(\mathbf{Z}(\vp)|_{\{z\}\times Y}\right)|_{g_{Y,z}}\leq C [\nabla^{z,k}_{\mathbf{f\cdots f}}\db\left(\mathbf{Z}(\vp)|_{\{z\}\times Y}\right)]_{C^\alpha(\{z\}\times Y,g_{Y,z})}
\end{equation}
and then bounding this by $C\delta_t^{k+1+\alpha}[\D^k\ddbar\vp]_{C^\alpha(B_R\times Y,g_t)}$ as in \eqref{lkjh} above. This establishes \eqref{drano_brano}.

Taking then the fiber average of \eqref{drano_brano}, we are left with bounding the fiber average of $\mathbf{Z}^{1,0}({\overline{\mathbf{W}}}^{0,1}(\vp))$. For this, we apply $\mathbf{Z}^{1,0}({\overline{\mathbf{W}}}^{0,1}(\cdot))$ to the relation
\begin{equation}\label{rel}0=\underline{\vp}(z)=\int_{\{z\}\times Y}\vp\omega_{F,z}^n,\end{equation}
which allows us to bound $\underline{\mathbf{Z}^{1,0}({\overline{\mathbf{W}}}^{0,1}(\vp))}$ in terms of the fiberwise $L^\infty$ norm of $\vp$ and of $\mathbf{Z}^{1,0}(\vp),{\overline{\mathbf{W}}}^{0,1}(\vp)$. Since $\underline{\vp}=0$, we can bound
\begin{equation}\begin{split}
\sup_{\{z\}\times Y}|\vp|&\leq C\sup_{\{z\}\times Y}|\nabla_{\mathbf{f}}\vp|_{\{z\}\times Y}|_{g_{Y,z}}\leq C[\nabla^k_{\mathbf{f\cdots f}}\ddbar\vp|_{\{z\}\times Y}]_{C^\alpha(\{z\}\times Y, g_{Y,z})}\\
&\leq C\delta_t^{k+2+\alpha}[\D^k_{\mathbf{f\cdots f}}\ddbar\vp|_{\{z\}\times Y}]_{C^\alpha(\{z\}\times Y,g_t)}
\leq C\delta_t^{k+2+\alpha}[\D^k\eta]_{C^\alpha(B_R\times Y,g_t)}
\end{split}\end{equation}
using Lemma \ref{l:braindead2}. As for $\mathbf{Z}^{1,0}(\vp)$, applying $\mathbf{Z}^{1,0}$ to \eqref{rel} allows us to bound $\underline{\mathbf{Z}^{1,0}(\vp)}$ by the fiberwise $L^\infty$ norm of $\vp$, and so by $C\delta_t^{k+2+\alpha}[\D^k\eta]_{C^\alpha(B_R\times Y,g_t)}$.

We can then bound the fiberwise $L^\infty$ norm of $\mathbf{Z}^{1,0}(\vp)$ in terms of $\underline{\mathbf{Z}^{1,0}(\vp)}$ and of the $L^\infty$ norm of the fiberwise gradient $\nabla_{\mathbf{f}}\left(\mathbf{Z}^{1,0}(\vp)|_{\{z\}\times Y}\right)$, which is bounded by $C\delta_t^{k+1+\alpha}[\D^k\ddbar\vp]_{C^\alpha(B_R\times Y,g_t)}$ as in \eqref{plethysm} and \eqref{lkjh} above. Putting all of these together completes the proof of \eqref{fiberavg}.

Now that \eqref{fiberavg} is proved, we can use it to prove \eqref{bbfriend}. Assume first that $k=0$ and bound the $L^\infty$ norm of $\eta(\mathbf{Z}^{1,0}, {\overline{\mathbf{W}}}^{0,1})$ on the fiber $\{z\}\times Y$ in terms of its fiberwise $C^\alpha$ seminorm and its fiberwise average
\begin{equation}\begin{split}
\sup_{\{z\}\times Y}|\eta(\mathbf{Z}^{1,0}, {\overline{\mathbf{W}}}^{0,1})|_{g_{Y,z}}&\leq \sup_B|\underline{\eta(\mathbf{Z}^{1,0}, {\overline{\mathbf{W}}}^{0,1})}|
 + C[\eta(\mathbf{Z}^{1,0}, {\overline{\mathbf{W}}}^{0,1})|_{\{z\}\times Y}]_{C^\alpha(\{z\}\times Y,g_{Y,z})}\\
&\leq  C\delta_t^{1+\alpha}[\ddbar\vp]_{C^\alpha(B_R\times Y,g_t)} +  C\delta_t^\alpha
[\eta]_{C^\alpha(B_R\times Y,g_t)}\leq  C\delta_t^\alpha
[\eta]_{C^\alpha(B_R\times Y,g_t)},
\end{split}\end{equation}
using \eqref{fiberavg}.
If $k \geq 1$, we instead first bound the fiberwise $L^\infty$ norm of $\eta(\mathbf{Z}^{1,0}, {\overline{\mathbf{W}}}^{0,1})$ in terms of the $L^\infty$ norm of its fiberwise gradient and its fiberwise average, and then invoke Lemma \ref{l:HT2:braindead}:
\begin{align}
\sup_{\{z\}\times Y}|\eta(\mathbf{Z}^{1,0}, {\overline{\mathbf{W}}}^{0,1})|_{g_{Y,z}}&\leq \sup_{B_R}|\underline{\eta(\mathbf{Z}^{1,0}, {\overline{\mathbf{W}}}^{0,1})}|+ \sup_{\{z\}\times Y} |\nabla_{\mathbf{f}}\eta(\mathbf{Z}^{1,0}, {\overline{\mathbf{W}}}^{0,1})|_{\{z\}\times Y}|_{g_{Y,z}}\\
&\leq C\delta_t^{k+1+\alpha}[\D^k\ddbar\vp]_{C^\alpha(B_R\times Y,g_t)} + C[\nabla^k_{\mathbf{f\cdots f}}\eta(\mathbf{Z}^{1,0}, {\overline{\mathbf{W}}}^{0,1})|_{\{z\}\times Y}]_{C^\alpha(\{z\}\times Y, g_{Y,z})} \\
&\leq C\delta_t^{k+1+\alpha}[\D^k\eta]_{C^\alpha(B_R\times Y,g_t)} + C\delta_t^{k+\alpha}[\D^k_{\mathbf{f\cdots f}}\eta(\mathbf{Z}^{1,0}, {\overline{\mathbf{W}}}^{0,1})|_{\{z\}\times Y}]_{C^\alpha(\{z\}\times Y,g_t)}\\
&\leq C\delta_t^{k+\alpha}[\D^k\eta]_{C^\alpha(B_R\times Y,g_t)},
\end{align}
which concludes the proof of the theorem.
\end{proof}

The following lemma and its proof were used in the proof of Theorem \ref{prop55}. All of this is essentially contained in the proof of \cite[Prop 5.5]{HT2}, but we make the argument explicit here for convenience.

\begin{lemma}\label{l:commutators} For any fixed type $(1,0)$ vector $Z$ on $\C^m$ and its trivial extension $\mathbf{Z}$ to $\C^m\times Y$,
\begin{equation}
(\mathbf{Z} \,\lrcorner\,i\partial\overline\partial\vp)|_{\{z\}\times Y} - \overline\partial(\mathbf{Z}(\vp)|_{\{z\}\times Y}) = L_1 \circledast D(\vp|_{\{z\}\times Y}) + L_2 \circledast D^2(\vp|_{\{z\}\times Y}).
\end{equation}
Here $D,D^2$ are the gradient and Hessian with respect to any fixed local coordinate system on $Y$, and $L_1,L_2$ are fixed smooth local coefficient fields depending only on the coordinate system.
\end{lemma}

\begin{proof}
We may assume that $Z=Z_j$, the trivial extension to $B\times Y$ of the $j$-th coordinate $(1,0)$ vector field on $\C^m$.
Fix any $y \in Y$ and a $J_Y$-holomorphic chart $(y^1, \ldots, y^n)$ on $\{z\}\times Y$ near $y$. Extend the functions $y^i$ to $J$-holomorphic functions $\hat{y}^i$ on a neighborhood of $(z,y)$ in the total space. The key property of the ``fibered'' chart $(z^1, \ldots, z^m, \hat{y}^1, \ldots, \hat{y}^n)$ is that
\begin{align}\label{e:qoelet}
\frac{\partial \vp}{\partial\hat{y}^p}\biggr|_{\{z\}\times Y} = \frac{\partial (\vp|_{\{z\}\times Y})}{\partial y^p}
\end{align}
for all local functions $\vp$.  Thus, expanding $\eta= i\partial\ov{\partial}\vp$ in terms of this chart,
\begin{align}
(Z_j \,\lrcorner\, \eta)|_{\{z\}\times Y}&= i \sum_{p=1}^n \frac{\partial}{\partial \ov{y}^p}\biggl(\frac{\partial\vp}{\partial z^j}\biggr|_{\{z\}\times Y}\biggr) d\ov{y}^p \label{e:kohelet2} \\&+i \sum_{p,q=1}^n \biggl(\biggl[\frac{\partial^2(\vp|_{\{z\}\times Y})}{\partial y^q \partial \ov{y}^p } Z_j(\hat{y}^q) \biggr]d\ov{y}^p - \biggl[\frac{\partial^2(\vp|_{\{z\}\times Y})}{\partial y^p \partial\ov{y}^q} Z_j(\ov{\hat{y}}^q)\biggr] dy^{p}\biggr). \label{e:kohelet3}
\end{align}
This is straightforward to check.

Since $J$ is fibered, the $(0,1)$-part of $Z_j$ with respect to $J$ is a section of $TY \otimes \C$, and \eqref{e:qoelet} says that $TY \otimes \C$ is generated by the vector fields $\partial/\partial \hat{y}^p$ and their complex conjugates. Thus, if we expand $Z_j$ in terms of our chart, there will be no $\partial/\partial\ov{z}^k$ components. Since $dz^k(Z_j) = \delta^k_j$, it follows that
\begin{align}\label{e:kohelet5}
Z_j = \frac{\partial}{\partial z^j} - \sum_{q=1}^n \biggl(a_{j}^q \frac{\partial}{\partial \hat{y}^q} + b_{j}^q\frac{\partial}{\partial\ov{\hat{y}}^q}\biggr),
\end{align}
for some local smooth functions $a_{j}^q, b_{j}^q.$ Thus,
\begin{align}\label{e:kohelet6}
\frac{\partial}{\partial \ov{y}^p}\biggl(\frac{\partial\vp}{\partial z^j}\biggr|_{\{z\}\times Y}\biggr) - \frac{\partial}{\partial \ov{y}^p}(Z_j(\vp)|_{\{z\}\times Y})
= \sum_{q=1}^n \biggl(\frac{\partial a^q_{j}}{\partial \ov{\hat{y}}^p} \frac{\partial(\vp|_{\{z\}\times Y})}{\partial y^q}
& + \frac{\partial b^q_{j}}{\partial \ov{\hat{y}}^p} \frac{\partial(\vp|_{\{z\}\times Y})}{\partial \ov{y}^q} \\
\label{e:kohelet7}
+\,a^q_{j}\frac{\partial^2(\vp|_{\{z\}\times Y})}{\partial \ov{y}^p\partial y^q}
&+ b^q_{j}\frac{\partial^2(\vp|_{\{z\}\times Y})}{\partial \ov{y}^p\partial \ov{y}^q}\biggr).
\end{align}
\end{proof}

Lastly, in the proof of Theorem \ref{prop55} we also used the following two variants of Lemma \ref{l:HT2:braindead}. Their proofs are very similar to each other and to the proof of Lemma \ref{l:HT2:braindead} in \cite{HT2}; see \cite[Lemma 3.3]{HT2}.

\begin{lemma}\label{l:braindead2}
Let $(Y,g)$ be a compact K\"ahler manifold.
Then for all $k \in \N$, $\alpha \in (0,1)$ there exists a constant $C_k = C_k(Y,g,\alpha)$ such that for all $\vp\in C^{k+2,\alpha}(Y)$,
\begin{align}\label{e:idioticestimate2}
\sum_{i=-1}^k\|\nabla^{i+2}\vp\|_{L^\infty(Y)} \leq C_k[\nabla^k\de\db\vp]_{C^{\alpha}(Y)}.
\end{align}
\end{lemma}

\begin{proof}
This is proved along the lines of Lemma \ref{l:HT2:braindead}.
We may assume without loss that $\vp$ has average zero. Suppose the lemma fails for some $k,\alpha$, so there exists a sequence $\vp_i \in C^{k+2,\alpha}(Y)$ with zero average with $[\nabla^k\de\db\vp_i]_{C^{\alpha}(Y)} < \frac{1}{i}\sum_{j=-1}^k\|\nabla^{j+2}\vp_i\|_{L^\infty(Y)}$. Dividing $\vp_i$ by $\sum_{j=-1}^k\|\nabla^{j+2}\vp_i\|_{L^\infty(Y)} > 0$, we may further assume that $\sum_{j=-1}^k\|\nabla^{j+2}\vp_i\|_{L^\infty(Y)} = 1$. In particular we obtain a uniform gradient bound for $\vp_i$, and since it has fiberwise average zero, we obtain from this that $\|\vp_i\|_{L^\infty(Y)}\leq C$ for all $i$. 



We thus have uniform bounds on $[\nabla^k\Delta \vp_i]_{C^{\alpha}(Y)}$ and $\|\vp_i\|_{L^\infty(Y)}$, and standard elliptic estimates give us a uniform $C^{k+2,\alpha}$ bound on $\vp_i$. Passing to a subsequence if needed, we may now assume that $\vp_i$ converges to some $\vp \in C^{k+2,\alpha}(Y)$ in the $C^{k+2,\beta}$ topology for every $\beta < \alpha$. By construction, this limit $\vp$ satisfies the following properties: $[\nabla^k\de\db\vp]_{C^{\alpha}(Y)} = 0$, $\sum_{j=-1}^k\|\nabla^{j+2}\vp\|_{L^\infty(Y)} = 1$, and $\vp$ has average zero. The first property implies that $\vp$ is smooth with $\nabla^{k+1}\de\db\vp = 0$.  We claim that this implies that $\nabla\de\db\vp=0$. This is clear when $k=0,$ while for $k\geq 1$, relying crucially on the fact that $\partial Y = \emptyset$, we have
\begin{equation}\int_Y |\nabla^k\de\db\vp|^2 = \int_Y \langle\nabla^{k-1}\de\db\vp,\nabla^*\nabla^k\de\db\vp\rangle  = 0,\end{equation}
and working backwards we get $\nabla^j\de\db\vp=0$ for $1\leq j\leq k$, proving our claim. But the claim then gives
\begin{equation}\int_Y |\de\db\vp|^2 = \int_Y \langle\db\vp,\de^*\de\db\vp\rangle  = 0,\end{equation}
so $\de\db\vp=0$ which implies that $\vp=0$, a contradiction.
\end{proof}

\begin{lemma}\label{l:braindead3}
Let $(Y,g)$ be a compact K\"ahler manifold.
Then for all $k \in \N$, $\alpha \in (0,1)$ there exists a constant $C_k = C_k(Y,g,\alpha)$ such that for all $\psi\in C^{k+1,\alpha}(Y,\mathbb{C})$,
\begin{align}\label{e:idioticestimate3}
\sum_{i=0}^k\|\nabla^{i+1}\psi\|_{L^\infty(Y)} \leq C_k[\nabla^k\db\psi]_{C^{\alpha}(Y)}.
\end{align}
\end{lemma}

\begin{proof}
Again, the proof is a simple modification of the proof of Lemma \ref{l:HT2:braindead}.
We may assume without loss that $\psi$ has average zero. Suppose the lemma fails for some $k,\alpha$, so there exists a sequence $\psi_i \in C^{k+1,\alpha}(Y)$ with zero average with $[\nabla^k\db\psi_i]_{C^{\alpha}(Y)} < \frac{1}{i}\sum_{j=0}^k\|\nabla^{j+1}\psi_i\|_{L^\infty(Y)}$. Dividing $\psi_i$ by $\sum_{j=0}^k\|\nabla^{j+1}\psi_i\|_{L^\infty(Y)} > 0$, we may further assume that $\sum_{j=0}^k\|\nabla^{j+1}\psi_i\|_{L^\infty(Y)} = 1$. In particular we obtain a uniform gradient bound for $\psi_i$, and since it has fiberwise average zero, we obtain from this that $\|\psi_i\|_{L^\infty(Y)}\leq C$ for all $i$. 

The next step is to obtain a uniform $C^{k+1,\alpha}$ bound on $\psi_i$. When $k\geq 1$ this follows from standard elliptic estimates, since we have uniform bounds on $\|\Delta\psi_i\|_{C^{k-1,\alpha}(Y)}$ and $\|\psi_i\|_{L^\infty(Y)}$. For $k=0$ we instead obtain our bound on $\|\psi_i\|_{C^{1,\alpha}(Y)}$ from the elliptic estimate
\begin{equation}\label{SCV}
\|\psi_i\|_{C^{1,\alpha}(Y)}\leq C(\|\db\psi_i\|_{C^\alpha(Y)}+\|\psi_i\|_{L^\infty(Y)}),
\end{equation}
which comes from the interior local estimate
\begin{equation}\label{SCV2}
\|u\|_{C^{1,\alpha}(B_R)}\leq C(\|\db u\|_{C^\alpha(B_{2R})}+\|u\|_{L^\infty(B_{2R})}),
\end{equation}
for all $C^{1,\alpha}$ complex-valued functions on the Euclidean ball $B_{2R}\subset\mathbb{C}^n$. This in turn follows from general theory of elliptic systems \cite{DN} applied to the elliptic operator $\db+\db^*$ acting on $V=\oplus_{j=0}^n\Lambda^{0,j}\mathbb{C}^n$, and applying the interior estimate in \cite{DN} to $(u,0,\dots,0)\in V$. When $k=0$ the RHS of \eqref{SCV} is uniformly bounded, and the desired bound follows.

Now that we have a uniform $C^{k+1,\alpha}$ bound on $\psi_i$, passing to a subsequence if needed, we may assume that $\psi_i$ converges to some $\psi \in C^{k+1,\alpha}(Y)$ in the $C^{k+1,\beta}$ topology for every $\beta < \alpha$. By construction, this limit $\psi$ satisfies the following properties: $[\nabla^k\db\psi]_{C^{\alpha}(Y)} = 0$, $\sum_{j=0}^k\|\nabla^{j+1}\psi\|_{L^\infty(Y)}= 1$, and $\psi$ has average zero. The first property implies that $\psi$ is smooth with $\nabla^{k+1}\db\psi = 0$. We claim that this implies that $\nabla\db\psi=0$. This is clear when $k=0,$ while for $k\geq 1$, relying crucially on the fact that $\partial Y = \emptyset$,
\begin{equation}\int_Y |\nabla^k\db\psi|^2 = \int_Y \langle\nabla^{k-1}\db\psi,\nabla^*\nabla^k\db\psi\rangle  = 0,\end{equation}
and working backwards we get $\nabla^j\db\psi=0$ for $1\leq j\leq k$, proving our claim. The claim gives in particular that $\de\db\psi=0$, which implies that $\psi=0$, a contradiction.
\end{proof}

\subsection{Local Schauder estimates}

In the course of the proof of our main theorem, we will also need the following local Schauder estimate which will be applied when linearizing the complex Monge-Amp\`ere equation. Let $(z_t,y_t) \to (z_\infty,y_\infty)$ be a convergent family of points in $B \times Y$. Consider the diffeomorphisms
\begin{equation}\Lambda_t:  B_{\psi e^{\frac{t}{2}}} \times Y \to B \times Y, \;\, (z,y) = \Lambda_t(\check{z},\check{y}) =(z_t + e^{-\frac{t}{2}}\check{z},\check{y}),\end{equation}
defined for all $t \geq t(\psi)$, where $\psi>0$ is any fixed number strictly less than $d(z_\infty,\partial B)$, and let $\check{J}_t = \Lambda_t^*J$ denote the pullback of the given fibered complex structure $J$ on $B \times Y$. Then $\check{J}_t$ converges to $\check{J}_\infty = J_{\mathbb{C}^m} + J_{Y,z_\infty}$ locally smoothly. Similarly, let $\check{\D}_t$ denote the pullback of the connection $\D$, so that $\check{\D}_t \to \check{\D}_\infty = \nabla^{\C^m} + \nabla^{g_{Y,z_\infty}}$ locally smoothly. Our new basepoint is $(\check{z}_t, \check{y}_t) = (0,y_t) \to (0,y_\infty)$.

\begin{proposition}\label{sonomatematicofrancese}
Let $U \subset \C^m \times Y$ be an open set containing $(0,y_\infty)$. Let $\check{g}_t$, $\check{\omega}_t^\sharp$ be Riemannian resp.~$\check{J}_t$-K\"ahler metrics on $U$ that converge locally smoothly to a Riemannian resp.~$\check{J}_\infty$-K\"ahler metric $\check{g}_\infty$, $\check{\omega}_\infty^\sharp$ on $U$. Then for all $a \in \N, \alpha \in (0,1)$ and $R>0$ there exist $t_0<\infty$ and $C<\infty$ such that for all $0 < \rho < R$, $t \geq t_0$, and all smooth real-valued $i\partial\overline\partial$-exact $\check{J}_t$-$(1,1)$-forms $\eta$ on $U$ we have that
\begin{equation}\label{bakunina}
[\check{\D}_t^a\eta]_{C^\alpha({B}_{\check{g}_t}((0,\check{y}_t),\rho))}\leq
C[\check{\D}_t^a\tr{\check\omega_t^\sharp}{\eta}]_{C^\alpha({B}_{\check{g}_t}((0,\check{y}_t),R))}
+C(R-\rho)^{-a-\alpha}\|\eta\|_{L^\infty({B}_{\check{g}_t}((0,\check{y}_t),R))},
\end{equation}
whenever ${B}_{\check{g}_t}((0,\check{y}_t),R)\subset U$, where the seminorms are those in \eqref{e:holderdef} and the metric used is $\check{g}_t$.
\end{proposition}

\begin{proof}
The first step is to prove that there exists $R_0>0$ such that \eqref{bakunina} holds for $0<\rho<R\leq R_0$. After this is achieved, we will remove this assumption with a standard covering argument.

There exists a $\sigma_0 > 0$ such that the metric $g_{\C^m} + g_{Y,z_\infty}$ admits a normal coordinate chart
\begin{equation}
\Phi: \mathbb{R}^{2m+2n} \supset B_{\sigma_0}(0) \to B_{g_{\C^m} + g_{Y,z_\infty}}((0,\check{y}_\infty),\sigma_0) \subset U
\end{equation}
centered at the point $(0,\check{y}_\infty)$. Let $\sigma \in (0,\sigma_0]$ be arbitrary but fixed, to be determined later. Choose $R_0$ so small and $t_0$ so large that $B_{\check{g}_t}((0,\check{y}_t),2r)$ is $\check{g}_t$-geodesically convex and moreover contained in $\Phi(B_\sigma(0))$ for all $r\leq R_0$ and $t \geq t_0$. Pulling back by $\Phi$, we may then pretend that all objects of interest are defined on (resp.~contained in) the ball $B_\sigma(0) \subset \mathbb{R}^{2m+2n}$. In particular, for all $r \leq R_0$ and $t \geq t_0$, the geodesic ball $\check{B}_{t}(r) := {B_{\check{g}_t}((0,\check{y}_t),r)}$ has compact closure in $B_\sigma(0)$. Moreover, the K\"ahler structures $\check{\omega}_t^\sharp, \check{J}_t$ on $B_\sigma(0)$ satisfy uniform (in $t$)  $C^\infty$ bounds.

Fix $0 < \rho < R \leq R_0$. For all $x \in \check{B}_t(\rho)$ define $\check{B}_x := B_{\check{g}_t}(x,\frac{1}{10}(R-\rho))$. Let $\lambda\check{B}_x$ denote the concentric $\check{g}_t$-ball of $\lambda$ times the radius. Then we claim that there exists a uniform $C$ such that for all $f$,
\begin{equation}\label{triplyanal}
\begin{split}
&\biggl(\sum_{b=0}^{a+2} (R-\rho)^{b-a-2-\alpha}\|\partial^b f\|_{L^\infty(\check{B}_x)}\biggr) + [\partial^{a+2}f]_{C^\alpha(\check{B}_x)}\\
&\leq C\biggl((R-\rho)^{-a-2-\alpha}\|f\|_{L^\infty(2\check{B}_x)} + \biggl(\sum_{b=0}^{a} (R-\rho)^{b-a-\alpha}\|\partial^b\Delta^{\check{\omega}_t^\sharp}f\|_{L^\infty(2\check{B}_x)}\biggr) + [\partial^a \Delta^{\check{\omega}_t^\sharp}f]_{C^\alpha(2\check{B}_x)}\biggr),
\end{split}
\end{equation}
where all norms and derivatives are defined with respect to the standard Euclidean metric on $\mathbb{R}^{2m+2n}$. Indeed, after scaling and stretching by $(R-\rho)^{-1}$ we can apply standard interior Schauder estimates for an essentially fixed pair of domains in $\mathbb{R}^{2m+2n}$ and an essentially fixed operator. See \cite[Thm 6.2]{GT} for the case $a = 0$, noting that their statement gives more than we are using here but not in a way that would help us directly. After scaling back to the original picture, this yields \eqref{triplyanal}.

We now eliminate the $\|\partial^b\Delta^{\check{\omega}_t^\sharp}f\|_{L^\infty}$ terms with $b \geq 1$ from the right-hand side of \eqref{triplyanal}. This can be done using standard interpolation inequalities, at the cost of having to pass from $2\check{B}_x$ to $3\check{B}_x$. For example, we can again stretch and scale by $(R-\rho)^{-1}$ to reduce to the case of two essentially fixed domains in $\R^{2m+2n}$, invoke \cite[Lemma 6.32]{GT}, and then again scale back to the original picture.

Let $\eta$ be an $i\partial\overline\partial$-exact $(1,1)$-form $\eta$ with respect to $\check{J}_t$ on $B_\sigma(0)$ as in the statement of the proposition. We can follow the proof of \cite[Prop 3.2]{He} to find a $\check{J}_t$-potential $f$ for $\eta$ on $3\check{B}_x$ with
\begin{equation}\label{yippiyayieh}
\|f\|_{L^\infty(2\check{B}_x)} \leq C(R-\rho)^2\|\eta\|_{L^\infty(3\check{B}_x)}.
\end{equation}
The proof in \cite{He} was written for a standard ball in $\C^{m+n}$ but all the ingredients carry over to small $\check{g}_t$-geodesic balls and the K\"ahler structure $\check\omega_t^\sharp, \check{J}_t$. Indeed, the standard Poincar\'e lemma formula can be made to work using radial $\check{g}_t$-geodesics, solving the $\overline\partial$-Neumann problem only requires uniform $\check{J}_t$-plurisubharmonicity of the squared $\check{g}_t$-distance, and Moser iteration for the $\check\omega_t^\sharp$-Laplacian also only depends on uniform geometry bounds for $\check{g}_t,\check\omega_t^\sharp,\check{J}_t^\sharp$. All of these hold after increasing $t_0$ and decreasing $\sigma$ if necessary. Here we are only using the smoothness of the (unrelated) tensors $\check{g}_\infty$ and $\check\omega_\infty^\sharp,\check{J}_\infty$, and are not using the normal coordinate property of $\Phi$ with respect to the metric $g_{\C^m} + g_{Y,z_\infty}$.

Inserting \eqref{yippiyayieh} into \eqref{triplyanal} together with the above interpolation argument, we obtain that
\begin{equation}\label{quadruplyanal}
\begin{split}
&\biggl(\sum_{b=0}^{a+2} (R-\rho)^{b-a-2-\alpha}\|\partial^{b} f\|_{L^\infty(\check{B}_x)}\biggr) + [\partial^{a+2}f]_{C^\alpha(\check{B}_x)}\\
&\leq C\biggl((R-\rho)^{-a-\alpha}\|\eta\|_{L^\infty(3\check{B}_x)} +  [\partial^a {\rm tr}^{\check{\omega}_t^\sharp}\eta]_{C^\alpha(3\check{B}_x)}\biggr).
\end{split}
\end{equation}
Our next claim is that this implies that
\begin{equation}
\begin{split}
\label{suredoes}
&\biggl(\sum_{b=0}^a (R-\rho)^{b-a-\alpha}\|\partial^b\eta\|_{L^\infty(\check{B}_x)}\biggr) + [\partial^a \eta]_{C^\alpha(\check{B}_x)}\\
& \leq C\biggl((R-\rho)^{-a-\alpha}\|\eta\|_{L^\infty(3\check{B}_x)} +  [\partial^a {\rm tr}^{\check{\omega}_t^\sharp}\eta]_{C^\alpha(3\check{B}_x)}\biggr).
\end{split}
\end{equation}
This is straightforward to check by writing $\eta = -\frac{1}{2}d(df \circ \check{J}_t)$ on the left-hand side of \eqref{suredoes} and applying the Leibniz rule (discrete as well as continuous). The main terms, where no derivatives or difference quotients land on $\check{J}_t$, can obviously be bounded by the corresponding terms (with two more derivatives of $f$ than of $\eta$) on the left-hand side of \eqref{quadruplyanal}. All of the other terms can similarly be absorbed into the left-hand side of \eqref{quadruplyanal} with a suitable index shift, thanks to the fact that $R-\rho = O(1)$.

Getting closer to the statement of the proposition, we now claim that \eqref{suredoes} implies that
\begin{equation}\label{doublyanal}
\begin{split}
&\biggl(\sum_{b=0}^{a} (R-\rho)^{b-a-\alpha}\|\partial^b\eta \|_{L^\infty(\check{B}_t(\rho))}\biggr) + [\partial^{a}\eta]_{C^\alpha(\check{B}_t(\rho))}\\
&\leq C\biggl((R-\rho)^{-a-\alpha}\|\eta\|_{L^\infty(\check{B}_t(R))} + [\partial^a {\rm tr}^{\check{\omega}_t^\sharp}\eta]_{C^\alpha(\check{B}_t(R))}\biggr).
\end{split}
\end{equation}
Indeed, taking the supremum of \eqref{suredoes} over all $x \in \check{B}_t(\rho)$ already takes care of the $\|\partial^b \eta\|_{L^\infty}$ terms on the left-hand side of \eqref{doublyanal}. To bound the $[\partial^{a}\eta]_{C^\alpha}$ term, note that for $x \neq x'$ in $\check{B}_t(\rho)$ we can estimate the $C^\alpha$ quotient of $\partial^{a}\eta$ at $x,x'$ in two ways. If $x' \not\in\check{B}_x$, then we can estimate it trivially using the triangle inequality, resulting in a term which is absorbed by the $\|\partial^{a}\eta\|_{L^\infty}$ term that we have just dealt with. On the other hand, if $x' \in \check{B}_x$, then we can use \eqref{suredoes} directly.

In \eqref{doublyanal}, we can replace the Euclidean metric used in the definition of the norms by $\check{g}_t$ because these two metrics are uniformly equivalent. Less trivially, we may also replace $\partial$ by $\check{\D}_t$ and the Euclidean parallel transport $P$ implicit in the definition of the $C^{\alpha}$ seminorm by the stretched $\P$-transport $\check{\P}_t$. And also, in the definition \eqref{e:holderdef} of our H\"older seminorms we only consider horizontal or minimal vertical $\P$-geodesics, while the in the Euclidean H\"older seminorms we connect points using the Euclidean segment (which in general is neither horizontal nor vertical), but this difference is immediately seen to be harmless (cf. \eqref{excellente}). Thus, once we show how to replace $\partial$ and $P$ by $\check{\D}_t$ and $\check{\P}_t$, the first step of the proposition follows by dropping the sum over $b$ from the left-hand side.

To compare $\partial$ to $\check{\D}_t$, write $\partial = \check{\D}_t + \check\Gamma_t$. Then $\check\Gamma_t \to \check\Gamma_\infty$ locally smoothly, where $\check\Gamma_\infty(0) = 0$ (recall that $0 \in \mathbb{R}^{2m+2n}$ is the origin of the normal coordinate chart $\Phi$ for the metric $g_{\C^m} + g_{Y,z_\infty}$ at $(0,\hat{y}_\infty)$, and that $\check\D_\infty$ is the Levi-Civita connection of this metric). Thus, the $C^{0}(B_\sigma(0))$ norm of $\check\Gamma_t$ goes to zero as $t \to \infty$ and $\sigma\to 0$, and for all $k \geq 1$ the $C^k(B_\sigma(0))$ norm of $\check{\Gamma}_t$ remains uniformly bounded as $t\to\infty$ and $\sigma\to 0$. By ODE estimates, the same statement holds for the $C^k(B_\sigma(0) \times B_\sigma(0))$ norms of $\check{\P}_t - P$. Thus, when we replace $\partial$ and $P$ by $\check{\D}_t$ and $\check{\P}_t$ in \eqref{doublyanal}, the terms of highest order in $\eta$ contained in the {error} (i.e., the $C^{a,\alpha}$ seminorm terms) come with a coefficient which is $o(1)$ as $t \to \infty$ and $\sigma \to 0$. All of the other error terms come with an explicit factor $\leq C(R-\rho)^\alpha$ compared to the already existing terms of the same order in $\eta$ in \eqref{doublyanal}. Thus, we obtain an estimate of the form
\begin{equation}
\begin{split}
&\biggl(\sum_{b=0}^{a} (R-\rho)^{b-a-\alpha}\|\check\D_t^b\eta \|_{L^\infty(\check{B}_t(\rho))}\biggr) + [\check{\D}_t^{a}\eta]_{C^\alpha(\check{B}_t(\rho))}\\
&\leq C\biggl((R-\rho)^{-a-\alpha}\|\eta\|_{L^\infty(\check{B}_t(R))} + [\check{\D}_t^a {\rm tr}^{\check{\omega}_t^\sharp}\eta]_{C^\alpha(\check{B}_t(R))}\biggr)\\
&+ o(1)\biggl(\biggl(\sum_{b=0}^{a} (R-\rho)^{b-a-\alpha}\|\check\D_t^b\eta \|_{L^\infty(\check{B}_t(R))}\biggr) + [\check{\D}_t^{a}\eta]_{C^\alpha(\check{B}_t(R))}\biggr),
\end{split}
\end{equation}
where the $o(1)$ is uniform as $t \to \infty$ and $\sigma\to 0$. Thus, fixing $\sigma \in (0,\sigma_0]$ sufficiently small and increasing $t_0$ if necessary, the desired result (for radii $0<\rho<R\leq R_0$) now follows from the iteration Lemma \ref{HT2-l:iterate}.

Lastly, we prove \eqref{bakunina} for all radii $0<\rho<R$. We may assume that $R\geq R_0$. Consider first the case when $\rho<\frac{R_0}{2}$. In this case we have $R_0-\rho>\frac{R_0}{2}>\rho$, so we can apply \eqref{bakunina} with larger radius $R_0-\rho<R_0\leq R,$ and get
\begin{equation}\begin{split}
[\check{\D}_t^a\eta]_{C^\alpha({B}_{\check{g}_t}((0,\check{y}_t),\rho))}&\leq
C[\check{\D}_t^a\tr{\check\omega_t^\sharp}{\eta}]_{C^\alpha({B}_{\check{g}_t}((0,\check{y}_t),R_0-\rho))}
+C(R_0-\rho)^{-a-\alpha}\|\eta\|_{L^\infty({B}_{\check{g}_t}((0,\check{y}_t),R_0-\rho))}\\
&\leq C[\check{\D}_t^a\tr{\check\omega_t^\sharp}{\eta}]_{C^\alpha({B}_{\check{g}_t}((0,\check{y}_t),R))}
+C(R-\rho)^{-a-\alpha}\|\eta\|_{L^\infty({B}_{\check{g}_t}((0,\check{y}_t),R))},
\end{split}
\end{equation}
where in the last line we used that $R_0-\rho\geq \frac{R_0}{2}\geq C^{-1}R>C^{-1}(R-\rho)$, since we allow $C$ to depend on $R$.

We can then assume that $\rho\geq \frac{R_0}{2}$.  Given two points $\check{x}_1=(\check{z}_1,\check{y}_1)$ and $\check{x}_2=(\check{z}_2,\check{y}_2)$ in $B_{\check{g}_t}((0,\check{y}_t),\rho)$, let $r=d^{\check{g}_t}(\check{x}_1,\check{x}_2)$. We can choose a sequence of points $\check{x}_1=\check{w}_0,\check{w}_1,\dots,\check{w}_N=\check{x}_2$ inside $B_{\check{g}_t}((0,\check{y}_t),\rho)$, where $N$ is at most $CR_0^{-1}$ (hence uniform), such that $r_j:=d^{\check{g}_t}(\check{w}_j,\check{w}_{j+1})< \frac{R_0}{4}$ is uniformly comparable to $r$ for $0\leq j\leq N-1$, and also $\check{w}_{j-1},\check{w}_{j+1}\in B_{\check{g}_t}(\check{w}_j,\frac{R_0}{4})\subset B_{\check{g}_t}((0,\check{y}_t),\rho)$ for $1\leq j\leq N-1$. Thus, for each $1\leq j\leq N-1$, the ball $B_{\check{g}_t}(\check{w}_j,\frac{R_0}{4}+\min(R-\rho,\frac{R_0}{4}))$ is contained in $B_{\check{g}_t}((0,\check{y}_t),R)$ and its radius is at most $\frac{R_0}{2}$, so we can apply \eqref{bakunina} to the balls $B_{\check{g}_t}(\check{w}_j,\frac{R_0}{4})$ and $B_{\check{g}_t}(\check{w}_j,\frac{R_0}{4}+\min(R-\rho,\frac{R_0}{4}))$ and estimate
\begin{equation}\label{merde}
\begin{split}
&\frac{|\check{\D}_t^a\eta(\check{w}_{j-1})-\mathbb{P}_{t,\check{w}_j\check{w}_{j-1}}(\check{\D}_t^a\eta(\check{w}_j))|_{\check{g}_t}}{r_j^\alpha}\\
&\leq C[\check{\D}_t^a\tr{\check\omega_t^\sharp}{\eta}]_{C^\alpha({B}_{\check{g}_t}(\check{w}_j,\frac{R_0}{4}+\min(R-\rho,\frac{R_0}{4}))}
+C\left(\min\left(R-\rho,\frac{R_0}{4}\right)\right)^{-a-\alpha}\|\eta\|_{L^\infty({B}_{\check{g}_t}(\check{w}_j,\frac{R_0}{4}+\min(R-\rho,\frac{R_0}{4}))}\\
&\leq C[\check{\D}_t^a\tr{\check\omega_t^\sharp}{\eta}]_{C^\alpha({B}_{\check{g}_t}((0,\check{y}_t),R))}
+C(R-\rho)^{-a-\alpha}\|\eta\|_{L^\infty({B}_{\check{g}_t}((0,\check{y}_t),R))}.
\end{split}
\end{equation}
where we used again that $\frac{R_0}{4}\geq C^{-1}R>C^{-1}(R-\rho)$. An analogous bound holds for $|\check{\D}_t^a\eta(\check{w}_{j+1})-\mathbb{P}_{t,\check{w}_j\check{w}_{j+1}}(\check{\D}_t^a\eta(\check{w}_j))|_{\check{g}_t}/r_j^\alpha$, and then to estimate the $C^\alpha$ difference quotient of $\check{\D}^a_t\eta$ between $\check{x}_1$ and $\check{x}_2$ we use \eqref{merde}, the triangle inequality, the uniform equivalence of $r_j$ and $r$, and the bound on the operator norm of $\P_t$ from Section \ref{s:tridiagonal} to bound
\begin{equation}
\begin{split}
\frac{|\check{\D}_t^a\eta(\check{x}_1)-\mathbb{P}_{t,\check{x}_2\check{x}_1}(\check{\D}_t^a\eta(\check{x}_2))|_{\check{g}_t}}{r^\alpha}&\leq
C\sum_{j=0}^{N-1}\frac{|\check{\D}_t^a\eta(\check{w}_j)-\mathbb{P}_{t,\check{w}_{j+1}\check{w}_{j}}(\check{\D}_t^a\eta(\check{w}_{j+1}))|_{\check{g}_t}}{r_j^\alpha}\\
&\leq C[\check{\D}_t^a\tr{\check\omega_t^\sharp}{\eta}]_{C^\alpha({B}_{\check{g}_t}((0,\check{y}_t),R))}
+C(R-\rho)^{-a-\alpha}\|\eta\|_{L^\infty({B}_{\check{g}_t}((0,\check{y}_t),R))}.
\end{split}
\end{equation}
and taking the supremum over $\check{x}_1$ and $\check{x}_2$ in $B_{\check{g}_t}((0,\check{y}_t),\rho)$ concludes the proof of \eqref{bakunina}.
\end{proof}

\section{Selection of obstruction functions}\label{sectref}

The main result of this section, the Selection Theorem \ref{ghost}, roughly speaking identifies a list of smooth functions $G_{j,p,k}, j\geq 2,$ which arise due to the unboundedness of the RHS of the Monge-Amp\`ere equation in the shrinking $C^{j,\alpha}$ norm. In order to even state precisely the properties that these functions satisfy, we first need to discuss how to turn a set of smooth functions on the total space into a fiberwise $L^2$ orthonormal set (with a suitable small error), which will be done in Section \ref{rueckert}, and we need to construct a certain approximate Green's operator in Section \ref{grz}.

\subsection{Approximate fiberwise Gram-Schmidt}\label{rueckert}

Let $B$ be a ball centered at the origin in $\mathbb{R}^d$ with real linear coordinates $z_1, \ldots, z_d$. Let $Y$ be a smooth compact manifold without boundary. Let $\Upsilon$ be a smooth fiberwise volume form on $B \times Y$, with fiberwise total volume equal to $1$. For clarity we will denote by $\Upsilon_z$ the induced volume form the fiber $\{z\}\times Y$. In our later applications, we will take $\Upsilon$ so that $\Upsilon_z$ is the unit-volume fiberwise Calabi-Yau volume form.

\begin{proposition}\label{iamdeathincarnate}
Suppose we are given an integer $J\geq 0$ and functions $F_1,\dots,F_N,$ $H_1,\dots,H_L\in C^\infty(B\times Y)$. Suppose also that the functions $\left\{H_k|_{\{z\}\times Y}\right\}_{1\leq k\leq L}$ are $L^2$-orthonormal (w.r.t. $\Upsilon$) for all $z\in B$.
Then we can find a concentric ball $B'\subset B$ and functions $G_{j}\in C^\infty(B'\times Y), 1\leq j\leq N'$ such that the functions $\left\{G_j|_{\{z\}\times Y},H_k|_{\{z\}\times Y}\right\}_{1\leq j\leq N',1\leq k\leq L}$ are $L^2$-orthonormal (w.r.t. $\Upsilon$) for all $z\in B'$, and there are functions $p_{ij},q_{ik}\in C^\infty(B'), 1\leq i\leq N, 1\leq j\leq N',1\leq k\leq L$ and $K_{i,\alpha} \in C^\infty(B'\times Y)$ ($\alpha\in\N^d$, $|\alpha| = J+1$) such that on $B'\times Y$ we have
\begin{equation}\label{buddha}
F_i(z,y) = \sum_{j=1}^{N'} p_{ij}(z) G_j(z,y)+\sum_{k=1}^L q_{ik}(z)H_k(z,y)+\sum_{|\alpha|=J+1} z^\alpha K_{i,\alpha}(z,y)
\end{equation}
for $1\leq i\leq N$.
\end{proposition}

\begin{rk}
The same statement holds in the real-analytic category if all the data are real-analytic, and the proof is verbatim the same.
\end{rk}

\begin{rk}
If we choose $H_1$ to be the constant function $1$, then $H_2, \ldots, H_L, G_1,\ldots,G_{N'}$ will have fiberwise average zero with respect to $\Upsilon$. If in addition some $F_i$ has fiberwise average zero, then the coefficient $q_{i1}(z)$ in the expansion \eqref{buddha} will be zero.
\end{rk}

\begin{rk}\label{miservi}
The errors $z^\alpha K_{i,\alpha}(z,y)$ in \eqref{buddha} have the following crucial property: if define diffeomorphisms
\begin{equation}\label{sigmat2}
\Sigma_t: B_{e^{\frac{t}{2}}} \times Y \to B \times Y, \;\, (z,y) = \Sigma_t(\check{z},\check{y}) =(e^{-\frac{t}{2}}\check{z},\check{y}),
\end{equation}
then we have that
\begin{equation}\label{marrow}
e^{\frac{Jt}{2}}\Sigma_t^*(z^\alpha K_{i,\alpha})\to 0,
\end{equation}
smoothly on $B_R\times Y$ for any fixed $R$. Indeed, trivially,
\begin{equation}\label{morrow}
\sup_{B_R}|\D^\iota \Sigma_t^*z^\alpha|\leq C(\iota,R)e^{-\frac{J+1}{2}t}
\end{equation}
for all $\iota\geq 0$. Using \eqref{morrow} it is then easy to conclude that \eqref{marrow} holds.
\end{rk}

\begin{proof}
It is enough to prove this for one single $F$ because if there are several then we can just process them one by one, adding the newly gained $G$s to the previously given list of $H$s in each step. Then begin by replacing $F$ by its fiberwise orthogonal projection onto the orthogonal complement of $H_1,\ldots,H_L$. If this projection (which we shall call $F$ from now on) is not identically zero on the central fiber, then we divide $F$ by its fiberwise $L^2$ norm (after concentrically shrinking the base if necessary) and stop---the proposition is proved, with one single function $G(z,y) := F(z,y)/\|F(z,\cdot)\|_{L^2}$ and with $K_\alpha \equiv 0$. So we may assume from now on that $F$ vanishes identically on the central fiber.

We will now construct an ordered tree such that the desired expansion can be read off by performing a depth-first traversal of the tree. An \emph{ordered tree} is a tree for which the children of each node are ordered, and \emph{depth-first search} is a way of ordering the whole set of nodes, see Figure \ref{figtree}. In particular, it makes sense to talk about previous and subsequent nodes with respect to depth-first search. In a depth-first search, at any given node  ${\bf n}$ of an ordered tree $T$, traversal of $T$ proceeds by ordering the whole set of nodes of the subtree $S$ of $T$ whose root is ${\bf n}$ (before moving on to nodes outside of $S$). Moreover, the first node visited after traversal of $S$ is the first unvisited child (with respect to the given ordering of the children) of the most recent ancestor of ${\bf n}$ not all of whose children have been visited.

\begin{figure}\label{figtree}
\includegraphics[width=80mm]{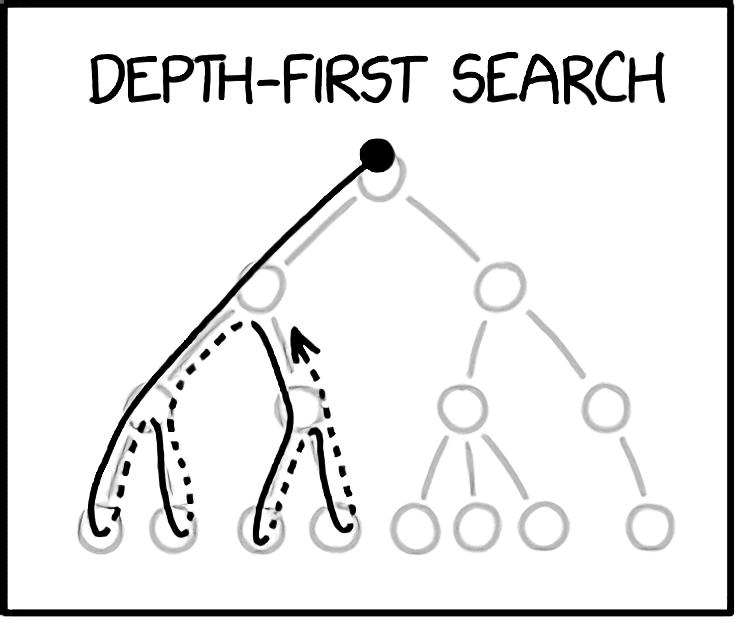}
\caption{Source: https://xkcd.com/2407/.}
\end{figure}

The nodes of our tree come in two flavors: {undecorated} and decorated. An \emph{undecorated node} is one of the form $(F)$, where $F$ is a smooth function on $B \times Y$ vanishing along the central fiber $\{0\}\times Y$. A \emph{decorated node} is one of the form $(F\,|\,G_1,\ldots,G_k)$, where $F$ is as before and $G_1, \ldots, G_k$ are smooth functions on $B \times Y$ that are fiberwise orthonormal in $L^2$. By construction, $G$s appearing as decoration in different nodes will be fiberwise orthogonal to each other and to the initially given list of $H$s, and the $F$ in each node (decorated or undecorated) will be fiberwise orthogonal to the span of the $H$s and all the $G$s appearing as decoration in all previous nodes (with respect to depth-first search). Also by construction, $k \in \{0,1,\ldots,d\}$.

We put an undecorated node $(F)$ as the root of the tree, where $F$ is the function constructed in the first paragraph of this proof. Then we construct our tree by iterating the following procedure.
\begin{itemize}
\item[1.] If there is no undecorated node of \emph{depth} $\leq J$ (i.e., descended from the root of the tree in at most $J$ generations), then stop. Otherwise, go to the first such node, $(F)$, with respect to the depth-first order of the tree that has been constructed so far, and continue as follows.
\item[2.] Define functions $F_1,\ldots,F_d \in C^\infty(B \times Y)$, where for $i \in \{1,\ldots,d\}$,
\begin{equation}
F_i(z,y) := \int_0^1 \frac{\partial F}{\partial z_i}(tz,y)\,dt.
\end{equation}
Moreover, let $\mathbf{O}$ be the set of all the initial $H$s, union the set of all the $G$s from all decorated nodes that come before $(F)$ in the depth-first ordering of the tree that has been constructed so far. By induction, this is a fiberwise $L^2$ orthonormal set of functions in $C^\infty(B \times Y)$.
\item[3.] On the central fiber, consider the finite-dimensional subspace $\mathcal{V} \subset C^\infty(\{0\}\times Y)$ spanned by the restrictions of elements of $\mathbf{O} \cup \{F_1,  \ldots, F_d\}$. The functions from $\mathbf{O}$ are an $L^2$-orthonormal subset of $\mathcal{V}$. Complete this set to a basis of $\mathcal{V}$ by adding elements from $F_1, \ldots, F_d$, which (after reordering) we may assume are $F_1, \ldots, F_{d_0}$ for some $d_0 \in \{0,1,\ldots,d\}$. Then, after shrinking $B$ concentrically, $F_{1}, \ldots, F_{{d_0}}$ will remain fiberwise linearly independent of each other and of $\mathbf{O}$ over every point of $B$. Thus, by applying Gram-Schmidt, we can construct a fiberwise orthonormal set $\{G_1, \ldots, G_{d_0}\} \subset C^\infty(B \times Y)$, fiberwise orthogonal to $\mathbf{O}$, such that the fiberwise span of $\mathbf{O} \cup \{G_1, \ldots, G_{d_0}\}$ is the same as the fiberwise span of $\mathbf{O} \cup \{F_{1},\ldots,F_{d_0}\}$ (which, on the central fiber, is the space $\mathcal{V}$). Change the undecorated node $(F)$ to the decorated node $(F\,|\,G_1,\ldots,G_{d_0})$.
\item[4.] Replace $F_{d_0+1},\ldots,F_{d}$ by their fiberwise orthogonal projections onto the fiberwise orthogonal complement of the fiberwise span of $\mathbf{O} \cup \{G_1,\ldots,G_{d_0}\}$. For each of the resulting functions $F_i \in C^\infty(B \times Y)$ ($i \in \{d_0+1,\ldots,d\}$), all of which vanish identically on the central fiber, we now create an undecorated child node $(F_i)$ of the decorated parent node $(F\,|\,G_1,\ldots,G_{d_0})$. Observe that the children are indeed naturally ordered at this point, even though the chosen order is arbitrary and any choice of an order would have been fine.
\end{itemize}

This procedure ends after finitely many steps because by construction the tree keeps growing in each step (unless the procedure ends) but can only contain at most $d^{J+1}$ nodes of either kind in total. Thus, eventually, there will be no undecorated nodes of depth $\leq J$ left. In particular, at the final stage, the only undecorated nodes will be of depth $=$ $J+1$, and will have no children.

We now obtain the desired expansion of $F$ as follows. Traverse the tree depth-first. At each decorated node $(F\,|\,G_1,\ldots,G_k)$, we may by construction write $F = z_1F_1 + \cdots + z_dF_d$, where each $F_i$ is {\bf either} a $C^\infty(B)$-linear combination of the $H$s, the decorating $G$s from previous nodes, and the decorating $G$s from the current node, {\bf or} equal to the $F$ stored in one of the (decorated or undecorated) children of the current node. We can ignore the contribution of $F_i$s of the first type because this is exactly what the non-remainder terms of the expansion \eqref{buddha} look like. For each $F_i$ of the second type, we wait until the depth-first traversal which continues with the first of these has arrived at the corresponding child node, and then either do nothing (if the child node is undecorated of depth $=$ $J+1$), or iterate the argument by further expanding $F_i$.

It remains to check that once the procedure has ended, the remainder terms are indeed of the form $\sum_{|\alpha|=J+1} z^\alpha K_\alpha(z,y)$ with $K_\alpha \in C^\infty(B \times Y)$. This can be done by induction on $J$. Indeed, for $J = 0$ the situation is clear. Assuming the proposition holds for $J-1$ instead of $J \geq 1$, consider the nodes of depth $1$, i.e., those (decorated) nodes $(F_i \,|\, \cdots )$ whose parent is the root $(F \,|\,\cdots)$ of the tree. Each of these nodes is the root of a subtree, and depth-first search on the whole tree restricts to depth-first search on the subtree. Depths of nodes in the subtree are one less than their depths in the whole tree. Thus, the inductive hypothesis applies to the subtree, producing remainders of the form $\sum_{|\alpha|=J} z^\alpha K_\alpha(z,y)$ in the expansion of each $F_i$. But each $F_i$ comes with a coefficient $z_i$ in the expansion of $F$.
\end{proof}

\subsection{Approximate Green operators}\label{grz}

Recall now the definition of $\omega_F$ which was given in the Introduction: first, we define $\omega_F=\omega_X+\ddbar\rho$ where $\rho$ is the smooth function on $B\times Y$ defined by the fact that $\omega_X|_{\{z\}\times Y}+\ddbar\rho(z,\cdot)$ is the unique Ricci-flat K\"ahler metric on $\{z\}\times Y$ cohomologous to the restriction of $\omega_X$, and by the normalization $\int_{\{z\}\times Y}\rho(z,y)\omega_X^n(y)=0$.

In this section we will also denote by $\omega_{\rm can}$ a fixed K\"ahler metric on $B$ (with no special properties), which in our later application to Theorems A and B will be the restriction to our ball $B$ of the unique global solution of the Monge-Amp\`ere equation \eqref{ma_limit}.

Given a smooth function $H$ with $\int_{\{z\}\times Y}H\omega_F^n=0$ for all $z\in B$ and a $(1,1)$-form $\alpha$ on $B\times Y$ we define a smooth function on $B$ by
\begin{equation}\label{proietto}
P_{t,H}(\alpha):=n({\rm pr}_B)_*\left(H\alpha\wedge\omega_F^{n-1}\right) + e^{-t}{\rm tr}^{\omega_{\rm can}}({\rm pr}_B)_*(H \alpha\wedge \omega_F^n).
\end{equation}
We think of this as an approximation to the fiberwise $L^2$-projection of ${\rm tr}^{\omega_t^\natural}\alpha$ onto ${\rm span}\,H$ (after scaling and stretching so that the fibers have unit size). Here are two simple yet crucial properties of this operator:
\begin{lemma}
Let $G$ be a smooth function on $B\times Y$. Then we have
\begin{equation}\label{dr1}
P_{t,H}(\ddbar G)(z)=\int_{\{z\}\times Y}H\mathfrak{L}_t(G)\;\omega_F^n,
\end{equation}
where we have defined $\mathfrak{L}_t(G)$ so that along $\{z\}\times Y$ it equals
\begin{equation}\label{dr2}
\mathfrak{L}_t(G):=\Delta^{{\omega}_F|_{\{z\}\times Y}}G+ e^{-t}\frac{((\sum_{\iota=1}^m Z_\iota \wedge \overline{Z}_\iota)\,\lrcorner\,(i\partial\overline\partial G \wedge \omega_F^n))|_{\{z\}\times Y}}{\omega_F^n|_{\{z\}\times Y}},
\end{equation}
where $Z_\iota$ denotes an orthonormal $(1,0)$ frame field with respect to $\omega_{\rm can}$ on $\C^m$, trivially extended to $\C^m\times Y$. Also, for every $(1,1)$-form $\beta$ on $B$ we have
\begin{equation}\label{frombase}
P_{t,H}({\rm pr}_B^*\beta)=0.
\end{equation}
\end{lemma}
\begin{proof}
The identity \eqref{dr1} is clear thanks to the definition \eqref{dr2}. For \eqref{frombase}, when we apply $P_{t,H}$ to ${\rm pr}_B^*\beta$ the first term in \eqref{proietto} vanishes by definition, and the second term is also seen to vanish by using the projection formula
$({\rm pr}_B)_*(H{\rm pr}_B^*\beta\wedge \omega_F^n)=\beta\wedge({\rm pr}_B)_*(H\omega_F^n),$ and the fact that $H$ has fiberwise average zero.
\end{proof}
For functions $A$ on the base and $G$ on the total space (with fiberwise average zero), we now define a new function $\mathfrak{G}_{t,k} = \mathfrak{G}_{t,k}(A,G)$. This should be thought of as an approximate right inverse of $\Delta^{\omega_t^\natural}$ applied to $AG$ (after scaling and stretching so that the fibers have unit size). Here we want to view $G$ as being fixed, while $A$ is variable except for the fact that $A$ is expected to be a polynomial or to behave like one. The precise properties of $\mathfrak{G}_{t,k}$ that we need will be proved in Lemmas \ref{paraminatrix}--\ref{Gstructure}. Somewhat strangely, $\mathfrak{G}_{t,k}$ will be a differential rather than an integral operator in the base directions.

Fix any $z\in B$. The aim is to define $\mathfrak{G}_{t,k}$ at all points $(z,y)$ in the fiber over $z$. We inductively define two sequences of functions $u_0, u_1, \ldots$ and $v_0, v_1, \ldots$, where the desired $\mathfrak{G}_{t,k}(z,y)$ will be given by $u_{k}(z,y)$.
Given $A,G$ as above, let
\begin{align}
u_0 := A(\Delta^{\omega_F|_{\{z\}\times Y}})^{-1}G,\label{voland1}\\
v_i := \mathfrak{L}_{t}(u_{i})-AG,\label{voland2}\\
u_{i+1} :=u_i  - (\Delta^{\omega_F|_{\{z\}\times Y}})^{-1}(v_i - \underline{v_i}),\label{voland3}
\end{align}
where $(\Delta^{\omega_F|_{\{z\}\times Y}})^{-1}$ applied to a function with fiberwise average zero denotes the unique solution with fiberwise average zero of the fiberwise Poisson equation. And finally,
\begin{equation}{\mathfrak{G}}_{t,k}({z},{y}) := u_{k}(z,y).\end{equation}

We immediately observe that $\mathfrak{G}_{t,k}(A,G)$ is bilinear in its two arguments, and it has fiberwise average zero.

\begin{rk}\label{prod1}
As an example, when the complex structure on $B\times Y$ is a product, $\omega_{\rm can}$ is just a Euclidean metric $\omega_{\C^m}$, $\omega_F$ is the pullback of a fixed Ricci-flat K\"ahler metric $\omega_Y$ on $Y$, and $G$ is also pulled back from $Y$, then we can compute that
\begin{equation}\label{inutilis}
\mathfrak{G}_{t,k}(A,G) = \sum_{\ell = 0}^{k} (-1)^\ell e^{-\ell t} (\Delta^{\C^m})^{\ell}A \cdot (\Delta^Y)^{-\ell-1}G,
\end{equation}
where $(\Delta^Y)^{-1}$ denotes the unique solution of the fiberwise Poisson equation (with zero average), see Lemma \ref{Gstructure} below and Remark \ref{prod2} for a sketch of proof. Applying the Laplacian of the product metric $\omega_{\C^m}+e^{-t}\omega_Y$ to this then clearly gives
\begin{equation}\label{inutilis2}
\Delta^{\omega_{\C^m}+e^{-t}\omega_Y}\mathfrak{G}_{t,k}(A,G) = e^tAG+(-1)^{k}e^{-k t}(\Delta^{\C^m})^{k+1}A \cdot (\Delta^Y)^{-k-1}G,
\end{equation}
and the last term vanishes if $A$ is a polynomial of degree $<2k+2$. This formula was our original motivation for the general construction of $\mathfrak{G}_{t,k}$.
\end{rk}

\begin{lemma}\label{paraminatrix}
For all functions $G,H$ on $B\times Y$ with  fiberwise average zero, there exist functions $\Phi_{\iota,k}(G,H)$ independent of $t$ on the base such that for all functions $A$ on the base and all $t$,
\begin{equation}\label{e:paraminatrix}
P_{t,H}(\ddbar(\mathfrak{G}_{t,k}(A,G))) = A\int_{\{z\}\times Y} GH\,\omega_F^n +  e^{-(2k+2)\frac{t}{2}} \sum_{\iota=0}^{2k+2} \Phi_{\iota,k}(G,H) \circledast \D^{\iota}A.
\end{equation}
\end{lemma}

\begin{proof}
Fix any point $z\in B$ and aim to verify \eqref{e:paraminatrix} at $z$. We compute
\begin{align}
P_{t,H}(\ddbar\mathfrak{G}_{t,k})(z) &= \int_{\{z\}\times Y} H \mathfrak{L}_{t}(\mathfrak{G}_{t,k})\,\omega_F^n\\
&= A(z) \int_{\{z\}\times Y} GH\,\omega_F^n + \int_{\{z\}\times Y} H v_{k}\,\omega_F^n,\label{iwonderwhy}
\end{align}
where both lines hold simply by definition from \eqref{dr1} and \eqref{voland2}. So we need to determine the structure of the second term in \eqref{iwonderwhy}. For this, observe that we can write
\begin{equation}\label{krasavice}
\mathfrak{L}_t(G) = (1+e^{-t}{\rm tr}^{\omega_{\rm can}}\omega_{F,\mathbf{bb}})\Delta^{\omega_F|_{\{z\}\times Y}}G+ e^{-t}\Delta^{\omega_{\rm can}}G+e^{-t}\D^2_{\mathbf{bf}}G\circledast \omega_{F,\mathbf{bf}},
\end{equation}
and so using \eqref{voland1}--\eqref{voland3}, along $\{z\}\times Y$,
\begin{equation}\begin{split}\label{wehret}
v_0 &= [e^{-t}\Delta^{\omega_{\rm can}} +  e^{-t}({\rm tr}^{\omega_{\rm can}}\omega_{F,\mathbf{bb}})\Delta^{\omega_F|_{\{z\}\times Y}}]
(A \cdot (\Delta^{\omega_F|_{\{z\}\times Y}})^{-1}G)\\
&+e^{-t}\D^2_{\mathbf{bf}}(A \cdot (\Delta^{\omega_F|_{\{z\}\times Y}})^{-1}G)\circledast \omega_{F,\mathbf{bf}},\end{split}\end{equation}
and for all $i \geq 0$, again along $\{z\}\times Y$,
\begin{align}
v_{i+1} &= \mathfrak{L}_{t}(u_{i+1}) - AG\label{strt}\\
&=v_i - \mathfrak{L}_{t}[(\Delta^{\omega_F|_{\{z\}\times Y}})^{-1}(v_i - \underline{v_i})]\label{e:koroviev}\\
&= \underline{v_i}- [e^{-t}\Delta^{\omega_{\rm can}} + e^{-t} ({\rm tr}^{\omega_{\rm can}}\omega_{F,\mathbf{bb}})
\Delta^{\omega_F|_{\{z\}\times Y}}](\Delta^{\omega_F|_{\{z\}\times Y}})^{-1}(v_i - \underline{v_i})\label{e:behemoth}\\
&-e^{-t}\D^2_{\mathbf{bf}}[(\Delta^{\omega_F|_{\{z\}\times Y}})^{-1}(v_i - \underline{v_i})]\circledast \omega_{F,\mathbf{bf}}.\label{nd}
\end{align}
It is now straightforward to prove by induction on $i$ that
\begin{equation}\label{yabbadabbadoo}
v_i = \sum_{\iota=0}^{2i+2} e^{-(i+1)t}\Phi_{\iota,i}(G) \circledast \D^{\iota}A + \sum_{\ell=1}^i\sum_{\iota=0}^{2\ell}e^{-\ell t}\Psi_{\iota,\ell}(G)\circledast \D^\iota A,
\end{equation}
where the $\Phi$'s live on the total space but the $\Psi$'s are from the base.
Indeed, the base case of the induction $i=0$ follows immediately from \eqref{wehret} remembering that $A$ is pulled back from the base. As for the induction step, we assume \eqref{yabbadabbadoo} for $i$ which gives
\begin{equation}\label{yabbadabbadoo2}
v_i-\underline{v_i} = \sum_{\iota=0}^{2i+2} e^{-(i+1)t}\Phi_{\iota,i}(G) \circledast \D^{\iota}A,
\end{equation}
\begin{equation}\label{yabbadabbadoo3}
\underline{v_i} = \sum_{\ell=1}^{i+1}\sum_{\iota=0}^{2\ell}e^{-\ell t}\Psi_{\iota,\ell}(G)\circledast \D^\iota A,
\end{equation}
where the functions $\Phi_{\iota,i}, \Psi_{\iota,\ell}$ have changed. We need to bound $v_{i+1}$, which is given by  \eqref{strt}--\eqref{nd}. The term $\underline{v_i}$ is controlled by \eqref{yabbadabbadoo3}, while for the other terms we first observe that thanks to \eqref{yabbadabbadoo2} we can write
\begin{equation}\label{yabbadabbadoo5}
(\Delta^{\omega_F|_{\{z\}\times Y}})^{-1}(v_i-\underline{v_i}) = \sum_{\iota=0}^{2i+2} e^{-(i+1)t}\Phi_{\iota,i}(G) \circledast \D^{\iota} A,
\end{equation}
where again $\Phi_{\iota,i}$ have changed, and using this it is straightforward to check that
\begin{equation}\label{yabbadabbadoo4}\begin{split}
&- [e^{-t}\Delta^{\omega_{\rm can}} + e^{-t} ({\rm tr}^{\omega_{\rm can}}\omega_{F,\mathbf{bb}})
\Delta^{\omega_F|_{\{z\}\times Y}}](\Delta^{\omega_F|_{\{z\}\times Y}})^{-1}(v_i - \underline{v_i})\\
&-e^{-t}\D^2_{\mathbf{bf}}[(\Delta^{\omega_F|_{\{z\}\times Y}})^{-1}(v_i - \underline{v_i})]\circledast \omega_{F,\mathbf{bf}}=\sum_{\iota=0}^{2i+4} e^{-(i+2)t}\Phi_{\iota,i}(G) \circledast \D^{\iota}A,
\end{split}\end{equation}
and \eqref{yabbadabbadoo} for $i+1$ now follows.

Setting $i =k$ in \eqref{yabbadabbadoo} and integrating it against $H$ over the fiber, the terms with $\Psi$ from the second sum disappear, and we are done.
\end{proof}

As a corollary of the above computations, we also obtain the following structure of $\mathfrak{G}_{t,k}$ itself, which is a direct generalization of \eqref{inutilis} in the product case.

\begin{lemma}\label{Gstructure}
For all smooth functions $G$ on the total space with  fiberwise average zero, there exist smooth functions $\Phi_{\iota,i}(G)$ on the total space, independent of $t$, such that for all functions $A$ on the base and all $t$,
\begin{equation}\label{e:Gstructure}
\mathfrak{G}_{t,k}(A,G) =  \sum_{\iota=0}^{2k}\sum_{i=\lceil \frac{\iota}{2} \rceil}^{k} e^{-it}(\Phi_{\iota,i}(G)\circledast \D^\iota A),
\end{equation}
and furthermore
\begin{equation}\label{sangennaro}
\Phi_{0,0}(G)=(\Delta^{\omega_F|_{\{\cdot\}\times Y}})^{-1}G.
\end{equation}
\end{lemma}

\begin{proof}
We combine \eqref{voland3} with \eqref{yabbadabbadoo5} to get
\begin{equation}\label{beneficiodicristocrocifisso}
u_{i+1} = u_i + \sum_{\iota=0}^{2i+2} e^{-(i+1)t}(\Phi_{\iota,i}(G)\circledast\D^\iota A).
\end{equation}
Hence, iterating, using \eqref{voland1}, and switching the order of summation,
\begin{equation}\label{feuerstein}
\mathfrak{G}_{t,k} = u_{k} = \sum_{\iota=0}^{2k}\sum_{i=\lceil\frac{\iota}{2}\rceil}^{k} e^{-it}(\Phi_{\iota,i}(G)\circledast\D^\iota A).
\end{equation}

To prove \eqref{sangennaro}, we are interested in the terms in the sum in \eqref{e:Gstructure} with $i=0$, and observe that none of these can come from the last term in \eqref{beneficiodicristocrocifisso} (with index $i$ there ranging from $0$ to $k-1$) since these terms have a coefficient at least as small as $e^{-t}$. Thus, the term with $i=0$ in \eqref{e:Gstructure} come purely from \eqref{voland1}, and \eqref{sangennaro} follows.
\end{proof}

\begin{rk}\label{prod2}
In the product situation described in Remark \ref{prod1}, \eqref{wehret} and \eqref{e:behemoth} imply by induction that
\begin{equation}v_i = v_i - \underline{v_i} = (-1)^i e^{-(i+1)t}(\Delta^{\C^m})^{i+1} A \cdot (\Delta^Y)^{-i-1}G.\end{equation}
In the proof of Lemma \ref{Gstructure} we then see that in the outer sum over $\iota$ we only have the even numbers $\iota = 2\ell$ ($0 \leq \ell \leq k$) and in the inner sum over $i$ we only see the leading term $i = \lceil \frac{\iota}{2}\rceil = \ell$, and \eqref{inutilis} follows.
\end{rk}

\subsection{A simple interpolation for polynomials}

The following simple lemma will be used extensively in this section.

\begin{lemma}\label{mushy}
For all $d,k \in \N_{\geq 1}$ there exists a constant $C = C(d,k)$ such that for all $0 \leq \iota < \kappa \leq k$, all $ R> 0$, and all polynomials $p$ of degree at most $k$ on $\R^d$ it holds that
\begin{equation}\label{shroom}
\sup_{B_R(0)} |\mathbb{D}^\kappa p| \leq C R^{\iota-\kappa}\sup_{B_R(0)}|\mathbb{D}^\iota p|.
\end{equation}
\end{lemma}

\begin{proof}
It is enough to prove this for $R = 1$, $\iota = 0$ and $\kappa = 1$, with $\mathbb{D}^\kappa p = \mathbb{D} p$ replaced by $\mathbb{D}_{\mathbf{e}_i}p$ on the left-hand side for any one of the standard basis vectors $\mathbf{e}_1, \ldots, \mathbf{e}_d$ on $\mathbb{R}^d$. Indeed, the case of a general $R > 0$ follows by scaling, and the case of a general $0 \leq \iota < \kappa \leq k$ follows by considering the component functions of $\mathbb{D}^\iota p$, which are again polynomials of degree at most $k$, and taking their iterated derivatives in directions $\mathbf{e}_1, \ldots, \mathbf{e}_d$. To prove the simplified statement, write $p(z) = \sum_{|\alpha|\leq k} p_\alpha z^\alpha$. Then
\begin{equation}
\frac{1}{k}\sup_{B_1(0)}|\mathbb{D}_{\mathbf{e}_i}p| = \sup_{|z|\leq 1}\left|\sum_{|\alpha|\leq k}  \frac{\alpha_i}{k} p_{\alpha} z^{\alpha-\mathbf{e}_i}\right| \leq \sum_{|\alpha|\leq k} |p_\alpha| =: \|p\|.
\end{equation}
Note that $\|p\|$ is simply the $\ell^1$ norm of $p$ with respect to the standard basis of the vector space $P$ of polynomials of degree at most $k$. Now observe that $p \mapsto \sup_{B_1(0)} |p|$ is also a norm on $P$ because a real polynomial that vanishes on an open set is the zero polynomial. So $\|p\| \leq C\sup_{B_1(0)}|p|$ as desired.
\end{proof}

\subsection{Statement of the Selection Theorem}

In this section we have two natural numbers $0\leq j\leq k$, fixed throughout, and a ball $B\subset\mathbb{C}^m$ centered at the origin (which in later applications in the blowup argument will be translated to be centered at our point of interest). We also denote by $\Upsilon=\omega_F^n$ the fiberwise Calabi-Yau volume form (scaled to fiberwise volume $1$ for convenience), and let $\Upsilon_z=\Upsilon|_{\{z\}\times Y}$.

Given $t,k$ and two smooth functions $A\in C^\infty(B,\mathbb{R})$ and $G\in C^\infty(B\times Y,\mathbb{R})$ with $\int_{\{z\}\times Y}G\Upsilon_z=0$ and $\int_{\{z\}\times Y}G^2\Upsilon_z=1$
for all $z\in B$, we constructed in Section \ref{grz} a function $\mathfrak{G}_{t,k}(A,G)\in C^\infty(B\times Y,\mathbb{R})$ with $\int_{\{z\}\times Y}\mathfrak{G}_{t,k}(A,G)\Upsilon_z=0$ for all $z\in B$, which is in some sense an approximate right inverse of $\Delta^{\omega^\natural_t}$ applied to $AG$ (cf. Lemma \ref{paraminatrix}).

Now, for all $i<j$ we suppose that we have smooth functions $G_{i,p,k}\in C^\infty(B\times Y,\mathbb{R}), 1\leq p\leq N_{i,k}$, which have fiberwise average zero and are fiberwise $L^2$ orthonormal, with $G_{i,p,k}=0$ when $i=0,1$. The main goal is to find smooth functions $G_{j,p,k}$ which satisfy a certain property that takes some work to describe.

First, we let $\delta_t>0$ be any fixed family of scalars which goes to zero as $t\to\infty$ but so that $\lambda_t:=\delta_te^{\frac{t}{2}}\to\infty$.

Consider now the diffeomorphisms as in \eqref{sigmat2}
\begin{equation}\label{sigmat}
\Sigma_t: B_{e^{\frac{t}{2}}} \times Y \to B \times Y, \;\, (z,y) = \Sigma_t(\check{z},\check{y}) =(e^{-\frac{t}{2}}\check{z},\check{y}),
\end{equation}
and for any function $u$ on $B\times Y$ we will write $\check{u}=\Sigma_t^*u$, and for a $2$-form $\alpha$ we will write $\check{\alpha}=e^t\Sigma_t^*\alpha$. In particular, note that $\check{\omega}_{\rm can}=e^t\Sigma_t^*\omega_{\rm can}$ is a K\"ahler metric uniformly equivalent to Euclidean (independent of $t$).

We will also need an intermediate ``hat picture'', obtained by factoring $\Sigma_t=\Psi_t\circ\Xi_t$ where
\begin{equation}
\Xi_t: B_{e^{\frac{t}{2}}} \times Y \to B_{\lambda_t} \times Y, \;\, (\hat{z},\hat{y}) = \Xi_t(\check{z},\check{y}) =(\delta_t\check{z},\check{y}),\end{equation}
\begin{equation}
\Psi_t: B_{\lambda_t} \times Y \to B \times Y, \;\, (z,y) = \Psi_t(\hat{z},\hat{y}) =(\lambda_t^{-1}\hat{z},\hat{y}),\end{equation}
and again given a function $u$ on $B\times Y$ we will write $\hat{u}=\Psi_t^*u$, and for a $2$-form $\alpha$ we will write $\hat{\alpha}=\lambda_t^2\Psi_t^*\alpha$.

We need a few more pieces of data.
Let $\eta^\ddagger_t$ be an arbitrary smooth family of $(1,1)$-forms on $B$ with coefficients polynomials of degree at most $j$ which satisfy $\hat{\eta}^\ddagger_t\to 0$ locally smoothly for $j>0$ (note that this implies that $\check{\eta}^\ddagger_t\to 0$ locally smoothly as well), and which converge to some constant $(1,1)$-form for $j=0$, and for $2\leq i\leq j$ let $A^\sharp_{t,i,p,k}$ be arbitrary polynomials of degree at most $j$ on $B$ such that $\hat{A}^\sharp_{t,i,p,k}=\lambda_t^2\Psi_t^*A^\sharp_{t,i,p,k}$ and $\check{A}^\sharp_{t,i,p,k}=e^t\Sigma_t^*A^\sharp_{t,i,p,k}$ satisfy that there is some $0<\alpha_0<1$ such that given any $R>0$ there is $C>0$ with
\begin{equation}\label{linftyyy0}
\|\D^\iota \hat{A}^\sharp_{t,i,p,k}\|_{L^\infty(\hat{B}_{R}\times Y,\hat{g}_t)}\leq C\delta_t^2 e^{-\alpha_0\frac{t}{2}},
\end{equation}
or equivalently
\begin{equation}\label{linftyyy}
\|\D^\iota \check{A}^\sharp_{t,i,p,k}\|_{L^\infty(\check{B}_{R\delta_t^{-1}}\times Y,\check{g}_t)}\leq C\delta_t^\iota e^{-\alpha_0\frac{t}{2}},
\end{equation}
for all $0\leq \iota\leq j$ and $t\geq 0$. With these, we define for $2\leq i\leq j$
\begin{equation}\label{dugger}
\gamma^{\sharp}_{t,i,k}=\sum_{p=1}^{N_{i,k}}\ddbar\mathfrak{G}_{t,k}(A^\sharp_{t,i,p,k},G_{i,p,k}),
\end{equation}
(compare \eqref{nuov} below), so $\gamma^{\sharp}_{t,j,k}$ depends on how we choose the functions $G_{j,p,k}$. It is proved by the arguments in \eqref{crocifisso4}, \eqref{crocifisso4bis} below that our assumption \eqref{linftyyy0} on $A^\sharp_{t,i,p,k}$ implies that for any $R>0$ there is $C>0$ with
\begin{equation}\label{crocifisso4hatto}
\|\D^\iota\hat{\gamma}^{\sharp}_{t,i,k}\|_{L^\infty(\hat{B}_{R}\times Y,\hat{g}_t)}\leq C\delta_t^{-\iota}e^{-\alpha_0 \frac{t}{2}},
\end{equation}
or equivalently
\begin{equation}\label{crocifisso4check}
\|\D^\iota\check{\gamma}^{\sharp}_{t,i,k}\|_{L^\infty(\check{B}_{R\delta_t^{-1}}\times Y,\check{g}_t)}\leq Ce^{-\alpha_0 \frac{t}{2}},
\end{equation}
for all $\iota\geq 0, t\geq 0$ and all $2\leq i\leq j$ 
(these estimates can be just taken as an assumption for now). Observe that the constants $C$ in \eqref{crocifisso4hatto}, \eqref{crocifisso4check} depend on the choice of the functions $G_{i,p,k}$ but the exponent $\alpha_0$ does not. We also define
\begin{equation}
\eta^\dagger_t=\sum_{i=2}^j\gamma^{\sharp}_{t,i,k},
\end{equation}
and
\begin{equation}
\omega^\sharp_t=\omega_{\rm can}+e^{-t}\omega_F+\eta^\dagger_t+\eta^\ddagger_t,
\end{equation}
and we assume that $\hat{\omega}^\sharp_t$ is a K\"ahler metric on $B_R\times Y$ for all $R>0$ and $t$ sufficiently large. This assumption is in fact automatically satisfied when $j>0$, using that $\hat{\eta}^\ddagger_t$ is pulled back from $B$ and goes to zero locally uniformly, and the estimate \eqref{crocifisso4hatto} with $\iota=0$ for $\hat{\eta}^\dagger_t$.

There is one more piece of background data, namely positive constants $c_t$ which are a polynomial in $e^{-t}$ of degree at most $m$ with constant coefficient $\binom{m+n}{n}$.

The key quantity we are interested in is then
\begin{equation}\label{stronzo}
\delta_t^{-j-\alpha}\left(c_t \frac{\check{\omega}_{\rm can}^m\wedge(\Sigma_t^*\omega_F)^n}{(\check\omega_t^\sharp)^{m+n}}-1\right).
\end{equation}

The following Selection Theorem allows us to find the desired obstruction functions:

\begin{theorem}[Selection Theorem]\label{ghost}
Suppose we are given $0\leq j\leq k$ and when $j>2$ we are also given smooth functions $G_{i,p,k}, 2\leq i\leq j-1, 1\leq p\leq N_{i,k},$ on $B\times Y$ which are fiberwise $L^2$ orthonormal, and have fiberwise average zero. Then there are a concentric ball $B'\subset B$ and smooth functions $G_{j,p,k}, 1\leq p\leq N_{j,k},$ on $B'\times Y$ (identically zero if $j=0,1$), with fiberwise average zero so that the $G_{i,p,k}, 2\leq i\leq j$ are all fiberwise $L^2$ orthonormal, with the following property:
if $\delta_t$ is any family of scalars with $\delta_t\to 0$ and $\delta_te^{\frac{t}{2}}\to\infty$, and if $A^\sharp_{t,i,p,k}, \eta^\dagger_t,\eta^\ddagger_t,\omega^\sharp_t,c_t$ are as above, and if \eqref{stronzo} converges locally uniformly on $\C^m\times Y$ to some limiting function $\mathcal{F}$,
then on $\Sigma_t^{-1}(B'\times Y)$ we can write \eqref{stronzo} as
\begin{equation}\label{stronza}
\delta_t^{-j-\alpha}\Sigma_t^*\left(f_{t,0}+\sum_{i=2}^{j}\sum_{p=1}^{N_{i,k}} f_{t,i,p}G_{i,p,k}\right) + o(1),
\end{equation}
where $f_{t,0},f_{t,i,p}$ are functions pulled back from $B'$ such that $\hat{f}_{t,0}=\Psi_t^*f_{t,0}$ and $\hat{f}_{t,i,p}=\Psi_t^*f_{t,i,p}$ converge locally smoothly to zero when $j>0$ (and which converge smoothly to some limit when $j=0$), and $o(1)$ is a term that converges locally smoothly to zero. Lastly, \eqref{stronzo} converges to $\mathcal{F}$ locally smoothly.
\end{theorem}
The functions $f_{t,0},f_{t,i,p}$ also depend on $k$, but for simplicity we have not made this explicit in the notation. We will refer to the functions $G_{i,p,k}$ as the obstruction functions. By definition they only arise for $i\geq 2$. The Selection Theorem is a cornerstone of our arguments, and will be used twice in the proof of our main results by blowup. The first time it will be employed in a fast-forming singularity scenario (subcase A below), where we will only use the improvement of regularity that \eqref{stronzo} having a uniform limit implies that it converges locally smoothly. The second time we will be in a regular-forming scenario (subcase B below), and there the full strength of \eqref{stronza} will be used. The crucial input which makes Theorem \ref{ghost} work is the exponential decay assumption in \eqref{crocifisso4check}.

\begin{rk}
It will be clear from the proof of the Selection Theorem \ref{ghost} that if instead of assuming that \eqref{stronzo} converges locally uniformly we just assume that it is locally uniformly bounded, then the conclusion \eqref{stronza} still holds, and the quantity in \eqref{stronzo} is locally uniformly bounded in all $C^k$ norms.
\end{rk}

\subsection{Proof of the Selection Theorem \ref{ghost}}
The proof is by induction on $j$, and since the base case $j=0$ and the next case $j=1$ are treated identically we present them together.
\subsubsection{The cases $j=0,1$}

In this case what we need to prove is simply that if \eqref{stronzo} converges locally uniformly to a limit then \eqref{stronza} holds with all $G_{i,p,k}=0$, and \eqref{stronzo} converges locally smoothly. By definition, when $j<2$ we have $\eta^\dagger_t=0$.
Since $\delta_te^{\frac{t}{2}}\to \infty$ and $j<2$, it follows that
$\delta_t^{-j-\alpha}e^{-t}\to 0.$ In particular, up to an error which is $O(\delta_t^{-j-\alpha}e^{-t})=o(1)$, we can replace $c_t$ in \eqref{stronzo} by $\binom{m+n}{n}$ (its leading term). Next, we can write with somewhat imprecise notation
\begin{equation}\label{imprecise}
\Sigma_t^*\omega_F=(\Sigma_t^*\omega_F)_{\mathbf{bb}}+(\Sigma_t^*\omega_F)_{\mathbf{bf}}+(\Sigma_t^*\omega_F)_{\mathbf{ff}}=:
e^{-t}\check\omega_{F,\mathbf{bb}}+e^{-\frac{t}{2}}\check\omega_{F,\mathbf{bf}}+\check\omega_{F,\mathbf{ff}},
\end{equation}
where the functions $\check\omega_{F,\mathbf{bb}},\check\omega_{F,\mathbf{bf}},\check\omega_{F,\mathbf{ff}}$ so defined are uniformly bounded on $B_{e^{\frac{t}{2}}}\times Y$. We then observe that
\begin{equation}\begin{split}
(\check\omega_t^\sharp)^{m+n}&=(\check{\omega}_{\rm can}+\Sigma_t^*\omega_F+\check{\eta}^\ddagger_t)^{m+n}=(\check{\omega}_{\rm can}+e^{-t}\check\omega_{F,\mathbf{bb}}+\check\omega_{F,\mathbf{ff}}+\check{\eta}^\ddagger_t)^{m+n}+O(e^{-t})\\
&=\binom{m+n}{n}(\check{\omega}_{\rm can}+e^{-t}\check\omega_{F,\mathbf{bb}}+\check{\eta}^\ddagger_t)^{m}(\Sigma_t^*\omega_F)_{\mathbf{ff}}^n+O(e^{-t}),
\end{split}\end{equation}
and so
\begin{equation}\begin{split}
\frac{\binom{m+n}{n}\check{\omega}_{\rm can}^m\wedge(\Sigma_t^*\omega_F)^n-(\check\omega_t^\sharp)^{m+n}}{(\check\omega_t^\sharp)^{m+n}}
&=\frac{\check{\omega}_{\rm can}^m\wedge(\Sigma_t^*\omega_F)^n-(\check{\omega}_{\rm can}+e^{-t}\check\omega_{F,\mathbf{bb}}+\check{\eta}^\ddagger_t)^{m}(\Sigma_t^*\omega_F)_{\mathbf{ff}}^n}{(\check{\omega}_{\rm can}+e^{-t}\check\omega_{F,\mathbf{bb}}+\check{\eta}^\ddagger_t)^{m}(\Sigma_t^*\omega_F)_{\mathbf{ff}}^n}+O(e^{-t})\\
&=\frac{\check{\omega}_{\rm can}^m-(\check{\omega}_{\rm can}+e^{-t}\check\omega_{F,\mathbf{bb}}+\check{\eta}^\ddagger_t)^m}{(\check{\omega}_{\rm can}+e^{-t}\check\omega_{F,\mathbf{bb}}+\check{\eta}^\ddagger_t)^m}+O(e^{-t})\\
&=\Sigma_t^*f_t+O(e^{-t}),
\end{split}\end{equation}
where the $O(e^{-t})$ is in the locally smooth topology and $f_t$ is some family of smooth functions pulled back from $B$. Passing to the hat picture, as described above, we let $\hat{f}_t=\Psi_t^*f_t$, so that $\Sigma_t^*f_t=\Xi_t^*\hat{f}_t$. When $j>0$ we have that $\hat{\eta}^\ddagger_t\to 0$ locally smoothly, and so it follows that $\hat{f}_t$ converge locally smoothly to zero, while when $j=0$ the constant forms $\hat{\eta}^\ddagger_t$ converge to some limit constant $(1,1)$-form on $B$, hence $\hat{f}_t$ converge locally smoothly to some limit. Multiplying this by $\delta_t^{-j-\alpha}$ and using that
$\delta_t^{-j-\alpha}e^{-t}=o(1)$ we see that indeed \eqref{stronza} holds with $G_{i,p,k}=0$, and of course the limit of \eqref{stronzo} is purely from the base. Lastly, since by assumption \eqref{stronzo} converges locally uniformly to $\mathcal{F}$, the same is true for $\delta_t^{-j-\alpha}\Sigma_t^*f_t$, and we need to prove that this convergence is locally smooth.

For this, we work in $\C^m$ in the hat picture, and expand in Taylor series
\begin{equation}\label{herrtaylor}
\hat{f}_t(\hat{z})=\sum_{|I|\leq j}\hat{a}_{I,t} \hat{z}^I+\sum_{|I|=j+1} \hat{b}_{I,t}(\hat{z}) \hat{z}^I=:\hat{p}_t(\hat{z})+\hat{r}_t(\hat{z}),
\end{equation}
where the Taylor coefficients $\hat{a}_{I,t}$ go to zero as $t\to\infty$ when $j>0$ and are bounded when $j=0$, and where the coefficients $\hat{b}_{I,t}(\hat{z})$ of the remainder $\hat{r}_t(\hat{z})$ are given by the usual formula
\begin{equation}
\hat{b}_{I,t}(\hat{z})=\frac{j+1}{I!}\int_0^1(1-s)^j D^I\hat{f}_t(s\hat{z})ds,
\end{equation}
from which we see that they go to zero locally smoothly as $t\to\infty$ when $j>0$ and are bounded when $j=0$. Then
\begin{equation}\label{resto}
(\Sigma_t^*f_t)(\check{z})=\hat{f}_t(\delta_t\check{z})=\sum_{|I|\leq j} \hat{a}_{I,t}\delta_t^{|I|}\check{z}^I+ \delta_t^{j+1}\sum_{|I|=j+1} \hat{b}_{I,t}(\delta_t \check{z}) \check{z}^I=(\Xi_t^*\hat{p}_t)(\check{z})+(\Xi_t^*\hat{r}_t)(\check{z}),
\end{equation}
from which we see that $\delta_t^{-j-\alpha}\Xi_t^*\hat{r}_t\to 0$ locally smoothly. Since by assumption $\delta_t^{-j-\alpha}\Sigma_t^*f_t$ converges locally uniformly to $\mathcal{F}$, it follows that the polynomials $\delta_t^{-j-\alpha}\Xi_t^*\hat{p}_t$ also converges locally uniformly to $\mathcal{F}$, and by Lemma \ref{mushy} this convergence is locally smooth as desired.

\subsubsection{The case $j\geq 2$: the initial list.}

As in the statement of the Selection Theorem \ref{ghost}, we suppose that the smooth functions $G_{i,p,k}$ with $2\leq i<j$ have already been selected on $B'\times Y$ (they are fiberwise orthonormal, and fiberwise orthogonal to the constants), and we now need to select the $G_{j,p,k}$'s. The process goes as follows: first we will select an initial list of $G_{j,p,k}$'s (this is the content of the current subsection). Then we will construct in \S \ref{iterr} a procedure that given these functions, it gives us some new functions to add to the list. This procedure will then need to be iterated a number of times (shrinking the ball $B'$ at each step without changing its notation) until the final list of functions is complete, and we then verify in \S \ref{veriff} that these satisfy the conclusion of the Selection Theorem \ref{ghost}.

We introduce a new index $\kappa$ to remember the number of the iteration that we are in, so initially $\kappa=0$. For each $\kappa$ we suppose we are given a ``list of obstruction functions at generations up to $\kappa$'' $G_{j,p,k}^{[q]}, 1\leq q\leq \kappa, 1\leq p\leq N^{[q]}_{j,k}$ which are smooth functions on $B'\times Y$ with fiberwise average zero and fiberwise orthonormal, with this list being the empty list when $\kappa=0$ (and $G_{j,p,k}^{[q]}$ depend only on $q$, not on $\kappa$). As in the statement of the Selection Theorem \ref{ghost}, these functions will be chosen independently of the auxiliary data $\delta_t,\eta^\ddagger_t,A^\sharp_{t,i,p,k},c_t$. Given these functions, and given also stretched polynomials $\check{A}^{\sharp,[q]}_{t,j,p,k}, 1\leq q\leq\kappa,$ of degree at most $j$ satisfying the assumption \eqref{linftyyy}, we construct as in \eqref{dugger}
\begin{equation}
\check{\gamma}^{\sharp,[q]}_{t,j,k}=\sum_{p=1}^{N^{[q]}_{j,k}}\ddbar\check{\mathfrak{G}}_{t,k}(\check{A}^{\sharp,[q]}_{t,j,p,k},\check{G}^{[q]}_{j,p,k}),
\end{equation}
for $1\leq q\leq\kappa$ (and set $\check{\gamma}^{\sharp,[0]}_{t,j,k}=0$), and let
\begin{equation}
\check{\omega}^{\sharp,[\kappa]}_t=\check{\omega}_{\rm can}+\Sigma_t^*\omega_F+\check{\gamma}^{\sharp}_{t,2,k}+\cdots+\check{\gamma}^\sharp_{t,j-1,k}+
\check{\gamma}^{\sharp,[1]}_{t,j,k}+\cdots+\check{\gamma}^{\sharp,[\kappa]}_{t,j,k}+\check{\eta}^\ddagger_t,
\end{equation}
which is a K\"ahler metric on $B\times Y$ for all $t$ large, using the bounds \eqref{crocifisso4check} for $\check{\gamma}^{\sharp,[q]}_{t,j,k}$. Of course, $\check{\omega}^{\sharp,[\kappa]}_t$ also depends on $j,k$, but we will not make this explicit in the notation for simplicity.

For each $\kappa\geq 0$ we will then consider the function
\begin{equation}\label{kulprit}
\mathcal{B}_t^{[\kappa]}:=\delta_t^{-j-\alpha}\left(c_t\frac{\check{\omega}_{\rm can}^m\wedge(\Sigma_t^*\omega_F)^n}{(\check\omega^{\sharp,[\kappa]}_t)^{m+n}}-1\right),
\end{equation}
which again implicitly also depends on $j,k$.
For convenience, we will say that a $t$-dependent function on $\Sigma_t^{-1}(B'\times Y)$ satisfies condition $(\star)$ if it is equal to
\begin{equation}\label{starr}
\delta_t^{-j-\alpha}\Sigma_t^*\left(f_{t,0}+\sum_{i=1}^N f_{t,i}h_i\right) + o(1),
\end{equation}
where the functions $f_{t,0},f_{t,i}$ are smooth and pulled back from $B'$ (and smooth in $t$), the functions $h_i$ are smooth on $B'\times Y$ with fiberwise average zero (with respect to $(\omega_F|_{\{z\}\times Y})^n$), the functions  $\hat{f}_{t,0}=\Psi_t^*f_{t,0}, \hat{f}_{t,i}=\Psi_t^*f_{t,i}$ converge locally smoothly to zero, and the $o(1)$ is a term that converges locally smoothly to zero. This definition is tailored to our desired conclusion in \eqref{stronza}.

\begin{rk}\label{usami}
Observe that if a $t$-dependent function satisfies $(\star)$ and it converges locally uniformly on $\C^m\times Y$ as $t\to\infty$, then this convergence is actually locally smooth. To see this, apply Proposition \ref{iamdeathincarnate} to these (with parameter $J=j+1$, and with the empty list of functions $H_k$), so up to shrinking $B'$ we may assume that the $h_i$ are fiberwise $L^2$ orthonormal (the errors coming from $\delta_t^{-j-\alpha}\Sigma_t^*(z^\alpha K_{i,\alpha})$ in \eqref{buddha} go to zero locally smoothly thanks to Remark \ref{miservi}, since $\delta_t^{-1}e^{-\frac{t}{2}}=o(1)$).
Then, as in \eqref{herrtaylor}, we expand in Taylor series in $\C^m$ in the hat picture
\begin{equation}\label{herrtaylor2}
\hat{f}_{t,i}(\hat{z})=\sum_{|I|\leq j}\hat{a}_{i,I,t} \hat{z}^I+\sum_{|I|=j+1} \hat{b}_{i,I,t}(\hat{z}) \hat{z}^I=:\hat{p}_{t,i}(\hat{z})+\hat{r}_{t,i}(\hat{z}),\quad 0\leq i\leq N,
\end{equation}
where the coefficients $\hat{a}_{i,I,t},\hat{b}_{i,I,t}(\hat{z})$ go to zero locally smoothly as $t\to\infty$, and as in \eqref{resto} we have that
$\delta_t^{-j-\alpha}\Xi_t^*\hat{r}_{t,i}\to 0$ locally smoothly. It follows that
\begin{equation}
\delta_t^{-j-\alpha}\Sigma_t^*\left(r_{t,0}+\sum_{i=1}^N r_{t,i}h_i\right)=\delta_t^{-j-\alpha}\Xi_t^*\left(\hat{r}_{t,0}+\sum_{i=1}^N \hat{r}_{t,i}\hat{h}_i\right)\to 0,
\end{equation}
locally smoothly as well, and so
\begin{equation}
\delta_t^{-j-\alpha}\Sigma_t^*\left(p_{t,0}+\sum_{i=1}^N p_{t,i}h_i\right)=\delta_t^{-j-\alpha}\Xi_t^*\left(\hat{p}_{t,0}+\sum_{i=1}^N \hat{p}_{t,i}\hat{h}_i\right)
\end{equation}
converges locally uniformly. Taking the fiberwise $L^2$ inner product against the constant $1$ and against each $h_i$ shows that $\delta_t^{-j-\alpha}\Sigma_t^*p_{t,i}, 0\leq i\leq N,$ also converges locally uniformly, and by Lemma \ref{mushy} this convergence is locally smooth. It then follows that $\delta_t^{-j-\alpha}\Sigma_t^*\left(p_{t,0}+\sum_{i=1}^N p_{t,i}h_i\right)$ converges locally smoothly, as desired.
\end{rk}

\begin{lemma}\label{solomon}
Suppose either $\kappa=0$ or $\kappa\geq 1$ and we have selected the functions $G_{j,p,k}^{[q]}$ as above for $1\leq q\leq \kappa.$ Then the function $\mathcal{B}_t^{[\kappa]}$ satisfies $(\star)$. Furthermore, we have
\begin{equation}\label{wulprit}
\delta_t^{j+\alpha}\mathcal{B}_t^{[\kappa]}=O(e^{-\alpha_0\frac{t}{2}})+o(1)_{\rm from\ base},
\end{equation}
where the term $O(e^{-\alpha_0\frac{t}{2}})$ is in $L^\infty_{\rm loc}$, while the last term is a function from the base which goes to zero locally smoothly.
\end{lemma}

\begin{proof}
For clarity we will split the proof into two steps.
For ease of notation, let us write
\begin{equation}
\omega^\square_t=\omega_{\rm can}+e^{-t}\omega_F+\eta^\ddagger_t,
\end{equation}
which is such that $\hat{\omega}^\square_t$ is a K\"ahler metric on $B_R\times Y$ for all $R>0$ and $t$ large,
and
\begin{equation}
\mathcal{C}_t=\left(c_t\frac{\omega_{\rm can}^m\wedge(e^{-t}\omega_F)^n}{(\omega^\square_t)^{m+n}}-1\right).
\end{equation}
The first step is to show that $\delta_t^{-j-\alpha}\Sigma_t^*\mathcal{C}_t$ satisfies $(\star)$.
Observe that
\begin{equation}\begin{split}
(\omega_t^\square)^{m+n}&=(\omega_{\rm can}+e^{-t}\omega_F+\eta^\ddagger_t)^{m+n}=
\binom{m+n}{n}e^{-nt}(\omega_{\rm can}+e^{-t}\omega_{F,\mathbf{bb}}+\eta^\ddagger_t)^m\wedge\omega_{F,\mathbf{ff}}^n\\
&+\sum_{q=1}^m
e^{-(n+q)t}\binom{m+n}{n+q}(\omega_{\rm can}+e^{-t}\omega_{F,\mathbf{bb}}+\eta^\ddagger_t)^{m-q}\wedge(\omega_{F,\mathbf{ff}}+\omega_{F,\mathbf{bf}})^{n+q}\\
&=:\binom{m+n}{n}e^{-nt}\left[(\omega_{\rm can}+e^{-t}\omega_{F,\mathbf{bb}}+\eta^\ddagger_t)^m\wedge\omega_{F,\mathbf{ff}}^n+e^{-t}\mathcal{D}_t\right],
\end{split}\end{equation}
and recalling that $c_t=\binom{m+n}{n}(1+e^{-t}\ti{c}_t)$ with $\ti{c}_t$ smoothly convergent, we can write $\mathcal{C}_t$ as
\begin{equation}\label{verylong}\begin{split}
&\frac{c_te^{-nt}\omega_{\rm can}^m\wedge\omega_F^n-(\omega^\square_t)^{m+n}}{(\omega^\square_t)^{m+n}}
=\frac{(1+e^{-t}\ti{c}_t)\omega_{\rm can}^m\wedge\omega_{F,\mathbf{ff}}^n
-(\omega_{\rm can}+e^{-t}\omega_{F,\mathbf{bb}}+\eta^\ddagger_t)^{m}
\wedge\omega_{F,\mathbf{ff}}^{n}-e^{-t}\mathcal{D}_t}{(\omega_{\rm can}+e^{-t}\omega_{F,\mathbf{bb}}+\eta^\ddagger_t)^{m}
\wedge\omega_{F,\mathbf{ff}}^{n}+e^{-t}\mathcal{D}_t}\\
&=\frac{(1+e^{-t}\ti{c}_t)\omega_{\rm can}^m-(\omega_{\rm can}+e^{-t}\omega_{F,\mathbf{bb}}+\eta^\ddagger_t)^m}{(\omega_{\rm can}+e^{-t}\omega_{F,\mathbf{bb}}+\eta^\ddagger_t)^m}\left(1+\frac{e^{-t}\mathcal{D}_t}{(\omega_{\rm can}+e^{-t}\omega_{F,\mathbf{bb}}+\eta^\ddagger_t)^{m}
\wedge\omega_{F,\mathbf{ff}}^{n}}\right)^{-1}\\
&-\frac{e^{-t}\mathcal{D}_t}{(\omega_{\rm can}+e^{-t}\omega_{F,\mathbf{bb}}+\eta^\ddagger_t)^{m}
\wedge\omega_{F,\mathbf{ff}}^{n}}\left(1+\frac{e^{-t}\mathcal{D}_t}{(\omega_{\rm can}+e^{-t}\omega_{F,\mathbf{bb}}+\eta^\ddagger_t)^{m}
\wedge\omega_{F,\mathbf{ff}}^{n}}\right)^{-1}\\
&=\left(e^{-t}\ti{c}_t+\sum_{q=0}^m\sum_{p=0}^{m-1}\frac{m!}{p!q!(m-p-q)!}\frac{\omega_{\rm can}^{p}e^{-qt}\omega_{F,\mathbf{bb}}^q(\eta^\ddagger_t)^{m-p-q}}{\omega_{\rm can}^m}\right)\\
&\left(1+\sum_{q=0}^m\sum_{p=0}^{m-1}\frac{m!}{p!q!(m-p-q)!}\frac{\omega_{\rm can}^{p}e^{-qt}\omega_{F,\mathbf{bb}}^q(\eta^\ddagger_t)^{m-p-q}}{\omega_{\rm can}^m}\right)^{-1}
\left(1+\frac{e^{-t}\mathcal{D}_t}{(\omega_{\rm can}+e^{-t}\omega_{F,\mathbf{bb}}+\eta^\ddagger_t)^{m}
\wedge\omega_{F,\mathbf{ff}}^{n}}\right)^{-1}\\
&-\frac{e^{-t}\mathcal{D}_t}{(\omega_{\rm can}+e^{-t}\omega_{F,\mathbf{bb}}+\eta^\ddagger_t)^{m}
\wedge\omega_{F,\mathbf{ff}}^{n}}\left(1+\frac{e^{-t}\mathcal{D}_t}{(\omega_{\rm can}+e^{-t}\omega_{F,\mathbf{bb}}+\eta^\ddagger_t)^{m}
\wedge\omega_{F,\mathbf{ff}}^{n}}\right)^{-1},\\
\end{split}\end{equation}
and from this we can see that $\delta_t^{-j-\alpha}\Sigma_t^*\mathcal{C}_t$ satisfies $(\star)$ as follows. First note that the functions
\begin{equation}\label{culprit}
\frac{\omega_{\rm can}^{p}e^{-qt}\omega_{F,\mathbf{bb}}^q(\eta^\ddagger_t)^{m-p-q}}{\omega_{\rm can}^m},
\end{equation}
with $q=0$ are pulled back from $B'$, while when $q>0$ this is not the case, but they are visibly of the form
$f_{t,0}+\sum_{i=1}^N f_{t,i}h_i$ with the same notation as above, where the $\hat{f}_{t,0},\hat{f}_{t,i}$ converge smoothly to zero at least as $O(e^{-qt})$, and the functions $h_i$ have fiberwise average zero and do not depend on the choice of $\eta^\ddagger_t$. An analogous statement holds for
\begin{equation}
\frac{e^{-t}\mathcal{D}_t}{(\omega_{\rm can}+e^{-t}\omega_{F,\mathbf{bb}}+\eta^\ddagger_t)^{m}
\wedge\omega_{F,\mathbf{ff}}^{n}},
\end{equation}
and so expanding $(1+x)^{-1}$ using the geometric sum, we thus see that $\delta_t^{-j-\alpha}\Sigma_t^*\mathcal{C}_t$ satisfies $(\star)$, and that
\begin{equation}
\mathcal{C}_t=O(e^{-t})+o(1)_{\rm from\ base},
\end{equation}
where $O(e^{-t})$ is in $L^\infty_{\rm loc}$.

The second step is then to write (using now the check picture)
\begin{equation}\label{werylong2}\begin{split}
&\delta_t^{j+\alpha}\mathcal{B}_t^{[\kappa]}\\
&=\Sigma_t^*\mathcal{C}_t-(\Sigma_t^*\mathcal{C}_t+1)\left(\frac{(\check\omega_t^{\sharp,[\kappa]})^{m+n}-(\check\omega_t^\square)^{m+n}}{(\check\omega_t^{\sharp,[\kappa]})^{m+n}}\right)\\
&=\Sigma_t^*\mathcal{C}_t-(\Sigma_t^*\mathcal{C}_t+1)\left(\sum_{i=1}^{m+n}\binom{m+n}{i}\frac{(\check{\gamma}^\sharp_{t,2,k}+\cdots+\check{\gamma}^\sharp_{t,j-1,k}+
\check{\gamma}^{\sharp,[1]}_{t,j,k}+\cdots+\check{\gamma}^{\sharp,[\kappa]}_{t,j,k}
)^i\wedge(\check\omega^\square_t)^{m+n-i}}{(\check\omega_t^\square)^{m+n}}\right)\\
&\left(1+\sum_{i=1}^{m+n}\binom{m+n}{i}\frac{(\check{\gamma}^\sharp_{t,2,k}+\cdots+\check{\gamma}^\sharp_{t,j-1,k}+
\check{\gamma}^{\sharp,[1]}_{t,j,k}+\cdots+\check{\gamma}^{\sharp,[\kappa]}_{t,j,k})^i\wedge(\check\omega_t^\square)^{m+n-i}}{(\check\omega_t^\square)^{m+n}}\right)^{-1},
\end{split}\end{equation}
and for all the terms with $\check{\gamma}^\sharp_{t,2,k}+\cdots+\check{\gamma}^\sharp_{t,j-1,k}+
\check{\gamma}^{\sharp,[1]}_{t,j,k}+\cdots+\check{\gamma}^{\sharp,[\kappa]}_{t,j,k}$ we recall from \eqref{dugger} that we have
\begin{equation}
\check{\gamma}^\sharp_{t,i,k}=\sum_{p=1}^{N_{i,k}}\ddbar\check{\mathfrak{G}}_{t,k}(\check{A}^\sharp_{t,i,p,k},\check{G}_{i,p,k}),\quad
\check{\gamma}^{\sharp,[q]}_{t,j,k}=\sum_{p=1}^{N^{[q]}_{j,k}}\ddbar\check{\mathfrak{G}}_{t,k}(\check{A}^{\sharp,[q]}_{t,j,p,k},\check{G}^{[q]}_{j,p,k}),
\end{equation}
with the bounds \eqref{crocifisso4hatto}, \eqref{crocifisso4check}, where the approximate Green operator $\check{\mathfrak{G}}_{t,k}$ is given schematically by
\begin{equation}\label{blobus}
\check{\mathfrak{G}}_{t,k}(\check{A},\check{G}) = \sum_{\iota=0}^{j}\sum_{\ell=\lceil \frac{\iota}{2} \rceil}^{k} e^{-\left(\ell-\frac{\iota}{2}\right) t}(\check{\Phi}_{\iota,\ell}(\check{G})\circledast \D^\iota \check{A}).
\end{equation}
by Lemma \ref{Gstructure} (since $\check{A}$ here is a polynomial of degree at most $j$). Plugging this into \eqref{werylong2} and arguing as we did above for $\delta_t^{-j-\alpha}\Sigma_t^*\mathcal{C}_t$ reveals that $\mathcal{B}_t^{[\kappa]}$ satisfies $(\star)$ and that \eqref{wulprit} holds (using also \eqref{linftyyy0} to show that $\hat{f}_{t,0},\hat{f}_{t,i}\to 0$ locally smoothly).
\end{proof}

For our initial list $G_{j,p,k}^{[1]}$ we would then naively like to take the $h_i$'s in \eqref{starr} for $\mathcal{B}_t^{[0]}$.  More precisely, we apply Proposition \ref{iamdeathincarnate} with the functions $F_i$ there equal to the $h_i$'s, and with the $H_k$ there equal to the $G_{i,p,k}, 2\leq i<j$ and with parameter $J=j+1$. Up to shrinking the ball $B'$ we thus obtain our desired list $G_{j,p,k}^{[1]}$ so that these together with the $G_{i,p,k}, 2\leq i<j$ are fiberwise orthonormal and orthogonal to the constants, and the $h_i$'s lie in the fiberwise linear span of the $G_{i,p,k}, 2\leq i<j$ together with the $G_{j,p,k}^{[1]}$ and the constants. Indeed, the errors coming from $\delta_t^{-j-\alpha}\Sigma_t^*(z^\alpha K_{i,\alpha})$ in \eqref{buddha} go to zero locally smoothly thanks to Remark \ref{miservi}, since $\delta_t^{-1}e^{-\frac{t}{2}}=o(1)$, and so they can be moved into the $o(1)$ term in the expansion \eqref{starr} for $\mathcal{B}_t^{[0]}$.

\subsubsection{The iterative procedure}\label{iterr}

Here for a given $\kappa\geq 1$ we assume we are given the list $G_{i,p,k}$ for $2\leq i<j$ and we have have constructed the list $G_{j,p,k}^{[q]}$ for $1\leq q\leq \kappa$, and hence we have the function $\mathcal{B}_t^{[\kappa]}$ in \eqref{kulprit}. Then we have
\begin{equation}\label{werylong}\begin{split}
\delta_t^{j+\alpha}(\mathcal{B}_t^{[\kappa]}-\mathcal{B}_t^{[\kappa-1]})=&-(\delta_t^{j+\alpha}\mathcal{B}_t^{[\kappa-1]}+1)\left(\frac{(\check\omega_t^{\sharp,[\kappa]})^{m+n}-(\check\omega_t^{\sharp,[\kappa-1]})^{m+n}}{(\check\omega_t^{\sharp,[\kappa]})^{m+n}}\right)\\
&=-(\delta_t^{j+\alpha}\mathcal{B}_t^{[\kappa-1]}+1)\left(\sum_{i=1}^{m+n}\binom{m+n}{i}\frac{(\check{\gamma}^{\sharp,[\kappa]}_{t,j,k}
)^i\wedge(\check\omega_t^{\sharp,[\kappa-1]})^{m+n-i}}{(\check\omega_t^{\sharp,[\kappa-1]})^{m+n}}\right)\\
&\left(1+\sum_{i=1}^{m+n}\binom{m+n}{i}\frac{(\check{\gamma}^{\sharp,[\kappa]}_{t,j,k})^i\wedge(\check\omega_t^{\sharp,[\kappa-1]})^{m+n-i}}{(\check\omega_t^{\sharp,[\kappa-1]})^{m+n}}\right)^{-1}.
\end{split}\end{equation}
Let us thus first look in detail at the term
\begin{equation}\label{wlog}
\sum_{i=1}^{m+n}\binom{m+n}{i}\frac{(\check{\gamma}^{\sharp,[\kappa]}_{t,j,k})^i\wedge(\check\omega_t^{\sharp,[\kappa-1]})^{m+n-i}}{(\check\omega_t^{\sharp,[\kappa-1]})^{m+n}}.
\end{equation}
Thanks to \eqref{crocifisso4check} we can write
\begin{equation}
\check\omega_t^{\sharp,[\kappa-1]}=\check{\omega}_{\rm can}+(\Sigma_t^*\omega_F)_{\mathbf{ff}}+O(e^{-\alpha_0\frac{t}{2}})+o(1)_{\rm from\ base},
\end{equation}
(here $O(\cdot), o(1)$ are in the locally smooth topology on $\C^m\times Y$) and so
\begin{equation}\begin{split}
\frac{(\check{\gamma}^{\sharp,[\kappa]}_{t,j,k})^i\wedge(\check\omega_t^{\sharp,[\kappa-1]})^{m+n-i}}{(\check\omega_t^{\sharp,[\kappa-1]})^{m+n}}
=\frac{(\check{\gamma}^{\sharp,[\kappa]}_{t,j,k})^i\wedge(\check{\omega}_{\rm can}+(\Sigma_t^*\omega_F)_{\mathbf{ff}})^{m+n-i}}{\binom{m+n}{n}\check{\omega}_{\rm can}^m\wedge(\Sigma_t^*\omega_F)^n_{\mathbf{ff}}}
(1+O(e^{-\alpha_0\frac{t}{2}})+o(1)_{\rm from\ base}).
\end{split}\end{equation}
We thus study the terms
\begin{equation}\label{vlog}
\frac{(\check{\gamma}^{\sharp,[\kappa]}_{t,j,k})^i\wedge(\check{\omega}_{\rm can}+(\Sigma_t^*\omega_F)_{\mathbf{ff}})^{m+n-i}}{\check{\omega}_{\rm can}^m\wedge(\Sigma_t^*\omega_F)^n_{\mathbf{ff}}},
\end{equation}
with $1\leq i\leq m+n$.
To do this, we use again Lemma \ref{Gstructure} and write schematically (for $1\leq q\leq \kappa$)
\begin{equation}\label{blob}
\check{\gamma}^{\sharp,[q]}_{t,j,k} = \sum_{p=1}^{N^{[q]}_{j,k}}\sum_{\iota=0}^{j}\sum_{\ell=\lceil \frac{\iota}{2} \rceil}^{k} e^{-\left(\ell-\frac{\iota}{2}\right) t}i\partial\overline\partial(\check{\Phi}_{\iota,\ell}(\check{G}^{[q]}_{j,p,k})\circledast \D^\iota \check{A}^{\sharp,[q]}_{t,j,p,k}),
\end{equation}
with the key relation
\begin{equation}\label{sangennaro2}
\check{\Phi}_{0,0}(\check{G})=(\Delta^{\omega_F|_{\{\cdot\}\times Y}})^{-1}\check{G},
\end{equation}
proved in \eqref{sangennaro}.
Decompose \eqref{blob} schematically into the sum of $6$ pieces
\begin{equation}\label{blobbe}\begin{split}
\check{\gamma}^{\sharp,[q]}_{t,j,k}&= \sum_{p=1}^{N^{[q]}_{j,k}}\sum_{\iota=0}^{j}\sum_{\ell=\lceil \frac{\iota}{2} \rceil}^{k} e^{-\left(\ell-\frac{\iota}{2}\right) t}\bigg\{(i\partial\overline\partial\check{\Phi}_{\iota,\ell}(\check{G}^{[q]}_{j,p,k}))_{\mathbf{ff}}\circledast \D^\iota \check{A}^{\sharp,[q]}_{t,j,p,k}
+
(i\partial\overline\partial\check{\Phi}_{\iota,\ell}(\check{G}^{[q]}_{j,p,k}))_{\mathbf{bf}}\circledast \D^\iota \check{A}^{\sharp,[q]}_{t,j,p,k}\\
&+
(i\partial\overline\partial\check{\Phi}_{\iota,\ell}(\check{G}^{[q]}_{j,p,k}))_{\mathbf{bb}}\circledast \D^\iota \check{A}^{\sharp,[q]}_{t,j,p,k}
+(i\partial\check{\Phi}_{\iota,\ell}(\check{G}^{[q]}_{j,p,k}))_{\mathbf{b}}\circledast \db\D^\iota \check{A}^{\sharp,[q]}_{t,j,p,k}
+(i\partial\check{\Phi}_{\iota,\ell}(\check{G}^{[q]}_{j,p,k}))_{\mathbf{f}}\circledast \db\D^\iota\check{A}^{\sharp,[q]}_{t,j,p,k}\\
&+\check{\Phi}_{\iota,\ell}(\check{G}^{[q]}_{j,p,k})\circledast i\de\db\D^\iota \check{A}^{\sharp,[q]}_{t,j,p,k}\bigg\}=:\sum_{\iota=0}^{j}\sum_{\ell=\lceil \frac{\iota}{2} \rceil}^{k} {\rm I}^{[q]}_{\iota,\ell}+\cdots+{\rm VI}^{[q]}_{\iota,\ell}.
\end{split}\end{equation}
Again the terms ${\rm I}^{[q]}_{\iota,\ell},\dots,{\rm VI}^{[q]}_{\iota,\ell}$ also depend on $j,k$ but we do not make this explicit.
Then we claim that the term in \eqref{wlog} is equal to (setting $q=\kappa$)
\begin{equation}\label{qulprit}\begin{split}
&\sum_{i=1}^{m+n}\binom{m+n}{i}\frac{(\check{\gamma}^{\sharp,[q]}_{t,j,k})^i\wedge(\check\omega_t^{\sharp,[q-1]})^{m+n-i}}{(\check\omega_t^{\sharp,[q-1]})^{m+n}}\\
&=\frac{(m+n)}{\binom{m+n}{n}}\frac{{\rm I}^{[q]}_{0,0}\wedge(\check{\omega}_{\rm can}+(\Sigma_t^*\omega_F)_{\mathbf{ff}})^{m+n-1}}{\check{\omega}_{\rm can}^m\wedge(\Sigma_t^*\omega_F)^n_{\mathbf{ff}}}
(1+O(e^{-\alpha_0\frac{t}{2}})+o(1)_{\rm from\ base})+F^{[q]}_t\\
&=\left(\sum_{p=1}^{N^{[q]}_{j,k}}\check{G}^{[q]}_{j,p,k}\check{A}^{\sharp,[q]}_{t,j,p,k}\right)
(1+O(e^{-\alpha_0\frac{t}{2}})+o(1)_{\rm from\ base})+ F^{[q]}_t,
\end{split}
\end{equation}
where given any $R>0$ there is a $C>0$ such that
\begin{equation}\label{qulprit2}
\|F^{[q]}_t\|_{L^\infty(\check{B}_{R\delta_t^{-1}}\times Y)}\leq C\delta_t^{\alpha_0}\|I^{[q]}_{0,0}\|_{L^\infty(\check{B}_{R\delta_t^{-1}}\times Y,\check{g}_t)},
\end{equation}
holds for all $t\geq 0$.

To prove \eqref{qulprit}, we first show that there is $C$ such that for all $t$ and all $R\leq C\delta_t^{-1}$ and all $1\leq q\leq \kappa$ we have
\begin{equation}\label{catilina}
C^{-1}\sum_{p=1}^{N^{[q]}_{j,k}}\|\check{A}^{\sharp,[q]}_{t,j,p,k}\|_{L^\infty(\check{B}_R)}
\leq \|{\rm I}^{[q]}_{0,0}\|_{L^\infty(\check{B}_R\times Y,\check{g}_t)}\leq C \sum_{p=1}^{N^{[q]}_{j,k}}\|\check{A}^{\sharp,[q]}_{t,j,p,k}\|_{L^\infty(\check{B}_R)}.
\end{equation}
The second inequality is obvious, as for the first one observe first that for every $z\in \check{B}_R$ the $(1,1)$-forms $\{\ddbar (\check{\Phi}_{0,0}(\check{G}^{[q]}_{j,p,k})|_{\{z\}\times Y})\}_{p}$ on $\{z\}\times Y$ are $\mathbb{R}$-linearly independent since using \eqref{sangennaro2} and taking the fiberwise trace, a nontrivial linear dependence among these would give a nonexistent linear dependence among the $\{\check{G}^{[q]}_{j,p,k}|_{\{z\}\times Y}\}_{p}$. But then
\begin{equation}
\|{\rm I}^{[q]}_{0,0}\|_{L^\infty(\{z\}\times Y,\check{g}_t)}=\sup_{\{z\}\times Y} \left|\sum_{p}\check{A}^{\sharp,[q]}_{t,j,p,k}(z)\ddbar (\check{\Phi}_{0,0}(\check{G}^{[q]}_{j,p,k})|_{\{z\}\times Y})\right|_{\omega_F|_{\{z\}\times Y}}
\end{equation}
and the RHS is a norm on the finite-dimensional $\mathbb{R}$-vector space of $(1,1)$-forms on $\{z\}\times Y$ spanned by $\{\ddbar (\check{\Phi}_{0,0}(\check{G}^{[q]}_{j,p,k})|_{\{z\}\times Y})\}_{p}$, which is therefore uniformly equivalent to $\sum_{p}|\check{A}^{\sharp,[q]}_{t,j,p,k}(z)|$, and \eqref{catilina} follows since for each $1\leq r\leq N^{[q]}_{j,k}$ we clearly have $\|\check{A}^{\sharp,[q]}_{t,j,r,k}\|_{L^\infty(\check{B}_R)}\leq \sup_{z\in \check{B}_R}\sum_{p}|\check{A}^{\sharp,[q]}_{t,j,p,k}(z)|$.

Next, for $(\iota,\ell)\neq (0,0)$ we can apply Lemma \ref{mushy} to balls of radius $R\delta_t^{-1}$ and get
\begin{equation}\begin{split}
\|{\rm I}^{[q]}_{\iota,\ell}\|_{L^\infty(\check{B}_{R\delta_t^{-1}}\times Y,\check{g}_t)}&\leq Ce^{-\left(\ell-\frac{\iota}{2}\right) t}\sum_{p}\|\D^\iota\check{A}^{\sharp,[q]}_{t,j,p,k}\|_{L^\infty(\check{B}_{R\delta_t^{-1}})}\leq Ce^{-\left(\ell-\frac{\iota}{2}\right)t}\delta_t^\iota\sum_{p}\|\check{A}^{\sharp,[q]}_{t,j,p,k}\|_{L^\infty(\check{B}_{R\delta_t^{-1}})}\\
&\leq Ce^{-\left(\ell-\frac{\iota}{2}\right)t}\delta_t^\iota\|{\rm I}^{[q]}_{0,0}\|_{L^\infty(\check{B}_{R\delta_t^{-1}}\times Y,\check{g}_t)},
\end{split}\end{equation}
using also \eqref{catilina}.

Arguing in the same way, and using also that $\check{\Phi}_{\iota,\ell}(\check{G}^{[q]}_{j,p,k})=\Sigma_t^*\Phi_{\iota,\ell}(G^{[q]}_{j,p,k})$ gains a factor of $e^{-\frac{t}{2}}$ for each base differentiation, we see that for all $\iota,\ell$ as in \eqref{blob} we have
\begin{equation}\begin{split}
\|{\rm II}^{[q]}_{\iota,\ell}\|_{L^\infty(\check{B}_{R\delta_t^{-1}}\times Y,\check{g}_t)}&\leq Ce^{-\frac{t}{2}}e^{-\left(\ell-\frac{\iota}{2}\right) t}\sum_{p}\|\D^\iota\check{A}^{\sharp,[q]}_{t,j,p,k}\|_{L^\infty(\check{B}_{R\delta_t^{-1}})}\leq Ce^{-\left(\ell-\frac{\iota}{2}+\frac{1}{2}\right)t}\delta_t^\iota\|{\rm I}^{[q]}_{0,0}\|_{L^\infty(\check{B}_{R\delta_t^{-1}}\times Y,\check{g}_t)},
\end{split}\end{equation}
and similarly
\begin{equation}
\|{\rm III}^{[q]}_{\iota,\ell}\|_{L^\infty(\check{B}_{R\delta_t^{-1}}\times Y,\check{g}_t)}\leq Ce^{-\left(\ell-\frac{\iota}{2}+1\right)t}\delta_t^\iota\|{\rm I}^{[q]}_{0,0}\|_{L^\infty(\check{B}_{R\delta_t^{-1}}\times Y,\check{g}_t)},
\end{equation}
\begin{equation}
\|{\rm IV}^{[q]}_{\iota,\ell}\|_{L^\infty(\check{B}_{R\delta_t^{-1}}\times Y,\check{g}_t)}\leq Ce^{-\left(\ell-\frac{\iota}{2}+\frac{1}{2}\right)t}\delta_t^{\iota+1}\|{\rm I}^{[q]}_{0,0}\|_{L^\infty(\check{B}_{R\delta_t^{-1}}\times Y,\check{g}_t)},
\end{equation}
\begin{equation}
\|{\rm V}^{[q]}_{\iota,\ell}\|_{L^\infty(\check{B}_{R\delta_t^{-1}}\times Y,\check{g}_t)}\leq Ce^{-\left(\ell-\frac{\iota}{2}\right)t}\delta_t^{\iota+1}\|{\rm I}^{[q]}_{0,0}\|_{L^\infty(\check{B}_{R\delta_t^{-1}}\times Y,\check{g}_t)},
\end{equation}
\begin{equation}
\|{\rm VI}^{[q]}_{\iota,\ell}\|_{L^\infty(\check{B}_{R\delta_t^{-1}}\times Y,\check{g}_t)}\leq Ce^{-\left(\ell-\frac{\iota}{2}\right)t}\delta_t^{\iota+2}\|{\rm I}^{[q]}_{0,0}\|_{L^\infty(\check{B}_{R\delta_t^{-1}}\times Y,\check{g}_t)},
\end{equation}
and so putting all of these together, and recalling that $\delta_t\gg e^{-\frac{t}{2}},$ we see that
\begin{equation}\label{zatan}
\|\bullet^{[q]}_{\iota,\ell}\|_{L^\infty(\check{B}_{R\delta_t^{-1}}\times Y,\check{g}_t)}\leq C\delta_t\|{\rm I}^{[q]}_{0,0}\|_{L^\infty(\check{B}_{R\delta_t^{-1}}\times Y,\check{g}_t)},
\end{equation}
whenever $\bullet\neq {\rm I}$ or $(\iota,\ell)\neq (0,0)$. Thus
\begin{equation}\label{reallyuseless}
\|\check{\gamma}^{\sharp,[q]}_{t,j,k}-{\rm I}^{[q]}_{0,0}\|_{L^\infty(\check{B}_{R\delta_t^{-1}}\times Y,\check{g}_t)}\leq C\delta_t\|{\rm I}^{[q]}_{0,0}\|_{L^\infty(\check{B}_{R\delta_t^{-1}}\times Y,\check{g}_t)}.
\end{equation}
Equation \eqref{reallyuseless} thus shows that to leading order only ${\rm I}^{[q]}_{0,0}$ contributes to $\check{\gamma}^{\sharp,[q]}_{t,j,k}$. For all $i\geq 2$, we can use \eqref{crocifisso4check} and obtain
\begin{equation}\label{zatan2}
\|(\check{\gamma}^{\sharp,[q]}_{t,j,k})^i\|_{L^\infty(\check{B}_{R\delta_t^{-1}}\times Y,\check{g}_t)}\leq  C\delta_t^{\alpha_0}
\|\check{\gamma}^{\sharp,[q]}_{t,j,k}\|_{L^\infty(\check{B}_{R\delta_t^{-1}}\times Y,\check{g}_t)}\leq C\delta_t^{\alpha_0}\|{\rm I}^{[q]}_{0,0}\|_{L^\infty(\check{B}_{R\delta_t^{-1}}\times Y,\check{g}_t)},
\end{equation}
and combining \eqref{zatan} and \eqref{zatan2} with \eqref{blobbe} easily gives the first equality in \eqref{qulprit} and \eqref{qulprit2}. The second equality in \eqref{qulprit} then follows from
\begin{equation}\label{sangue}
\frac{(m+n)}{\binom{m+n}{n}}\frac{{\rm I}^{[q]}_{0,0}\wedge(\check{\omega}_{\rm can}+(\Sigma_t^*\omega_F)_{\mathbf{ff}})^{m+n-1}}{\check{\omega}_{\rm can}^m\wedge(\Sigma_t^*\omega_F)^n_{\mathbf{ff}}}=\sum_{p=1}^{N^{[q]}_{j,k}}\check{G}^{[q]}_{j,p,k}\check{A}^{\sharp,[q]}_{t,j,p,k},
\end{equation}
which is a consequence of \eqref{sangennaro2}.

Next, combining \eqref{crocifisso4check} with \eqref{qulprit}, \eqref{qulprit2}, \eqref{catilina} gives
\begin{equation}\label{qulprit3}
\left\|\sum_{i=1}^{m+n}\binom{m+n}{i}\frac{(\check{\gamma}^{\sharp,[q]}_{t,j,k})^i\wedge(\check\omega_t^{\sharp,[q-1]})^{m+n-i}}{(\check\omega_t^{\sharp,[q-1]})^{m+n}}\right\|_{L^\infty(\check{B}_{R\delta_t^{-1}}\times Y,\check{g}_t)}\leq  C\delta_t^{\alpha_0},
\end{equation}
so inserting \eqref{wulprit}, \eqref{qulprit}, \eqref{qulprit2}, \eqref{qulprit3} into \eqref{werylong} gives
\begin{equation}\label{panacea}
\mathcal{B}_t^{[\kappa]}=\mathcal{B}_t^{[\kappa-1]}-\delta_t^{-j-\alpha}
\left(\sum_{p=1}^{N^{[\kappa]}_{j,k}}\check{G}^{[\kappa]}_{j,p,k}\check{A}^{\sharp,[\kappa]}_{t,j,p,k}\right)(1+o(1)_{\rm from\ base})
+\delta_t^{-j-\alpha}E^{[\kappa]}_t+o(1),
\end{equation}
where the error function $E^{[\kappa]}_t$ satisfies the estimate
\begin{equation}\label{panacea2}
\|E^{[\kappa]}_t\|_{L^\infty(\check{B}_{R\delta_t^{-1}}\times Y)}\leq C\delta_t^{\alpha_0}\|I^{[\kappa]}_{0,0}\|_{L^\infty(\check{B}_{R\delta_t^{-1}}\times Y,\check{g}_t)}.
\end{equation}
Since by Lemma \ref{solomon} the functions $\mathcal{B}_t^{[\kappa]}$ and $\mathcal{B}_t^{[\kappa-1]}$ satisfy $(\star)$, so does the error term $\delta_t^{-j-\alpha}E^{[\kappa]}_t$. We thus consider the functions $h_i$ that arise in the expression \eqref{starr} for $\delta_t^{-j-\alpha}E^{[\kappa]}_t$. We then apply Proposition \ref{iamdeathincarnate} with the functions $F_i$ there equal to the $h_i$'s, and with the $H_k$ there equal to the $\check{G}_{i,p,k}, 2\leq i<j$ together with the $\check{G}_{j,p,k}^{[q]}, 1\leq q\leq\kappa$, and with parameter $J=j+1$. Up to shrinking the ball $B'$ again we thus obtain our desired list $\check{G}_{j,p,k}^{[\kappa+1]}$ so that these together with the $\check{G}_{i,p,k}, 2\leq i<j$  and the $\check{G}_{j,p,k}^{[q]}, 1\leq q\leq\kappa$ are fiberwise orthonormal and orthogonal to the constants, and the functions $h_i$ in the expression \eqref{starr} for $\delta_t^{-j-\alpha}E^{[\kappa]}_t$ lie in the fiberwise linear span of all these $\check{G}$'s and the constants.  Again, the errors coming from $\delta_t^{-j-\alpha}\Sigma_t^*(z^\alpha K_{i,\alpha})$ in \eqref{buddha} go to zero locally smoothly thanks to Remark \ref{miservi}, since $\delta_t^{-1}e^{-\frac{t}{2}}=o(1)$, and so they can be moved into the $o(1)$ term in the expansion \eqref{starr} for $\delta_t^{-j-\alpha}E^{[\kappa]}_t$. This is the step from $\kappa$ to $\kappa+1$ in the iterative procedure to select these functions.

Iterating \eqref{panacea} gives
\begin{equation}\label{panacea3}
\mathcal{B}_t^{[\kappa]}=\mathcal{B}_t^{[0]}-\delta_t^{-j-\alpha}
\left(\sum_{q=1}^\kappa\sum_{p=1}^{N^{[q]}_{j,k}}\check{G}^{[q]}_{j,p,k}\check{A}^{\sharp,[q]}_{t,j,p,k}\right)(1+o(1)_{\rm from\ base})
+\delta_t^{-j-\alpha}\sum_{q=1}^\kappa E^{[q]}_t+o(1),
\end{equation}
with
\begin{equation}\label{panacea4}
\|E^{[q]}_t\|_{L^\infty(\check{B}_{R\delta_t^{-1}}\times Y)}\leq C\delta_t^{\alpha_0}\|I^{[q]}_{0,0}\|_{L^\infty(\check{B}_{R\delta_t^{-1}}\times Y,\check{g}_t)},
\end{equation}
with \eqref{wulprit} implying in particular the crude estimate
\begin{equation}\label{lucium}
\|\mathcal{B}_t^{[0]}\|_{L^\infty(\check{B}_{R\delta_t^{-1}}\times Y)}\leq C\delta_t^{-j-\alpha},
\end{equation}
and by construction for each $1\leq q\leq\kappa$ the functions $h_i$ in the expression \eqref{starr} for $\delta_t^{-j-\alpha}E^{[q]}_t$ lie in the fiberwise linear span of the
$\check{G}_{i,p,k}, 2\leq i<j$ and $\check{G}_{j,p,k}^{[r]}, 1\leq r\leq q+1$
together with the constants.

\subsubsection{Iteration and conclusion}\label{veriff}

We now repeat the iterative step that we just described (shrinking also $B'$ at each step) until step $\ov{\kappa}:=\lceil\frac{j+\alpha}{\alpha_0}\rceil$ and then we stop, so the last set of functions which are added to the list are the $G_{j,p,k}^{[\ov{\kappa}+1]}$. Our choice of $\ov{\kappa}$ is made so that $\delta_t^{-j-\alpha}\delta_t^{(\ov{\kappa}+1)\alpha_0}\to 0$. The resulting $G_{j,p,k}^{[q]}$ with $1\leq q\leq\ov{\kappa}+1$ are then renamed simply $G_{j,p,k}$. These, together with the $G_{i,p,k}, 2\leq i<j$, are the obstruction functions that we seek. It remains to show that the statement of the Selection Theorem \ref{ghost} holds with this choice of obstruction functions. By definition, the quantity in \eqref{stronzo} equals $\mathcal{B}_t^{[\ov{\kappa}+1]}$, which satisfies $(\star)$ thanks to Lemma \ref{solomon}. From Remark \ref{usami} we see that if it converges locally uniformly, then it converges locally smoothly, which is the last claim in the Selection Theorem \ref{ghost}. We are then left with showing that if $\mathcal{B}_t^{[\ov{\kappa}+1]}$ converges locally uniformly, then \eqref{stronza} holds.

For this, we go back to \eqref{panacea3} setting $\kappa=\ov{\kappa}+1$, and write it as
\begin{equation}
\mathcal{B}_t^{[\ov{\kappa}+1]}=\mathcal{B}_t^{[0]}+\sum_{q=1}^{\ov{\kappa}+1}(D^{[q]}_t+\ti{E}^{[q]}_t)+o(1),
\end{equation}
where $o(1)$ is in the smooth topology and we have set
\begin{equation}
D^{[q]}_t=-\delta_t^{-j-\alpha}\left(\sum_{p=1}^{N^{[q]}_{j,k}}\check{G}^{[q]}_{j,p,k}\check{A}^{\sharp,[q]}_{t,j,p,k}\right)(1+o(1)_{\rm from\ base}),\quad \ti{E}^{[q]}_t=\delta_t^{-j-\alpha}E^{[q]}_t,
\end{equation}
which thanks to \eqref{catilina}, \eqref{sangue} and \eqref{panacea4} satisfy
\begin{equation}\label{lutto}
\|D^{[q]}_t\|_{L^\infty(\check{B}_{R\delta_t^{-1}}\times Y)}\leq C \delta_t^{-j-\alpha}\|{\rm I}^{[q]}_{0,0}\|_{L^\infty(\check{B}_{R\delta_t^{-1}}\times Y,\check{g}_t)},\quad
\|\ti{E}^{[q]}_t\|_{L^\infty(\check{B}_{R\delta_t^{-1}}\times Y)}\leq C\delta_t^{\alpha_0}\|D^{[q]}_t\|_{L^\infty(\check{B}_{R\delta_t^{-1}}\times Y)}.
\end{equation}
Thanks to our definition of the $G_{i,p,k}$, the quantity
\begin{equation}
\mathcal{B}_t^{[0]}+\sum_{q=1}^{\ov{\kappa}+1}D^{[q]}_t+\sum_{q=1}^{\ov{\kappa}}\ti{E}^{[q]}_t,
\end{equation}
is already of the desired form \eqref{stronza} (and lies in the fiberwise span of the $G_{i,p,k}, 2\leq i\leq j$), and so we will be done if we can show that
\begin{equation}\label{mainclaim}
\|\ti{E}^{[\ov{\kappa}+1]}_t\|_{L^\infty(\check{B}_{R\delta_t^{-1}}\times Y)}=o(1),
\end{equation}
for any given $R$. Indeed, we have shown earlier that $\ti{E}^{[\ov{\kappa}+1]}_t$ satisfies $(\star)$, so \eqref{mainclaim} together with Remark \ref{usami} would imply that it is $o(1)$ locally smoothly.

To prove \eqref{mainclaim}, for $1\leq q\leq \ov{\kappa}+1$ we denote by $P^{[q]}$ the fiberwise $L^2$ orthogonal projection operator onto the span of the $\check{G}_{j,p,k}^{[q]}$'s (although these functions had been renamed, here we revert to the previous notation as it will make the rest of the proof clearer). Since by assumption the $\mathcal{B}_t^{[\ov{\kappa}+1]}$ converges locally uniformly, so does $\mathcal{B}_t^{[0]}+\sum_{q=1}^{\ov{\kappa}+1}(D^{[q]}_t+\ti{E}^{[q]}_t)$ and so does $P^{[q]}$ applied to it, hence by definition
\begin{equation}\label{hook1}
\left\|P^{[1]}{\mathcal{B}}^{[0]}+D^{[1]}_t+P^{[1]}\sum_{q=1}^{\ov{\kappa}+1}\ti{E}^{[q]}_t\right\|_{L^\infty(\check{B}_{R\delta_t^{-1}}\times Y)}\leq C,
\end{equation}
and for $2\leq r\leq \ov{\kappa}+1,$
\begin{equation}\label{hook2}
\left\|D^{[r]}_t+P^{[r]}\sum_{q=r-1}^{\ov{\kappa}+1}\ti{E}^{[q]}_t\right\|_{L^\infty(\check{B}_{R\delta_t^{-1}}\times Y)}\leq C,
\end{equation}
where we used that by construction $P^{[r]}{\mathcal{B}}^{[0]}=0$ for $r\geq 2$ and $P^{[r]}\ti{E}^{[q]}_t=0$ for $q<r-1$ (since the functions $h_i$ that arise in the expression \eqref{starr} for $\ti{E}^{[q]}_t$ lie in the fiberwise linear span of the $\check{G}_{i,p,k}, 2\leq i<j$ together with the $\check{G}_{j,p,k}^{[s]}, 1\leq s\leq q+1$). We then claim that by induction on $1\leq r\leq \ov{\kappa}+1$ we have
\begin{equation}\label{hook3}
\left\|D^{[r]}_t\right\|_{L^\infty(\check{B}_{R\delta_t^{-1}}\times Y)}\leq C+C\delta_t^{(r-1)\alpha_0}\left\|\mathcal{B}_t^{[0]}\right\|_{L^\infty(\check{B}_{R\delta_t^{-1}}\times Y)}
+C\delta_t^{\alpha_0}\sum_{q=r+1}^{\ov{\kappa}+1}\left\|D^{[q]}_t\right\|_{L^\infty(\check{B}_{R\delta_t^{-1}}\times Y)},
\end{equation}
for all $t$ sufficiently large. Once this is proved then \eqref{mainclaim} follows by taking $r=\ov{\kappa}+1$
\begin{equation}
\left\|D^{[\ov{\kappa}+1]}_t\right\|_{L^\infty(\check{B}_{R\delta_t^{-1}}\times Y)}\leq C+C\delta_t^{\ov{\kappa}\alpha_0}\left\|\mathcal{B}_t^{[0]}\right\|_{L^\infty(\check{B}_{R\delta_t^{-1}}\times Y)},
\end{equation}
and using \eqref{lutto} and \eqref{lucium} we obtain
\begin{equation}
\left\|\ti{E}^{[\ov{\kappa}+1]}_t\right\|_{L^\infty(\check{B}_{R\delta_t^{-1}}\times Y)}\leq C\delta_t^{\alpha_0}+C\delta_t^{(\ov{\kappa}+1)\alpha_0}\left\|\mathcal{B}_t^{[0]}\right\|_{L^\infty(\check{B}_{R\delta_t^{-1}}\times Y)}\leq C\delta_t^{\alpha_0}
+C\delta_t^{(\ov{\kappa}+1)\alpha_0}\delta_t^{-j-\alpha}=o(1),
\end{equation}
as desired.

To prove \eqref{hook3} we first show the case $r=1$. For this, we use \eqref{hook1} and \eqref{lutto} to bound
\begin{equation}\begin{split}
\left\|D^{[1]}_t\right\|_{L^\infty(\check{B}_{R\delta_t^{-1}}\times Y)}&\leq C+\|P^{[1]}{\mathcal{B}}^{[0]}\|_{L^\infty(\check{B}_{R\delta_t^{-1}}\times Y)}+\left\|P^{[1]}\sum_{q=1}^{\ov{\kappa}+1}\ti{E}^{[q]}_t\right\|_{L^\infty(\check{B}_{R\delta_t^{-1}}\times Y)}\\
&\leq C+C\|{\mathcal{B}}^{[0]}\|_{L^\infty(\check{B}_{R\delta_t^{-1}}\times Y)}+C\sum_{q=1}^{\ov{\kappa}+1}\left\|\ti{E}^{[q]}_t\right\|_{L^\infty(\check{B}_{R\delta_t^{-1}}\times Y)}\\
&\leq C+C\|{\mathcal{B}}^{[0]}\|_{L^\infty(\check{B}_{R\delta_t^{-1}}\times Y)}+C\delta_t^{\alpha_0}\sum_{q=1}^{\ov{\kappa}+1}\left\|D^{[q]}_t\right\|_{L^\infty(\check{B}_{R\delta_t^{-1}}\times Y)},
\end{split}\end{equation}
and we may assume that $C\delta_t^{\alpha_0}\leq\frac{1}{2}$ so the term with $q=1$ in the sum on the RHS can be absorbed by the LHS, thus proving \eqref{hook3} with $r=1$. As for the induction step, for $r\geq 1$ we can use \eqref{hook2} with $r+1$ instead of $r$, together with \eqref{lutto} and the induction hypothesis \eqref{hook3} to bound
\begin{equation}\begin{split}
\left\|D^{[r+1]}_t\right\|_{L^\infty(\check{B}_{R\delta_t^{-1}}\times Y)}&\leq C+\left\|P^{[r+1]}\sum_{q=r}^{\ov{\kappa}+1}\ti{E}^{[q]}_t\right\|_{L^\infty(\check{B}_{R\delta_t^{-1}}\times Y)}\\
&\leq C+C\sum_{q=r}^{\ov{\kappa}+1}\left\|\ti{E}^{[q]}_t\right\|_{L^\infty(\check{B}_{R\delta_t^{-1}}\times Y)}\\
&\leq C+C\delta_t^{\alpha_0}\sum_{q=r}^{\ov{\kappa}+1}\left\|D^{[q]}_t\right\|_{L^\infty(\check{B}_{R\delta_t^{-1}}\times Y)}\\
&\leq C+C\delta_t^{\alpha_0}\left\|D^{[r]}_t\right\|_{L^\infty(\check{B}_{R\delta_t^{-1}}\times Y)}+C\delta_t^{\alpha_0}\left\|D^{[r+1]}_t\right\|_{L^\infty(\check{B}_{R\delta_t^{-1}}\times Y)}\\
&\ \ \ +C\delta_t^{\alpha_0}\sum_{q=r+2}^{\ov{\kappa}+1}\left\|D^{[q]}_t\right\|_{L^\infty(\check{B}_{R\delta_t^{-1}}\times Y)}\\
&\leq C+C\delta_t^{r\alpha_0}\left\|\mathcal{B}_t^{[0]}\right\|_{L^\infty(\check{B}_{R\delta_t^{-1}}\times Y)}
+C\delta_t^{\alpha_0}\left\|D^{[r+1]}_t\right\|_{L^\infty(\check{B}_{R\delta_t^{-1}}\times Y)}\\
&\ \ \ +C\delta_t^{\alpha_0}\sum_{q=r+2}^{\ov{\kappa}+1}\left\|D^{[q]}_t\right\|_{L^\infty(\check{B}_{R\delta_t^{-1}}\times Y)},
\end{split}\end{equation}
and again we may assume that $C\delta_t^{\alpha_0}\leq\frac{1}{2}$ and absorb the term $C\delta_t^{\alpha_0}\left\|D^{[r+1]}_t\right\|_{L^\infty(\check{B}_{R\delta_t^{-1}}\times Y)}$ on the RHS by the LHS, thus proving \eqref{hook3}. This completes the proof of \eqref{mainclaim} and hence also of the Selection Theorem \ref{ghost}.

\section{The asymptotic expansion theorem}\label{s:mainest}

This section contains the proof of the asymptotic expansion Theorem \ref{shutupandsuffer}, and is the main part of the paper.

\subsection{Statement of the asymptotic expansion}

Let us first explain the precise assumptions for the theorem. We work on $B\times Y$, where $B\subset \C^m$ is the unit ball and $Y$ is a closed real $2n$-fold. The product $B\times Y$ is equipped with a complex structure $J$ such that ${\rm pr}_B$ is $(J,J_{\C^m})$-holomorphic, and with a K\"ahler metric $\omega_X$, and such that the fibers $\{z\}\times Y$ are Calabi-Yau $n$-folds. When later in Section \ref{sectmain} we will use the expansion to prove Theorems A and B, we will have a compact Calabi-Yau manifold $X$ with fiber space structure $f$ as in Section \ref{sectsetup}, and $B$ is a ball in the base whose closure does not meet the critical locus $f(S)$, over which $f$ is smoothly trivial and so $f^{-1}(B)$ is diffeomorphic to $B\times Y$ and the complex structure of $X$ defines our complex structure $J$.

As in \eqref{pstwlrllk} we define a semi-Ricci-flat form $\omega_F$ on $B\times Y$ by $\omega_F=\omega_X+\ddbar\rho$ where $\rho_z=\rho(z,\cdot)$ is such that $\omega_X|_{\{z\}\times Y}+\ddbar\rho_z$ is Ricci-flat K\"ahler on $\{z\}\times Y$ and $\int_{\{z\}\times Y}\rho_z\omega_X^n=0$. We suppose that we have a K\"ahler metric $\omega_{\rm can}$ on $B$, and define
$\omega^\natural_t=\omega_{\rm can}+e^{-t}\omega_F$, which we assume is a K\"ahler metric for all $t\geq 0$. We also assume we have Ricci-flat K\"ahler metrics $\omega^\bullet_t,t\geq 0,$ on $B\times Y$, of the form $\omega^\bullet_t=\omega^\natural_t+\ddbar\psi_t$ which solve the complex Monge-Amp\`ere equation
\begin{equation}\label{ma}
(\omega^\bullet_t)^{m+n}=(\omega^\natural_t+\ddbar\psi_t)^{m+n}=c_te^{-nt}\omega_{\rm can}^m\wedge\omega_F^n,
\end{equation}
where $c_t$ is a polynomial in $e^{-t}$ of degree at most $m$ with constant coefficient $\binom{m+n}{n}$. We also crucially assume that
on $B\times Y$ we have
\begin{equation}\label{fucker}
C^{-1}\omega^\natural_t\leq \omega^\bullet_t\leq C\omega^\natural_t,
\end{equation}
and that
\begin{equation}\label{fucker2}
\psi_t\to 0,
\end{equation}
weakly as distributions as $t\to\infty$.

We are given $0\leq j\leq k$, and given any point $z\in B$, fix a ball $B_r(z)\subset B$ and apply the Selection Theorem \ref{ghost} to it. This produces for us a much smaller ball $B'=B_{r'}(z)\subset B$ and a list $G_{i,p,k},2\leq i\leq j, 1\leq p\leq N_{i,k}$ of smooth functions on $B'\times Y$. For each of these, we define
\begin{equation}\label{proietta}
P_{t,i,p,k}=P_{t,G_{i,p,k}},
\end{equation}
as in \eqref{proietto}. Recall also that we defined an approximate Green's operator $\mathfrak{G}_{t,k}$ in \S \ref{grz}, and that for a smooth function $F$ on $B\times Y$ we denote by $\underline{F}\in C^\infty(B)$ its fiberwise average with respect to $\omega_F|_{\{z\}\times Y}$.

Throughout the whole proof we will also fix a family of shrinking product Riemannian metrics $g_t$ on $B\times Y$, for example $g_t=g_{\C^m}+e^{-t}g_{Y,z_0}$ for some fixed $z_0\in B$. These will only be used to measure norms and distances, so the particular choice of $g_t$ will not matter. Lastly, the H\"older seminorms that we use are those defined in \S \ref{holderiano}.

\begin{theorem}\label{shutupandsuffer}
For all $0\leq j \leq k$, given any $z\in B$ we can find a ball $B'\subset B$ centered at $z$ and functions $G_{i,p,k}$ as above, such that on $B'\times Y$ we have a decomposition
\begin{equation}\label{decomponiti}
\omega^\bullet_t = \omega_t^\natural + \gamma_{t,0} + \gamma_{t,1,k} + \cdots + \gamma_{t,j,k} + \eta_{t,j,k},
\end{equation}
with the following properties. First of all, for every $0<\alpha<1$ and every smaller ball $B''\Subset B'$ there is $C>0$ such that on $B''\times Y$ we have for all $t\geq 0$:
\begin{equation}\label{eins}
\|\D^\iota\eta_{t,j,k}\|_{L^\infty(B''\times Y, g_t)}\leq Ce^{\frac{\iota-j-\alpha}{2}t}\;\,\text{for all}\;\,0\leq \iota\leq j,
\end{equation}
\begin{equation}\label{zwei}
[\D^j\eta_{t,j,k}]_{C^\alpha(B''\times Y,g_t)}\leq C.
\end{equation}
Furthermore, we have
\begin{equation}\label{fuenf}
\gamma_{t,0}=\ddbar\underline{\psi_{t}},\quad \gamma_{t,1,k}=0,\quad \gamma_{t,i,k} = \sum_{p=1}^{N_{i,k}}\ddbar\mathfrak{G}_{t,k}(A_{t,i,p,k},G_{i,p,k}) \;\,(2\leq i \leq j),
\end{equation}
where the
\begin{equation}\label{sechs}
A_{t,i,p,k} = P_{t,i,p,k}(\eta_{t,i-1,k}) \;\,(2\leq i \leq j),
\end{equation} are functions from the base, and we have the estimates
\begin{equation}\label{vier3}
\|\D^\iota\gamma_{t,0}\|_{L^\infty(B'',g_{\C^m})} = o(1) \;\,(0 \leq \iota \leq j),
\end{equation}
\begin{equation}\label{vier2}
\quad [\D^j\gamma_{t,0}]_{C^\alpha(B'',g_{\C^m})}\leq C,
\end{equation}
\begin{equation}\label{drei2}
\|\D^\iota A_{t,i,p,k}\|_{L^\infty(B'',g_{\C^m})}\leq Ce^{-(i+1+\alpha)(1-\frac{\iota}{j+2+\alpha})\frac{t}{2}}
\;\,(0 \leq \iota \leq j+2,\;\, 2 \leq i \leq j,\;\,1\leq p\leq N_{i,k}),
\end{equation}
\begin{equation}\label{drei3a}
\|\D^{j+2+\iota} A_{t,i,p,k}\|_{L^\infty(B'',g_{\C^m})}\leq Ce^{\iota\frac{t}{2}},
\;\,(0 \leq \iota \leq 2k,\;\,2 \leq i \leq j,\;\,1\leq p\leq N_{i,k}),
\end{equation}
\begin{equation}\label{sieben}\begin{split}
\sup_{x,x'\in B''\times Y}\sum_{i=2}^j \sum_{p=1}^{N_{i,k}}\sum_{\iota = -2}^{2k} &e^{-\iota\frac{t}{2}}\Bigg(\frac{|\D^{j+2+\iota} A_{t,i,p,k}(x)-\P_{x'x}(\D^{j+2+\iota}A_{t,i,p,k}(x'))|_{g_t}}{d^{g_t}(x,x')^\alpha}\Bigg)\leq C,
\end{split}
\end{equation}
where as usual the supremum in \eqref{sieben} is taken over pairs of points that are either horizontally or vertically joined.
\end{theorem}

\begin{rk}
To help understand this statement, morally each piece $\gamma_{t,i,k}$ should be thought of as having $\de\db$-potential of the form $e^{-\frac{i+2}{2}t}$ $\times$ (fixed function on total space), so $\gamma_{t,i,k}$ would be then bounded in the shrinking $C^i(g_t)$ but not in the shrinking $C^{i,\alpha}(g_t)$ for any $\alpha>0$. In practice, the potential of $\gamma_{t,i,k}$ is not quite of this form, but recalling \eqref{e:Gstructure} we see that its ``leading term'' is indeed of the form $\sum_p A_{t,i,p,k}(\mathrm{fixed}\ \mathrm{function})_p$, and the $L^\infty$ bound for $A_{t,i,p,k}$ in \eqref{drei2} just barely falls short of $e^{-\frac{i+2}{2}t}$. Furthermore, as we will see in the proof of Theorem \ref{mthmB}, the functions $A_{t,2,p,k}$ do satisfy the ``optimal'' $L^\infty$ bound by $e^{-2t}$. As for the estimates in \eqref{sieben}, they give in particular bounds for the H\"older seminorms of derivatives of $A_{t,i,p,k}$ of order $j+3$ to $j+2+2k$ which blow up exponentially fast, and will be used in Section \ref{sectmain} to derive uniform $C^{j}(g_X)$ bounds for the pieces $\gamma_{t,i,k}, 2\leq i\leq j$, since these are defined using the approximate Green's operator $\mathfrak{G}_{t,k}$ that involves a high number of derivatives (cf. Lemma \ref{Gstructure}).
\end{rk}
\begin{rk}
Observe also that when $j=k=0$ the statement of theorem is analogous to our earlier work \cite[Thm 1.4]{HT2}, although not exactly identical since a different parallel transport is used there. In this very special case, the proof of Theorem \ref{shutupandsuffer} gives a more streamlined and improved proof of that result.
\end{rk}

\begin{rk}
When $j\leq 1$, the decomposition \eqref{decomponiti} simplifies, since $\gamma_{t,1,k}=0$, and the theorem just says that $\omega^\bullet_t-\omega^\natural_t$ is bounded in a shrinking $C^{j,\alpha}$-type norm, and goes to zero in shrinking $C^j$. However, already $\gamma_{t,2,k}$ does not vanish in general (see the proof of Theorem \ref{mthmB}, and the corresponding discussion in the Introduction), and so this strong statement that holds for $j=0,1,$ fails for $j\geq 2$. This is one of the main reasons why the statement and proof of Theorem \ref{shutupandsuffer} are complicated.
\end{rk}

The proof of Theorem \ref{shutupandsuffer} occupies the rest of this section (indeed, essentially all of the remainder of the paper), and will be divided into subsections.

\subsection{Set-up of an inductive scheme, and initial reductions}

We start with the given point $z\in B$ with a ball $B'\subset B$ centered at $z$. For a given $k = 0,1,2,\ldots$ the proof proceeds by induction on $j = 0,1,\ldots,k$.
We will treat both the base case and the induction step at once, so, given $k$, we work at some $0 \leq j \leq k$ where if $j>0$
we assume we already have the decomposition \eqref{decomponiti} at step $j-1$ which satisfies \eqref{eins}---\eqref{sieben}, and we aim to refine the decomposition \eqref{decomponiti} to step $j$ (and define it for $j=0$) and prove \eqref{eins}---\eqref{sieben} if $j>0$ and \eqref{eins}, \eqref{zwei}, \eqref{vier3}, \eqref{vier2} if $j=0$.

First we give the details on how the decomposition \eqref{decomponiti} is constructed.  When $j=0$ we simply set $\gamma_{t,0}=\ddbar\underline{\psi_{t}}, \eta_{t,0,k}=\ddbar(\psi_t-\underline{\psi_{t}})$, so that \eqref{decomponiti} holds. For $j>0$ we assume by induction that we have the decomposition \eqref{decomponiti} at step $j-1$
\begin{equation}\label{decomponitiolde}
\omega^\bullet_t = \omega_t^\natural + \gamma_{t,0} + \gamma_{t,2,k} + \cdots + \gamma_{t,j-1,k} + \eta_{t,j-1,k},
\end{equation}
on some ball $B'$ centered at $z$, and we wish to decompose $\eta_{t,j-1,k}=\eta_{t,j,k}+\gamma_{t,j,k}$ on some possibly smaller ball.

As indicated above, when $j>2$ up to shrinking the ball $B'$ we may assume that we have already selected smooth functions $G_{i,p,k}, 2\leq i\leq j-1, 1\leq p\leq N_{i,k}$, which are fiberwise $L^2$ orthonormal, and have fiberwise average zero. When $j\geq 2$ we then apply the Selection Theorem \ref{ghost} which up to shrinking $B'$ gives us a list of functions $G_{j,p,k}, 1\leq p\leq N_{j,k}$ on $B'\times Y$ with fiberwise average zero and so that the $G_{i,p,k}, 2\leq i\leq j$ are all fiberwise $L^2$ orthonormal, and the conclusion of the Selection Theorem \ref{ghost} holds. For ease of notation, we will rename $B'$ to $B$ in all of the following.

With these functions, we have the projections $P_{t,i,p,k}=P_{t,G_{i,p,k}}$ as in \eqref{proietto}, and we then define
$A_{t,j,p,k} = P_{t,j,p,k}(\eta_{t,j-1,k})$ and $\gamma_{t,j,k} = \sum_{p=1}^{N_{j,k}}\ddbar\mathfrak{G}_{t,k}(A_{t,j,p,k},G_{j,p,k})$, as in \eqref{fuenf}, \eqref{sechs}, where $\mathfrak{G}_{t,k}$ was constructed in Section \ref{grz}. Lastly, we define $\eta_{t,j,k}=\eta_{t,j-1,k}-\gamma_{t,j,k},$ so that \eqref{decomponiti} holds at step $j$, together with \eqref{fuenf}, \eqref{sechs}.

The goal is thus to prove the estimates \eqref{eins}, \eqref{zwei}, \eqref{vier2}, \eqref{drei2}, \eqref{drei3a} and \eqref{sieben} if $j>0$, and estimates \eqref{eins}, \eqref{zwei}, \eqref{vier3} and \eqref{vier2} if $j=0$.

\subsubsection{The bounds that hold thanks to the induction hypothesis}

First, we need to use the induction hypothesis \eqref{eins} to obtain the uniform bound in \eqref{drei2} with $\iota=0$ for the functions $A_{t,i,p,k}$ ($2 \leq i \leq j \leq k$), which was defined by \eqref{sechs}. Of course, we only need to establish this bound for $i=j$, since those for $2\leq i\leq j-1$ are already known to hold by induction, but since it makes no difference here we work with an arbitrary $2\leq i\leq j$.
Note that for any $(1,1)$-form $\alpha$ on the total space,
\begin{equation}\label{veryclever}
\|P_{t,i,p,k}(\alpha)\|_{L^\infty(B)}\leq Ce^{-t}\|\alpha\|_{L^\infty(B\times Y,g_t)}.
\end{equation}
Indeed, if we call $K=\|\alpha\|_{L^\infty(B\times Y,g_t)}$ then on $B\times Y$ we have
\begin{equation}|\alpha_{\mathbf{ff}}| \leq CKe^{-t},\end{equation}
\begin{equation}|\alpha_{\mathbf{bf}}| \leq CKe^{-\frac{t}{2}},\end{equation}
\begin{equation}|\alpha_{\mathbf{bb}}| \leq CK,\end{equation}
which can be then fed into \eqref{proietto} to obtain \eqref{veryclever}.

By induction \eqref{eins} (and its analogs for all $2\leq i \leq j-1$) we know that
\begin{equation}\|\eta_{t,i-1,k}\|_{L^\infty(B\times Y,g_t)}\leq C e^{\frac{-i+1-\beta}{2}t},\end{equation}
for all $\beta<1$. For the rest of the proof, it suffices to take $\beta=\alpha$, say, which we will do from now on.
Then \eqref{veryclever} gives
\begin{equation}\|P_{t,i,p,k}(\eta_{t,i-1,k})\|_{L^\infty(B)}\leq Ce^{\frac{-i-1-\alpha}{2}t}.\end{equation}
Putting these together we get
\begin{equation}\label{uniforme}
\|A_{t,i,p,k}\|_{L^\infty(B)}\leq Ce^{\frac{-i-1-\alpha}{2}t},
\end{equation}
for all $2 \leq i \leq j$, $1\leq p\leq N_{i,k}$.

\subsubsection{Proving \eqref{eins}, \eqref{vier3}, \eqref{drei2} and \eqref{drei3a}}

The logic now is the following. Suppose first that we have proved \eqref{zwei}, \eqref{vier2} and \eqref{sieben}, and use them to quickly establish \eqref{eins}, \eqref{vier3}, \eqref{drei2}, \eqref{drei3a} for $j>0$, and \eqref{eins}, \eqref{vier3} for $j=0$. And then, and this is the main task that will be carried out in the later subsections, we shall establish
that \eqref{zwei}, \eqref{vier2} and \eqref{sieben} do hold (this last one is of course vacuous when $j=0,1$).

First, let us dispense with the case $j=0$, where by assumption we assume that \eqref{zwei} and \eqref{vier2} hold. Then Theorem \ref{prop55} applies to $\eta_{t,0,k}$ and it shows that \eqref{eins} follows from \eqref{zwei}. Next, recall that $\gamma_{t,0}=\ddbar\underline{\psi_{t}}$ and $\psi_t\to 0$ in the weak topology by \eqref{fucker2}. This implies that $\underline{\psi_{t}}\to 0$ weakly, and hence also
\begin{equation}\label{fucker3}
\gamma_{t,0}\to 0,
\end{equation}
weakly. On the other hand, the bound \eqref{fucker} and the fiber integration argument in \cite[p.436]{To} give
\begin{equation}\label{caz4}
\|\gamma_{t,0}\|_{L^\infty(B,g_{\C^m})}\leq C.
\end{equation}
Combining this with \eqref{vier2} then gives a uniform $C^{\alpha}$ bound for $\gamma_{t,0}$ (up to shrinking $B$) and so $\gamma_{t,0}\to 0$ in $C^{\alpha'}$ for any $\alpha'<\alpha$, which proves \eqref{vier3}, and completes our discussion of the case when $j=0$.

We assume then that $j>0$. First observe that the exact same arguments show that \eqref{eins} follows from \eqref{zwei} and Theorem \ref{prop55} applied to $\eta_{t,j,k}$, and that \eqref{vier3} follows from \eqref{vier2} and \eqref{fucker}, \eqref{fucker2}.

The remaining task is then to prove \eqref{drei2}, \eqref{drei3a}, making use of \eqref{sieben} together with \eqref{uniforme}, that we have just established holds thanks to induction hypotheses.

First, we discuss \eqref{drei2}. We feed the $\D^{j+2}$ part of \eqref{sieben} and the uniform bound \eqref{uniforme} into the interpolation inequality in Proposition \ref{l:higher-interpol} (here $R$ is the radius of $B$, which is comparable to $1$) and obtain
\begin{equation}\sum_{\iota=1}^{j+2}(R-\rho)^\iota \|\D^\iota A_{t,i,p,k}\|_{L^\infty(B_\rho,g_{\C^m})}\leq C(R-\rho)^{j+2+\alpha}+Ce^{\frac{-i-1-\alpha}{2}t},\end{equation}
and choosing
\begin{equation}R-\rho=e^{-\frac{i+1+\alpha}{j+2+\alpha}\frac{t}{2}}.\end{equation}
proves \eqref{drei2}, and it also in particular gives a uniform $C^{j+2,\alpha}$ bound on the $A_{t,i,p,k}$'s. Note that here we have only used the $\D^{j+2}$ part of \eqref{sieben}. It is possible to do slightly better for $\iota \leq j+1$ by using the full strength of \eqref{sieben} and some more elaborate interpolations. While similar improvements will in fact be important below, we do not pursue them here because they are more awkward to state and don't make enough of a difference to the statement of Theorem \ref{shutupandsuffer}.

Finally, we discuss \eqref{drei3a}.
From \eqref{sieben} we get in particular that
\begin{equation}[\D^{j+2+\iota}A_{t,i,p,k}]_{C^\alpha(B,g_{\C^m})}\leq Ce^{\frac{\iota}{2}t},\quad 0\leq \iota \leq 2k,\end{equation}
and doing interpolation between this and $|\D^{j+2} A_{t,i,p,k}|=o(1)$ (which is a slightly suboptimal consequence of \eqref{drei2}) we get in particular the bound
\begin{equation}\label{stroh}
\|\D^{j+2+\iota}A_{t,i,p,k}\|_{L^\infty(B',g_{\C^m})}=o(e^{\frac{\iota}{2}t}),\quad  0\leq \iota\leq 2k,
\end{equation}
which is even stronger than the statement of \eqref{drei3a}.

\subsubsection{Setting up the proof of \eqref{zwei}, \eqref{vier2} and \eqref{sieben}}

Thanks to the previous section, it remains to establish the estimates \eqref{zwei}, \eqref{vier2} and \eqref{sieben}. This is the main task and will occupy almost all the rest of the paper.

To this effect, suppose we had a uniform bound on the sum of the difference quotients
\begin{equation}\label{putative}\begin{split}
\mathcal{D}_{t,j}(x,x')&:=\sum_{i=2}^j \sum_{p=1}^{N_{i,k}}\sum_{\iota = -2}^{2k} \Bigg( e^{-\iota\frac{t}{2}}\frac{|\D^{j+2+\iota}A_{t,i,p,k}(x)-\P_{x'x}(\D^{j+2+\iota}A_{t,i,p,k}(x'))|_{g_{t}(x)}}{d^{g_{t}}(x,x')^\alpha}\Bigg)\\
&+\frac{|\D^j\gamma_{t,0}(x)-\P_{x'x}(\D^j\gamma_{t,0})(x'))|_{g_{t}(x)}}{d^{g_{t}}(x,x')^\alpha}+\frac{|\D^j\eta_{t,j,k}(x)-\P_{x'x}(\D^j\eta_{t,j,k})(x'))|_{g_{t}(x)}}{d^{g_{t}}(x,x')^\alpha}.
\end{split}\end{equation}
for all $x\neq x'\in B''\times Y$ which are either horizontally or vertically joined, and all $t\geq 0$, then it is clear that the bounds \eqref{zwei}, \eqref{vier2} and \eqref{sieben} hold. Observe that for each fixed $t$ we have $\lim_{x'\to x}\mathcal{D}_{t,j}(x,x')=0$ since all the objects are smooth.

\subsection{Set-up of the primary (nonlinear) blowup argument}\label{mavaffa}

Thus, to complete the proof of Theorem \ref{shutupandsuffer} it suffices to prove a uniform bound for \eqref{putative}. We shall assume from now on that $B=B_1$ w.r.t. the Euclidean metric $\omega_{\C^m}$, and by replacing $B$ with a slightly smaller ball we will assume without loss that all our objects are defined (and satisfy the assumptions of Theorem \ref{shutupandsuffer}) on a concentric ball $B_{1+\sigma}$ for some $\sigma>0$. To bound \eqref{putative}, define a function $\mu_{j,t}$ on $B \times Y$ (with usual variables $x=(z,y)$) by
\begin{align}
\mu_{j,t}(x) = ||z|-1|^{j+\alpha} \sup_{\substack{x'=(z',y')\ s.t.\ |z'-z|<\frac{1}{4}||z|-1|\\x'\ \mathrm{and}\ x\ \mathrm{horizontally}\ \mathrm{or}\ \mathrm{vertically}\ \mathrm{joined}}} \mathcal{D}_{t,j}(x,x').
\end{align}
Then it is clearly sufficient to prove the inequality
\begin{align}\label{strongk2}
\max_{B\times Y} \mu_{j,t} \leq C.
\end{align}

If \eqref{strongk2} is false, then
$\limsup_{t\to\infty} \max_{B\times Y} \mu_{j,t}=\infty.$
For simplicity of notation we will assume that $\lim_{t\to\infty}\max_{B\times Y} \mu_{j,t}=\infty$ (usually this will be true only along some sequence $t_i\to \infty$). Choose $x_t = (z_t, y_t) \in B\times Y$ such that the maximum of $\mu_{j,t}$ is achieved at $x_t$, and define $\lambda_t$ by
\begin{equation}\lambda_t^{j+\alpha}=\sup_{\substack{x'=(z',y')\ s.t.\ |z'-z_t|<\frac{1}{4}||z_t|-1|\\x'\ \mathrm{and}\ x_t\ \mathrm{horizontally}\ \mathrm{or}\ \mathrm{vertically}\ \mathrm{joined}}}\mathcal{D}_{t,j}(x_t,x').\end{equation}
Let us note for later purposes that after passing to a subsequence,
\begin{align}\label{whoops666}
z_t \to z_\infty \in \overline{B}, \;\, y_t \to y_\infty \in Y.
\end{align}
Now $\lambda_t \to \infty$ since otherwise $\max_{B\times Y} \mu_{j,t}$ would be uniformly bounded. Let us also choose any point
$x_t' = (z_t', y_t')$  with $|z'_t-z_t|\leq \frac{1}{4}||z_t|-1|$ realizing the sup in the definition of $\lambda_t$. Since  $\lim_{x'\to x}\mathcal{D}_{t,j}(x,x')=0$, we may assume without loss that $x_t'\neq x_t$.
Consider the diffeomorphisms
\begin{equation}\Psi_t:B_{\lambda_t}\times Y\to B\times Y, \;\, (z,y) = \Psi_t(\hat{z},\hat{y})=(\lambda_t^{-1}\hat{z},\hat{y}),\end{equation}
pull back any contravariant $2$-tensor via $\Psi_t^*$, rescale it by $\lambda_t^2$ and denote the new object with a hat, for example $\hat{\omega}^\bullet_t=\lambda_t^2\Psi_t^*\omega_t^\bullet$, etc. Define also $\hat{A}_{t,i,p,k} = \lambda_t^{2}\Psi_t^*A_{t,i,p,k}, \;\, \hat{x}_t=\Psi_t^{-1}(x_t), \;\, \hat{x}'_t=\Psi_t^{-1}(x'_t).$
We thus have
\begin{align}
\label{caz01}\hat{g}_t = g_{\C^m} +  \lambda_t^{2}e^{-t}g_{Y,z_0}, \;\,
\hat{\omega}_t^\natural = \hat{\omega}_{\rm can} +  \lambda_t^{2}e^{-t}\Psi_t^*\omega_F,\;\,
C^{-1}\hat{\omega}^\natural_t \leq \hat{\omega}^\bullet_t \leq C\hat{\omega}^\natural_t.
\end{align}
Also, from \eqref{caz4} we get
\begin{equation}\label{caz4new}
\|\hat{\gamma}_{t,0}\|_{L^\infty(B_{\lambda_t},g_{\C^m})}\leq C,
\end{equation}
and from \eqref{uniforme}
\begin{equation}\label{caz5new}
\|\hat{A}_{t,i,p,k}\|_{L^\infty(B_{\lambda_t})}\leq C\lambda_t^2e^{\frac{-i-1-\alpha}{2}t}=C\delta_t^2e^{\frac{-i+1-\alpha}{2}t},
\end{equation}
for all $2 \leq i \leq j$, $1\leq p\leq N_{i,k}$, where we define
\begin{equation}\delta_t=\lambda_te^{-\frac{t}{2}}.\end{equation}
For simplicity of notation, we will not decorate the $\D$ operator with a hat, since its usage will always be clear from the context.
We also define a stretched projection operator $\hat{P}_{t,i,p,k}$, in analogy with \eqref{proietto} and \eqref{proietta}, by
\begin{equation}\label{proiettohat}
\hat{P}_{t,i,p,k}(\alpha)=n({\rm pr}_B)_*\left(\hat{G}_{i,p,k}\alpha\wedge\Psi_t^*\omega_F^{n-1}\right) + \delta_t^2{\rm tr}^{\hat{\omega}_{\rm can}}({\rm pr}_B)_*(\hat{G}_{i,p,k} \alpha\wedge \Psi_t^*\omega_F^n),
\end{equation}
where $\hat{G}_{i,p,k}=\Psi_t^*G_{i,p,k}$. Its relation with the unstretched projection $P_{t,i,p,k}$ is the following: if $\beta$ is a $(1,1)$-form on $B\times Y$, and if we let $\Gamma_t:B_{\lambda_t}\to B$ be given by $\Gamma_t(\hat{z})=\lambda_t^{-1}\hat{z}$, then it is immediate to verify that we have
\begin{equation}
\hat{P}_{t,i,p,k}(\Psi_t^*\beta)=\Gamma_t^*P_{t,i,p,k}(\beta).
\end{equation}
Observe that
\begin{align}
\lambda_t^{j+\alpha}&=\mathcal{D}_{t,j}(x_t,x'_t)=\sum_{i=2}^j\sum_{p=1}^{N_{i,k}}\sum_{\iota = -2}^{2k} \Bigg( e^{-\iota\frac{t}{2}}\frac{|\D^{j+2+\iota}A_{t,i,p,k}(x)-\P_{x'_tx_t}(\D^{j+2+\iota}A_{t,i,p,k}(x'))|_{g_{t}(x)}}{d^{g_{t}}(x_t,x'_t)^\alpha}\Bigg)\\
&+\frac{|\D^j\gamma_{t,0}(x_t)-\P_{x'_tx_t}(\D^j\gamma_{t,0}(x'_t))|_{g_{t}(x_t)}}{d^{g_{t}}(x_t,x'_t)^\alpha}
+\frac{|\D^j\eta_{t,j,k}(x_t)-\P_{x'_tx_t}(\D^j\eta_{t,j,k}(x'_t))|_{g_{t}(x_t)}}{d^{g_{t}}(x_t,x'_t)^\alpha}\\
&=\sum_{i=2}^j\sum_{p=1}^{N_{i,k}}\sum_{\iota = -2}^{2k}
\Bigg(e^{-\iota\frac{t}{2}}\frac{|\D^{j+2+\iota}\hat{A}_{t,i,p,k}(\hat{x}_t)-\P_{\hat{x}'_t\hat{x}_t}(\D^{j+2+\iota}\hat{A}_{t,i,p,k}(\hat{x}'_t))|_{\Psi_t^*g_{t}(\hat{x}_t)}}{d^{\Psi_t^*g_{t}}(\hat{x}_t,\hat{x}'_t)^\alpha}
\Bigg) \lambda_t^{-2}\\
&+\frac{|\D^j\hat{\gamma}_{t,0}(\hat{x}_t)-\P_{\hat{x}'_t\hat{x}_t}(\D^j\hat{\gamma}_{t,0}(\hat{x}'_t))
|_{\Psi_t^*g_{t}(\hat{x}_t)}}{d^{\Psi_t^*g_{t}}(\hat{x}_t,\hat{x}'_t)^\alpha} \lambda_t^{-2}
+\frac{|\D^j\hat{\eta}_{t,j,k}(\hat{x}_t)-\P_{\hat{x}'_t\hat{x}_t}(\D^j\hat{\eta}_{t,j,k}(\hat{x}'_t))
|_{\Psi_t^*g_{t}(\hat{x}_t)}}{d^{\Psi_t^*g_{t}}(\hat{x}_t,\hat{x}'_t)^\alpha} \lambda_t^{-2}\\
&= \sum_{i=2}^j \sum_{p=1}^{N_{i,k}}\sum_{\iota= -2}^{2k} \Bigg(
\delta_t^{\iota}\frac{|\D^{j+2+\iota}\hat{A}_{t,i,p,k}(\hat{x}_t)-\P_{\hat{x}'_t\hat{x}_t}(\D^{j+2+\iota}\hat{A}_{t,i,p,k}(\hat{x}'_t))|_{\hat{g}_t(\hat{x}_t)}}{d^{\hat{g}_t}(\hat{x}_t,\hat{x}'_t)^\alpha}\Bigg)\lambda_t^{j+\alpha}\\
&+\frac{|\D^j\hat{\gamma}_{t,0}(\hat{x}_t)-\P_{\hat{x}'_t\hat{x}_t}(\D^j\hat{\gamma}_{t,0}(\hat{x}'_t))
|_{\hat{g}_t(\hat{x}_t)}}{d^{\hat{g}_t}(\hat{x}_t,\hat{x}'_t)^\alpha} \lambda_t^{j+\alpha}
+\frac{|\D^j\hat{\eta}_{t,j,k}(\hat{x}_t)-\P_{\hat{x}'_t\hat{x}_t}(\D^j\hat{\eta}_{t,j,k}(\hat{x}'_t))
|_{\hat{g}_t(\hat{x}_t)}}{d^{\hat{g}_t}(\hat{x}_t,\hat{x}'_t)^\alpha} \lambda_t^{j+\alpha},
\end{align}
which implies that
\begin{equation}\label{blow22}\begin{split}
&\sum_{i=2}^j\sum_{p=1}^{N_{i,k}}\sum_{\iota = -2}^{2k} \Bigg(
\delta_t^{\iota}\frac{|\D^{j+2+\iota}\hat{A}_{t,i,p,k}(\hat{x}_t)-\P_{\hat{x}'_t\hat{x}_t}(\D^{j+2+\iota}\hat{A}_{t,i,p,k}(\hat{x}'_t))|_{\hat{g}_t(\hat{x}_t)}}{d^{\hat{g}_t}(\hat{x}_t,\hat{x}'_t)^\alpha}\Bigg)\\
&+\frac{|\D^j\hat{\gamma}_{t,0}(\hat{x}_t)-\P_{\hat{x}'_t\hat{x}_t}(\D^j\hat{\gamma}_{t,0}(\hat{x}'_t))
|_{\hat{g}_t(\hat{x}_t)}}{d^{\hat{g}_t}(\hat{x}_t,\hat{x}'_t)^\alpha}
+\frac{|\D^j\hat{\eta}_{t,j,k}(\hat{x}_t)-\P_{\hat{x}'_t\hat{x}_t}(\D^j\hat{\eta}_{t,j,k}(\hat{x}'_t))
|_{\hat{g}_t(\hat{x}_t)}}{d^{\hat{g}_t}(\hat{x}_t,\hat{x}'_t)^\alpha}
=1.
\end{split}
\end{equation}

Now recall that $\hat{x}_t'$ was chosen to maximize the difference quotient of \eqref{blow22} among all points $\hat{x}_t' = (\hat{z}_t', \hat{y}_t')$  with $|\hat{z}'_t-\hat{z}_t|<\frac{1}{4}||\hat{z}_t|-\lambda_t|$ which are horizontally or vertically joined to $\hat{x}_t$.
Moreover, the point $\hat{x}_t$ itself maximizes the quantity
\begin{align}
\mu_{j,t}(\Psi_t(\hat{x}))= ||\hat{z}|-\lambda_t|^{j+\alpha} \sup_{\substack{\hat{x}'=(\hat{z}',\hat{y}')\ s.t.\ |\hat{z}'-\hat{z}|<\frac{1}{4}||\hat{z}|-\lambda_t|\\\hat{x}'\ \mathrm{and}\ \hat{x}\ \mathrm{horizontally}\ \mathrm{or}\ \mathrm{vertically}\ \mathrm{joined}}} \hat{\mathcal{D}}_{t,j}(\hat{x},\hat{x}')
\end{align}
among all $\hat{x} = (\hat{z},\hat{y}) \in B_{\lambda_t} \times Y$,
where we used the obvious notation
\begin{equation}\begin{split}
\hat{\mathcal{D}}_{t,j}(\hat{x},\hat{x}')&:=\sum_{i=2}^j \sum_{p=1}^{N_{i,k}}\sum_{\iota = -2}^{2k} \Bigg( \delta_t^\iota\frac{|\D^{j+2+\iota}\hat{A}_{t,i,p,k}(\hat{x})-\P_{\hat{x}'\hat{x}}(\D^{j+2+\iota}\hat{A}_{t,i,p,k}(\hat{x}'))|_{\hat{g}_{t}(\hat{x})}}{d^{\hat{g}_{t}}(\hat{x},\hat{x}')^\alpha}\Bigg)\\
&+\frac{|\D^j\hat{\gamma}_{t,0}(\hat{x})-\P_{\hat{x}'\hat{x}}(\D^j\hat{\gamma}_{t,0})(\hat{x}'))|_{\hat{g}_{t}(\hat{x})}}{d^{\hat{g}_{t}}(\hat{x},\hat{x}')^\alpha}
+\frac{|\D^j\hat{\eta}_{t,j,k}(\hat{x})-\P_{\hat{x}'\hat{x}}(\D^j\hat{\eta}_{t,j,k})(\hat{x}'))|_{\hat{g}_{t}(\hat{x})}}{d^{\hat{g}_{t}}(\hat{x},\hat{x}')^\alpha}.
\end{split}\end{equation}
We observe that
\begin{align}\label{fuck_the_boundary2}
||\hat{z}_t|-\lambda_t|=\max \mu_{j,t}^{\frac{1}{j+\alpha}} \to \infty,
\end{align}
and since by definition $|\hat{z}_t-\hat{z}'_t|\leq \frac{1}{4}||\hat{z}_t|-\lambda_t|$, we also have
\begin{align}\label{fuck_the_boundary}
||\hat{z}'_t|-\lambda_t|\geq\frac{3}{4}||\hat{z}_t|-\lambda_t|\to\infty.
\end{align}
This tells us that if we pass to a pointed limit with basepoint $\hat{x}_t$, the boundary moves away to infinity and the limit space will be complete. We also learn that for all $\hat{x} = (\hat{z}, \hat{y}) \in B_{\lambda_t} \times Y$,
\begin{align}\label{e:hubbabubba}
 \sup_{\substack{\hat{x}'=(\hat{z}',\hat{y}')\ s.t.\ |\hat{z}'-\hat{z}|<\frac{1}{4}||\hat{z}|-\lambda_t|\\\hat{x}'\ \mathrm{and}\ \hat{x}\ \mathrm{horizontally}\ \mathrm{or}\ \mathrm{vertically}\ \mathrm{joined}}} \hat{\mathcal{D}}_{t,j}(\hat{x},\hat{x}')\leq \frac{||\hat{z}_t|-\lambda_t|^{j+\alpha}}{||\hat{z}|-\lambda_t|^{j+\alpha}}.
 \end{align}

Using the triangle inequality and \eqref{fuck_the_boundary2}, \eqref{fuck_the_boundary}, we deduce in particular that
there exists a $C$ such that for all $R > 0$ there exists a $t_R$ and for $t \geq t_R$ and $\hat{z}^\bullet_t=\hat{z}_t$ or $\hat{z}'_t$, we have
\begin{align}
\label{saves_our_asses2}
\sup_{\substack{\hat{x},\hat{x}' \in \hat{B}_R(\hat{z}^\bullet_t)\times Y\\\hat{x}'\ \mathrm{and}\ \hat{x}\ \mathrm{horizontally}\ \mathrm{or}\ \mathrm{vertically}\ \mathrm{joined}}} \hat{\mathcal{D}}_{t,j}(\hat{x},\hat{x}')\leq C,
\end{align}
where here and in all the following, $\hat{B}_R(\hat{z})$ will denote a Euclidean ball in $\C^m$ (the hat is just a decoration to remind the reader that we are in the hat picture). Furthermore, if the center is one of the two points $\hat{z}_t$ and $\hat{z}'_t$ we may omit the center and the $Y$ factor and write simply $\hat{B}_R$ for short. Obviously \eqref{saves_our_asses2} implies
\begin{equation}\label{saves_our_asses}
\sum_{i=2}^j \sum_{p=1}^{N_{i,k}}\sum_{\iota = -2}^{2k} \Bigg( \delta_t^\iota[\D^{j+2+\iota}\hat{A}_{t,i,p,k}]_{C^\alpha(\hat{B}_R,\hat{g}_t)}\Bigg)
+[\D^j\hat{\gamma}_{t,0}]_{C^\alpha(\hat{B}_R,\hat{g}_t)}+[\D^j\hat{\eta}_{t,j,k}]_{C^\alpha(\hat{B}_R,\hat{g}_t)}\leq C,
\end{equation}
for every fixed $R$.

\subsection{The non-escaping property}

\begin{proposition}
Assume that $\delta_t\leq C$. Then we have the upper bound
\begin{align}\label{blow223}
d^{\hat{g}_t}(\hat{x}_t, \hat{x}_t')\leq C.
\end{align}
\end{proposition}

When $\delta_t\to\infty$ the proof below does not work, but we will soon show in \eqref{strabik} that an even better bound holds in that case.

\begin{proof}
To start, \eqref{saves_our_asses} gives in particular
\begin{equation}\label{danaos}
[\D^{j+2+\iota}\hat{A}_{t,i,p,k}]_{C^\alpha(\hat{B}_R,\hat{g}_t)}\leq C\delta_t^{-\iota},\quad -2\leq \iota\leq 2k,
\end{equation}
for any given $R$,
while from \eqref{caz5new}
\begin{equation}\label{kaz5new}
\|\hat{A}_{t,i,p,k}\|_{L^\infty(\hat{B}_R,\hat{g}_t)}\leq C\delta_t^2e^{\frac{-i+1-\alpha}{2}t},
\end{equation}
and we can interpolate between these two, first taking \eqref{danaos} with $\iota=-2$ to get
\begin{align}
\sum_{\ell=1}^{j}(R-\rho)^\ell\|\D^\ell\hat{A}_{t,i,p,k}\|_{L^\infty(\hat{B}_\rho)}&\leq C(R-\rho)^{j+\alpha}[\D^{j} \hat{A}_{t,i,p,k}]_{C^\alpha(\hat{B}_R)}+C\delta_t^2e^{\frac{-i+1-\alpha}{2}t}\\
&\leq C(R-\rho)^{j+\alpha}\delta_t^2+C\delta_t^2e^{\frac{-i+1-\alpha}{2}t},
\end{align}
so picking $R-\rho\sim \left(e^{\frac{-i+1-\alpha}{2}t}\right)^{\frac{1}{j+\alpha}}$ (which is small) we get
\begin{equation}\label{610}
\|\D^\ell\hat{A}_{t,i,p,k}\|_{L^\infty(\hat{B}_R,\hat{g}_t)}\leq C\delta_t^2\left(e^{\frac{-i+1-\alpha}{2}t}\right)^{1-\frac{\ell}{j+\alpha}},\quad 0\leq \ell\leq j,
\end{equation}
for any fixed $R$. Next, we take \eqref{danaos} with $0<\iota+2=:\ell$ and interpolate
\begin{align}
(R-\rho)^\ell\|\D^{j+\ell}\hat{A}_{t,i,p,k}\|_{L^\infty(\hat{B}_\rho)}&\leq C(R-\rho)^{\ell+\alpha}[\D^{j+\ell} \hat{A}_{t,i,p,k}]_{C^\alpha(\hat{B}_R)}+C\|\D^j\hat{A}_{t,i,p,k}\|_{L^\infty(\hat{B}_R)}\\
&\leq C(R-\rho)^{\ell+\alpha}\delta_t^{-\ell+2}+
C\delta_t^2\left(e^{\frac{-i+1-\alpha}{2}t}\right)^{\frac{\alpha}{j+\alpha}},
\end{align}
so picking $R-\rho\sim \delta_t^{\frac{\ell}{\ell+\alpha}}\left(e^{\frac{-i+1-\alpha}{2}t}\right)^{\frac{\alpha}{(j+\alpha)(\ell+\alpha)}}$ (which is small because of our assumption $\delta_t\leq C$) we get
\begin{equation}\label{611}
\|\D^{j+\ell}\hat{A}_{t,i,p,k}\|_{L^\infty(\hat{B}_R,\hat{g}_t)}\leq C\delta_t^{-\ell+2}\left(\delta_t^\ell \left(e^{\frac{-i+1-\alpha}{2}t}\right)^{\frac{\alpha}{j+\alpha}}
\right)^{\frac{\alpha}{\ell+\alpha}},\quad 1\leq \ell\leq 2k+2,
\end{equation}
for any fixed $R$,
and so combining \eqref{610} with $\ell=j$ and \eqref{611} we see that
\begin{equation}\label{ferentes}
\delta_t^\iota\|\D^{j+2+\iota}\hat{A}_{t,i,p,k}\|_{L^\infty(\hat{B}_R,\hat{g}_t)}=o\left(\delta_t^{\frac{(\iota+2)\alpha}{\iota+2+\alpha}}\right)=o(1),\quad -2\leq \iota \leq  2k,
\end{equation}
for any fixed $R$, and applying this to balls centered at $\hat{z}_t$ and $\hat{z}'_t$ gives in particular
\begin{equation}\label{timeo}
\sum_{i=2}^j\sum_{p=1}^{N_{i,k}}\sum_{\iota = -2}^{2k}
\delta_t^{\iota}|\D^{j+2+\iota}\hat{A}_{t,i,p,k}(\hat{x}_t)-\P_{\hat{x}'_t\hat{x}_t}(\D^{j+2+\iota}\hat{A}_{t,i,p,k}(\hat{x}'_t))|_{\hat{g}_t(\hat{x}_t)}=
o(1).
\end{equation}

Next, again thanks to \eqref{saves_our_asses}, for any fixed $R$ we have
\begin{equation}\label{chiodofisso}
[\D^j\hat{\eta}_{t,j,k}]_{C^\alpha(\hat{B}_R,\hat{g}_t)}\leq C,
\end{equation}
and Theorem \ref{prop55} gives (in the case when $\delta_t$ does not go to zero we need to assume here that $R$ is sufficiently large, which is of course allowed)
\begin{equation}\label{chiodini}
\|\D^j\hat{\eta}_{t,j,k}\|_{L^\infty(\hat{B}_R,\hat{g}_t)}\leq C\delta_t^\alpha,
\end{equation}
and applying this to balls centered at $\hat{z}_t$ and $\hat{z}'_t$ gives in particular (by also invoking Section \ref{s:tridiagonal})
\begin{equation}\label{timeo2}
|\D^j\hat{\eta}_{t,j,k}(\hat{x}_t)-\P_{\hat{x}'_t\hat{x}_t}(\D^j\hat{\eta}_{t,j,k}(\hat{x}'_t))
|_{\hat{g}_t(\hat{x}_t)}
\leq C\delta_t^\alpha.
\end{equation}

Similarly, from \eqref{saves_our_asses} we get
\begin{equation}
[\D^j\hat{\gamma}_{t,0}]_{C^\alpha(\hat{B}_R,\hat{g}_t)}\leq C,
\end{equation}
for any fixed $R$, and interpolating between this and \eqref{caz4new} easily gives
\begin{equation}\label{timeo3}
|\D^j\hat{\gamma}_{t,0}(\hat{x}_t)-\P_{\hat{x}'_t\hat{x}_t}(\D^j\hat{\gamma}_{t,0}(\hat{x}'_t))
|_{\hat{g}_t(\hat{x}_t)}
\leq C.
\end{equation}

Then combining \eqref{blow22} with \eqref{timeo}, \eqref{timeo2} and \eqref{timeo3} we obtain the desired bound \eqref{blow223}.
\end{proof}

We are now in a position to study the possible complete pointed limit spaces of
\begin{equation}(B_{\lambda_t}\times Y, \hat{g}_t,\hat{x}_t)\end{equation}
as $t\to\infty$. Modulo translations in the $\mathbb{C}^m$ factor, we may assume that $\hat{x}_t = (0,\hat{y}_t)\in \mathbb{C}^m\times Y$. Recall here that $\hat{y}_t\to \hat{y}_\infty\in Y$ by \eqref{whoops666}. At this point three cases need to be considered, again up to passing to a subsequence: (1) $\delta_t\to\infty$; (2) $\delta_t \to \delta \in (0,\infty)$, and without loss of generality $\delta = 1$; and (3) $\delta_t\to 0$. Observe that thanks to \eqref{whoops666}, it holds in all three cases that the complex structure converges locally smoothly to the product complex structure $J_{\C^m} + J_{Y,z_\infty}.$

\subsection{Case 1: the blowup is $\C^{m+n}$} In this case we assume that $\delta_t\to \infty$.

Recall from \eqref{fuck_the_boundary2} and \eqref{fuck_the_boundary} that the two points $\hat{z}_t$ (which we have translated to be equal to $0$) and $\hat{z}'_t$ satisfy that
\begin{equation}
||\hat{z}_t|-\lambda_t|\to\infty, \quad ||\hat{z}'_t|-\lambda_t|\geq \frac{3}{4}||\hat{z}_t|-\lambda_t|\to\infty,
\quad |\hat{z}_t-\hat{z}'_t|\leq \frac{1}{4}||\hat{z}_t|-\lambda_t|.
\end{equation}
Apply the diffeomorphism
\begin{equation}\Xi_t:  B_{e^{\frac{t}{2}}} \times Y \to B_{\lambda_t} \times Y, \;\, (\hat{z},\hat{y}) = \Xi_t(\check{z},\check{y}) =(\delta_t\check{z},\check{y}),\end{equation}
and as usual multiply all the pulled back contravariant $2$-tensors by $\delta_t^{-2}$, and denote the new objects by a check. Denote also $\check{A}_{t,i,p,k}=\delta_t^{-2}\Xi_t^*\hat{A}_{t,i,p,k}$. Then $\check{g}_t$ is locally uniformly Euclidean, and the points $\check{z}_t$ and $\check{z}'_t$ satisfy
\begin{equation}
||\check{z}_t|-e^{\frac{t}{2}}|=\delta_t^{-1}||\hat{z}_t|-\lambda_t|, \quad ||\check{z}'_t|-e^{\frac{t}{2}}|=\delta_t^{-1}||\hat{z}'_t|-\lambda_t|,
\quad |\check{z}_t-\check{z}'_t|=\delta_t^{-1}|\hat{z}_t-\hat{z}'_t|.
\end{equation}
Now the balls $\check{B}_1(\check{z}_t), \check{B}_1(\check{z}'_t)$ need not be contained in $B_{e^{\frac{t}{2}}}$, since we do not have any relation among $||\hat{z}_t|-\lambda_t|$ and $\delta_t$, but they are compactly contained in the larger ball $B_{(1+\sigma)e^{\frac{t}{2}}}$ for $t$ sufficiently large (cf. the beginning of \S \ref{mavaffa}), and on this larger ball times $Y$ the metric $\check{\omega}^\bullet_t$ is Ricci-flat and uniformly Euclidean, so standard local estimates for the complex Monge-Amp\`ere equation \cite[Proposition 2.3]{HT2} give us uniform $C^\infty$ estimates for $\check{\omega}^\bullet_t$ on compact subsets of $B_{(1+\sigma)e^{\frac{t}{2}}} \times Y$, and in particular on $\check{B}_1(\check{z}_t)\times Y$ and $\check{B}_1(\check{z}'_t)\times Y$.

Thanks to the definitions \eqref{fuenf}, \eqref{sechs} as fiber integrations, we deduce easily from this that the following objects (proved in this order) also have uniform uniform $C^\infty$ estimates on $\check{B}_1(\check{z}_t)\times Y$ and $\check{B}_1(\check{z}'_t)\times Y$: \begin{equation}\check{\gamma}_{t,0}, \check{\eta}_{t,0,k},\check{A}_{t,2,p,k},\check{\gamma}_{t,2,k},\check{\eta}_{t,2,k},\dots,
\check{A}_{t,j,p,k},\check{\gamma}_{t,j,k},\check{\eta}_{t,j,k},\end{equation}
where the stretched projection $\check{P}_{t,i,p,k}$ which is used here to define $\check{A}_{t,i,p,k}$ is defined in an analogous manner as in \eqref{proiettohat}, with $\check{G}_{i,p,k}=\Xi_t^*\hat{G}_{i,p,k}$.
Transferring these estimates back to the hat picture we obtain in particular that for $\hat{p}=\hat{z}_t,\hat{z}'_t,$
\begin{equation}\begin{split}\label{kk2}&\|\D^j\hat{\gamma}_{t,0}\|_{L^\infty(\hat{B}_{\delta_t}(\hat{p})\times Y,\hat{g}_t)}+\|\D^j\hat{\eta}_{t,j,k}\|_{L^\infty(\hat{B}_{\delta_t}(\hat{p})\times Y,\hat{g}_t)}\\
&+\sum_{i=2}^j\sum_{p=1}^{N_{i,k}}\sum_{\iota=-2}^{2k}\delta_t^\iota\|\D^{j+2+\iota}\hat{A}_{t,i,p,k}\|_{L^\infty(\hat{B}_{\delta_t}(\hat{p})\times Y,\hat{g}_t)}\leq C\delta_t^{-j},
\end{split}\end{equation}
\begin{equation}\begin{split}
\label{kk1}&[\D^j\hat{\gamma}_{t,0}]_{C^\alpha(\hat{B}_{\delta_t}(\hat{p})\times Y,\hat{g}_t)}+[\D^j\hat{\eta}_{t,j,k}]_{C^\alpha(\hat{B}_{\delta_t}(\hat{p})\times Y,\hat{g}_t)}\\
&+\sum_{i=2}^j\sum_{p=1}^{N_{i,k}}\sum_{\iota=-2}^{2k}\delta_t^\iota[\D^{j+2+\iota}\hat{A}_{t,i,p,k}]_{C^\alpha(\hat{B}_{\delta_t}(\hat{p})\times Y,\hat{g}_t)}\leq C\delta_t^{-j-\alpha}.
\end{split}\end{equation}
Using \eqref{kk2} and the triangle inequality to bound the numerators in \eqref{blow22} (using also the discussion in Section \ref{s:tridiagonal} to bound uniformly the operator norm of $\P$) gives
\begin{equation}\label{strabik}
d^{\hat{g}_t}(\hat{x}_t,\hat{x}'_t)^\alpha\leq C \delta_t^{-j},
\end{equation}
so the two points $\hat{z}_t$ and $\hat{z}'_t$ are nonescaping, and colliding when $j>0$. Thus $\hat{x}'_t$ belongs to $\hat{B}_{\delta_t}(\hat{z}_t)\times Y$ for all $t$ large, and so applying \eqref{kk1} shows that the quantity in \eqref{blow22}, which equals $1$, is also bounded above by $C\delta_t^{-j-\alpha}$, an obvious contradiction.

\subsection{Case 2: the blowup is $\C^m \times Y$}

 In this case we have that $\delta_t\to 1$, without loss of generality. We now have that
\begin{equation} \hat{g}^\natural_t \to g_{\rm can}(z_{\infty}) + g_{Y,z_\infty}=:g_P, \;\,   {\rm in}\;\,C^\infty_{\rm loc}(\C^m \times Y),\end{equation}
and the complex structure converges to a product, where recall that we have arranged that $z_\infty=0$.
We also know that $d^{\hat{g}_t}(\hat{x}_t, \hat{x}_t') \leq C$ by \eqref{blow223}, so passing to a subsequence we may assume that $\hat{x}_t' \to \hat{x}_\infty'$.

Thanks to \eqref{caz01}, we can apply standard local estimates for the complex Monge-Amp\`ere equation \cite[Proposition 2.3]{HT2} on small balls to obtain $C^\infty_{\rm loc}(\C^m \times Y)$ bounds for $\hat\omega_t^\bullet$. As in Case 1, thanks to the definitions as fiber integrations, we deduce easily from this that the following objects (proved in this order) also have uniform $C^\infty$ estimates:
\begin{equation}\hat{\gamma}_{t,0}, \hat{\eta}_{t,0,k},\hat{A}_{t,2,p,k},\hat{\gamma}_{t,2,k},\hat{\eta}_{t,2,k},\dots,
\hat{A}_{t,j,p,k},\hat{\gamma}_{t,j,k},\hat{\eta}_{t,j,k}.\end{equation}
This has many useful consequences. First, going into \eqref{blow22} (and using also Remark \ref{archiloco} to compare H\"older norms), we can see that all the objects appearing there are $C^\infty_{\rm loc}(\C^m \times Y)$ bounded, so by estimating the $C^\alpha$ difference quotients in \eqref{blow22} by $C^\beta$ ones for any $\beta > \alpha$, we conclude that the two points $\hat{x}_t = (0,\hat{y}_t)$ and $\hat{x}_t'$ have $\hat{g}_t$-distance uniformly bounded away from zero. Second, up to passing to a subsequence, we may assume that all of the above objects converge in $C^\infty_{\rm loc}(\C^m \times Y)$.

Let us denote by $\hat{\gamma}_{\infty,0}$ the smooth limit of $\hat{\gamma}_{t,0}$. We first consider the case $j>0$, where we claim that $\hat{\gamma}_{\infty,0}=0$. Indeed, by \eqref{vier3} and \eqref{vier2} with $j=0$, which hold by induction, we have a locally uniform $C^{0,\alpha}$ bound for $\gamma_{t,0}$, and since this also weakly converges to zero (thanks to \eqref{fucker3}), we conclude in particular that $\gamma_{t,0}\to 0$ locally uniformly, and this statement implies directly that $\hat{\gamma}_{t,0}\to 0$ locally uniformly (and hence locally smoothly), as desired.

We also know that $\hat{A}_{t,i,p,k}\to 0$ locally uniformly (thanks to \eqref{caz5new}),
and so it follows that $\hat{A}_{t,i,p,k}$ also goes to zero locally smoothly.

Recalling that by definition we have
\begin{equation}\label{expa}
\hat{\gamma}_{t,i,k} = \sum_{p=1}^{N_{i,k}}i\partial\overline\partial \hat{\mathfrak{G}}_{t,k}(\hat{A}_{t,i,p,k},\hat{G}_{i,p,k}),
\end{equation}
and that Lemma \ref{Gstructure} writes this as
\begin{equation}\label{expa2}
\hat{\gamma}_{t,i,k} = i\partial\overline\partial\sum_{p=1}^{N_{i,k}}\sum_{\iota=0}^{2k}\sum_{\ell=\lceil \frac{\iota}{2} \rceil}^{k} e^{-\left(\ell-\frac{\iota}{2}\right) t}\delta_t^{\iota}(\hat{\Phi}_{\iota,\ell}(\hat{G}_{i,p,k})\circledast \D^\iota \hat{A}_{t,i,p,k} ),
\end{equation}
we conclude that $\hat{\gamma}_{t,i,k}, 2\leq i\leq j,$ also go to zero locally smoothly.

The $C^\infty_{\rm loc}(\C^m \times Y)$ limit $\hat\omega_\infty^\bullet = \omega_P + \hat\eta_{\infty,j,k}$ of $\hat\omega_t^\bullet = \hat\omega^\natural_{t} + \hat\gamma_{t,0} + \hat\gamma_{t,2,k} +\cdots+\hat{\gamma}_{t,j,k}+ \hat\eta_{t,j,k}$ will be a Ricci-flat K\"ahler metric uniformly equivalent to the standard $\omega_P$. By the Liouville Theorem from \cite{He} (see also \cite{LLZ}), we see that $\nabla^{g_P}\hat\omega_\infty^\bullet = 0$, and hence
\begin{equation}\label{lioufuck}
\nabla^{g_P}\hat\eta_{\infty,j,k} = 0.
\end{equation}

Going back to \eqref{blow22}, we thus obtain a contradiction because the left-hand side of \eqref{blow22} manifestly converges to zero: the denominators converge to a strictly positive constant, and the numerators all go to zero (thanks to the above $C^\infty_{\rm loc}$ vanishing of the $\hat{A}_{t,i,p,k}$'s and $\hat{\gamma}_{t,0}$, and \eqref{lioufuck}). This completes the proof of Case 2 when $j>0$.

Assume now that $j=0$. Passing to the limit as above we obtain
\begin{equation}\hat{\omega}^\bullet_\infty=\omega_P+\hat{\gamma}_{\infty,0}+\hat{\eta}_{\infty,0,k},
\end{equation}
and we know that $\hat{\gamma}_{\infty,0}+\hat{\eta}_{\infty,0,k}$ is $\de\db$-exact by \cite[Proposition 3.11]{HT2}, say $\hat{\gamma}_{\infty,0}+\hat{\eta}_{\infty,0,k}=\ddbar\Upsilon$ for some smooth function $\Psi$ on $\C^m\times Y$. But recall that by definition $\hat{\eta}_{t,0,k}=\ddbar(\hat{\psi}_t-\underline{\hat{\psi}_t})$ has local uniform $C^\infty$ bounds, and since its potentials have fiberwise average zero, up to passing to a subsequence they converge locally smoothly to a limit smooth function $\zeta$ on $\C^m\times Y$ with fiberwise average zero, and $\hat{\eta}_{\infty,0,k}=\ddbar\zeta$. We conclude that $\hat{\gamma}_{\infty,0}=\ddbar (\Upsilon-\zeta)$, but we also know that $\hat{\gamma}_{\infty,0}$ is pulled back from the base $\mathbb{C}^m$, so we can write it as $\hat{\gamma}_{\infty,0}=\ddbar\Xi$ for some function smooth $\Xi$ on $\mathbb{C}^m$. Comparing these, it follows that $\Upsilon=\zeta+\Xi+(\text{pluriharmonic function})$, and since $Y$ is compact the pluriharmonic is also pulled back from $\mathbb{C}^m$. Taking the fiber average we thus conclude that
\begin{equation}
\ddbar\underline{\Upsilon}=\ddbar\Xi=\hat{\gamma}_{\infty,0}.
\end{equation}
As above the Liouville Theorem from \cite{He} gives $\nabla^{g_P}\hat\omega_\infty^\bullet = 0$, and so $\nabla^{g_P}\ddbar\Upsilon=0$. Since
\begin{equation}
\ddbar\underline{\Upsilon}=({\rm pr}_{\mathbb{C}^m})_*(\ddbar\Upsilon),
\end{equation}
applying $\nabla^{g_P}$ to this shows that
\begin{equation}
0=\nabla^{g_P}\ddbar\underline{\Upsilon}=\nabla^{g_P}\hat{\gamma}_{\infty,0},
\end{equation}
and so we again have that \eqref{lioufuck} holds and we conclude as in the case $j>0$ above.

\subsection{Case 3: the blowup is $\C^m$ (modulo linear regularity)}

We finally assume that $\delta_t\to 0$. This is by far the hardest case.

First, we can interpolate between the uniform bounds \eqref{caz4new} and the seminorm bound $[\D^j\hat{\gamma}_{t,0}]_{C^\alpha(B_R)}\leq C$ from \eqref{saves_our_asses}, to see that $\hat{\gamma}_{t,0}$ has a uniform $C^{j,\alpha}_{\rm loc}(\C^m)$ bound. By Ascoli-Arzel\`a, up to passing to a sequence $t_i\to\infty$, $\hat{\gamma}_{t,0}$ converges in $C^{j,\beta}_{\rm loc}(\C^m)$ for $\beta<\alpha$ to some limit $\hat{\gamma}_{\infty,0}$. When $j>0$ the limit $\hat{\gamma}_{\infty,0}$ is in fact zero, by the same argument that we used in Case 2 above using \eqref{fucker3}, but we cannot conclude that this holds when $j=0$.
Thus in particular this gives
\begin{equation}\label{pettytheft}
\|\D^\iota\hat{\gamma}_{t,0}\|_{L^\infty(\hat{B}_{R},\hat{g}_t)}=\begin{cases}O(1),\quad \iota=j=0,\\
o(1),\quad j>0,\ 0\leq \iota\leq j,
\end{cases}\end{equation}
for all $0\leq\iota\leq j$ and fixed $R$, where as usual $\hat{B}_R$ here can be centered at $\hat{z}_t$ or $\hat{z}'_t$.
In particular, \begin{equation}\label{timeo4}
|\D^j\hat{\gamma}_{t,0}(\hat{x}_t)-\P_{\hat{x}'_t\hat{x}_t}(\D^j\hat{\gamma}_{t,0}(\hat{x}'_t))
|_{\hat{g}_t(\hat{x}_t)}=\begin{cases}O(1),\quad \iota=j=0,\\
o(1),\quad j>0,\ 0\leq\iota\leq j.
\end{cases}
\end{equation}
Now the key claim is the following non-colliding estimate: there exists an $\epsilon>0$ such that for all $t$ it holds that
\begin{equation}\label{noncolla}
d^{\hat{g}_t}(\hat{x}_t, \hat{x}_t') \geq \epsilon.
\end{equation}
Assuming \eqref{noncolla}, let us quickly complete the proof of Theorem \ref{shutupandsuffer}.
Indeed from \eqref{blow22} together with \eqref{noncolla} we know that
\begin{equation}\label{cokka}\begin{split}
&\sum_{i=2}^j\sum_{p=1}^{N_{i,k}}\sum_{\iota = -2}^{2k}
\delta_t^{\iota}|\D^{j+2+\iota}\hat{A}_{t,i,p,k}(\hat{x}_t)-\P_{\hat{x}'_t\hat{x}_t}(\D^{j+2+\iota}\hat{A}_{t,i,p,k}(\hat{x}'_t))|_{\hat{g}_t(\hat{x}_t)}\\
&+|\D^j\hat{\gamma}_{t,0}(\hat{x}_t)-\P_{\hat{x}'_t\hat{x}_t}(\D^j\hat{\gamma}_{t,0}(\hat{x}'_t))
|_{\hat{g}_t(\hat{x}_t)}
+|\D^j\hat{\eta}_{t,j,k}(\hat{x}_t)-\P_{\hat{x}'_t\hat{x}_t}(\D^j\hat{\eta}_{t,j,k}(\hat{x}'_t))
|_{\hat{g}_t(\hat{x}_t)}\geq\ve^\alpha,
\end{split}\end{equation}
but from \eqref{timeo}, \eqref{timeo2} and \eqref{timeo4} we see that all terms on the LHS go to zero when $j>0$, which is a contradiction. On the other hand, when $j=0$, we see from \eqref{chiodini} that the $C^\beta_{\rm loc}(\C^m\times Y)$ limit $\hat{\omega}^\bullet_\infty$ of $\hat{\omega}^\bullet_t$ is of the form
\begin{equation}
\hat{\omega}^\bullet_\infty=\omega_{\C^m}+\hat{\gamma}_{\infty,0},
\end{equation}
and by \eqref{cokka} together with \eqref{timeo2} we see that $\hat{\omega}^\bullet_\infty$ is not constant on $\mathbb{C}^m$. We also know that $\hat{\omega}^\bullet_\infty$ has $C^\alpha_{\rm loc}(\C^m)$ coefficients and is uniformly equivalent to $\omega_{\C^m}$ by \eqref{caz01}. The argument in \cite[\S 5.3.3]{HT2}, which originates from \cite[Theorem 4.1]{To}, then shows that $\hat{\omega}^\bullet_\infty$ is a smooth Ricci-flat K\"ahler metric on $\C^m$, and we thus obtain a contradiction to the Liouville Theorem \cite[Theorem 2.4]{HT2}.

Thus, the proof of Theorem \ref{shutupandsuffer} is now complete modulo the crucial linear regularity claim \eqref{noncolla}, which will occupy almost all the rest of the paper.

\subsection{Set-up of the secondary (linear) blowup argument in Case 3}

If the desired estimate \eqref{noncolla} was false, then, since $\hat{x}_t \neq \hat{x}_t'$ for all $t$, there would exist a sequence $t_i \to \infty$ such that $d_{t_i} = d^{\hat{g}_{t_i}}(\hat{x}_{t_i}, \hat{x}_{t_i}') \to 0$. As usual, we will pretend that $d_t = d^{\hat{g}_t}(\hat{x}_t, \hat{x}_t') \to 0$. Define also a new parameter
\begin{equation}
\ve_t=d_t^{-1}\delta_t,
\end{equation}
and consider the diffeomorphisms
\begin{equation}\Theta_t:  B_{d_t^{-1}\lambda_t} \times Y \to B_{\lambda_t} \times Y, \;\, (\hat{z},\hat{y}) = \Theta_t(\tilde{z},\tilde{y}) = (d_t \tilde{z}, \tilde{y}).\end{equation}
Pull back all our objects under $\Theta_t$, multiply the metrics and $2$-forms by $d_t^{-2}$, and denote the resulting objects by the same letters with each hat replaced by a tilde. Define also $\ti{A}_{t,i,p,k}=d_t^{-2}\Theta_t^*\hat{A}_{t,i,p,k}$. Then first of all
\begin{align}\tilde{g}_{t} = g_{\C^m} +  \epsilon_t^2 g_{Y,z_0},\quad \tilde{\omega}_t^\natural = \ti{\omega}_{\rm can} + \epsilon_t^2\Theta_t^*\Psi_t^*\omega_F.
\end{align}
Secondly, thanks to \eqref{ma},
\begin{align}
(\ti{\omega}^\bullet_t)^{m+n}=(\tilde{\omega}^\natural_t + \ti{\gamma}_{t,0}+\ti{\gamma}_{t,2,k}+\cdots+\ti{\gamma}_{t,j,k}+\tilde{\eta}_{t,j,k})^{m+n} = c_t e^{\tilde{H}_t}(\tilde{\omega}^\natural_{t})^{m+n},\label{zumteufel151}
\end{align}
where the constants $c_t$ converge to $\binom{m+n}{n}$ and
\begin{align}\label{calmund}
\tilde{H}_t =  \log \frac{\ti{\omega}_{\rm can}^m \wedge (\epsilon_t^2\Theta_t^*\Psi_t^*\omega_F)^n}{(\ti{\omega}_{\rm can}+\epsilon_t^2\Theta_t^*\Psi_t^*\omega_F)^{m+n}}.
\end{align}

Next, from \eqref{blow22}, \eqref{saves_our_asses}, there is $C$ such that for all $R$ there is $t_R$ such that for all $t\geq t_R$
\begin{align}
\sum_{i=2}^j \sum_{p=1}^{N_{i,k}}\sum_{\iota = -2}^{2k} \Bigg( \ve_t^\iota[\D^{j+2+\iota}\ti{A}_{t,i,p,k}]_{C^\alpha(\ti{B}_{Rd_t^{-1}}(\ti{z}_t)\times Y,\ti{g}_t)}\Bigg)\nonumber\\
+[\D^j\ti{\gamma}_{t,0}]_{C^\alpha(\ti{B}_{Rd_t^{-1}}(\ti{z}_t)\times Y,\ti{g}_t)}+[\D^j\ti{\eta}_{t,j,k}]_{C^\alpha(\ti{B}_{Rd_t^{-1}}(\ti{z}_t)\times Y,\ti{g}_t)}\leq Cd_t^{j+\alpha},\label{whoknows112}\\
\sum_{i=2}^j\sum_{p=1}^{N_{i,k}}\sum_{\iota = -2}^{2k}\ve_t^\iota\frac{|\D^{j+2+\iota}\ti{A}_{t,i,p,k}(\ti{x}_t) -  \P_{\ti{x}'_t\ti{x}_t}(\D^{j+2+\iota}\ti{A}_{t,i,p,k}(\ti{x}'_t))|_{\ti{g}_t(\ti{x}_t)}}{d^{\ti{g}_t}(\ti{x}_t,\ti{x}'_t)^\alpha}\nonumber\\
+\frac{|\D^j\ti{\gamma}_{t,0}(\ti{x}_t) -  \P_{\ti{x}'_t\ti{x}_t}(\D^j\ti{\gamma}_{t,0}(\ti{x}'_t))|_{\ti{g}_t(\ti{x}_t)}}{d^{\ti{g}_t}(\ti{x}_t,\ti{x}'_t)^\alpha}
+\frac{|\D^j\ti\eta_{t,j,k}(\ti{x}_t) -  \P_{\ti{x}'_t\ti{x}_t}(\D^j\ti\eta_{t,j,k}(\ti{x}'_t))|_{\ti{g}_t(\ti{x}_t)}}{d^{\ti{g}_t}(\ti{x}_t,\ti{x}'_t)^\alpha}
=d_t^{j+\alpha},\label{zumteufel161}\\
d^{\tilde{g}_t}(\tilde{x}_t,\tilde{x}_t') = 1.\label{zumteufel162}
\end{align}

It is now the time to separate each of our objects into a jet part and a remainder. As it turns out, this separation for $\ti{\eta}_{t,j,k}$ will not be needed (this is an improvement over the analogous point in \cite{HT2}).

First, from now on $\ti{B}_R$ will always denote $\ti{B}_R(\ti{z}_t)\times Y$. This will always include the other blowup point $\ti{z}'_t$ provided $R>1$. As usual, $[\cdot]_{C^\alpha(\ti{B}_R,\ti{g}_t)}$ will denote the $\ti{g}_t$-seminorm defined using $\P$ as in \eqref{e:holderdef}, and we will also write
\begin{equation}
[\tau]_{C^\alpha_{\rm base}(\ti{B}_R,\ti{g}_t)} := \sup \left\{\frac{|\tau(x_0) - \P^\gamma_{t_1,t_0}(\tau(x_1))|_{\ti{g}_t}}{d^{\ti{g}_t}(x_0,x_1)^\alpha} \;\Bigg|\;
\parbox{59mm}{$\gamma: [t_0,t_1] \to \ti{B}_R(\ti{z}_t)\times Y$ horizontal\\ $\P$-geodesic
$x_i := \gamma(t_i)$ for $i = 0,1$ } \right\}
\end{equation}
which is a $\ti{g}_t$-seminorm where we only consider pairs of points that are horizontally joined.

Let us then discuss the jet subtraction for $\ti{A}_{t,i,p,k}$. Define a polynomial function $\ti{A}^\sharp_{t,i,p,k}$ as the $j$-jet of $\ti{A}_{t,i,p,k}$ at $\ti{z}_t=0$  with respect to the standard coordinates on $\C^m$, and define \begin{equation}\ti{A}^*_{t,i,p,k}=\ti{A}_{t,i,p,k}-\ti{A}^\sharp_{t,i,p,k},\end{equation}
so that $\ti{A}^*_{t,i,p,k}$ vanishes to order $j+1$ at $\ti{z}_t=0$.

Recall also that $\ti{\omega}_{\rm can}=d_t^{-2}\lambda_t^2\Theta_t^*\Psi_t^*\omega_{\rm can}$ where $\Psi_t\circ\Theta_t(z,y)=(d_t\lambda_t^{-1}z,y)$.
For all $\iota\geq 0$ and $0<\beta<1$ we have
\begin{equation}\label{krk}\begin{split}
&\|\D^\iota \ti{\omega}_{\rm can}\|_{L^\infty(\ti{B}_{\lambda_td_t^{-1}}(0),\ti{g}_t)}=d_t^{\iota}\lambda_t^{-\iota}\|\D^\iota\omega_{\rm can}\|_{L^\infty(B_{1}(0),g_t)}\leq Cd_t^{\iota}\lambda_t^{-\iota},\\
&[\D^\iota \ti{\omega}_{\rm can}]_{C^\beta(\ti{B}_{\lambda_td_t^{-1}}(0),\ti{g}_t)}=d_t^{\iota+\beta}\lambda_t^{-\iota-\beta}[\D^\iota\omega_{\rm can}]_{C^\beta(B_{1}(0),g_t)}\leq Cd_t^{\iota+\beta}\lambda_t^{-\iota-\beta}.\end{split}
\end{equation}

Defining $\ti{\psi}_t=d_t^{-2}\lambda_t^2\Theta_t^*\Psi_t^*\psi_t$, and letting $\underline{\ti{\psi}_t}$ be its fiberwise average, we have $\ti{\gamma}_{t,0}=\ddbar\underline{\ti{\psi}_t}$.
We again need to perform a jet subtraction to $\ti{\gamma}_{t,0}$ by defining a polynomial function $\ti{\psi}^\sharp_t$ as the $(j+2)$-jet of $\underline{\ti{\psi}_t}$ at $\ti{z}_t=0$  with respect to the standard coordinates on $\C^m$ and letting
\begin{equation}\ti{\eta}^\ddagger_t=\ddbar \ti{\psi}^\sharp_t,\quad \ti{\eta}^\diamond_t=\ti{\gamma}_{t,0}-\ti{\eta}^\ddagger_t.\end{equation}
so that $\ti{\eta}^\diamond_t$ vanishes to order $j+1$ at $x_t$.

Let us introduce some new notation.
Recall that we have defined
\begin{equation}\ti{\gamma}_{t,i,k} = \sum_{p=1}^{N_{i,k}}\ddbar\ti{\mathfrak{G}}_{t,k}(\ti{A}_{t,i,p,k},\ti{G}_{i,p,k}),\end{equation}
so let us split $\ti{A}_{t,i,p,k}=\ti{A}_{t,i,p,k}^*+\ti{A}_{t,i,p,k}^\sharp$ and define
\begin{equation}\label{nuov}\ti{\eta}^\circ_t=\sum_{i=2}^j\sum_{p=1}^{N_{i,k}}\ddbar\ti{\mathfrak{G}}_{t,k}(\ti{A}^*_{t,i,p,k},\ti{G}_{i,p,k}),\quad
\ti{\eta}^\dagger_t=\sum_{i=2}^j\sum_{p=1}^{N_{i,k}}\ddbar\ti{\mathfrak{G}}_{t,k}(\ti{A}^\sharp_{t,i,p,k},\ti{G}_{i,p,k}),\end{equation}
so that we have
\begin{equation}\ti{\omega}^\bullet_t=\ti{\omega}^\natural_t+\ti{\eta}^\ddagger_t+\ti{\eta}^\diamond_t+\ti{\eta}^\circ_t+\ti{\eta}^\dagger_t+\ti{\eta}_{t,j,k}.\end{equation}
Let us also write
\begin{equation}\ti{\omega}^\sharp_t =\ti{\omega}^\natural_t+\ti{\eta}^\dagger_t+\ti{\eta}^\ddagger_t,\end{equation}
so that
\begin{equation}\label{bunny}\ti{\omega}^\bullet_t=\ti{\omega}^\sharp_t+\ti{\eta}^\diamond_t+\ti{\eta}^\circ_t+\ti{\eta}_{t,j,k}.\end{equation}

Clearly equations \eqref{whoknows112} and \eqref{zumteufel161} hold verbatim with $\ti{A}_{t,i,p,k}$ and $\ti{\gamma}_{t,0}$ replaced by $\ti{A}^*_{t,i,p,k}$ and $\ti{\eta}^\diamond_t$ respectively, i.e. for any fixed $R$

\begin{equation}\label{whoknows113}\begin{split}
\sum_{i=2}^j \sum_{p=1}^{N_{i,k}}\sum_{\iota = -2}^{2k} \Bigg( \ve_t^\iota[\D^{j+2+\iota}\ti{A}^*_{t,i,p,k}]_{C^\alpha(\ti{B}_{Rd_t^{-1}},\ti{g}_t)}\Bigg)
+[\D^j\ti{\eta}^\diamond_t]_{C^\alpha(\ti{B}_{Rd_t^{-1}},\ti{g}_t)}+[\D^j\ti{\eta}_{t,j,k}]_{C^\alpha(\ti{B}_{Rd_t^{-1}},\ti{g}_t)}\leq Cd_t^{j+\alpha},
\end{split}
\end{equation}
\begin{equation}\label{zumteufel163}
\begin{split}
d_t^{-j-\alpha}\Bigg(&\sum_{i=2}^j\sum_{p=1}^{N_{i,k}}\sum_{\iota = -2}^{2k}\ve_t^\iota\frac{|\D^{j+2+\iota}\ti{A}^*_{t,i,p,k}(\ti{x}_t) -  \P_{\ti{x}'_t\ti{x}_t}(\D^{j+2+\iota}\ti{A}^*_{t,i,p,k}(\ti{x}'_t))|_{\ti{g}_t(\ti{x}_t)}}{d^{\ti{g}_t}(\ti{x}_t,\ti{x}'_t)^\alpha}\\
&+\frac{|\D^j\ti{\eta}^\diamond_t(\ti{x}_t) -  \P_{\ti{x}'_t\ti{x}_t}(\D^j\ti{\eta}^\diamond_t(\ti{x}'_t))|_{\ti{g}_t(\ti{x}_t)}}{d^{\ti{g}_t}(\ti{x}_t,\ti{x}'_t)^\alpha}
+\frac{|\D^j\ti{\eta}_{t,j,k}(\ti{x}_t) -  \P_{\ti{x}'_t\ti{x}_t}(\D^j\ti{\eta}_{t,j,k}(\ti{x}'_t))|_{\ti{g}_t(\ti{x}_t)}}{d^{\ti{g}_t}(\ti{x}_t,\ti{x}'_t)^\alpha}\Bigg)
=1.
\end{split}
\end{equation}

\subsection{Estimates on the solution components and on the background data }

The following section is the technical heart of the paper. Having split up the Monge-Amp\`ere equation into background, jets, and ``good'' parts as above, we now derive precise estimates on the various components, which will ultimately allow us to expand and linearize the Monge-Amp\`ere equation.

In the following sections, the radius $R$ will be any fixed radius, unless otherwise specified.

\subsubsection{Estimates for $\ti{\eta}_{t,j,k}$}

Recalling \eqref{saves_our_asses}, we can apply Theorem \ref{prop55} to get
\begin{equation}\label{crocefesso}\|\D^\iota\hat{\eta}_{t,j,k}\|_{L^\infty(\hat{B}_{R},\hat{g}_t)}\leq C\delta_t^{j+\alpha-\iota},\quad [\D^\iota\hat{\eta}_{t,j,k}]_{C^\alpha(\hat{B}_{R},\hat{g}_t)}\leq C \delta_t^{j-\iota},\end{equation}
for $0\leq \iota\leq j$, which in the tilde picture becomes
\begin{equation}\label{crocifisso7}
d_t^{-\iota}\|\D^\iota\ti{\eta}_{t,j,k}\|_{L^\infty(\ti{B}_{Rd_t^{-1}},\ti{g}_t)}\leq C\delta_t^{j+\alpha-\iota},\quad d_t^{-\iota-\alpha}[\D^\iota\ti{\eta}_{t,j,k}]_{C^\alpha(\ti{B}_{Rd_t^{-1}},\ti{g}_t)}\leq C \delta_t^{j-\iota},
\end{equation}
for $0\leq \iota\leq j,$ or equivalently
\begin{equation}\label{utilissimo}
d_t^{-j-\alpha}\|\D^\iota\ti{\eta}_{t,j,k}\|_{L^\infty(\ti{B}_{Rd_t^{-1}},\ti{g}_t)}\leq C\ve_t^{j+\alpha-\iota},\quad d_t^{-j-\alpha}[\D^\iota\ti{\eta}_{t,j,k}]_{C^\alpha(\ti{B}_{Rd_t^{-1}},\ti{g}_t)}\leq C \ve_t^{j-\iota}.
\end{equation}

\subsubsection{Estimates for $\ti{\gamma}_{t,0}, \ti{\eta}^\ddagger_t$ and $\ti{\eta}^\diamond_t$}

For any fixed $R$, thanks to \eqref{pettytheft} we have
\begin{equation}\label{beneficio2rep}
\|\D^\iota\hat{\gamma}_{t,0}\|_{L^\infty(\hat{B}_{R},\hat{g}_t)}=\begin{cases}O(1),\quad \iota=j=0,\\
o(1),\quad j>0,\ 0\leq \iota\leq j,
\end{cases}
\end{equation}
and from \eqref{saves_our_asses} we get
\begin{equation}\label{dicristo2}
[\D^j\hat{\gamma}_{t,0}]_{C^\alpha(\hat{B}_{R},\hat{g}_t)}\leq C,
\end{equation}
and so for $0\leq \iota\leq j,$
\begin{equation}\label{crocifisso2}
d_t^{-\iota}\|\D^\iota\ti{\gamma}_{t,0}\|_{L^\infty(\ti{B}_{Rd_t^{-1}},\ti{g}_t)}=\begin{cases}O(1),\quad \iota=j=0,\\
o(1),\quad j>0,\ 0\leq\iota\leq j,
\end{cases} \quad d_t^{-\iota-\alpha}[\D^\iota\ti{\gamma}_{t,0}]_{C^\alpha(\ti{B}_{Rd_t^{-1}},\ti{g}_t)}=\begin{cases}o(1),\quad 0\leq \iota<j,\\
O(1),\quad \iota=j.
\end{cases}
\end{equation}
Since $\ti{\eta}^\ddagger_t$ is annihilated by $[\D^j\cdot]$, it follows from \eqref{dicristo2} that
\begin{equation}\label{dicristo8}
[\D^j\ti{\eta}^\diamond_t]_{C^\alpha(\ti{B}_{Rd_t^{-1}},\ti{g}_t)}=[\D^j\ti{\gamma}_{t,0}]_{C^\alpha(\ti{B}_{Rd_t^{-1}},\ti{g}_t)}\leq Cd_t^{j+\alpha},
\end{equation}
which integrating along segments (starting at $\ti{x}_t$ where $\ti{\eta}^\diamond_t$ vanishes to order $j+1$) gives
\begin{equation}\label{crocifisso8}
d_t^{-\iota}\|\D^\iota\ti{\eta}^\diamond_t\|_{L^\infty(\ti{B}_S,\ti{g}_t)}\leq Cd_t^{j+\alpha-\iota}S^{j+\alpha-\iota}, \quad d_t^{-\iota-\alpha}[\D^\iota\ti{\eta}^\diamond_t]_{C^\alpha(\ti{B}_S,\ti{g}_t)}\leq Cd_t^{j-\iota}S^{j-\iota},
\end{equation}
for $0\leq \iota\leq j$ and $S\leq Rd_t^{-1}$. In particular, taking $S=\ti{R} \ve_t$ for $\ti{R}\leq R\delta_t^{-1}$ and get
\begin{equation}\label{crocifisso8bis}
d_t^{-\iota}\|\D^\iota\ti{\eta}^\diamond_t\|_{L^\infty(\ti{B}_{\ti{R}\ve_t},\ti{g}_t)}\leq C\delta_t^{j+\alpha-\iota}\ti{R}^{j+\alpha-\iota}, \quad d_t^{-\iota-\alpha}[\D^\iota\ti{\eta}^\diamond_t]_{C^\alpha(\ti{B}_{\ti{R}\ve_t},\ti{g}_t)}\leq C\delta_t^{j-\iota}\ti{R}^{j-\iota},
\end{equation}
which will prove useful when $\ve_t\geq C^{-1}$.

On the other hand, from its definition as a jet $\ti{\eta}^\ddagger_t$ inherits from \eqref{crocifisso2} the bounds
\begin{equation}\label{crocifisso9}
d_t^{-\iota}\|\D^\iota\ti{\eta}^\ddagger_t\|_{L^\infty(\ti{B}_{S},\ti{g}_t)}=\begin{cases}O(1),\quad \iota=j=0,\\
o(1),\quad j>0,\ 0\leq\iota\leq j,\\
0,\quad \iota>j,
\end{cases} \quad d_t^{-\iota-\alpha}[\D^\iota\ti{\eta}^\ddagger_t]_{C^\alpha(\ti{B}_{S},\ti{g}_t)}=\begin{cases}o(1),\quad 0\leq \iota<j,\\
0,\quad \iota\geq j.
\end{cases}
\end{equation}
for $S\leq Rd_t^{-1}$.

\subsubsection{Estimates for $\ti{A}^*_{t,i,p,k}$}

Since $\hat{A}^\sharp_{t,i,p,k}$ is a polynomial from the base of degree at most $j$, it follows from \eqref{saves_our_asses} that
\begin{equation}\label{lilith}
[\D^j\hat{A}^*_{t,i,p,k}]_{C^\alpha(\hat{B}_R,\hat{g}_t)}\leq C\delta_t^2,
\end{equation}
for all fixed $R$, which integrating along segments (starting at $\hat{x}_t$ where $\hat{A}^*_{t,i,p,k}$ vanishes to order $j+1$) gives
\begin{equation}\label{noreasontolabelthistoo}
\|\D^\iota\hat{A}^*_{t,i,p,k}\|_{L^\infty(\hat{B}_R,\hat{g}_t)}\leq CR^{j+\alpha-\iota}\delta_t^2,\quad [\D^\iota\hat{A}^*_{t,i,p,k}]_{C^\alpha(\hat{B}_R,\hat{g}_t)}\leq CR^{j-\iota}\delta_t^2,
\end{equation}
for $0\leq \iota\leq j$, and taking the radius $R=S d_t$ we can translate to the tilde picture
\begin{equation}\label{coveted}
d_t^{-\iota+2}\|\D^\iota\ti{A}^*_{t,i,p,k}\|_{L^\infty(\ti{B}_S,\ti{g}_t)}\leq Cd_t^{j+\alpha-\iota}\delta_t^2 S^{j+\alpha-\iota},\quad d_t^{-\iota+2-\alpha}[\D^\iota\ti{A}^*_{t,i,p,k}]_{C^\alpha(\ti{B}_S,\ti{g}_t)}\leq Cd_t^{j-\iota}\delta_t^2S^{j-\iota},
\end{equation}
for $0\leq \iota\leq j$ and $S\leq Rd_t^{-1}$.

For derivatives of order higher than $j$, recall that from \eqref{saves_our_asses} we have the bounds
\begin{equation}\label{ass_saved}
[\D^{j+\ell}\hat{A}^*_{t,i,p,k}]_{C^\alpha(\hat{B}_R,\hat{g}_t)}\leq C\delta_t^{2-\ell},
\end{equation}
for $0\leq\ell\leq 2k+2$ and $R$ fixed. For $1\leq\ell\leq 2k+2$ we then interpolate between this and the bound \eqref{noreasontolabelthistoo} with $\iota=j$ as follows:
\begin{equation}\label{intermed}\begin{split}
(R-\rho)^\ell\|\D^{j+\ell}\hat{A}^*_{t,i,p,k}\|_{L^\infty(\hat{B}_\rho,\hat{g}_t)}&\leq C(R-\rho)^{\ell+\alpha}[\D^{j+\ell}\hat{A}^*_{t,i,p,k}]_{C^\alpha(\hat{B}_R,\hat{g}_t)}+C\|\D^j\hat{A}^*_{t,i,p,k}\|_{L^\infty(\hat{B}_R,\hat{g}_t)}\\
&\leq C(R-\rho)^{\ell+\alpha}\delta_t^{2-\ell}+CR^{\alpha}\delta_t^{2}.
\end{split}\end{equation}
We choose $\rho=S\delta_t$ with $S>1$ fixed, let $A>1$ solve $\frac{A^{\frac{\alpha}{\ell+\alpha}}}{A-1}=S^{\frac{\ell}{\ell+\alpha}}$ and define $R=AS\delta_t$, so that
$R-\rho=\delta_t(AS)^{\frac{\alpha}{\ell+\alpha}}=(R^\alpha\delta_t^\ell)^{\frac{1}{\ell+\alpha}}$, and so we obtain
\begin{equation}\label{notlabeled}\|\D^{j+\ell}\hat{A}^*_{t,i,p,k}\|_{L^\infty(\hat{B}_{S\delta_t},\hat{g}_t)}\leq C_{S}\delta_t^{2-\ell+\alpha},\end{equation}
which in the tilde picture becomes
\begin{equation}\label{coveted2a}
d_t^{-j-\ell+2}\|\D^{j+\ell}\ti{A}^*_{t,i,p,k}\|_{L^\infty(\ti{B}_{S\ve_t},\ti{g}_t)}\leq C_{S} \delta_t^{2-\ell+\alpha} \end{equation}
for all $1\leq \ell\leq 2k+2$ and $S>1$ (hence also trivially for all $S$ fixed).
It will also be useful to rewrite \eqref{coveted} and \eqref{coveted2a} as
\begin{equation}\label{ottimo_a}\begin{split}
&d_t^{-j-\alpha}\|\D^{\iota}\ti{A}^*_{t,i,p,k}\|_{L^\infty(\ti{B}_S,\ti{g}_t)}\leq C \ve_t^2S^{j+\alpha-\iota},\\
&d_t^{-j-\alpha}\|\D^{j+\ell}\ti{A}^*_{t,i,p,k}\|_{L^\infty(\ti{B}_{S\ve_t},\ti{g}_t)}\leq C_{S} \ve_t^{2-\ell+\alpha},
\end{split}\end{equation}
where $0\leq \iota\leq j, 1\leq \ell\leq 2k+2$ and $S$ is fixed.

These are useful when $\ve_t\geq C^{-1}$. If on the other hand $\ve_t\to 0$, then in \eqref{intermed} we choose $R=2Sd_t$ with $S$ fixed, and pick $\rho=R-(R^\alpha\delta_t^\ell)^{\frac{1}{\ell+\alpha}}\geq \frac{R}{2}=Sd_t$ (provided $t$ sufficiently large, since $(R^\alpha\delta_t^\ell)^{\frac{1}{\ell+\alpha}}=(2S)^{\frac{\alpha}{\ell+\alpha}}d_t\ve_t^{\frac{\ell}{\ell+\alpha}}$; the lower bound for $t$ here does not even depend on $S$ provided $S\geq 1$) to obtain
\begin{equation}\label{covetedx}\|\D^{j+\ell}\hat{A}^*_{t,i,p,k}\|_{L^\infty(\hat{B}_{Sd_t},\hat{g}_t)}\leq C \delta_t^{2-\ell}(d_t^\alpha\delta_t^\ell)^{\frac{\alpha}{\ell+\alpha}}S^{\frac{\alpha^2}{\ell+\alpha}}=C\delta_t^{2-\ell}d_t^\alpha\ve_t^{\frac{\ell\alpha}{\ell+\alpha}} S^{\frac{\alpha^2}{\ell+\alpha}},\end{equation}
which in the tilde picture becomes
\begin{equation}\label{coveted2b}
d_t^{-j-\ell+2}\|\D^{j+\ell}\ti{A}^*_{t,i,p,k}\|_{L^\infty(\ti{B}_{S},\ti{g}_t)}\leq C\delta_t^{2-\ell}d_t^\alpha\ve_t^{\frac{\ell\alpha}{\ell+\alpha}} S^{\frac{\alpha^2}{\ell+\alpha}},
\end{equation}
for all $1\leq \ell\leq 2k+2$ and fixed $S$.
Again we can rewrite \eqref{coveted} and \eqref{coveted2b} as
\begin{equation}\label{ottimo}\begin{split}
&d_t^{-j-\alpha}\|\D^{\iota}\ti{A}^*_{t,i,p,k}\|_{L^\infty(\ti{B}_S,\ti{g}_t)}\leq C \ve_t^2S^{j+\alpha-\iota},\\
&d_t^{-j-\alpha}\|\D^{j+\ell}\ti{A}^*_{t,i,p,k}\|_{L^\infty(\ti{B}_S,\ti{g}_t)}\leq C \ve_t^{2-\ell}(S^\alpha \ve_t^\ell)^{\frac{\alpha}{\ell+\alpha}},
\end{split}\end{equation}
where $0\leq \iota\leq j, 1\leq \ell\leq 2k+2$ and $S$ fixed.

\subsubsection{Estimates for $\ti{A}^\sharp_{t,i,p,k}$}

From \eqref{610} we have
\begin{equation}
\|\D^\iota\hat{A}_{t,i,p,k}\|_{L^\infty(\hat{B}_R,\hat{g}_t)}\leq C\delta_t^2\left(e^{\frac{-i+1-\alpha}{2}t}\right)^{1-\frac{\iota}{j+\alpha}},\quad 0\leq \iota\leq j,
\end{equation}
for all given $R$ (with $C$ independent of $R$), and since $\hat{A}^\sharp_{t,i,p,k}$ is the $j$-jet of $\hat{A}_{t,i,p,k}$ at $\hat{z}_t=0$ then in particular all the coefficients of the polynomial $\hat{A}^\sharp_{t,i,p,k}$ have size bounded by $C\delta_t^2e^{\frac{-i+1-\alpha}{2}\frac{\alpha}{j+\alpha}t}$, and so
\begin{equation}\label{prechecazzo}\begin{split}
&\|\D^\iota\hat{A}^\sharp_{t,i,p,k}\|_{L^\infty(\hat{B}_R,\hat{g}_t)}\leq C\max(1,R^{j-\iota})\delta_t^2e^{\frac{-i+1-\alpha}{2}\frac{\alpha}{j+\alpha}t},\\
&[\D^\iota\hat{A}^\sharp_{t,i,p,k}]_{C^\beta(\hat{B}_R,\hat{g}_t)}\leq C\max(1,R^{j-\iota-\beta})\delta_t^2e^{\frac{-i+1-\alpha}{2}\frac{\alpha}{j+\alpha}t},
\end{split}\end{equation}
for $0\leq \iota\leq j, 0<\beta<1$, and so taking $R=Sd_t$ we obtain in the tilde picture
\begin{equation}\label{checazzo}
d_t^{-\iota+2}\|\D^\iota\ti{A}^\sharp_{t,i,p,k}\|_{L^\infty(\ti{B}_S,\ti{g}_t)}\leq C\max(1,S^{j-\iota}d_t^{j-\iota})\delta_t^2e^{\frac{-i+1-\alpha}{2}\frac{\alpha}{j+\alpha}t},
\end{equation}
\begin{equation}\label{checazzo2}
d_t^{-\iota+2-\beta}[\D^\iota\ti{A}^\sharp_{t,i,p,k}]_{C^\beta(\ti{B}_S,\ti{g}_t)}\leq C\max(1,S^{j-\iota-\beta}d_t^{j-\iota-\beta})\delta_t^2e^{\frac{-i+1-\alpha}{2}\frac{\alpha}{j+\alpha}t},
\end{equation}
for $0\leq \iota\leq j, 0<\beta<1, S\leq Cd_t^{-1}$.

\subsubsection{Estimates for $\ti{\eta}^\circ_t$}

By definition, we seek to bound derivatives of
\begin{equation}\label{krummung}\hat{\eta}^\circ_t=\sum_{i=2}^j\sum_{p=1}^{N_{i,k}}\ddbar\hat{\mathfrak{G}}_{t,k}(\hat{A}^*_{t,i,p,k},\hat{G}_{i,p,k}),\end{equation}
where from Lemma \ref{Gstructure} we have
\begin{equation}\label{expa3}
\hat{\eta}^\circ_t = i\partial\overline\partial\sum_{i=2}^j\sum_{p=1}^{N_{i,k}}\sum_{\iota=0}^{2k}\sum_{\ell=\lceil \frac{\iota}{2} \rceil}^{k} e^{-\left(\ell-\frac{\iota}{2}\right) t}\delta_t^{\iota}(\hat{\Phi}_{\iota,\ell}(\hat{G}_{i,p,k})\circledast \D^\iota \hat{A}^*_{t,i,p,k}),
\end{equation}
and we will apply $\D^r$ to this ($r\leq j$). Note that in general we do not have that $\ddbar$ is schematically of type $\D^2$, because the product metrics $g_{z,t}$ that are used to define $\D$ are in general not K\"ahler with respect to the complex structure (which after our stretchings we denote by $\hat{J}^\natural_t$, as in \cite[\S 5]{HT2}). Rather, using that for a function $u$ we have $\ddbar u =\frac{1}{2}d(\hat{J}^\natural_t du),$ we can write schematically
\begin{equation}\label{preQ}
\ddbar = \hat{J}^\natural_t\circledast\D^2 + (\D \hat{J}^\natural_t)\circledast \D,
\end{equation}
and hence for $r\geq 0$
\begin{equation}\label{Q2}
\D^r\ddbar = \sum_{s=0}^{r+1}(\D^{r+1-s} \hat{J}^\natural_t)\circledast \D^{s+1}.
\end{equation}
Observe that for every $\iota\geq 0$ and fixed $R$ we have the bounds
\begin{equation}\label{Q}
\|\D^\iota \hat{J}^\natural_t\|_{L^\infty(\hat{B}_R,\hat{g}_t)}\leq C\delta_t^{-\iota},\quad [\D^\iota \hat{J}^\natural_t]_{C^\alpha(\hat{B}_R,\hat{g}_t)}\leq C\delta_t^{-\iota-\alpha},
\end{equation}
which are straightforward after noting that
we have $(\D^\iota\hat{J}^\natural_t)_{\mathbf{f}\cdots\mathbf{f}}^{\mathbf{b}}=0$ because the fibers of ${\rm pr}_B$ are $\hat{J}^\natural_t$-complex and when differentiating purely vertically (say in the fiber over $\hat{z}$) $\D^\iota$ is the same as $\nabla^{\hat{z},\iota}$, the iterated covariant derivative of the product metric Riemannian metric $g_{\C^m}+g_{Y,\hat{z}}$.
Using \eqref{Q2} we have the schematics
\begin{equation}\label{schemas}
\D^{r}\ddbar\left(\hat{\Phi}_{\iota,\ell}(\hat{G}_{i,p,k})\circledast \D^\iota \hat{A}^*_{t,i,p,k}\right)=\sum_{s=0}^{r+1}\sum_{i_1+i_2=s+1}(\D^{r+1-s} \hat{J}^\natural_t)\circledast\D^{i_1}\hat{\Phi}_{\iota,\ell}(\hat{G}_{i,p,k})\circledast \D^{i_2+\iota}\hat{A}^*_{t,i,p,k},
\end{equation}
and so we obtain
\begin{equation}\label{idiot2}
\D^r\hat{\eta}^\circ_t=\sum_{i=2}^j\sum_{p=1}^{N_{i,k}}\sum_{\iota=0}^{2k}\sum_{\ell=\lceil \frac{\iota}{2} \rceil}^{k}\sum_{s=0}^{r+1}\sum_{i_1+i_2=s+1} e^{-\left(\ell-\frac{\iota}{2}\right) t}\delta_t^{\iota}(\D^{r+1-s} \hat{J}^\natural_t)\circledast\D^{i_1}\hat{\Phi}_{\iota,\ell}(\hat{G}_{i,p,k})\circledast \D^{i_2+\iota}\hat{A}^*_{t,i,p,k}.
\end{equation}
To estimate these sums we will also use the simple bounds (worst-case scenario)
\begin{equation}\|\D^i\hat{\Phi}\|_{L^\infty(\hat{B}_R,\hat{g}_t)}\leq C\delta_t^{-i},\quad [\D^i\hat{\Phi}]_{C^\alpha(\hat{B}_R,\hat{g}_t)}\leq C\delta_t^{-i-\alpha}.\end{equation}
Taking \eqref{noreasontolabelthistoo} with radius $R=\ti{R}\delta_t$ with $\ti{R}$ fixed, together with \eqref{notlabeled} gives
\begin{equation}
\|\D^q\hat{A}^*_{t,i,p,k}\|_{L^\infty(\hat{B}_{\ti{R}\delta_t},\hat{g}_t)}\leq C_{\ti{R}}\delta_t^{j+2+\alpha-q},\quad 0\leq q\leq 2k+2+j,
\end{equation}
and similarly from \eqref{noreasontolabelthistoo} and \eqref{ass_saved}
\begin{equation}
[\D^q\hat{A}^*_{t,i,p,k}]_{C^\alpha(\hat{B}_{\ti{R}\delta_t},\hat{g}_t)}\leq C_{\ti{R}}\delta_t^{j+2-q},\quad 0\leq q\leq 2k+2+j,
\end{equation}
and so in \eqref{idiot2} using also \eqref{Q} we can bound
\begin{equation}\label{zena}
\|(\D^{r+1-s} \hat{J}^\natural_t)\circledast\D^{i_1}\hat{\Phi}_{\iota,\ell}(\hat{G}_{i,p,k})\circledast \D^{i_2+\iota}\hat{A}^*_{t,i,p,k}\|_{L^\infty(\hat{B}_{\ti{R}\delta_t},\hat{g}_t)}\leq C_{\ti{R}}\delta_t^{-r-1+s}\delta_t^{j+2+\alpha-i_1-i_2-\iota}=C_{\ti{R}}\delta_t^{j+\alpha-r-\iota},
\end{equation}
\begin{equation}\begin{split}\label{zdva}
[(\D^{r+1-s}& \hat{J}^\natural_t)\circledast\D^{i_1}\hat{\Phi}_{\iota,\ell}(\hat{G}_{i,p,k})\circledast \D^{i_2+\iota}\hat{A}^*_{t,i,p,k}]_{C^\alpha(\hat{B}_{\ti{R}\delta_t},\hat{g}_t)}\\
&\leq C_{\ti{R}}\delta_t^{-r-1+s-\alpha}\delta_t^{j+2+\alpha-i_1-i_2-\iota}+C_{\ti{R}}\delta_t^{-r-1+s}\delta_t^{j+2-i_1-i_2-\iota}\leq
C_{\ti{R}}\delta_t^{j-r-\iota},
\end{split}\end{equation}
and combining \eqref{idiot2}, \eqref{zena}, \eqref{zdva} gives
\begin{equation}\label{crocifisso3ghost}
\|\D^r\hat{\eta}^\circ_t\|_{L^\infty(\hat{B}_{\ti{R}\delta_t},\hat{g}_t)}\leq C_{\ti{R}}\delta_t^{j+\alpha-r},\quad [\D^r\hat{\eta}^\circ_t]_{C^\alpha(\hat{B}_{\ti{R}\delta_t},\hat{g}_t)}\leq C_{\ti{R}}\delta_t^{j-r},
\end{equation}
for $0\leq r\leq j$, which in the tilde picture becomes
\begin{equation}\label{crocifisso3}
d_t^{-r}\|\D^r\ti{\eta}^\circ_t\|_{L^\infty(\ti{B}_{\ti{R}\ve_t},\ti{g}_t)}\leq C_{\ti{R}}\delta_t^{j+\alpha-r},\quad d_t^{-r-\alpha}[\D^r\ti{\eta}^\circ_t]_{C^\alpha(\ti{B}_{\ti{R}\ve_t},\ti{g}_t)}\leq C_{\ti{R}}\delta_t^{j-r},
\end{equation}
and which we can rewrite as
\begin{equation}\label{utilissimo2}
d_t^{-j-\alpha}\|\D^r\ti{\eta}^\circ_t\|_{L^\infty(\ti{B}_{\ti{R}\ve_t},\ti{g}_t)}\leq C_{\ti{R}}\ve_t^{j+\alpha-r},\quad d_t^{-j-\alpha}[\D^r\ti{\eta}^\circ_t]_{C^\alpha(\ti{B}_{\ti{R}\ve_t},\ti{g}_t)}\leq C_{\ti{R}} \ve_t^{j-r}.
\end{equation}

These estimates will be useful in the case when $\ve_t\geq C^{-1}$. On the other hand when $\ve_t\to 0$, we shall take only derivatives and difference quotient in the base directions, however the $\de\db$ in \eqref{expa3} can still be in all directions, so in the term
\begin{equation}
(\D^{r+1-s} \hat{J}^\natural_t)\circledast\D^{i_1}\hat{\Phi}_{\iota,\ell}(\hat{G}_{i,p,k}),
\end{equation}
in \eqref{idiot2} we can have at most two fiber derivatives. To bound this term, observe that if we take $u$ fiber derivatives with $u\leq 2$ (we shall denote this by $\D^{i_1}_{\#\mathbf{f}=u}$)
and the remaining $i_1-u$ derivatives are in base directions, then recalling that $\hat{\Phi}_{\iota,\ell}=\Psi_t^*\Phi_{\iota,\ell}, \hat{J}^\natural_t=\Psi_t^*J$ we have
\begin{equation}
\|\D^{i_1}_{\#\mathbf{f}=u}\hat{\Phi}_{\iota,\ell}(\hat{G}_{i,p,k})\|_{L^\infty(\hat{B}_R,\hat{g}_t)}\leq C\lambda_t^{-i_1+u}\delta_t^{-u},\quad
[\D^{i_1}_{\#\mathbf{f}=u}\hat{\Phi}_{\iota,\ell}(\hat{G}_{i,p,k})]_{C^\alpha_{\rm base}(\hat{B}_{R},\hat{g}_t)}\leq C\lambda_t^{-i_1+u-\alpha}\delta_t^{-u},
\end{equation}
\begin{equation}\label{Q3}
\|\D^\iota_{\#\mathbf{f}=u} \hat{J}^\natural_t\|_{L^\infty(\hat{B}_R,\hat{g}_t)}\leq C\lambda_t^{-\iota+u}\delta_t^{-u},\quad [\D^\iota_{\#\mathbf{f}=u} \hat{J}^\natural_t]_{C^\alpha_{\rm base}(\hat{B}_R,\hat{g}_t)}\leq C\lambda_t^{-\iota+u-\alpha}\delta_t^{-u},
\end{equation}
\begin{equation}\label{Q4}
\|((\D^{r+1-s} \hat{J}^\natural_t)\circledast\D^{i_1}\hat{\Phi}_{\iota,\ell}(\hat{G}_{i,p,k}))_{\#\mathbf{f}=u}\|_{L^\infty(\hat{B}_R,\hat{g}_t)}\leq C\lambda_t^{-r-1+s-i_1+u}\delta_t^{-u},
\end{equation}
\begin{equation}\label{Q5}
[((\D^{r+1-s} \hat{J}^\natural_t)\circledast\D^{i_1}\hat{\Phi}_{\iota,\ell}(\hat{G}_{i,p,k}))_{\#\mathbf{f}=u}]_{C^\alpha_{\rm base}(\hat{B}_{R},\hat{g}_t)}\leq C\lambda_t^{-r-1+s-i_1+u-\alpha}\delta_t^{-u},
\end{equation}
and so we can use \eqref{noreasontolabelthistoo} with radius $R=S d_t$ with $S$ fixed, together with \eqref{covetedx} to bound
\begin{equation}\label{ena2}\begin{split}
&\delta_t^\iota\|((\D^{r+1-s} \hat{J}^\natural_t)\circledast\D^{i_1}\hat{\Phi}_{\iota,\ell}(\hat{G}_{i,p,k}))_{\#\mathbf{f}=u}\circledast \D^{i_2+\iota}\hat{A}^*_{t,i,p,k}\|_{L^\infty(\hat{B}_{Sd_t},\hat{g}_t)}\\
&\leq C_S\delta_t^\iota\lambda_t^{-r-1+s-i_1+u}\delta_t^{-u}\begin{cases}
\delta_t^2 d_t^{j+\alpha-i_2-\iota}, \quad i_2+\iota\leq j\\
d_t^\alpha\delta_t^{2+j-i_2-\iota}\ve_t^{\frac{(i_2+\iota-j)\alpha}{i_2+\iota-j+\alpha}}, \quad i_2+\iota>j
\end{cases}\\
&\leq C_Sd_t^{j+\alpha-r},
\end{split}
\end{equation}
since $i_2+u\leq r+2$ and using also $u\leq 2$ (and that $\lambda_t^{-1}$ and $\ve_t$ are bounded and only appear here with nonnegative powers). On the other hand, using \eqref{noreasontolabelthistoo}, \eqref{ass_saved} we have
\begin{equation}\label{dva2}\begin{split}
&\delta_t^\iota[((\D^{r+1-s} \hat{J}^\natural_t)\circledast\D^{i_1}\hat{\Phi}_{\iota,\ell}(\hat{G}_{i,p,k}))_{\#\mathbf{f}=u}\circledast \D^{i_2+\iota}\hat{A}^*_{t,i,p,k}]_{C^\alpha_{\rm base}(\hat{B}_{Sd_t},\hat{g}_t)}\\
&\leq
C_Sd_t^{j+\alpha-r}\lambda_t^{-\alpha}+
C_S\delta_t^\iota\lambda_t^{-r-1+s-i_1+u}\delta_t^{-u}\begin{cases}
\delta_t^2 d_t^{j-i_2-\iota}, \quad i_2+\iota\leq j\\
\delta_t^{2+j-i_2-\iota}, \quad i_2+\iota>j
\end{cases}\\
&\leq
C_Sd_t^{j-r},
\end{split}\end{equation}
and in particular we obtain
\begin{equation}\label{crocifisso3bis}
d_t^{-r}\|\D^r_{\mathbf{b\cdots b}}\ti{\eta}^\circ_t\|_{L^\infty(\ti{B}_{S},\ti{g}_t)}\leq C_Sd_t^{j+\alpha-r},\quad d_t^{-r-\alpha}[\D^r_{\mathbf{b\cdots b}}\ti{\eta}^\circ_t]_{C^\alpha_{\mathrm{base}}(\ti{B}_S,\ti{g}_t)}\leq C_Sd_t^{j-r},
\end{equation}
for $0\leq r\leq j$ and $S$ fixed.

It is important to make here the following observation: in \eqref{ena2}, whenever in the above estimates we converted a $\delta_t$ into a $d_t$ (using that $\ve_t\to 0$), or we threw away a negative power of $\lambda_t$ or the term $e^{-\left(\ell-\frac{\iota}{2}\right)t}$, the actual result is $o(d_t^{j+\alpha-r})$ rather than $O(d_t^{j+\alpha-r})$. For later use, in the case when $r=j$, we need to identify exactly which terms in \eqref{ena2} are not a priori $o(d_t^{\alpha})$. Inspecting the above bounds, we must have $i_1=u+s-j-1$ so that the factor of $\lambda_t^{-1}$ is absent, and hence $i_2=j+2-u$, and we must also have $\ell=\frac{\iota}{2}$ so that the exponential term is absent. Let us first examine the case when $i_2+\iota\leq j$.  This means that $j+2-u+\iota\leq j$ i.e. $\iota\leq u-2\leq 0$, and so we must have $\ell=\iota=0, u=2,i_2=j,j+1-s+i_1=2$, so in particular $j-1\leq s\leq j+1$ and the terms $\sum_{s=j-1}^{j+1}((\D^{j+1-s} \hat{J}^\natural_t)\circledast\D^{s+1-j}\hat{\Phi}_{0,0}(\hat{G}_{i,p,k}))_{\#\mathbf{f}=2}$ here are just equal to $(\ddbar\hat{\Phi}_{0,0}(\hat{G}_{i,p,k}))_{\mathbf{ff}}$. The resulting term is $O(d_t^\alpha)$ but not a priori smaller, and thanks to \eqref{sangennaro} it equals
\begin{equation}
\sum_{i=2}^j\sum_{p=1}^{N_{i,k}}(\ddbar(\Delta^{\Psi_t^*\omega_F|_{\{\cdot\}\times Y}})^{-1}\hat{G}_{i,p,k})_{\mathbf{ff}} \D^j\hat{A}^*_{t,i,p,k}.
\end{equation}
On the other hand, in the second case when $i_2+\iota>j$ we always get $o(d_t^\alpha)$ thanks to the term $\ve_t^{\frac{(i_2+\iota-j)\alpha}{i_2+\iota-j+\alpha}}$ which goes to zero by assumption.
So the conclusion is that in $L^\infty(\ti{B}_S,\ti{g}_t)$ we have
\begin{equation}\label{gimme}
d_t^{-j-\alpha}\D^j_{\mathbf{b\cdots b}}\ti{\eta}^\circ_t
=d_t^{-j-\alpha}\sum_{i=2}^j\sum_{p=1}^{N_{i,k}}(\ddbar(\Delta^{\Psi_t^*\omega_F|_{\{\cdot\}\times Y}})^{-1}\ti{G}_{i,p,k})_{\mathbf{ff}} \D^j\ti{A}^*_{t,i,p,k}+o(1).
\end{equation}

Lastly, let us show that when $\ve_t\to 0$ if replace the shrinking metrics $\ti{g}_t$ with the fixed metric $g_{X}$ we do get
\begin{equation}\label{crocifisso3ter}
d_t^{-r}\|\D^r\ti{\eta}^\circ_t\|_{L^\infty(\ti{B}_S,g_{X})}\leq C_Sd_t^{j+\alpha-r},\quad d_t^{-r-\alpha}[\D^r\ti{\eta}^\circ_t]_{C^\alpha(\ti{B}_S,g_{X})}\leq C_Sd_t^{j-r},
\end{equation}
for $0\leq r\leq j$ and fixed $S$.
This is proved similarly to \eqref{crocifisso3bis}, using that all derivatives of $\ti{\Phi}_{\iota,\ell}$ and $\ti{J}^\natural_t$ have uniformly bounded norm with respect to $g_X$. Briefly, in the tilde picture we have
\begin{equation}\label{krktk}
\ti{\eta}^\circ_t = i\partial\overline\partial\sum_{i=2}^j\sum_{p=1}^{N_{i,k}}\sum_{\iota=0}^{2k}\sum_{\ell=\lceil \frac{\iota}{2} \rceil}^{k} e^{-\left(\ell-\frac{\iota}{2}\right) t}\ve_t^{\iota}(\ti{\Phi}_{\iota,\ell}(\ti{G}_{i,p,k})\circledast \D^\iota \ti{A}^*_{t,i,p,k}),
\end{equation}
and we apply $\D^r$ to this ($r\leq j$), using again the schematics
\begin{equation}\label{schemata}
\D^{r}\ddbar\left(\ti{\Phi}_{\iota,\ell}(\ti{G}_{i,p,k})\circledast \D^\iota \ti{A}^*_{t,i,p,k}\right)=\sum_{s=0}^{r+1}\sum_{i_1+i_2=s+1}(\D^{r+1-s}\ti{J}^\natural_t)\circledast\D^{i_1}\ti{\Phi}_{\iota,\ell}(\ti{G}_{i,p,k})\circledast \D^{i_2+\iota}\ti{A}^*_{t,i,p,k},
\end{equation}
we can bound
\begin{equation}\begin{split}
&d_t^{-r}\ve_t^\iota\|(\D^{r+1-s}\ti{J}^\natural_t)\circledast\D^{i_1}\ti{\Phi}_{\iota,\ell}(\ti{G}_{i,p,k})\circledast \D^{i_2+\iota}\ti{A}^*_{t,i,p,k}\|_{L^\infty(\ti{B}_{S},g_X)}\\
&\leq C_Sd_t^{-r-1+s-i_1}\delta_t^\iota \begin{cases}
\delta_t^2 d_t^{j+\alpha-i_2-\iota}, \quad i_2+\iota\leq j\\
d_t^\alpha\delta_t^{2+j-i_2-\iota}\ve_t^{\frac{(i_2+\iota-j)\alpha}{i_2+\iota-j+\alpha}}, \quad i_2+\iota>j
\end{cases}\\
&\leq C_Sd_t^{j+\alpha-r},
\end{split}
\end{equation}
and
\begin{equation}\begin{split}
&d_t^{-r-\alpha}\ve_t^\iota[(\D^{r+1-s}\ti{J}^\natural_t)\circledast\D^{i_1}\ti{\Phi}_{\iota,\ell}(\ti{G}_{i,p,k})\circledast \D^{i_2+\iota}\ti{A}^*_{t,i,p,k}]_{C^\alpha(\ti{B}_{S},g_X)}\\
&\leq
C_Sd_t^{j-r}+
C_Sd_t^{-r-1+s-i_1}\delta_t^\iota\begin{cases}
\delta_t^2 d_t^{j-i_2-\iota}, \quad i_2+\iota\leq j\\
\delta_t^{2+j-i_2-\iota}, \quad i_2+\iota>j
\end{cases}\\
&\leq
C_Sd_t^{j-r},
\end{split}\end{equation}
and \eqref{crocifisso3ter} follows.

\subsubsection{Estimates for $\ti{\eta}^\dagger_t$}

We seek to bound
\begin{equation}\hat{\eta}^\dagger_t=\sum_{i=2}^j\sum_{p=1}^{N_{i,k}}\ddbar\hat{\mathfrak{G}}_{t,k}(\hat{A}^\sharp_{t,i,p,k},\hat{G}_{i,p,k}),\end{equation}
where from Lemma \ref{Gstructure} we have
\begin{equation}\label{expa4}
\hat{\eta}^\dagger_t = i\partial\overline\partial\sum_{i=2}^j\sum_{p=1}^{N_{i,k}}\sum_{\iota=0}^{2k}\sum_{\ell=\lceil \frac{\iota}{2} \rceil}^{k} e^{-\left(\ell-\frac{\iota}{2}\right) t}\delta_t^{\iota}(\hat{\Phi}_{\iota,\ell}(\hat{G}_{i,p,k})\circledast \D^\iota \hat{A}^\sharp_{t,i,p,k}),
\end{equation}
and we will apply $\D^r$ to this, for any $r\geq 0$.  Recall from \eqref{prechecazzo} that
\begin{equation}\begin{split}
&\|\D^\iota\hat{A}^\sharp_{t,i,p,k}\|_{L^\infty(\hat{B}_R,\hat{g}_t)}\leq C\max(1,R^{j-\iota})\delta_t^2e^{\frac{-i+1-\alpha}{2}\frac{\alpha}{j+\alpha}t},\\
&[\D^\iota\hat{A}^\sharp_{t,i,p,k}]_{C^\beta(\hat{B}_R,\hat{g}_t)}\leq C\max(1,R^{j-\iota-\beta})\delta_t^2e^{\frac{-i+1-\alpha}{2}\frac{\alpha}{j+\alpha}t},
\end{split}\end{equation}
for $0\leq \iota\leq j, 0<\beta<1$ and fixed $R$, while of course derivatives of order $>j$ vanish.
We have to use again the schematics analogous to \eqref{schemas} and the bounds (for $0<\beta<1$ and fixed $R$)
\begin{equation}\label{ena3}
\|(\D^{r+1-s} \hat{J}^\natural_t)\circledast\D^{i_1}\hat{\Phi}_{\iota,\ell}(\hat{G}_{i,p,k})\circledast \D^{i_2+\iota}\hat{A}^\sharp_{t,i,p,k}\|_{L^\infty(\hat{B}_R,\hat{g}_t)}\leq C\delta_t^{-r-1+s-i_1}\delta_t^2e^{\frac{-i+1-\alpha}{2}\frac{\alpha}{j+\alpha}t}\leq C\delta_t^{-r}e^{\frac{-i+1-\alpha}{2}\frac{\alpha}{j+\alpha}t},
\end{equation}
\begin{equation}\label{dva3}
[(\D^{r+1-s} \hat{J}^\natural_t)\circledast\D^{i_1}\hat{\Phi}_{\iota,\ell}(\hat{G}_{i,p,k})\circledast \D^{i_2+\iota}\hat{A}^\sharp_{t,i,p,k}]_{C^\beta(\hat{B}_R,\hat{g}_t)}\leq C\delta_t^{-r-1+s-i_1-\beta}\delta_t^2e^{\frac{-i+1-\alpha}{2}\frac{\alpha}{j+\alpha}t}\leq C\delta_t^{-r-\beta}e^{\frac{-i+1-\alpha}{2}\frac{\alpha}{j+\alpha}t}.
\end{equation}
Overall, this gives
\begin{equation}\label{crocifisso4}
d_t^{-r}\|\D^r\ti{\eta}^\dagger_t\|_{L^\infty(\ti{B}_S,\ti{g}_t)}\leq C\delta_t^{-r}e^{\frac{-1-\alpha}{2}\frac{\alpha}{j+\alpha}t},\quad
d_t^{-r-\beta}[\D^r\ti{\eta}^\dagger_t]_{C^\beta(\ti{B}_S,\ti{g}_t)}\leq C\delta_t^{-r-\beta}e^{\frac{-1-\alpha}{2}\frac{\alpha}{j+\alpha}t},
\end{equation}
for all $r\geq 0, 0<\beta<1$ and $S\leq Rd_t^{-1}$ ($R$ fixed),
while if we take only derivatives in the base directions, then in \eqref{ena3} and \eqref{dva3} there are at most $2$ fiber derivatives landing on $\hat{\Phi}$ and $\hat{J}^\natural_t$ (and the $C^\beta$ difference quotient is in the base only), which implies that the negative powers of $\delta_t$ at the end of \eqref{ena3} and \eqref{dva3} disappear and we get
\begin{equation}\label{crocifisso4bis}
d_t^{-r}\|\D^r_{\mathbf{b\cdots b}}\ti{\eta}^\dagger_t\|_{L^\infty(\ti{B}_S,\ti{g}_t)}\leq Ce^{\frac{-1-\alpha}{2}\frac{\alpha}{j+\alpha}t},\quad
d_t^{-r-\beta}[\D^r_{\mathbf{b\cdots b}}\ti{\eta}^\dagger_t]_{C^\beta_{\mathrm{base}}(\ti{B}_S,\ti{g}_t)}\leq Ce^{\frac{-1-\alpha}{2}\frac{\alpha}{j+\alpha}t},
\end{equation}
for all $r\geq 0, 0<\beta<1$, $S\leq Rd_t^{-1}$ ($R$ fixed), and all of these go to zero. On the other hand, if in \eqref{ena3}, \eqref{dva3} we use the fixed metric $g_X$ instead of $\hat{g}_t$ then all derivatives of $\hat{\Phi}_{\iota,\ell}$ and $\hat{J}^\natural_t$ are uniformly bounded and we obtain
\begin{equation}\label{crocifisso4ter}
\|\D^r\hat{\eta}^\dagger_t\|_{L^\infty(\hat{B}_R,g_X)}\leq C\delta_t^2e^{\frac{-1-\alpha}{2}\frac{\alpha}{j+\alpha}t},\quad
[\D^r\hat{\eta}^\dagger_t]_{C^\beta(\hat{B}_R,g_X)}\leq C\delta_t^2e^{\frac{-1-\alpha}{2}\frac{\alpha}{j+\alpha}t},
\end{equation}
for all $r\geq 0, 0<\beta<1$ and fixed $R$. Lastly, we will also need a similar estimate using the fixed metric $g_X$ but in the tilde picture when $\ve_t\to 0,$ which says that
\begin{equation}\label{crocifisso4quater}
\|\D^r\ti{\eta}^\dagger_t\|_{L^\infty(\ti{B}_S,g_X)}\leq C\ve_t^2e^{\frac{-1-\alpha}{2}\frac{\alpha}{j+\alpha}t},\quad
[\D^r\ti{\eta}^\dagger_t]_{C^\beta(\ti{B}_S,g_X)}\leq C\ve_t^2e^{\frac{-1-\alpha}{2}\frac{\alpha}{j+\alpha}t},
\end{equation}
for $r\geq 0, 0<\beta<1,$ fixed $S$, and again these all go to zero. Briefly, we have
\begin{equation}
\ti{\eta}^\dagger_t = i\partial\overline\partial\sum_{i=2}^j\sum_{p=1}^{N_{i,k}}\sum_{\iota=0}^{2k}\sum_{\ell=\lceil \frac{\iota}{2} \rceil}^{k} e^{-\left(\ell-\frac{\iota}{2}\right) t}\ve_t^{\iota}(\ti{\Phi}_{\iota,\ell}(\ti{G}_{i,p,k})\circledast \D^\iota \ti{A}^\sharp_{t,i,p,k}),
\end{equation}
and we apply $\D^r$ to this, using the schematics analogous to \eqref{schemata}, and using that all derivatives of $\ti{\Phi}_{\iota,\ell}$ and $\ti{J}^\natural_t$ have uniformly bounded norm with respect to $g_X$
we can bound
\begin{equation}\begin{split}
d_t^{-r}\ve_t^\iota&\|(\D^{r+1-s} \ti{J}^\natural_t)\circledast\D^{i_1}\ti{\Phi}_{\iota,\ell}(\ti{G}_{i,p,k})\circledast \D^{i_2+\iota}\ti{A}^\sharp_{t,i,p,k}\|_{L^\infty(\ti{B}_{S},g_X)}\\
&\leq Cd_t^{-r-1+s-i_1}\delta_t^\iota\delta_t^2 e^{\frac{-i+1-\alpha}{2}\frac{\alpha}{j+\alpha}t}
\leq Cd_t^{-r} \ve_t^2 e^{\frac{-i+1-\alpha}{2}\frac{\alpha}{j+\alpha}t},
\end{split}
\end{equation}
and
\begin{equation}\begin{split}
d_t^{-r-\beta}\ve_t^\iota&[(\D^{r+1-s} \ti{J}^\natural_t)\circledast\D^{i_1}\ti{\Phi}_{\iota,\ell}(\ti{G}_{i,p,k})\circledast \D^{i_2+\iota}\ti{A}^\sharp_{t,i,p,k}]_{C^\beta(\ti{B}_{S},g_X)}\\
&\leq Cd_t^{-r-1+s-i_1-\beta}\delta_t^\iota\delta_t^2 e^{\frac{-i+1-\alpha}{2}\frac{\alpha}{j+\alpha}t}
\leq Cd_t^{-r-\beta} \ve_t^2 e^{\frac{-i+1-\alpha}{2}\frac{\alpha}{j+\alpha}t},
\end{split}\end{equation}
and \eqref{crocifisso4quater} follows.

\subsubsection{Estimates for $\ti{\omega}^\sharp_t$}

Since $\tilde\omega_t^\sharp = \tilde\omega_t^\bullet - \tilde\eta_{t,j,k}-\ti{\eta}^\circ_t-\ti{\eta}^\diamond_t$, we can see, using \eqref{crocifisso7}, \eqref{crocifisso8}, \eqref{crocifisso3} and the fact that the complex structure has uniformly bounded $\ti{g}_t$-norm,
that for any given $R$, $\tilde\omega_t^\sharp$ is a K\"ahler form on $\ti{B}_R$ for all $t$ sufficiently large, with associated metric uniformly equivalent to $\tilde{g}_t$ (thanks to \eqref{fucker}).

We claim that we have the bounds
\begin{equation}\label{crocifisso5}
d_t^{-\iota}\|\D^\iota\tilde{\omega}^\sharp_t\|_{L^\infty(\ti{B}_S,\ti{g}_t)}\leq C\delta_t^{-\iota},\quad d_t^{-\iota-\beta}[\D^\iota\tilde{\omega}^\sharp_t]_{C^\beta(\ti{B}_S,\ti{g}_t)}\leq C\delta_t^{-\iota-\beta},
\end{equation}
for all $\iota\geq 0$ (possibly larger than $j$), $0<\beta<1$, and $S\leq Rd_t^{-1}$ ($R$ fixed).

Indeed, recall that
\begin{equation}\ti{\omega}^\sharp_t=\ti{\omega}_{\rm can}+\ve_t^2\Theta_t^*\Psi_t^*\omega_F+\ti{\eta}^\dagger_t+\ti{\eta}^\ddagger_t.\end{equation}
The term $\ti{\omega}_{\rm can}$ is bounded by \eqref{krk}, the term $\ti{\eta}^\dagger_t$ is bounded by \eqref{crocifisso4}, and $\ti{\eta}^\ddagger_t$ by \eqref{crocifisso9} (and this one vanishes when differentiated more than $j$ times). Lastly, for the term $\ve_t^2  \Theta_t^*\Psi_t^*\omega_F$ we have
\begin{equation}\label{lameduck}
d_t^{-\iota}\|\D^\iota(\ve_t^2\Theta_t^*\Psi_t^*\omega_F)\|_{L^\infty(\ti{B}_S,\ti{g}_t)}=
\lambda_t^{-\iota}\|\D^\iota(e^{-t}\omega_F)\|_{L^\infty(B_{Sd_t\lambda_t^{-1}},g_t)}\leq C\delta_t^{-\iota},
\end{equation}
\begin{equation}\label{tameduck}
d_t^{-\iota-\beta}[\D^\iota(\ve_t^2\Theta_t^*\Psi_t^*\omega_F)]_{C^\beta(\ti{B}_S,\ti{g}_t)}=
\lambda_t^{-\iota-\beta}[\D^\iota(e^{-t}\omega_F)]_{C^\beta(B_{Sd_t\lambda_t^{-1}},g_t)}\leq C\delta_t^{-\iota-\beta},\end{equation}
by simple ``index counting'', for all $\iota\geq 0, 0<\beta<1$, $S\leq Rd_t^{-1}$ ($R$ fixed), and \eqref{crocifisso5} now follows.

Lastly, we claim that if we move only in the base directions, then we get
\begin{equation}\label{crocifisso5bis}
d_t^{-\iota}\|\D^\iota_{\mathbf{b\cdots b}}\tilde{\omega}^\sharp_t\|_{L^\infty(\ti{B}_S,\ti{g}_t)}\leq \begin{cases}
O(1),\quad \iota=0,\\
o(1), \quad \iota>0,\end{cases}\quad d_t^{-\iota-\beta}[\D^\iota_{\mathbf{b\cdots b}}\tilde{\omega}^\sharp_t]_{C^\beta_{\mathrm{base}}(\ti{B}_S,\ti{g}_t)}=o(1),
\end{equation}
for all $\iota\geq 0, 0<\beta<1$ and fixed $S$. Indeed this follows by the same argument as above, replacing \eqref{crocifisso4} by \eqref{crocifisso4bis}, and observing that when we go base-only the bounds in \eqref{lameduck}, \eqref{tameduck} improve trivially to
\begin{equation}\label{lameduck_redux}
d_t^{-\iota}\|\D^\iota_{\mathbf{b\cdots b}}(\ve_t^2\Theta_t^*\Psi_t^*\omega_F)\|_{L^\infty(\ti{B}_S,\ti{g}_t)}\leq C\lambda_t^{-\iota},
\end{equation}
\begin{equation}\label{tameduck_redux}
d_t^{-\iota-\beta}[\D^\iota_{\mathbf{b\cdots b}}(\ve_t^2\Theta_t^*\Psi_t^*\omega_F)]_{C^\beta_{\mathrm{base}}(\ti{B}_S,\ti{g}_t)}\leq C\lambda_t^{-\iota-\beta},
\end{equation}
for all $\iota\geq 0, 0<\beta<1$, $S\leq Cd_t^{-1}$.

\subsubsection{Expansion of the Monge-Amp\`ere equation}

Using the decomposition \eqref{bunny} for the Ricci-flat metrics $\ti{\omega}^\bullet_t$, we expand the Monge-Amp\`ere equation \eqref{zumteufel151} as
\begin{align}\label{satanas}
&\tr{\tilde\omega_t^\sharp}{(\tilde\eta_{t,j,k}+\tilde\eta_t^\diamond+\tilde\eta_t^\circ)}\\
&+\sum_{i+\iota+p\geq 2}\frac{(m+n)!}{i!\iota!p!(m+n-i-\iota-p)!}\frac{(\tilde\eta_{t,j,k})^i \wedge (\tilde\eta_t^\diamond)^\iota\wedge(\tilde\eta_t^\circ)^p \wedge (\tilde\omega_t^\sharp)^{m+n-i-\iota-p}}{(\tilde\omega_t^\sharp)^{m+n}}
= c_t e^{\tilde{H}_t} \frac{(\tilde\omega^\natural_t)^{m+n}}{(\tilde\omega_t^\sharp)^{m+n}} - 1.\nonumber
\end{align}

For ease of notation, call
\begin{equation}\mathcal{C}_{i\iota p}=\frac{(\tilde\eta_{t,j,k})^i \wedge (\tilde\eta_t^\diamond)^\iota\wedge(\tilde\eta_t^\circ)^p \wedge (\tilde\omega_t^\sharp)^{m+n-i-\iota-p}}{(\tilde\omega_t^\sharp)^{m+n}}.\end{equation}

\medskip

\begin{proposition}\label{subclaim12}
For any fixed $R$, we have
\begin{equation}\label{cacanew2}
d_t^{-j-\alpha}\left[\D^j_{\mathbf{b\cdots b}}\left(c_te^{\tilde{H}_t} \frac{(\tilde\omega^\natural_t)^{m+n}}{(\tilde\omega_t^\sharp)^{m+n}}-1\right)\right]_{C^\alpha_{\rm base}(\ti{B}_R,\ti{g}_t)}=o(1),
\end{equation}
\begin{equation}\label{cacanew}
\forall i+\iota+p\geq 2:\, d_t^{-j-\alpha}[\D^j_{\mathbf{b\cdots b}}\mathcal{C}_{i\iota p}]_{C^\alpha_{\rm base}(\ti{B}_R,\ti{g}_t)}=o(1),
\end{equation}
for all $t$ sufficiently large.
Combining these with \eqref{satanas} we obtain
\begin{equation}\label{killmenow}
d_t^{-j-\alpha}[\D^j_{\mathbf{b\cdots b}}(\tr{\tilde\omega_t^\sharp}{(\tilde\eta_{t,j,k}+\ti{\eta}^\diamond_t+\ti{\eta}^\circ_t)})]_{C^\alpha_{\rm base}(\ti{B}_R,\ti{g}_t)}=o(1).
\end{equation}
On the other hand, if we assume $\ve_t\geq C^{-1}$ then for every fixed $R$ and $0\leq a\leq j$ we have
\begin{equation}\label{cacanew3}
d_t^{-j-\alpha}\left\|\D^a\left(c_te^{\tilde{H}_t} \frac{(\tilde\omega^\natural_t)^{m+n}}{(\tilde\omega_t^\sharp)^{m+n}}-1\right)\right\|_{L^\infty(\ti{B}_{R\ve_t},\ti{g}_t)}\leq C\ve_t^{j+\alpha-a},
\end{equation}
\begin{equation}\label{cacanew3b}
d_t^{-j-\alpha}\left[\D^a\left(c_te^{\tilde{H}_t} \frac{(\tilde\omega^\natural_t)^{m+n}}{(\tilde\omega_t^\sharp)^{m+n}}-1\right)\right]_{C^\alpha(\ti{B}_{R\ve_t},\ti{g}_t)}\leq C\ve_t^{j-a},
\end{equation}
\begin{equation}\label{cacanew4}
\forall i+\iota+p\geq 2:\, d_t^{-j-\alpha}\|\D^a\mathcal{C}_{i\iota p}\|_{L^\infty(\ti{B}_{R\ve_t},\ti{g}_t)}\leq C\delta_t^{j+\alpha}\ve_t^{j+\alpha-a},
\end{equation}
\begin{equation}\label{cacanew4b}
\forall i+\iota+p\geq 2:\, d_t^{-j-\alpha}[\D^a\mathcal{C}_{i\iota p}]_{C^\alpha(\ti{B}_{R\ve_t},\ti{g}_t)}\leq C\delta_t^{j+\alpha}\ve_t^{j-a},
\end{equation}
for $t$ sufficiently large.
\end{proposition}

\begin{proof}
First, let us prove \eqref{cacanew2}.
Using the definition of $\ti{H}_t$, we can write
\begin{equation}e^{\ti{H}_t}\frac{(\tilde\omega^\natural_t)^{m+n}}{(\tilde\omega_t^\sharp)^{m+n}} =\frac{\ti{\omega}_{\rm can}^m\wedge(\ve_t^2\Theta_t^*\Psi_t^*\omega_F)^n}{(\tilde\omega_t^\sharp)^{m+n}}.\end{equation}
To bound the numerator we shall employ \eqref{krk}, \eqref{lameduck_redux} and \eqref{tameduck_redux}.
As for the denominator $(\tilde\omega_t^\sharp)^{m+n}$, let us temporarily use the shorthand
$b=(\tilde\omega_t^\sharp)^{m+n},$
so that we can write schematically (again omitting combinatorial factors)
\begin{equation}\label{kombinacja4}
\D^{r}(b^{-1})=\sum_{\ell=1}^{r}\sum_{j_1+\cdots+j_\ell=r}\frac{\D^{j_1}b}{b}\cdots \frac{\D^{j_{\ell}}b}{b}b^{-1},
\end{equation}
and it then suffices to use \eqref{crocifisso5bis}, and \eqref{cacanew2} follows.

Estimate \eqref{cacanew} also follows easily from \eqref{crocifisso7},  \eqref{crocifisso8} and \eqref{crocifisso3bis}, together with \eqref{crocifisso5bis} and \eqref{kombinacja4}.

On the other hand, assuming now that $\ve_t\geq C^{-1}$, estimates \eqref{cacanew4} and \eqref{cacanew4b} follow from
\eqref{crocifisso7}, \eqref{crocifisso8bis}, \eqref{crocifisso3}, together with \eqref{crocifisso5} and \eqref{kombinacja4}.

Lastly, to prove \eqref{cacanew3}, \eqref{cacanew3b}, using \eqref{satanas}, \eqref{cacanew4} and \eqref{cacanew4b} it suffices to show that
\begin{equation}
d_t^{-j-\alpha}\left\|\D^a\tr{\tilde\omega_t^\sharp}{(\tilde\eta_{t,j,k}+\tilde\eta_t^\diamond+\tilde\eta_t^\circ)}\right\|_{L^\infty(\ti{B}_{R\ve_t},\ti{g}_t)}\leq C\ve_t^{j+\alpha-a},
\end{equation}
\begin{equation}
d_t^{-j-\alpha}\left[\D^a\tr{\tilde\omega_t^\sharp}{(\tilde\eta_{t,j,k}+\tilde\eta_t^\diamond+\tilde\eta_t^\circ)}\right]_{C^\alpha(\ti{B}_{R\ve_t},\ti{g}_t)}\leq C\ve_t^{j-a},
\end{equation}
for $R$ fixed and $0\leq a\leq j$, which is also a direct consequence of
\eqref{crocifisso7}, \eqref{crocifisso8bis}, \eqref{crocifisso3}, together with \eqref{crocifisso5} and \eqref{kombinacja4}.
\end{proof}

We are now in position to derive a contradiction on each of the three possible scenarios (up to passing to a sequence $t_i\to\infty$ as usual), according to whether $\ve_t\to \infty,$ $\ve_t$ remains bounded away from $0$ and $\infty$, or $\ve_t\to 0$. \medskip\

\subsection{Subcase A: $\epsilon_t \to \infty$.}\label{endgameA}

We are aiming to obtain a contradiction by showing that the LHS of \eqref{zumteufel163} is $o(1)$. Recall that by construction $d^{\ti{g}_t}(\ti{x}_t,\ti{x}'_t)=1$. As mentioned in the Introduction, the idea in this subcase is to kill all contributions to the blowup quantity \eqref{zumteufel163} using Schauder estimates for the linearized PDE on balls in $\C^{m+n}$. The argument is quite long because of the complexity of the quantitative estimates satisfied by all the pieces in the decomposition of the solution $\ti{\eta}_t$, so we give a brief outline. The main task will be to show \eqref{fleabag2}, which morally says that the contribution ``of the whole solution $\ti{\eta}_t$'' goes to zero. Once we know this, we need to use a noncancellation property to kill the contributions of each individual term in the blowup quantity (which is a sum of norms rather than the norm of a sum). The noncancellation property roughly speaking shows that the contributions of the different pieces to the blowup quantity cannot cancel each other, and will ultimately allow us to kill their contributions one at the time in subsections \ref{jjj1}, \ref{jjj2} and \ref{jjj3}. Since the noncancellation property is self-contained and a bit lengthy, we discuss it first.

\subsubsection{The noncancellation property}

The following noncancellation property is stated and proved in the hat picture, and it will then be transferred to the tilde picture.

\begin{proposition}\label{noncancell} Let $\hat{B}_R := B^{\C^m}(\hat{z}_t, R) \times Y$. The following inequality holds for all $0<\alpha<1,$ $a \in \mathbb{N}$ and all $R$:
\begin{align}
[\D^{a}\hat{\gamma}_{t,0}]_{C^{\alpha}(\hat{B}_R)} \leq
C\left([\D^a_{\mathbf{b\cdots b}}\hat\eta_t]_{C^\alpha_{\rm base}(\hat{B}_R,g_X)} +\left(\frac{R}{\lambda_t}\right)^{1-\alpha}\sum_{b=0}^{a} \lambda_t^{-\alpha}\|\D^b \hat\eta_t \|_{L^{\infty}(\hat{B}_R,g_X)}\right), \label{fedeliallalinea}
\end{align}
The following inequality holds for all $a\in\N$, $2 \leq i \leq j$, $1 \leq p \leq N_{i,k}$:
\begin{align}
&[\D^{a}\hat{A}_{t,i,p,k}]_{C^\alpha(\hat{B}_R)} \leq
C\delta_t^2\bigg([\D^a_{\mathbf{b\cdots b}}\hat\eta_t]_{C^\alpha_{\rm base}(\hat{B}_R,\hat{g}_t)}\label{fedeliallalinea2}\\
&+\left(\frac{R}{\lambda_t}\right)^{1-\alpha}\sum_{b=0}^{a} \lambda_t^{-\alpha}\bigg(\|\D^b\hat{\eta}^\ddagger_t\|_{L^{\infty}(\hat{B}_R,\hat{g}_t)}
+\|\D^b(\hat{\eta}^\diamond_t+\hat{\eta}^\circ_t+\hat{\eta}_{t,j,k}) \|_{L^{\infty}(\hat{B}_R,\hat{g}_t)}+\delta_t^{-2}\|\D^b \hat\eta^\dagger_t \|_{L^{\infty}(\hat{B}_R,g_X)}\bigg)\bigg)\label{fedeliallalinea4}\\
&+C\sum_{\iota=2}^{i-1}\sum_{q = 1}^{N_{\iota,k}}\left(\delta_t^{2k+2}[\D^{a + 2k + 2}\hat{A}_{t,\iota,q,k}]_{C^\alpha(\hat{B}_R)} + \left(\frac{R}{\lambda_t}\right)^{1-\alpha}\sum_{b=0}^{a + 2k + 2} e^{-(2k+2)\frac{t}{2}}\lambda_t^{b-a-\alpha} \|\D^{b}\hat{A}_{t,\iota,q,k}\|_{L^{\infty}(\hat{B}_R)} \right).\label{fedeliallalinea3}
\end{align}
\end{proposition}

\begin{proof}
We begin by writing $\hat\gamma_{t,0}$ and $\hat{A}_{t,i,p,k}$ as the pushforwards of certain forms on the total space that are roughly proportional to the whole solution $\hat\eta_t$. Next, we discuss how to slide $\D^a$ into a pushforward. Then we discuss how to slide a H\"older difference quotient into a pushforward. At the end we explain how to put everything together to get the desired inequalities stated above. \medskip\

\noindent \emph{Claim 1}: We have the following representations of $\hat\gamma_{t,0}$ and $\hat{A}_{t,i,p,k}$ as pushforwards of $\hat\eta_t$:
\begin{equation}\label{fiber1oo}
\hat{\gamma}_{t,0}=({\rm pr}_B)_*\left(\hat{\eta}_t\wedge\Psi_t^*\omega_F^n\right),
\end{equation}
and, writing $\hat{P}_{t,i,p,k}$ for the stretched projection in \eqref{proiettohat},
\begin{align}
\hat{A}_{t,i,p,k} &= \hat{P}_{t,i,p,k}(\hat{\eta}_t) + e^{-(2k+2)\frac{t}{2}} \sum_{\iota=2}^{i-1}\sum_{q=1}^{N_{\iota,k}} \sum_{\kappa=0}^{2k+2}\hat{\Phi}_{\kappa,i,p,k,\iota,q}  \circledast \lambda_t^\kappa \D^\kappa \hat{A}_{t,\iota,q,k},\label{e:herrdeninger}
\end{align}
where the functions $\hat{\Phi}_{\kappa,i,p,k,\iota,q}$ are from the base.\\

\noindent \emph{Proof of Claim 1}: We can assume without loss that $\int_{\{\hat{z}\}\times Y}\Psi_t^*{\omega}_F^n|_{\{\hat{z}\}\times Y}=1$ for all $\hat{z}$, so we can write by definition
\begin{equation}\underline{\hat{\psi}_t}=({\rm pr}_B)_*\left(\hat{\psi}_t\Psi_t^*\omega_F^n\right),\end{equation}
and note that $({\rm pr}_B)_*$ commutes with $\de$ and $\db$ since ${\rm pr}_B$ is holomorphic, while recalling that $\hat{J}^\natural_t=\Psi_t^*J$ and decorating here the $\db$ operators for clarity we have
\begin{equation}
\db^{\hat{J}^\natural_t}(\Psi_t^*\omega_F^n)=\Psi_t^*(\db^J\omega_F^n)=0,
\end{equation}
(using that $\omega_F^n$ is $d$-closed and of type $(n,n)$ with respect to $J$), and similarly for $\de$ and so
\begin{equation}\label{fiber1}
\hat{\gamma}_{t,0}=\ddbar\underline{\hat{\psi}_t}=({\rm pr}_B)_*\left(\hat{\eta}_t\wedge\Psi_t^*\omega_F^n\right).
\end{equation}
Again by definition
\begin{align}\hat{A}_{t,i,p,k} &= \hat{P}_{t,i,p,k}(\hat\eta_{t,i-1,k}) = \hat{P}_{t,i,p,k}(\hat\eta_t) - \hat{P}_{t,i,p,k}(\hat\gamma_{t,0}) - \sum_{\iota=2}^{i-1} \hat{P}_{t,i,p,k}(\hat\gamma_{t,\iota,k}).
\end{align}
The term $\hat{P}_{t,i,p,k}(\hat\gamma_{t,0})$ is zero thanks to \eqref{frombase}. For the remaining terms we invoke Lemma \ref{paraminatrix}, transplanted to the hat picture:
\begin{align}
\hat{P}_{t,i,p,k}(\hat\gamma_{t,\iota,k}) &= \sum_{q=1}^{N_{\iota,k}} \hat{P}_{t,i,p,k}(i\partial\overline\partial\hat{\mathfrak{G}}_{t,k}(\hat{A}_{t,\iota,q,k},\hat{G}_{\iota,q,k}))\\
&= \sum_{q=1}^{N_{\iota,k}}\left( \hat{A}_{t,\iota,q,k}\int_{\{\hat{z}\}\times Y} \hat{G}_{i,p,k}\hat{G}_{\iota,q,k}\Psi_t^*\omega_F^n + e^{-(2k+2)\frac{t}{2}}\sum_{\kappa=0}^{2k+2}\hat\Phi_{\kappa,i,p,k}(\hat{G}_{\iota,q,k}) \circledast \lambda_t^\kappa \D^\kappa\hat{A}_{t,\iota,q,k}\right).
\end{align}
Here the first term vanishes by orthogonality of the $\hat{G}_{i,p,k}$'s, and the rest is exactly what we are claiming.\hfill$\Box$\medskip\

\noindent \emph{Claim 2}: For any form $\xi$ on $B\times Y$ of degree $(n,n)$ or $(n+1,n+1)$ and any $a\geq 0,$ at any point $z\in B$ we have
\begin{equation}\label{abc}
\D^a ({\rm pr}_B)_*\xi=({\rm pr}_B)_*\left(\nabla^{z,a}_{\mathbf{b\cdots b}}\xi\right).
\end{equation}

\noindent \emph{Proof of Claim 2}: By definition we have
\begin{equation}\label{radial}
\D^a ({\rm pr}_B)_*\xi=\de^{a}_{\mathbf{b\cdots b}}({\rm pr}_B)_*\xi,
\end{equation}
and to evaluate this at a point $z\in B$ recall that $\nabla^z$ is a product connection and in the base directions its Christoffel symbols are just zero. Assume first that $\xi$ is an $(n,n)$-form, so the RHS of \eqref{radial} equals
\begin{equation}\label{mithridates1}\begin{split}
\int_{\{z\}\times Y} \left(\nabla^{z,a}_{\mathbf{b\cdots b}}\xi\right)=({\rm pr}_B)_*\left(\nabla^{z,a}_{\mathbf{b\cdots b}}\xi\right),
\end{split}\end{equation}
as desired. Similarly, when $\xi$ is an $(n+1,n+1)$-form, the RHS of \eqref{radial} equals
\begin{equation}\label{mithridates}\begin{split}
\left(\int_{\{z\}\times Y} \nabla^{z,a}_{\mathbf{b\cdots b}}\left(\iota_{\frac{\de}{\de \ov{z}^j}}\iota_{\frac{\de}{\de z^i}}\xi\right)\right)dz^i\wedge d\ov{z}^j&=\left(\int_{\{z\}\times Y} \iota_{\frac{\de}{\de \ov{z}^j}}\iota_{\frac{\de}{\de z^i}}\left(\nabla^{z,a}_{\mathbf{b\cdots b}}\xi\right)\right)dz^i\wedge d\ov{z}^j\\
&=({\rm pr}_B)_*\left(\nabla^{z,a}_{\mathbf{b\cdots b}}\xi\right),
\end{split}\end{equation}
where in the first equality we used $\nabla^z_{\mathbf{b}}\left({\frac{\de}{\de z^i}}\right)=0$. \hfill$\Box$\medskip\

\noindent \emph{Claim 3}: Denote by $\Phi: \{\hat{z}_1\}\times Y \to \{\hat{z}_2\}\times Y$ the map given by $\Phi(\hat{z}_1,\hat{y})=(\hat{z}_2,\hat{y})$. Then there is a constant $C$ such that for any smooth form $\xi$ on the total space we have
\begin{equation}\label{e:cerebrum}
|(({\rm pr}_B)_*\xi)(\hat{z}_1) - (({\rm pr}_B)_*\xi)(\hat{z}_2) | \leq C\sup_{\hat{x}_1\in \{\hat{z}_1\}\times Y}|\xi(\hat{x}_1) - \P_{\Phi(\hat{x}_1)\hat{x}_1}\xi(\Phi(\hat{x}_1))|_{g_X(\hat{x}_1)} + C\lambda_t^{-1}(\sup_{\{\hat{z}_2\}\times Y} |\xi|_{g_X})|\hat{z}_1- \hat{z}_2|.
\end{equation}

\noindent \emph{Proof of Claim 3}: Pick any fixed orthonormal vectors $Z_1, \ldots, Z_r$ on the base ($r = ({\rm deg}\,\xi)-2n$), and denote by the same notation their trivial extension to horizontal vector fields on the total space. Then $\xi(Z_1,\ldots,Z_r, -, \ldots, -)$ is a $2n$-form on the total space, hence has at most $2n$ fiber components in each of its indecomposable terms, and only those indecomposable terms with exactly $2n$ fiber components survive the restriction. Note also note that $\P$-transport commutes with restriction since it preserves the horizontal-vertical decomposition (cf. Section \ref{s:tridiagonal}).

Now write the fiberwise restriction of $\xi(Z_1,\ldots,Z_r, -, \ldots, -)$ to any fiber $\{\hat{z}\}\times Y$ as a scalar function $f(\hat{x})$ times the fiberwise Ricci-flat volume form $\Upsilon_{\hat{z}}:=\Psi_t^*{\omega}_F^n|_{\{\hat{z}\}\times Y}$.
This now allows us to estimate
\begin{align}
|((&{\rm pr}_B)_*\xi)(\hat{z}_1) - (({\rm pr}_B)_*\xi)(\hat{z}_2) | \\
&\leq\biggl|(({\rm pr}_B)_*\xi)(\hat{z}_2)(Z_1,\ldots, Z_r) - \int_{\hat{x}_1 \in \{\hat{z}_1\}\times Y} \P_{\Phi(\hat{x}_1)\hat{x}_1}[f(\Phi(\hat{x}_1))\Upsilon_{\hat{z}_2}({\Phi(\hat{x}_1)})]\biggr|\label{lop1}\\
&+\biggl|(({\rm pr}_B)_*\xi)(\hat{z}_1)(Z_1,\ldots, Z_r) - \int_{\hat{x}_1 \in \{\hat{z}_1\}\times Y} \P_{\Phi(\hat{x}_1)\hat{x}_1}[f(\Phi(\hat{x}_1))\Upsilon_{\hat{z}_2}({\Phi(\hat{x}_1)})]\biggr|,\label{lop2}
\end{align}
and we bound the term in \eqref{lop1} by
\begin{align}&\biggl|(({\rm pr}_B)_*\xi)(\hat{z}_2)(Z_1,\ldots, Z_r) - \int_{\hat{x}_1 \in \{\hat{z}_1\}\times Y} \P_{\Phi(\hat{x}_1)\hat{x}_1}[f(\Phi(\hat{x}_1))\Upsilon_{\hat{z}_2}({\Phi(\hat{x}_1)})]\biggr| \\
&=
\biggl|\int_{\hat{x}_2 \in \{\hat{z}_2\}\times Y} f(\hat{x}_2)(\Upsilon_{\hat{z}_2}-(\Phi^{-1})^*\P_{\Phi(\hat{x}_1)\hat{x}_1}\Upsilon_{\hat{z}_2})(\hat{x}_2)\biggr|
\leq C\lambda_t^{-1}(\sup_{\{\hat{z}_2\}\times Y} |\xi|_{g_X})|\hat{z}_1 - \hat{z}_2|,
\end{align}
using here that $|\Upsilon_{\hat{z}_2}-(\Phi^{-1})^*\P_{\Phi(\hat{x}_1)\hat{x}_1}\Upsilon_{\hat{z}_2}|(\hat{x}_2)$ before stretching has size proportional to the unstretched distance $|z_1 - z_2|$, which becomes $\lambda_t^{-1}|\hat{z}_1 - \hat{z}_2|$ after stretching.
Having estimated the term in \eqref{lop1}, it remains to bound the term in \eqref{lop2} by writing it as
\begin{align}
&\biggl|\int_{\hat{x}_1 \in \{\hat{z}_1\}\times Y} \left(\xi(\hat{x}_1)(Z_1,\ldots,Z_r,-,\ldots,-)|_{\{\hat{z}_1\}\times Y}   - \P_{\Phi(\hat{x}_1)\hat{x}_1}[\xi(\Phi(\hat{x}_1))(Z_1,\ldots,Z_r,-,\ldots,-)|_{\{\hat{z}_2\}\times Y} ]\right) \biggr|\\
&\leq C\sup_{\hat{x}_1\in \{\hat{z}_1\}\times Y}|\xi(\hat{x}_1) - \P_{\Phi(\hat{x}_1)\hat{x}_1}\xi(\Phi(\hat{x}_1))|_{g_X(\hat{x}_1)},
\end{align}
which is clear.\hfill $\Box$\medskip\

With these $3$ claims we can now complete the proof of Proposition \ref{noncancell}, starting with \eqref{fedeliallalinea}. Combining Claims 1--2 gives
\begin{equation}\label{multiradial}
\D^a\hat{\gamma}_{t,0}=({\rm pr}_B)_*\left(\nabla^{z,a}_{\mathbf{b\cdots b}}(\hat{\eta}_t\wedge\Psi_t^*\omega_F^n)\right),
\end{equation}
and applying Claim 3 gives
\begin{equation}\label{uniradial}
[\D^{a}\hat{\gamma}_{t,0}]_{C^{\alpha}(\hat{B}_R)}\leq C[\nabla^{z,a}_{\mathbf{b\cdots b}}(\hat{\eta}_t\wedge\Psi_t^*\omega_F^n)]_{C^\alpha_{\rm base}(\hat{B}_R,g_X)}+C\lambda_t^{-1}R^{1-\alpha}\|\nabla^{z,a}_{\mathbf{b\cdots b}}(\hat{\eta}_t\wedge\Psi_t^*\omega_F^n)\|_{L^{\infty}(\hat{B}_R,g_X)},
\end{equation}
and we can use the Leibniz rule to bound
\begin{equation}\label{longa}\begin{split}
[\nabla^{z,a}_{\mathbf{b\cdots b}}(\hat{\eta}_t\wedge\Psi_t^*\omega_F^n)]_{C^\alpha_{\rm base}(\hat{B}_R,g_X)}&\leq C\sum_{b=0}^a
\bigg([\nabla^{z,b}_{\mathbf{b\cdots b}}\hat{\eta}_t]_{C^\alpha_{\rm base}(\hat{B}_R,g_X)}\|\nabla^{z,a-b}_{\mathbf{b\cdots b}}(\Psi_t^*\omega_F^n)\|_{L^{\infty}(\hat{B}_R,g_X)}\\
&+\|\nabla^{z,b}_{\mathbf{b\cdots b}}\hat{\eta}_t\|_{L^{\infty}(\hat{B}_R,g_X)}
[\nabla^{z,a-b}_{\mathbf{b\cdots b}}(\Psi_t^*\omega_F^n)]_{C^\alpha_{\rm base}(\hat{B}_R,g_X)}\bigg)\\
&\leq C\sum_{b=0}^a
\bigg(\lambda_t^{b-a}[\nabla^{z,b}_{\mathbf{b\cdots b}}\hat{\eta}_t]_{C^\alpha_{\rm base}(\hat{B}_R,g_X)}+\lambda_t^{b-a-1}R^{1-\alpha}\|\nabla^{z,b}_{\mathbf{b\cdots b}}\hat{\eta}_t\|_{L^{\infty}(\hat{B}_R,g_X)}\bigg)\\
&\leq C[\nabla^{z,a}_{\mathbf{b\cdots b}}\hat{\eta}_t]_{C^\alpha_{\rm base}(\hat{B}_R,g_X)}+C\left(\frac{R}{\lambda_t}\right)^{1-\alpha}\sum_{b=0}^{a} \lambda_t^{b-a-\alpha}\|\nabla^{z,b}_{\mathbf{b\cdots b}} \hat\eta_t \|_{L^{\infty}(\hat{B}_R,g_X)},
\end{split}\end{equation}
where in the last line for $b<a$ we have used
\begin{equation}
[\nabla^{z,b}_{\mathbf{b\cdots b}}\hat{\eta}_t]_{C^\alpha_{\rm base}(\hat{B}_R,g_X)}\leq R^{1-\alpha}\|\nabla^{z,b+1}_{\mathbf{b\cdots b}}\hat{\eta}_t\|_{L^\infty(\hat{B}_R,g_X)},
\end{equation}
as in Lemma \ref{kathoum}. Similarly, we bound the last term in \eqref{uniradial} by
\begin{equation}\label{duoradial}
C\lambda_t^{-1}R^{1-\alpha}\|\nabla^{z,a}_{\mathbf{b\cdots b}}(\hat{\eta}_t\wedge\Psi_t^*\omega_F^n)\|_{L^{\infty}(\hat{B}_R,g_X)}
\leq C\lambda_t^{-1}R^{1-\alpha}\sum_{b=0}^a\lambda_t^{b-a}\|\nabla^{z,b}_{\mathbf{b\cdots b}}\hat{\eta}_t\|_{L^{\infty}(\hat{B}_R,g_X)},
\end{equation}
and combining \eqref{uniradial}, \eqref{longa} and \eqref{duoradial} gives
\begin{equation}\label{prims}
[\D^{a}\hat{\gamma}_{t,0}]_{C^{\alpha}(\hat{B}_R)}\leq C[\nabla^{z,a}_{\mathbf{b\cdots b}}\hat{\eta}_t]_{C^\alpha_{\rm base}(\hat{B}_R,g_X)}+C\left(\frac{R}{\lambda_t}\right)^{1-\alpha}\sum_{b=0}^{a} \lambda_t^{b-a-\alpha}\|\nabla^{z,b}_{\mathbf{b\cdots b}} \hat\eta_t \|_{L^{\infty}(\hat{B}_R,g_X)}.
\end{equation}
To complete the proof of \eqref{fedeliallalinea} we need to convert the $\nabla^{z,p}$s into $\D^p$s.
For this we use \eqref{domineddio}, which for all $c\geq 2$ gives
\begin{equation}\label{musil1}
\D^c_{\mathbf{b\cdots b}}\hat{\eta}_t=\nabla^{z,c}_{\mathbf{b\cdots b}}\hat{\eta}_t+\sum_{p=0}^{c-2}(\nabla^{z,p}\hat{\eta}_t\circledast \nabla^{c-2-p}_{z,y,\ti{z}}\hat{\mathbf{A}})_{\mathbf{b\cdots b}},\end{equation}
where $\hat{\mathbf{A}}$ denotes the $\mathbf{A}$ tensor in the hat picture. This can be iterated as
\begin{equation}\label{musilz}\begin{split}
&\nabla^{z,c}_{\mathbf{b\cdots b}}\alpha=\D^c_{\mathbf{b\cdots b}}\alpha\\
&+\sum_{p=0}^{c-2}\sum_{r=0}^{\lfloor\frac{p}{2}\rfloor}\sum_{p_0=0}^{p-2}\sum_{p_1=0}^{p_0-2}\cdots\sum_{p_r=0}^{p_{r-1}-2}\bigg(\D^{p_r}\alpha\circledast \nabla^{c-2-p}\hat{\mathbf{A}}\circledast \nabla^{p-2-p_0}\hat{\mathbf{A}}\circledast\nabla^{p_0-2-p_1}\hat{\mathbf{A}}\cdots \nabla^{p_{r-1}-2-p_r}\hat{\mathbf{A}}\bigg)_{\mathbf{b\cdots b}}.
\end{split}\end{equation}
We have the very crude estimate
\begin{equation}\label{musil2}
\|\nabla^{a}\hat{\mathbf{A}}\|_{L^\infty(\hat{B}_R,g_X)}\leq C\lambda_t^{-1},\quad [\nabla^{a}\hat{\mathbf{A}}]_{C^\alpha_{\rm base}(\hat{B}_R,g_X)}\leq CR^{1-\alpha}\lambda_t^{-1},
\end{equation}
for any $a\geq 0$, coming from the fact that at least one of the indices of $\hat{\mathbf{A}}$ must be in the base direction (otherwise it is zero) and thanks to the stretching $\Psi_t$ this gives a factor of $\lambda_t^{-1}$. Using \eqref{musilz} and \eqref{musil2} and arguing as in \eqref{longa}, \eqref{duoradial} we obtain
\begin{equation}\label{prism}
[\nabla^{z,a}_{\mathbf{b\cdots b}}\hat{\eta}_t]_{C^\alpha_{\rm base}(\hat{B}_R,g_X)}\leq [\D^{a}_{\mathbf{b\cdots b}}\hat{\eta}_t]_{C^\alpha_{\rm base}(\hat{B}_R,g_X)}+C\left(\frac{R}{\lambda_t}\right)^{1-\alpha}\sum_{b=0}^{a} \lambda_t^{-\alpha}\|\D^{b}\hat\eta_t \|_{L^{\infty}(\hat{B}_R,g_X)},
\end{equation}
and similarly
\begin{equation}\label{pmirs}
\left(\frac{R}{\lambda_t}\right)^{1-\alpha}\sum_{b=0}^{a} \lambda_t^{b-a-\alpha}\|\nabla^{z,b}_{\mathbf{b\cdots b}} \hat\eta_t \|_{L^{\infty}(\hat{B}_R,g_X)}\leq
C\left(\frac{R}{\lambda_t}\right)^{1-\alpha}\sum_{b=0}^{a} \lambda_t^{-\alpha}\|\D^{b} \hat\eta_t \|_{L^{\infty}(\hat{B}_R,g_X)},
\end{equation}
and so \eqref{fedeliallalinea} follows from \eqref{prims}, \eqref{prism} and \eqref{pmirs}.

We then move on to the proof of \eqref{fedeliallalinea2}--\eqref{fedeliallalinea3}. Thanks to \eqref{e:herrdeninger}, we have
\begin{equation}\label{loo}
[\D^{a}\hat{A}_{t,i,p,k}]_{C^\alpha(\hat{B}_R)}= [\D^a\hat{P}_{t,i,p,k}(\hat{\eta}_t)]_{C^\alpha(\hat{B}_R)} + e^{-(2k+2)\frac{t}{2}} \sum_{\iota=2}^{i-1}\sum_{q=1}^{N_{\iota,k}} \sum_{\kappa=0}^{2k+2}\lambda_t^\kappa \left[\D^a\left(\hat{\Phi}_{\kappa,i,p,k,\iota,q}  \circledast \D^\kappa \hat{A}_{t,\iota,q,k}\right)\right]_{C^\alpha(\hat{B}_R)},
\end{equation}
and let us first discuss the first term on the RHS. Recall that
\begin{equation}
\hat{P}_{t,i,p,k}(\hat{\eta}_t)=n({\rm pr}_B)_*(\hat{G}_{i,p,k}\hat{\eta}_t\wedge\Psi_t^*\omega_F^{n-1})+
\delta_t^2{\rm tr}^{\hat{\omega}_{\rm can}}({\rm pr}_B)_*(\hat{G}_{i,p,k}\hat{\eta}_t\wedge\Psi_t^*\omega_F^{n}),
\end{equation}
and the estimate
\begin{equation}\label{pizgrgr}\begin{split}
&\delta_t^2[\D^a{\rm tr}^{\hat{\omega}_{\rm can}}({\rm pr}_B)_*(\hat{G}_{i,p,k}\hat{\eta}_t\wedge\Psi_t^*\omega_F^{n})]_{C^\alpha(\hat{B}_R)}\\
&\leq
C\delta_t^2\left([\D^a_{\mathbf{b\cdots b}}\hat\eta_t]_{C^\alpha_{\rm base}(\hat{B}_R,g_X)} +\left(\frac{R}{\lambda_t}\right)^{1-\alpha}\sum_{b=0}^{a}\lambda_t^{-\alpha}\|\D^b \hat\eta_t \|_{L^{\infty}(\hat{B}_R,g_X)}\right),
\end{split}\end{equation}
is proved exactly like in the proof of \eqref{fedeliallalinea}. We then use the decomposition
\begin{equation}\label{scrat}
\hat{\eta}_t=\hat{\eta}^\ddagger_t+\hat{\eta}^\diamond_t+\hat{\eta}^\circ_t+\hat{\eta}_{t,j,k}+\hat{\eta}^\dagger_t,\end{equation}
and bound trivially
\begin{equation}\label{pizgrgr2}
\|\D^b \hat\eta_t \|_{L^{\infty}(\hat{B}_R,g_X)}\leq\|\D^b \hat{\eta}^\ddagger_t \|_{L^{\infty}(\hat{B}_R,\hat{g}_t)}+ \|\D^b (\hat{\eta}^\diamond_t+\hat{\eta}^\circ_t+\hat{\eta}_{t,j,k}) \|_{L^{\infty}(\hat{B}_R,\hat{g}_t)}+\|\D^b \hat\eta^\dagger_t \|_{L^{\infty}(\hat{B}_R,g_X)}.
\end{equation}
On the other hand, to bound
\begin{equation}\label{mafioso}
[\D^a({\rm pr}_B)_*(\hat{G}_{i,p,k}\hat{\eta}_t\wedge\Psi_t^*\omega_F^{n-1})]_{C^\alpha(\hat{B}_R)},
\end{equation}
we use Claim 2 to get
\begin{equation}\label{crda}
\D^a ({\rm pr}_B)_*(\hat{G}_{i,p,k}\hat{\eta}_t\wedge\Psi_t^*\omega_F^{n-1})=({\rm pr}_B)_*\left(\nabla^{z,a}_{\mathbf{b\cdots b}}(\hat{G}_{i,p,k}\hat{\eta}_t\wedge\Psi_t^*\omega_F^{n-1})\right),
\end{equation}
while notice that now since the quantities inside the pushforwards are $2n$-forms, only the $\mathbf{ff}$-components of $\hat{\eta}_t$ appear. More precisely, if $\eta$ is a $2$-form on $\hat{B}_R\times Y$, we define $\eta_{\mathbf{ff}}$ as the $2$-form on $\hat{B}_R\times Y$ obtained from $\eta$ by deleting all components that are not purely fiber-fiber, using the splitting in \eqref{porra}. We then apply Claim 3 and we bound \eqref{mafioso} by
\begin{equation}\begin{split}
&C[\nabla^{z,a}_{\mathbf{b\cdots b}}(\hat{G}_{i,p,k}(\hat{\eta}_t)_{\mathbf{ff}}\wedge\Psi_t^*\omega_F^{n-1})]_{C^\alpha_{\rm base}(\hat{B}_R,g_X)}+C\lambda_t^{-1}R^{1-\alpha}\|\nabla^{z,a}_{\mathbf{b\cdots b}}(\hat{G}_{i,p,k}(\hat{\eta}_t)_{\mathbf{ff}}\wedge\Psi_t^*\omega_F^{n-1})\|_{L^\infty(\hat{B}_R,g_X)},
\end{split}\end{equation}
which as in \eqref{longa} and \eqref{duoradial} is bounded by
\begin{equation}\label{zerp}
C[\nabla^{z,a}_{\mathbf{b\cdots b}}(\hat{\eta}_t)_{\mathbf{ff}}]_{C^\alpha_{\rm base}(\hat{B}_R,g_X)}+C\left(\frac{R}{\lambda_t}\right)^{1-\alpha}\sum_{b=0}^{a} \lambda_t^{b-a-\alpha}\|\nabla^{z,b}_{\mathbf{b\cdots b}}(\hat{\eta}_t)_{\mathbf{ff}} \|_{L^{\infty}(\hat{B}_R,g_X)}.
\end{equation}
As before, for any $(1,1)$-form $\alpha$ on $\hat{B}_R\times Y$ we use the conversion
\begin{equation}\label{sgherro}
\D^c_{\mathbf{b\cdots b}}(\alpha_{\mathbf{ff}})=\nabla^{z,c}_{\mathbf{b\cdots b}}(\alpha_{\mathbf{ff}})+\sum_{p=0}^{c-2}(\nabla^{z,p}(\alpha_{\mathbf{ff}})\circledast \nabla^{c-2-p}_{z,y,\ti{z}}\hat{\mathbf{A}})_{\mathbf{b\cdots b}},\end{equation}
which can be iterated
\begin{equation}\label{scagnozzo}\begin{split}
&\nabla^{z,c}_{\mathbf{b\cdots b}}(\alpha_{\mathbf{ff}})=\D^c_{\mathbf{b\cdots b}}(\alpha_{\mathbf{ff}})\\
&+\sum_{p=0}^{c-2}\sum_{r=0}^{\lfloor\frac{p}{2}\rfloor}\sum_{p_0=0}^{p-2}\sum_{p_1=0}^{p_0-2}\cdots\sum_{p_r=0}^{p_{r-1}-2}\bigg(\D^{p_r}(\alpha_{\mathbf{ff}})\circledast \nabla^{c-2-p}\hat{\mathbf{A}}\circledast \nabla^{p-2-p_0}\hat{\mathbf{A}}\circledast\nabla^{p_0-2-p_1}\hat{\mathbf{A}}\cdots \nabla^{p_{r-1}-2-p_r}\hat{\mathbf{A}}\bigg)_{\mathbf{b\cdots b}}.
\end{split}\end{equation}
Applying this conversion with $\alpha=\hat{\eta}_t$ and bounding the $\hat{\mathbf{A}}$ tensor terms by \eqref{musil2} we see that \eqref{zerp} is bounded by
\begin{equation}\label{zerp2}
C[\D^a_{\mathbf{b\cdots b}}(\hat{\eta}_t)_{\mathbf{ff}}]_{C^\alpha_{\rm base}(\hat{B}_R,g_X)}+C\left(\frac{R}{\lambda_t}\right)^{1-\alpha}\sum_{b=0}^{a} \lambda_t^{-\alpha}\|\D^b(\hat{\eta}_t)_{\mathbf{ff}} \|_{L^{\infty}(\hat{B}_R,g_X)}.
\end{equation}
At this point we bring in the decomposition \eqref{scrat}, and noting that $(\hat{\eta}^\ddagger_t)_{\mathbf{ff}}=0$ since $\hat{\eta}^\ddagger_t$ is pulled back from the base, we can bound \eqref{zerp2} by
\begin{equation}\label{zerp3}\begin{split}
&C[\D^a_{\mathbf{b\cdots b}}(\hat{\eta}_t)_{\mathbf{ff}}]_{C^\alpha_{\rm base}(\hat{B}_R,g_X)}+C\left(\frac{R}{\lambda_t}\right)^{1-\alpha}\sum_{b=0}^{a} \lambda_t^{-\alpha}\left(\|\D^b(\hat{\eta}^\dagger_t)_{\mathbf{ff}} \|_{L^{\infty}(\hat{B}_R,g_X)}+\|\D^b(\hat{\eta}^\diamond_t+\hat{\eta}^\circ_t+\hat{\eta}_{t,j,k})_{\mathbf{ff}} \|_{L^{\infty}(\hat{B}_R,g_X)}\right)\\
&\leq C\delta_t^2\Bigg([\D^a_{\mathbf{b\cdots b}}(\hat{\eta}_t)_{\mathbf{ff}}]_{C^\alpha_{\rm base}(\hat{B}_R,\hat{g}_t)}\\
&\quad\quad\quad\quad+\left(\frac{R}{\lambda_t}\right)^{1-\alpha}\sum_{b=0}^{a} \lambda_t^{-\alpha}\left(\delta_t^{-2}\|\D^b(\hat{\eta}^\dagger_t)_{\mathbf{ff}} \|_{L^{\infty}(\hat{B}_R,g_X)}+\|\D^b(\hat{\eta}^\diamond_t+\hat{\eta}^\circ_t+\hat{\eta}_{t,j,k})_{\mathbf{ff}} \|_{L^{\infty}(\hat{B}_R,\hat{g}_t)}\right)\Bigg)\\
&\leq C\delta_t^2\left([\D^a_{\mathbf{b\cdots b}}\hat{\eta}_t]_{C^\alpha_{\rm base}(\hat{B}_R,\hat{g}_t)}+\left(\frac{R}{\lambda_t}\right)^{1-\alpha}\sum_{b=0}^{a} \lambda_t^{-\alpha}\left(\delta_t^{-2}\|\D^b\hat{\eta}^\dagger_t \|_{L^{\infty}(\hat{B}_R,g_X)}+\|\D^b(\hat{\eta}^\diamond_t+\hat{\eta}^\circ_t+\hat{\eta}_{t,j,k}) \|_{L^{\infty}(\hat{B}_R,\hat{g}_t)}\right)\right),
\end{split}
\end{equation}
where the last inequality holds because $\D$ and $\P$ preserve the three subbundles of \eqref{porra} and these subbundles are orthogonal with respect to any of our product Riemannian metrics (and $g_X$ is uniformly equivalent to a product metric), and because restriction of contravariant tensors to subspaces is norm nonincreasing.

The bound \eqref{zerp3} together with \eqref{pizgrgr} and \eqref{pizgrgr2} gives us the desired bound \eqref{fedeliallalinea2}--\eqref{fedeliallalinea4} for the first term on the RHS of \eqref{loo}. Lastly, we deal with the second term on the RHS of \eqref{loo}.
Using the Leibniz rule we can bound it by
\begin{equation}\label{noreas}\begin{split}
e^{-(2k+2)\frac{t}{2}} &\sum_{\iota=2}^{i-1}\sum_{q=1}^{N_{\iota,k}} \sum_{\kappa=0}^{2k+2}\sum_{b=0}^a\lambda_t^\kappa\bigg([\D^{a-b}\hat{\Phi}_{\kappa,i,p,k,\iota,q}]_{C^\alpha(\hat{B}_R)}
\|\D^{b+\kappa}\hat{A}_{t,\iota,q,k}\|_{L^\infty(\hat{B}_R)}\\
&+\|\D^{a-b}\hat{\Phi}_{\kappa,i,p,k,\iota,q}\|_{L^\infty(\hat{B}_R)}
[\D^{b+\kappa}\hat{A}_{t,\iota,q,k}]_{C^\alpha(\hat{B}_R)}\bigg),
\end{split}\end{equation}
and we can bound
\begin{equation}\begin{split}
\lambda_t^\kappa[\D^{a-b}\hat{\Phi}_{\kappa,i,p,k,\iota,q}]_{C^\alpha(\hat{B}_R)}&\leq \lambda_t^\kappa R^{1-\alpha}\|\D^{a-b+1}\hat{\Phi}_{\kappa,i,p,k,\iota,q}\|_{L^\infty(\hat{B}_R)}\\
&\leq
C R^{1-\alpha}\lambda_t^{b+\kappa-a-1}=C\left(\frac{R}{\lambda_t}\right)^{1-\alpha}\lambda_t^{b+\kappa-a-\alpha},
\end{split}\end{equation}
and for $b+\kappa<a+2k+2$
\begin{equation}\begin{split}
\lambda_t^\kappa\|\D^{a-b}\hat{\Phi}_{\kappa,i,p,k,\iota,q}\|_{L^\infty(\hat{B}_R)}[\D^{b+\kappa}\hat{A}_{t,\iota,q,k}]_{C^\alpha(\hat{B}_R)}&\leq C \lambda_t^{b+\kappa-a}R^{1-\alpha}\|\D^{b+\kappa+1}\hat{A}_{t,\iota,q,k}\|_{L^\infty(\hat{B}_R)}\\
&=C\left(\frac{R}{\lambda_t}\right)^{1-\alpha}\lambda_t^{b+\kappa+1-a-\alpha}\|\D^{b+\kappa+1}\hat{A}_{t,\iota,q,k}\|_{L^\infty(\hat{B}_R)},
\end{split}\end{equation}
and so \eqref{noreas} is bounded by
\begin{equation}
Ce^{-(2k+2)\frac{t}{2}} \sum_{\iota=2}^{i-1}\sum_{q=1}^{N_{\iota,k}}\bigg(\lambda_t^{2k+2}[\D^{a+2k+2}\hat{A}_{t,\iota,q,k}]_{C^\alpha(\hat{B}_R)}
+\left(\frac{R}{\lambda_t}\right)^{1-\alpha}\sum_{c=0}^{a+2k+2} \lambda_t^{c-a-\alpha}\|\D^{c}\hat{A}_{t,\iota,q,k}\|_{L^\infty(\hat{B}_R)}\bigg),
\end{equation}
which is exactly \eqref{fedeliallalinea3}.
\end{proof}

\subsubsection{Killing the contribution from $\ti{A}_{t,i,p,k}$: the main claim \eqref{sogno}}

Now that we have established the noncancellation property, the first goal is to show that the first piece of \eqref{zumteufel163} goes to zero, namely
\begin{equation}\label{agognata}
d_t^{-j-\alpha}\sum_{i=2}^j\sum_{p=1}^{N_{i,k}}\sum_{a = j}^{j+2+2k}\ve_t^{a-j-2}\frac{|\D^{a}\ti{A}_{t,i,p,k}(\ti{x}_t) -  \P_{\ti{x}'_t\ti{x}_t}(\D^{a}\ti{A}_{t,i,p,k}(\ti{x}'_t))|_{\ti{g}_t(\ti{x}_t)}}{d^{\ti{g}_t}(\ti{x}_t,\ti{x}'_t)^\alpha}=o(1),
\end{equation}
and to do this we will prove the more precise estimate
\begin{equation}\label{superagognata}
d_t^{-j-\alpha}\sum_{i=2}^j\sum_{p=1}^{N_{i,k}}\ve_t^{a-j-2}[\D^a\ti{A}_{t,i,p,k}]_{C^\beta(\ti{B}_{O(\ve_t)},\ti{g}_t)}\leq C\ve_t^{\alpha-\beta},
\end{equation}
for all $a\geq j$ and $\alpha\leq\beta<1$, where somewhat abusively in the rest of this section the notation $O(\ve_t)$ for a radius $R$ will mean that it satisfies $C_0\ve_t\leq R\leq C_1\ve_t$, where $C_0$ is the fixed large constant in \eqref{e:tilt0} (so that the $\ti{g}_t$-geodesic ball centered at $\ti{x}_t$ with radius $R$ contains a Euclidean ball of radius $R/2$ times the whole $Y$ fiber). This choice will allow us to apply the interpolation inequalities in Proposition \ref{l:higher-interpol} on $\ti{B}_{\rho}\subset\ti{B}_R$ when $R$ and $\rho$ are $O(\ve_t)$.

Observe that taking \eqref{superagognata} with $\beta>\alpha$ implies \eqref{agognata} since it gives in particular an $o(1)$ bound for the $C^\beta$ seminorm on the $\ti{g}_t$-geodesic ball centered at $\ti{x}_t$ of radius $2$, and hence an $o(1)$ bound for the $C^\alpha$ seminorm on this same ball, which includes the other blowup point $\ti{x}'_t$ and thus implies \eqref{agognata}.

To start, we wish to employ \eqref{fedeliallalinea2}--\eqref{fedeliallalinea3} with H\"older exponent equal to $\beta$ and radius $R=d_tS, S=O(\ve_t)$ (so $(R/\lambda_t)^{1-\beta}\leq Ce^{-(1-\beta)\frac{t}{2}}$). To do this, we first use  \eqref{crocifisso7}, \eqref{crocifisso8bis}, \eqref{crocifisso9}, \eqref{crocifisso3} and \eqref{crocifisso4ter} to bound the term
\begin{equation}\begin{split}
&e^{-(1-\beta)\frac{t}{2}}\sum_{b=0}^{a} \lambda_t^{-\beta}\left(\|\D^b\hat{\eta}^\ddagger_t\|_{L^{\infty}(\hat{B}_{d_tS},\hat{g}_t)}
+\|\D^b(\hat{\eta}^\diamond_t+\hat{\eta}^\circ_t+\hat{\eta}_{t,j,k}) \|_{L^{\infty}(\hat{B}_{d_tS},\hat{g}_t)}+\delta_t^{-2}\|\D^b \hat\eta^\dagger_t \|_{L^{\infty}(\hat{B}_{d_tS},g_X)}\right)\\
&\leq Ce^{-(1-\beta)\frac{t}{2}}\sum_{b=0}^{a} \lambda_t^{-\beta}\left(O(1)
+\|\D^b(\hat{\eta}^\diamond_t+\hat{\eta}^\circ_t+\hat{\eta}_{t,j,k}) \|_{L^{\infty}(\hat{B}_{d_tS},\hat{g}_t)}+Ce^{\frac{-1-\alpha}{2}\frac{\alpha}{j+\alpha}t}\right)\\
&\leq Ce^{-(1-\beta)\frac{t}{2}} \lambda_t^{-\beta}(O(1)+\delta_t^{j+\alpha-a})+Ce^{-(1-\beta)\frac{t}{2}}\sum_{b=j+1}^{a} \lambda_t^{-\beta}\|\D^b(\hat{\eta}^\diamond_t+\hat{\eta}^\circ_t+\hat{\eta}_{t,j,k}) \|_{L^{\infty}(\hat{B}_{d_tS},\hat{g}_t)}.
\end{split}\end{equation}

Using this, we can transfer \eqref{fedeliallalinea2}--\eqref{fedeliallalinea3} to the tilde picture and multiplying it by $\delta_t^{a-j-2}d_t^{\beta-\alpha}$ we then get for all $a\geq j, \alpha\leq\beta<1,$
\begin{equation}\label{bracciodiferro}\begin{split}
&d_t^{-j-\alpha}\ve_t^{a-j-2}[\D^{a}\ti{A}_{t,i,p,k}]_{C^\beta(\ti{B}_{O(\ve_t)},\ti{g}_t)} \leq
C\ve_t^{a-j}d_t^{-j-\alpha}[\D^a_{\mathbf{b\cdots b}}\ti\eta_t]_{C^\beta_{\rm base}(\ti{B}_{O(\ve_t)},\ti{g}_t)}\\
&+Ce^{-(1-\beta)\frac{t}{2}}\lambda_t^{-\beta}\delta_t^{a-j}d_t^{\beta-\alpha}(O(1)+\delta_t^{j+\alpha-a})\\
&+Ce^{-(1-\beta)\frac{t}{2}}\sum_{b=j+1}^{a} \lambda_t^{-\beta}\delta_t^{a-j}d_t^{-b+\beta-\alpha}\|\D^b(\ti{\eta}_{t,j,k}+\ti{\eta}^\diamond_t+
\ti{\eta}^\circ_t)\|_{L^\infty(\ti{B}_{O(\ve_t)},\ti{g}_t)}\\
&+C\sum_{\ell=2}^{i-1}\sum_{q = 1}^{N_{\ell,k}}\Bigg(d_t^{-j-\alpha}\ve_t^{2k+a-j}[\D^{a + 2k + 2}\ti{A}_{t,\ell,q,k}]_{C^\beta(\ti{B}_{O(\ve_t)},\ti{g}_t)}\\
&+ \sum_{b=0}^{a + 2k + 2} d_t^{-j-\alpha} \ve_t^{a-j-2}d_t^{a-b+\beta}e^{-(2k+2+1-\beta)\frac{t}{2}}\lambda_t^{b-a-\beta} \|\D^{b}\ti{A}_{t,\ell,q,k}\|_{L^{\infty}(\ti{B}_{O(\ve_t)},\ti{g}_t)} \Bigg),
\end{split}\end{equation}
where the constants depend on $a$,
so the goal \eqref{superagognata} is to show that all terms on the RHS of \eqref{bracciodiferro} are bounded by $C\ve_t^{\alpha-\beta}$,
for all $a\geq j$. This is clear for the term on the second line
\begin{equation}\label{ominide}
Ce^{-(1-\beta)\frac{t}{2}}\lambda_t^{-\beta}\delta_t^{a-j}d_t^{\beta-\alpha}(O(1)+\delta_t^{j+\alpha-a})=o(\ve_t^{\alpha-\beta}).
\end{equation}
Next we discuss the term on the first line of \eqref{bracciodiferro}. The first useful observation is that thanks to \eqref{crocifisso9} and \eqref{crocifisso4bis} for all $a\geq j$ we have
\begin{equation}\label{bstrd2}
\begin{split}
\ve_t^{a-j}d_t^{-j-\alpha}[\D^a_{\mathbf{b\cdots b}}\ti{\eta}_t]_{C^\beta_{\rm base}(\ti{B}_{O(\ve_t)},\ti{g}_t)}&=\ve_t^{a-j}d_t^{-j-\alpha}[\D^a_{\mathbf{b\cdots b}}(\ti{\eta}_{t,j,k}+\ti{\eta}^\diamond_t+\ti{\eta}^\circ_t)]_{C^\beta_{\rm base}(\ti{B}_{O(\ve_t)},\ti{g}_t)}+C\delta_t^{a-j}d_t^{\beta-\alpha}o(1)\\
&=\ve_t^{a-j}d_t^{-j-\alpha}[\D^a_{\mathbf{b\cdots b}}(\ti{\eta}_{t,j,k}+\ti{\eta}^\diamond_t+\ti{\eta}^\circ_t)]_{C^\beta_{\rm base}(\ti{B}_{O(\ve_t)},\ti{g}_t)}+o(\ve_t^{\alpha-\beta}).
\end{split}
\end{equation}
We thus proceed to bound the first term on the RHS of \eqref{bstrd2}, and the idea is to use the Schauder estimates in Proposition \ref{sonomatematicofrancese}. In order to do that we need to check that after applying the diffeomorphisms
\begin{equation}\Pi_t:  B_{e^{\frac{t}{2}}} \times Y \to B_{d_t^{-1}\lambda_t} \times Y, \;\, (\ti{z},\ti{y}) = \Pi_t(\check{z},\check{y}) =(\ve_t\check{z},\check{y}),\end{equation}
then on $\Pi_t^{-1}(\ti{B}_{\ve_t}\times Y)=\check{B}_1\times Y$ the metrics $\check{g}_t=\ve_t^{-2}\Pi_t^*\ti{g}_t$ and $\check{\omega}^\sharp_t=\ve_t^{-2}\Pi_t^*\ti{\omega}^\sharp_t$ are smoothly convergent. This is obvious for $\check{g}_t$ and for $\check{\omega}^\sharp_t$
we have
\begin{equation}
\check{\omega}^\sharp_t=\check{\omega}_{\rm can}+\Pi_t^*\Theta_t^*\Psi_t^*\omega_F+\check{\eta}^\dagger_t+\check{\eta}^\ddagger_t,
\end{equation}
where $\check{\omega}_{\rm can}+\Pi_t^*\Theta_t^*\Psi_t^*\omega_F$ is smoothly convergent, $\check{\eta}^\dagger_t$ goes smoothly to zero by \eqref{crocifisso4}, and $\check{\eta}^\ddagger_t$ goes smoothly to zero when $j>0$ by \eqref{crocifisso9}, and it (sub)converges smoothly when $j=0$ since in that case it is a sequence of constant coefficient forms with uniformly bounded $L^\infty$ norm. Thus Proposition \ref{sonomatematicofrancese} applies, so given two radii $\rho<R$ which are $O(\ve_t)$, and $\alpha\leq\beta<1$, let $\ti{R}=\rho+\frac{1}{2}(R-\rho)$ (which is of course also $O(\ve_t)$) apply the proposition to bound
\begin{equation}\label{schauder}
\begin{split}
\ve_t^{a-j}d_t^{-j-\alpha}[\D^a(\ti{\eta}_{t,j,k}+\ti{\eta}^\diamond_t+\ti{\eta}^\circ_t)]_{C^\beta(\ti{B}_\rho,\ti{g}_t)}&\leq
C\ve_t^{a-j}d_t^{-j-\alpha}[\D^a\tr{\ti\omega_t^\sharp}{(\ti{\eta}_{t,j,k}+\ti{\eta}^\diamond_t+\ti{\eta}^\circ_t)}]_{C^\beta(\ti{B}_{\ti{R}},\ti{g}_t)}\\
&+C\ve_t^{a-j}d_t^{-j-\alpha}(R-\rho)^{-a-\beta}\|\ti{\eta}_{t,j,k}+\ti{\eta}^\diamond_t+\ti{\eta}^\circ_t\|_{L^\infty(\ti{B}_{\ti{R}},\ti{g}_t)}\\
&\leq C\ve_t^{a-j}d_t^{-j-\alpha}[\D^a\tr{\ti\omega_t^\sharp}{(\ti{\eta}_{t,j,k}+\ti{\eta}^\diamond_t+\ti{\eta}^\circ_t)}]_{C^\beta(\ti{B}_{\ti{R}},\ti{g}_t)}\\
&+C\ve_t^{a+\alpha}(R-\rho)^{-a-\beta},
\end{split}
\end{equation}
where for the $L^\infty$ term we used the estimates from \eqref{utilissimo}, \eqref{crocifisso8} and \eqref{utilissimo2} which together give
\begin{equation}\label{cojone}
d_t^{-j-\alpha}\|\D^\iota(\ti{\eta}_{t,j,k}+\ti{\eta}^\diamond_t+\ti{\eta}^\circ_t)\|_{L^\infty(\ti{B}_{O(\ve_t)},\ti{g}_t)}\leq C\ve_t^{j+\alpha-\iota},\quad d_t^{-j-\alpha}[\D^\iota(\ti{\eta}_{t,j,k}+\ti{\eta}^\diamond_t+\ti{\eta}^\circ_t)]_{C^\alpha(\ti{B}_{O(\ve_t)},\ti{g}_t)}\leq C\ve_t^{j-\iota},
\end{equation}
for $0\leq\iota\leq j$.

The main claim is then that for all $a\geq j$, radii $\rho<R$ which are $O(\ve_t)$, and $\alpha\leq\beta<1$, letting $\ti{R}=\rho+\frac{1}{2}(R-\rho)$, we have the bound
\begin{equation}\label{sogno}\begin{split}
C\ve_t^{a-j}d_t^{-j-\alpha}[\D^a\tr{\ti\omega_t^\sharp}{(\ti{\eta}_{t,j,k}+\ti{\eta}^\diamond_t+\ti{\eta}^\circ_t)}]_{C^\beta(\ti{B}_{\ti{R}},\ti{g}_t)}&\leq
\frac{1}{2}\ve_t^{a-j}d_t^{-j-\alpha}[\D^a(\ti{\eta}_{t,j,k}+\ti{\eta}^\diamond_t+\ti{\eta}^\circ_t)]_{C^\beta(\ti{B}_R,\ti{g}_t)}\\
&+C\ve_t^{\alpha-\beta}+C\ve_t^{\alpha-\beta+a}(R-\rho)^{-a}.
\end{split}
\end{equation}
where the constants depend on $a,\beta$.

Before we delve into the proof of the main claim \eqref{sogno}, it is useful to derive some consequences from it. Specifically, suppose that \eqref{sogno} has been proved for some value of $a\geq j$ and $\alpha\leq\beta<1$, then combining it with \eqref{schauder} gives
\begin{equation}\begin{split}
\ve_t^{a-j}d_t^{-j-\alpha}[\D^a(\ti{\eta}_{t,j,k}+\ti{\eta}^\diamond_t+\ti{\eta}^\circ_t)]_{C^\beta(\ti{B}_\rho,\ti{g}_t)}
&\leq \frac{1}{2}\ve_t^{a-j}d_t^{-j-\alpha}[\D^a(\ti{\eta}_{t,j,k}+\ti{\eta}^\diamond_t+\ti{\eta}^\circ_t)]_{C^\beta(\ti{B}_R,\ti{g}_t)}\\
&+C\ve_t^{\alpha-\beta}+C\ve_t^{\alpha-\beta+a}(R-\rho)^{-a}+C\ve_t^{\alpha+a}(R-\rho)^{-a-\beta},
\end{split}
\end{equation}
for all radii $\rho<R$ which are $O(\ve_t)$, and then the iteration Lemma \ref{HT2-l:iterate} gives that
\begin{equation}\label{osondesto}
\ve_t^{a-j}d_t^{-j-\alpha}[\D^a(\ti{\eta}_{t,j,k}+\ti{\eta}^\diamond_t+\ti{\eta}^\circ_t)]_{C^\beta(\ti{B}_\rho,\ti{g}_t)}\leq
C\ve_t^{\alpha-\beta}+C\ve_t^{\alpha-\beta+a}(R-\rho)^{-a}+C\ve_t^{\alpha+a}(R-\rho)^{-a-\beta},
\end{equation}
for all radii $\rho<R$ which are $O(\ve_t)$ and all $a\geq j$. In other words, this means
\begin{equation}\label{osondesto2}
\ve_t^{a-j}d_t^{-j-\alpha}[\D^a(\ti{\eta}_{t,j,k}+\ti{\eta}^\diamond_t+\ti{\eta}^\circ_t)]_{C^\beta(\ti{B}_{O(\ve_t)},\ti{g}_t)}\leq
C\ve_t^{\alpha-\beta},
\end{equation}
which can be inserted into \eqref{bstrd2} and get
\begin{equation}\label{fleabag2}
\ve_t^{a-j}d_t^{-j-\alpha}[\D^a_{\mathbf{b\cdots b}}\ti{\eta}_t]_{C^\beta_{\rm base}(\ti{B}_{O(\ve_t)},\ti{g}_t)}\leq C\ve_t^{\alpha-\beta},
\end{equation}
which would give the desired bound for the first term on the RHS of \eqref{bracciodiferro}. Furthermore, we can interpolate between \eqref{osondesto2} and the $L^\infty$ bound in \eqref{cojone} using  Proposition \ref{l:higher-interpol} on balls of radius $O(\ve_t)$ and obtain
\begin{equation}\label{kkrz}
\ve_t^{b-j-\alpha}d_t^{-j-\alpha}\|\D^b(\ti{\eta}_{t,j,k}+\ti{\eta}^\diamond_t+\ti{\eta}^\circ_t)\|_{L^\infty(\ti{B}_{O(\ve_t)},\ti{g}_t)}\leq C,
\quad \ve_t^{b-j+\beta-\alpha}d_t^{-j-\alpha}[\D^b(\ti{\eta}_{t,j,k}+\ti{\eta}^\diamond_t+\ti{\eta}^\circ_t)]_{C^\beta(\ti{B}_{O(\ve_t)},\ti{g}_t)}\leq C,
\end{equation}
for $0\leq b\leq a$.

Now that we have derived some consequences from the main claim \eqref{sogno}, we proceed to prove it by induction on $a\geq j$. In the inductive step, \eqref{sogno} can be assumed to hold for smaller values of $a\geq j$, hence \eqref{kkrz} will hold for these smaller values. In the base case of the induction when $a=j$, the bounds in \eqref{kkrz} are already known to hold when $\alpha=\beta$ thanks to \eqref{cojone}, but at this point they have not yet been established when $a=j$ and $\beta>\alpha$.

Now, to prove claim \eqref{sogno} we need to convert $[\D^a\tr{\ti\omega_t^\sharp}{(\ti{\eta}_{t,j,k}+\ti{\eta}^\diamond_t+\ti{\eta}^\circ_t)}]_{C^\beta}$ into $[\D^a(\ti{\eta}_{t,j,k}+\ti{\eta}^\diamond_t+\ti{\eta}^\circ_t)]_{C^\beta}$ with a small coefficient in front, and this is where the PDE \eqref{satanas} comes in, which in the tilde picture can be written (ignoring combinatorial factors) as
\begin{equation}\label{satanasti}
d_t^{-j-\alpha}\tr{\ti\omega_t^\sharp}{(\ti{\eta}_{t,j,k}+\ti{\eta}^\diamond_t+\ti{\eta}^\circ_t)}
=d_t^{-j-\alpha}\left(c_t \frac{\ti{\omega}_{\rm can}^m\wedge(\Theta_t^*\Psi_t^*\omega_F)^n}{(\ti\omega_t^\sharp)^{m+n}}-1\right)
-d_t^{-j-\alpha}\sum_{p\geq 2}\frac{(\ti{\eta}_{t,j,k}+\ti{\eta}^\diamond_t+\ti{\eta}^\circ_t)^p\wedge (\ti\omega_t^\sharp)^{m+n-p}}{(\ti\omega_t^\sharp)^{m+n}}.
\end{equation}
We then multiply \eqref{satanasti} by $\ve_t^{a-j}, a\geq j,$ and we need to see what bounds are available for the pieces on the RHS of \eqref{satanasti}.

\subsubsection{The piece of \eqref{satanasti} with the nonlinearities}\label{k2}

This is the term
\begin{equation}\label{feixianxing}
\mathcal{N}:=C\ve_t^{a-j}d_t^{-j-\alpha}\sum_{p\geq 2}\frac{(\ti{\eta}_{t,j,k}+\ti{\eta}^\diamond_t+\ti{\eta}^\circ_t)^p\wedge (\ti\omega_t^\sharp)^{m+n-p}}{(\ti\omega_t^\sharp)^{m+n}}.
\end{equation}
For any $a\geq 0$, we take $\D^a$ of \eqref{feixianxing} and expand it schematically as
\begin{equation}\label{totalmess}
C\sum_{\substack{p\geq 2\\ \mu+\psi+\sigma=a}}(\D^\mu(\ve_t^{a-j}d_t^{-j-\alpha}(\ti{\eta}_{t,j,k}+\ti{\eta}^\diamond_t+\ti{\eta}^\circ_t)^p))(\D^\psi((\ti{\omega}^\sharp_t)^{m+n-p}))(\D^\sigma(((\ti{\omega}^\sharp_t)^{m+n})^{-1})),
\end{equation}

To bound this, first observe that from \eqref{crocifisso5} we can bound for all $\psi\geq 0$
\begin{equation}\label{plz3}
[\D^\psi\tilde{\omega}^\sharp_t]_{C^\beta(\ti{B}_{\ve_t},\ti{g}_t)}\leq C\ve_t^{-\beta-\psi},
\quad \|\D^\psi\tilde{\omega}^\sharp_t\|_{L^\infty(\ti{B}_{\ve_t},\ti{g}_t)}\leq C\ve_t^{-\psi},
\end{equation}
and similarly for the reciprocal of $(\tilde{\omega}^\sharp_t)^{m+n}$. As for $\ti{\eta}_{t,j,k}+\ti{\eta}^\diamond_t+\ti{\eta}^\circ_t$, we claim that for $\alpha\leq \beta<1$ and $0\leq \mu<a$ we have
\begin{equation}\label{plz4}
\|\D^\mu(\ti{\eta}_{t,j,k}+\ti{\eta}^\diamond_t+\ti{\eta}^\circ_t)\|_{L^\infty(\ti{B}_{O(\ve_t)},\ti{g}_t)}\leq C\delta_t^{j+\alpha}\ve_t^{-\mu},\quad
[\D^\mu(\ti{\eta}_{t,j,k}+\ti{\eta}^\diamond_t+\ti{\eta}^\circ_t)]_{C^\beta(\ti{B}_{O(\ve_t)},\ti{g}_t)}\leq C\delta_t^{j+\alpha}\ve_t^{-\beta-\mu},\end{equation}
and indeed in the base case of the induction $a=j$ these are simple consequences of \eqref{cojone} plus interpolation (Proposition \ref{l:higher-interpol}), while for $a>j$ these are given by the bounds \eqref{kkrz} which hold by induction.

Next, we take $a\geq j$ in \eqref{totalmess} and we take $[\D^a\mathcal{N}]_{C^\beta(\ti{B}_{\ti{R}},\ti{g}_t)}$, where $\ti{R}=O(\ve_t)$ is as above, and $\alpha\leq\beta<1$. The seminorm will distribute onto all the terms in the obvious sense. \\

\noindent
{\bf Case 1. }Let us first consider any term with a $\D^a$ derivative of $(\ti{\eta}_{t,j,k}+\ti{\eta}^\diamond_t+\ti{\eta}^\circ_t)$ on which the difference quotient lands, which is then multiplied by at least one other decorated $\ti{\eta}$ piece (since $p\geq 2$) which is $o(1)$ by \eqref{plz4}, and thus the sum of all such terms is bounded above by
\begin{equation}\frac{1}{4}\ve_t^{a-j}d_t^{-j-\alpha}[\D^a(\ti{\eta}_{t,j,k}+\ti{\eta}^\diamond_t+\ti{\eta}^\circ_t)]_{C^\beta(\ti{B}_R,\ti{g}_t)}.\end{equation}

\noindent
{\bf Case 2. }Next, consider any term with a $\D^a$ derivative of $(\ti{\eta}_{t,j,k}+\ti{\eta}^\diamond_t+\ti{\eta}^\circ_t)$ (we may assume $a>0$, otherwise this term has already been discussed), with the difference quotient landing somewhere else. Each such term is thus the product of two pieces, the first being $\ve_t^{a-j}d_t^{-j-\alpha}\D^a(\ti{\eta}_{t,j,k}+\ti{\eta}^\diamond_t+\ti{\eta}^\circ_t)$ and the second piece with the difference quotient landing on the product of the remaining terms (and these terms contain at least one other decorated $\ti{\eta}$ piece since $p\geq 2$). Thanks to \eqref{plz3}, \eqref{plz4} (this last one applies since $a>0$) this second piece is bounded by $C\delta_t^{j+\alpha}\ve_t^{-\beta}=o(\ve_t^{-\beta}).$ On the other hand, to bound the first piece we interpolate with Proposition \ref{l:higher-interpol} from radius $\ti{R}=\rho+\frac{1}{2}(R-\rho)$ to radius $R$
\begin{equation}\label{kuz1}\begin{split}
\ve_t^{a-j}d_t^{-j-\alpha}\|\D^a(\ti{\eta}_{t,j,k}+\ti{\eta}^\diamond_t+\ti{\eta}^\circ_t)\|_{L^\infty(\ti{B}_{\ti{R}},\ti{g}_t)}&\leq
C\ve_t^{a-j}d_t^{-j-\alpha}(R-\rho)^{\beta}[\D^a(\ti{\eta}_{t,j,k}+\ti{\eta}^\diamond_t+\ti{\eta}^\circ_t)]_{C^\beta(\ti{B}_R,\ti{g}_t)}\\
&+C\ve_t^{a-j}d_t^{-j-\alpha}(R-\rho)^{-a}\|\ti{\eta}_{t,j,k}+\ti{\eta}^\diamond_t+\ti{\eta}^\circ_t\|_{L^\infty(\ti{B}_R,\ti{g}_t)}\\
&\leq C\ve_t^{a-j}d_t^{-j-\alpha}\ve_t^{\beta}[\D^a(\ti{\eta}_{t,j,k}+\ti{\eta}^\diamond_t+\ti{\eta}^\circ_t)]_{C^\beta(\ti{B}_R,\ti{g}_t)}\\
&+C\ve_t^{a-j}d_t^{-j-\alpha}(R-\rho)^{-a}\delta_t^{j+\alpha},
\end{split}\end{equation}
so multiplying these bounds for the two pieces, in this case the sum of all such terms overall can be bounded by
\begin{equation}\begin{split}
&o(1)C\ve_t^{a-j}d_t^{-j-\alpha}[\D^a(\ti{\eta}_{t,j,k}+\ti{\eta}^\diamond_t+\ti{\eta}^\circ_t)]_{C^\beta(\ti{B}_R,\ti{g}_t)}
+o(1)\ve_t^{-\beta}C\ve_t^{a-j}d_t^{-j-\alpha}(R-\rho)^{-a}\delta_t^{j+\alpha}\\
&\leq \frac{1}{4}\ve_t^{a-j}d_t^{-j-\alpha}[\D^a(\ti{\eta}_{t,j,k}+\ti{\eta}^\diamond_t+\ti{\eta}^\circ_t)]_{C^\beta(\ti{B}_R,\ti{g}_t)}+
o(1)\ve_t^{\alpha-\beta+a}(R-\rho)^{-a}.
\end{split}\end{equation}

\noindent
{\bf Case 3. }In all the remaining terms which are not covered by cases 1 and 2, the largest number of derivatives that can land on $\ve_t^{a-j}d_t^{-j-\alpha}(\ti{\eta}_{t,j,k}+\ti{\eta}^\diamond_t+\ti{\eta}^\circ_t)$ (including difference quotient) is strictly less than $a$, i.e. it is either of the form $r$ with $0\leq r<a$, or of the form $r+\beta$ with $0\leq r<a$. Thus the inductive estimates \eqref{plz4} hold, and the piece of this term where these derivatives land is bounded by $C\ve_t^{a-r+\alpha}$ (resp. $C\ve_t^{a-r+\alpha-\beta}$), while the second piece of this term where the remaining $a+\beta-r$ (resp. $a-r$) derivatives land is bounded by $o(\ve_t^{-a-\beta+r})$ (resp. $o(\ve_t^{-a+r})$) by  \eqref{plz3}, \eqref{plz4}. Thus, every such term overall is bounded by $o(\ve_t^{\alpha-\beta}).$\\

Combining the discussion of these $3$ cases, we have thus covered all possible terms and the conclusion is that for all $a\geq j, \alpha\leq\beta<1$ we have
\begin{equation}\label{piece2}\begin{split}
[\D^a\mathcal{N}]_{C^\beta(\ti{B}_{\ti{R}},\ti{g}_t)}&\leq \frac{1}{2}\ve_t^{a-j}d_t^{-j-\alpha}[\D^a(\ti{\eta}_{t,j,k}+\ti{\eta}^\diamond_t+\ti{\eta}^\circ_t)]_{C^\beta(\ti{B}_R,\ti{g}_t)}\\
&+o(\ve_t^{\alpha-\beta})+o(\ve_t^{\alpha-\beta+a})(R-\rho)^{-a}.
\end{split}\end{equation}

\subsubsection{The piece of \eqref{satanasti} with $c_t \frac{\ti{\omega}_{\rm can}^m\wedge(\Theta_t^*\Psi_t^*\omega_F)^n}{(\ti\omega_t^\sharp)^{m+n}}-1$}\label{k3}

The claim about this is that for all $a\geq 0$ and $1>\beta\geq\alpha$ and all fixed $R$ we have
\begin{align}\label{piece3}
\ve_t^{a-j}d_t^{-j-\alpha}\left\|\D^a\left(c_t\frac{\ti{\omega}_{\rm can}^m\wedge(\Theta_t^*\Psi_t^*\omega_F)^n}{(\ti\omega_t^\sharp)^{m+n}}-1\right)\right\|_{L^\infty(\ti{B}_{R\ve_t}\times Y,\ti{g}_t)}&\leq C\ve_t^{\alpha},\\
\ve_t^{a-j}d_t^{-j-\alpha}\left[\D^a\left(c_t\frac{\ti{\omega}_{\rm can}^m\wedge(\Theta_t^*\Psi_t^*\omega_F)^n}{(\ti\omega_t^\sharp)^{m+n}}-1\right)\right]_{C^\beta(\ti{B}_{R\ve_t}\times Y,\ti{g}_t)}&\leq C\ve_t^{\alpha-\beta},\nonumber
\end{align}
where as usual the constants depend on $a,\beta$.

First, observe that \eqref{piece3} for $a\leq j$ and $\beta=\alpha$ follows immediately from \eqref{cacanew3}, \eqref{cacanew3b}.

Next, to prove \eqref{piece3} for $a\geq j$ and $1>\beta\geq\alpha$, we apply the diffeomorphism
\begin{equation}\Pi_t:  B_{e^{\frac{t}{2}}} \times Y \to B_{d_t^{-1}\lambda_t} \times Y, \;\, (\ti{z},\ti{y}) = \Pi_t(\check{z},\check{y}) =(\ve_t\check{z},\check{y}),\end{equation}
and as usual multiply all the pulled back contravariant $2$-tensors by $\ve_t^{-2}$, and denote the new objects by a check. This way $\Pi_t^{-1}(\ti{B}_{R\ve_t}\times Y)=\check{B}_R\times Y$, the metric $\check{g}_t$ is uniformly Euclidean, and we have by definition $\Psi_t\circ\Theta_t\circ\Pi_t=\Sigma_t$, where $\Sigma_t$ was defined in \eqref{sigmat}.

Transferring to the check picture gives
\begin{equation}\label{ooga}\begin{split}
&\delta_t^{-j-\alpha}\left\|\D^a\left(c_t\frac{\check{\omega}_{\rm can}^m\wedge(\Sigma_t^*\omega_F)^n}{(\check\omega_t^\sharp)^{m+n}}-1\right)\right\|_{L^\infty(\check{B}_{R}\times Y,\check{g}_t)}\\
&=\ve_t^{-\alpha}\ve_t^{a-j}d_t^{-j-\alpha}\left\|\D^a\left(c_t\frac{\ti{\omega}_{\rm can}^m\wedge(\Theta_t^*\Psi_t^*\omega_F)^n}{(\ti\omega_t^\sharp)^{m+n}}-1\right)\right\|_{L^\infty(\ti{B}_{R\ve_t}\times Y,\ti{g}_t)},
\end{split}\end{equation}
\begin{equation}\label{chaka}\begin{split}
&\delta_t^{-j-\alpha}\left[\D^a\left(c_t\frac{\check{\omega}_{\rm can}^m\wedge(\Sigma_t^*\omega_F)^n}{(\check\omega_t^\sharp)^{m+n}}-1\right)\right]_{C^\beta(\check{B}_{R}\times Y,\check{g}_t)}\\
&=\ve_t^{\beta-\alpha}\ve_t^{a-j}d_t^{-j-\alpha}\left[\D^a\left(c_t\frac{\ti{\omega}_{\rm can}^m\wedge(\Theta_t^*\Psi_t^*\omega_F)^n}{(\ti\omega_t^\sharp)^{m+n}}-1\right)\right]_{C^\beta(\ti{B}_{R\ve_t}\times Y,\ti{g}_t)}.
\end{split}\end{equation}
It then follows from \eqref{piece3} with $a\leq j,\beta=\alpha$ that
$\delta_t^{-j-\alpha}\left(c_t \frac{\check{\omega}_{\rm can}^m\wedge(\Sigma_t^*\omega_F)^n}{(\check\omega_t^\sharp)^{m+n}}-1\right)$ is locally uniformly bounded in $C^{j,\alpha}$ for the (essentially fixed) metric $\check{g}_t$, so by Ascoli-Arzel\`a up to passing to sequence it converges locally uniformly on $\C^m\times Y$. Furthermore, the functions $\hat{A}^\sharp_{t,i,p,k}$ satisfy the bounds \eqref{linftyyy0} with $\alpha_0=\frac{\alpha(1+\alpha)}{j+\alpha}$ thanks to \eqref{prechecazzo}. We can thus apply the last statement in the Selection Theorem \ref{ghost} which implies that $\delta_t^{-j-\alpha}\left(c_t \frac{\check{\omega}_{\rm can}^m\wedge(\Sigma_t^*\omega_F)^n}{(\check\omega_t^\sharp)^{m+n}}-1\right)$ converges locally smoothly, so in particular all of its derivatives are locally uniformly bounded on $\check{B}_R\times Y$. Thus the LHS of \eqref{ooga} and \eqref{chaka} with are uniformly bounded for all $a,\beta$ and $R$, and this completes the proof of \eqref{piece3}.

\subsubsection{Completion of the proof of the main claim \eqref{sogno} and killing the contribution from $\ti{A}_{t,i,p,k}$}\label{jjj1}

Combining the discussion of the 2 pieces in sections \ref{k2} and \ref{k3}, namely combining \eqref{piece2} and \eqref{piece3} with \eqref{satanasti}, completes the proof by induction of \eqref{sogno} and hence of \eqref{fleabag2}. This means that the first term on the RHS of \eqref{bracciodiferro} is $O(\ve_t^{\alpha-\beta})$.

The second line of \eqref{bracciodiferro} was already treated in \eqref{ominide}. As for the third line, which is
\begin{equation}\label{krkz}
e^{-(1-\beta)\frac{t}{2}}\sum_{b=j+1}^{a} \lambda_t^{-\beta}\delta_t^{a-j}d_t^{-b+\beta-\alpha}\|\D^b (\ti{\eta}_{t,j,k}+\ti{\eta}^\diamond_t+\ti{\eta}^\circ_t) \|_{L^{\infty}(\ti{B}_{O(\ve_t)},\ti{g}_t)},
\end{equation}
where $a\geq j+1$, we also show that this term is $o(\ve_t^{\alpha-\beta})$ by using \eqref{kkrz} which gives
\begin{equation}\begin{split}
\|\D^b(\ti{\eta}_{t,j,k}+\ti{\eta}^\diamond_t+\ti{\eta}^\circ_t)\|_{L^\infty(\ti{B}_{O(\ve_t)},\ti{g}_t)}&\leq C\ve_t^{j-b+\alpha}d_t^{j+\alpha},
\end{split}\end{equation}
for $j+1\leq b\leq a$, and so each term in the sum in \eqref{krkz} is bounded by
\begin{equation}\begin{split}
&Ce^{-(1-\beta)\frac{t}{2}}\lambda_t^{-\beta}\delta_t^{a-j}d_t^{-b+\beta-\alpha}\ve_t^{j-b+\alpha}d_t^{j+\alpha}=Ce^{-(1-\beta)\frac{t}{2}}\lambda_t^{-\beta}\delta_t^{a-b+\beta}\ve_t^{\alpha-\beta}=o(\ve_t^{\alpha-\beta}),
\end{split}\end{equation}
which establishes our claim.

To recap what we have achieved so far, combining \eqref{bracciodiferro} with the bounds \eqref{fleabag2} for the first term on the RHS and those that we just discussed for the second and third line, we obtain for $a\geq j$
\begin{equation}\label{brazodeferaza}\begin{split}
&d_t^{-j-\alpha}\ve_t^{a-j-2}[\D^{a}\ti{A}_{t,i,p,k}]_{C^\beta(\ti{B}_{O(\ve_t)},\ti{g}_t)} \leq
C\ve_t^{\alpha-\beta}\\
&+C\sum_{\ell=2}^{i-1}\sum_{q = 1}^{N_{\ell,k}}\Bigg(d_t^{-j-\alpha}\ve_t^{2k+a-j}[\D^{a + 2k + 2}\ti{A}_{t,\ell,q,k}]_{C^\beta(\ti{B}_{O(\ve_t)},\ti{g}_t)}\\
 &+ \sum_{b=0}^{a + 2k + 2} d_t^{-j-\alpha} \ve_t^{a-j-2}d_t^{a-b+\beta}e^{-(2k+2+1-\beta)\frac{t}{2}}\lambda_t^{b-a-\beta} \|\D^{b}\ti{A}_{t,\ell,q,k}\|_{L^{\infty}(\ti{B}_{O(\ve_t)},\ti{g}_t)} \Bigg),
\end{split}\end{equation}
and we now use this to prove \eqref{superagognata} by showing by induction on $2\leq i\leq j$ that
for all $a\geq j$ we have
\begin{equation}\label{agone}
d_t^{-j-\alpha}\sum_{p=1}^{N_{i,k}}\ve_t^{a-j-2}[\D^a\ti{A}_{t,i,p,k}]_{C^\beta(\ti{B}_{O(\ve_t)},\ti{g}_t)}\leq C\ve_t^{\alpha-\beta},
\end{equation}
where the constant depends on $a$.
The base of the induction $i=2$ is exactly \eqref{brazodeferaza}.
For the induction step, we assume that \eqref{agone} holds up to $i-1$, then the second line of \eqref{brazodeferaza} can also be bounded by $C\ve_t^{\alpha-\beta}$. As for the third line of \eqref{brazodeferaza}, first consider the terms with $0\leq b\leq j$. For these, we use directly \eqref{610} which gives in particular
\begin{equation}\label{fleamarket}d_t^{-b}\|\D^{b}\ti{A}_{t,\ell,q,k}\|_{L^{\infty}(\ti{B}_{\ve_t},\ti{g}_t)}=o(\ve_t^2),\end{equation}
and so we can bound these terms by
\begin{equation}d_t^{-j-\alpha} \ve_t^{a-j-2}d_t^{a+\beta}e^{-(2k+2+1-\beta)\frac{t}{2}}\lambda_t^{b-a-\beta}\ve_t^2 o(1)=
\delta_t^{a-j}d_t^{\beta-\alpha}e^{-(2k+2+1-\beta)\frac{t}{2}}\lambda_t^{b-a-\beta}o(1)=o(\ve_t^{\alpha-\beta}),\end{equation}
since $b\leq j\leq a$.
Next, look at the terms with $j<b\leq a+2k+2$. For these we use interpolation between \eqref{fleamarket} with $b=j$ and
\begin{equation}d_t^{-j-\alpha}\ve_t^{b-j-2}[\D^b\ti{A}_{t,\ell,q,k}]_{C^\beta(\ti{B}_{O(\ve_t)},\ti{g}_t)}\leq C\ve_t^{\alpha-\beta},\end{equation}
for $b\geq j$ and $\ell\leq i-1$, which comes from the induction hypothesis \eqref{agone}. Interpolating gives, for $j<b\leq  a+2k+2$,
\begin{equation}\begin{split}
(R-\rho)^b\|\D^b\ti{A}_{t,\ell,q,k}\|_{L^\infty(\ti{B}_\rho,\ti{g}_t)}&\leq C
(R-\rho)^{b+\beta}[\D^b\ti{A}_{t,\ell,q,k}]_{C^\beta(\ti{B}_R,\ti{g}_t)}+C(R-\rho)^j\|\D^j\ti{A}_{t,\ell,q,k}\|_{L^\infty(\ti{B}_R,\ti{g}_t)}\\
&\leq C(R-\rho)^{b+\beta}d_t^{j+\alpha}\ve_t^{2+j-b+\alpha-\beta}+C(R-\rho)^jd_t^{j}\ve_t^2o(1),
\end{split}\end{equation}
with $\rho,R$ which are $O(\ve_t)$ and we get
\begin{equation}
\|\D^b\ti{A}_{t,\ell,q,k}\|_{L^\infty(\ti{B}_{O(\ve_t)},\ti{g}_t)}\leq C\ve_t^{2+j+\alpha-b}d_t^{j+\alpha}+C\ve_t^{2+j-b}d_t^{j},
\end{equation}
and so the terms in the third line of \eqref{brazodeferaza} with $b>j$ can be bounded by
\begin{equation}\begin{split}
&d_t^{-j-\alpha} \ve_t^{a-j-2}d_t^{a-b+\beta}e^{-(2k+2+1-\beta)\frac{t}{2}}\lambda_t^{b-a-\beta}\ve_t^{2+j+\alpha-b}d_t^{j+\alpha}\\
&+d_t^{-j-\alpha} \ve_t^{a-j-2}d_t^{a-b+\beta}e^{-(2k+2+1-\beta)\frac{t}{2}}\lambda_t^{b-a-\beta}\ve_t^{2+j-b}d_t^{j}\\
&=d_t^{\beta-\alpha}\delta_t^{a-b+\alpha}e^{-(2k+2+1-\beta)\frac{t}{2}}\lambda_t^{b-a-\beta}+
d_t^{\beta-\alpha}\delta_t^{a-b}e^{-(2k+2+1-\beta)\frac{t}{2}}\lambda_t^{b-a-\beta}\\
&=d_t^{\beta-\alpha}\lambda_t^{\alpha-\beta}e^{-(a-b+\alpha+2k+2+1-\beta)\frac{t}{2}}
+d_t^{\beta-\alpha}e^{-(a-b+2k+2+1-\beta)\frac{t}{2}}\lambda_t^{-\beta}\\
&=o(\ve_t^{\alpha-\beta}),
\end{split}\end{equation}
since $b\leq a+2k+2$. This completes the inductive proof of \eqref{agone}, and hence this also completes the proof of \eqref{superagognata} and of \eqref{agognata}.

\subsubsection{Killing the contribution from $\ti{\eta}^\diamond_t$}\label{jjj2}
The first claim is that
\begin{equation}\label{dq1}
d_t^{-j-\alpha}[\D^j\ti{\eta}^\diamond_t]_{C^\beta(\ti{B}_{O(\ve_t)},\ti{g}_t)}\leq Cd_t^{-j-\alpha}[\D^j(\hat{\eta}^\diamond_t+\hat{\eta}^\circ_t+\hat{\eta}_{t,j,k})]_{C^\beta(\ti{B}_{O(\ve_t)},\ti{g}_t)}+o(\ve_t^{\alpha-\beta}).
\end{equation}
To prove this, we start by recalling that
\begin{equation}\label{decomposto}
\ti{\eta}_t=\ti{\eta}^\diamond_t+\ti{\eta}^\ddagger_t+\ti{\eta}^\dagger_t+\ti{\eta}^\circ_t+\ti{\eta}_{t,j,k},\end{equation}
and by taking \eqref{fedeliallalinea} on $\hat{B}_{\delta_t}$ (which translates to $\ti{B}_{\ve_t}$ in the tilde picture) we bound the term ($0\leq b\leq j$)
\begin{equation}\label{carafa}\begin{split}
\|\D^b\hat{\eta}_t\|_{L^\infty(\hat{B}_{\delta_t},g_X)}&\leq C\|\D^b(\hat{\eta}_t-\hat{\eta}^\dagger_t)\|_{L^\infty(\hat{B}_{\delta_t},\hat{g}_t)}
+\|\D^b\hat{\eta}^\dagger_t\|_{L^\infty(\hat{B}_{\delta_t},g_X)}\\
&\leq C\|\D^b(\hat{\eta}^\diamond_t+\hat{\eta}^\ddagger_t+\hat{\eta}^\circ_t+\hat{\eta}_{t,j,k})\|_{L^\infty(\hat{B}_{\delta_t},\hat{g}_t)}+C\delta_t^2e^{\frac{-1-\alpha}{2}\frac{\alpha}{j+\alpha}t}\\
&=O(1),
\end{split}
\end{equation}
using \eqref{crocefesso}, \eqref{crocifisso8}, \eqref{crocifisso9}, \eqref{crocifisso3ghost} and \eqref{crocifisso4ter} to bound all the pieces. Using this, we can transfer \eqref{fedeliallalinea} to the tilde picture and multiply it by $d_t^{\beta-\alpha}$ to get
\begin{equation}\label{fedeliz}\begin{split}
d_t^{-j-\alpha}[\D^j\ti{\eta}^\diamond_t]_{C^\beta(\ti{B}_{O(\ve_t)},\ti{g}_t)}&=d_t^{-j-\alpha}[\D^j\ti{\gamma}_{t,0}]_{C^\beta(\ti{B}_{O(\ve_t)},\ti{g}_t)}\\
&\leq Cd_t^{-j-\alpha}[\D^j_{\mathbf{b\cdots b}}\ti{\eta}_t]_{C^\beta_{\rm base}(\ti{B}_{O(\ve_t)},\ti{g}_t)}+Cd_t^{\beta-\alpha}e^{-(1-\beta)\frac{t}{2}}\lambda_t^{-\beta}\\
&= Cd_t^{-j-\alpha}[\D^j_{\mathbf{b\cdots b}}\ti{\eta}_t]_{C^\beta_{\rm base}(\ti{B}_{O(\ve_t)},\ti{g}_t)}+o(\ve_t^{\alpha-\beta}),
\end{split}
\end{equation}
and bringing in \eqref{bstrd2} with $a=j$ completes the proof of \eqref{dq1}.

We then insert \eqref{osondesto2} with $a=j$ into \eqref{dq1} and obtain
\begin{equation}\label{osonfesso}
d_t^{-j-\alpha}[\D^j\ti{\eta}^\diamond_t]_{C^\beta(\ti{B}_{O(\ve_t)},\ti{g}_t)}\leq C\ve_t^{\alpha-\beta}.
\end{equation}
Taking this with $\beta>\alpha$ gives an $o(1)$ bound for the $C^\beta$ seminorm of $d_t^{-j-\alpha}\ti{\eta}^\diamond_t$ on the $\ti{g}_t$-geodesic ball centered at $\ti{x}_t$ of radius $2$, and hence an $o(1)$ bound for the $C^\alpha$ seminorm on this same ball, which includes the other blowup point $\ti{x}'_t$ and thus the contribution of $\ti{\eta}^\diamond_t$ to \eqref{zumteufel163} goes to zero.

\subsubsection{Killing the contribution from $\ti{\eta}_{t,j,k}$}\label{jjj3}

Lastly, in order to obtain the final contradiction to \eqref{zumteufel163}, we must show that the contribution of $\ti{\eta}_{t,j,k}$ also goes to zero.
But as above,  \eqref{osondesto2} with $a=j$ implies in particular an $o(1)$ bound for the $C^\alpha$ seminorm of $\ti{\eta}_{t,j,k}+\ti{\eta}^\diamond_t+\ti{\eta}^\circ_t$ on  the $\ti{g}_t$-geodesic ball $\ti{B}_2(\ti{x}_t)$ centered at $\ti{x}_t$ of radius $2$, while \eqref{osonfesso} implies a similar $o(1)$ bound for the $C^\alpha$ seminorm of $\ti{\eta}^\diamond_t$ on $\ti{B}_2(\ti{x}_t)$. Thus, to conclude, it suffices to show that
\begin{equation}\label{kz0}
d_t^{-j-\alpha}[\D^j\ti{\eta}^\circ_t]_{C^\alpha(\ti{B}_2(\ti{x}_t),\ti{g}_t)}=o(1).
\end{equation}
Observe that the ball $\ti{B}_2(\ti{x}_t)$ is contained in what we denote by $\ti{B}_S$ for any fixed $S$ sufficiently large (since by definition this is the product of a Euclidean ball of radius $S$ in the base times $Y$).
Then recall that from Lemma \ref{Gstructure} we have
\begin{equation}\label{kleinearschloch}
\ti{\eta}^\circ_t = i\partial\overline\partial\sum_{i=2}^j\sum_{p=1}^{N_{i,k}}\sum_{\iota=0}^{2k}\sum_{\ell=\lceil \frac{\iota}{2} \rceil}^{k} e^{-\left(\ell-\frac{\iota}{2}\right) t}\ve_t^{\iota}(\ti{\Phi}_{\iota,\ell}(\ti{G}_{i,p,k})\circledast \D^\iota \ti{A}^*_{t,i,p,k}),
\end{equation}
and the bounds from \eqref{coveted}, \eqref{coveted2a}
\begin{equation}\label{sommario}
d_t^{-\iota+2}\|\D^\iota\ti{A}^*_{t,i,p,k}\|_{L^\infty(\ti{B}_{\ve_t},\ti{g}_t)}\leq C\delta_t^{j+2+\alpha-\iota},
\end{equation}
for $0\leq \iota\leq j+2+2k$.  On the other hand taking \eqref{superagognata} with $\beta>\alpha$ gives
\begin{equation}
d_t^{-\iota+2-\alpha}[\D^\iota\ti{A}^*_{t,i,p,k}]_{C^\beta(\ti{B}_{O(\ve_t)},\ti{g}_t)}\leq C\ve_t^{\alpha-\beta}\delta_t^{j+2-\iota}=o(\delta_t^{j+2-\iota}),
\end{equation}
for $j\leq\iota\leq j+2+2k$, and as before this implies that
\begin{equation}\label{eiosottoni}
d_t^{-\iota+2-\alpha}[\D^\iota\ti{A}^*_{t,i,p,k}]_{C^\alpha(\ti{B}_2(\ti{x}_t),\ti{g}_t)}=o(\delta_t^{j+2-\iota}).
\end{equation}
for $j\leq\iota\leq j+2+2k$.
Also, simply taking \eqref{coveted} with $S$ fixed gives
\begin{equation}\label{eiosottoni2}
d_t^{-\iota+2-\alpha}[\D^\iota\ti{A}^*_{t,i,p,k}]_{C^\alpha(\ti{B}_S,\ti{g}_t)}\leq Cd_t^{j-\iota}\delta_t^{2}S^{j-\iota}=C\delta_t^{j+2-\iota}\ve_t^{-j+\iota}S^{j-\iota}=o(\delta_t^{j+2-\iota}),
\end{equation}
for $0\leq \iota<j$. On the other hand, for the functions $\ti{\Phi}:=\ti{\Phi}_{\iota,\ell}(\ti{G}_{i,p,k})$ we have the estimates
\begin{equation}\label{mofo}
\|\D^\iota\ti{\Phi}\|_{L^\infty(\ti{B}_{\ve_t},\ti{g}_t)}\leq C\ve_t^{-\iota},
\end{equation}
\begin{equation}\label{mofo2}
[\D^\iota\ti{\Phi}]_{C^\alpha(\ti{B}_{S},\ti{g}_t)}\leq CS^{1-\alpha}\|\D^{\iota+1}\ti{\Phi}\|_{L^\infty(\ti{B}_{S},\ti{g}_t)}\leq CS^{1-\alpha}\ve_t^{-\iota-1}=o(\ve_t^{-\iota-\alpha}),
\end{equation}
for all $\iota\geq 0$, $0<\alpha<1$ and $S$ fixed, while for the complex structure $\ti{J}^\natural_t=\Theta_t^*\hat{J}^\natural_t$ the bounds \eqref{Q} transform to
\begin{equation}\label{Q!}
\|\D^\iota \ti{J}^\natural_t\|_{L^\infty(\ti{B}_{Rd_t^{-1}},\ti{g}_t)}\leq C\ve_t^{-\iota},
\end{equation}
for $R$ fixed, and so for fixed $S$ we get as in \eqref{mofo2}
\begin{equation}\label{Q!!}
[\D^\iota \ti{J}^\natural_t]_{C^\alpha(\ti{B}_{S},\ti{g}_t)}\leq CS^{1-\alpha}\|\D^{\iota+1}\ti{J}^\natural_t\|_{L^\infty(\ti{B}_{S},\ti{g}_t)}\leq CS^{1-\alpha}\ve_t^{-\iota-1}=o(\ve_t^{-\iota-\alpha}).
\end{equation}
Now we can use all the estimates \eqref{sommario}, \eqref{eiosottoni}, \eqref{eiosottoni2}, \eqref{mofo}, \eqref{mofo2}, \eqref{Q!} and \eqref{Q!!} to prove \eqref{kz0} by arguing like in the proof of \eqref{crocifisso3} as follows: write schematically
\begin{equation}\label{schematic}
\D^{j}\ddbar\left(\ti{\Phi}_{\iota,\ell}(\ti{G}_{i,p,k})\circledast \D^\iota \ti{A}^*_{t,i,p,k}\right)=\sum_{s=0}^{j+1}\sum_{i_1+i_2=s+1}(\D^{j+1-s}\ti{J}^\natural_t)\circledast\D^{i_1}\ti{\Phi}_{\iota,\ell}(\ti{G}_{i,p,k})\circledast \D^{i_2+\iota}\ti{A}^*_{t,i,p,k},
\end{equation}
and estimate
\begin{equation}\label{pet}\begin{split}
&d_t^{-j-\alpha}\ve_t^\iota[(\D^{j+1-s}\ti{J}^\natural_t)\circledast\D^{i_1}\ti{\Phi}_{\iota,\ell}(\ti{G}_{i,p,k})\circledast \D^{i_2+\iota}\ti{A}^*_{t,i,p,k}]_{C^\alpha(\ti{B}_2(\ti{x}_t),\ti{g}_t)}\\
&\leq Cd_t^{-j-\alpha}\ve_t^\iota o(\ve_t^{-j-1+s-i_1-\alpha})d_t^{i_2+\iota-2}\delta_t^{j+2+\alpha-i_2-\iota}+
Cd_t^{-j-\alpha}\ve_t^\iota \ve_t^{-j-1+s-i_1}d_t^{i_2+\iota-2+\alpha}o(\delta_t^{j+2-i_2-\iota})\\
&=o(1),\end{split}
\end{equation}
and combining \eqref{kleinearschloch}, \eqref{schematic} and \eqref{pet} proves \eqref{kz0}. This concludes Subcase A.

\subsection{Subcase B:  $\epsilon_{t} \to 1$ (without loss). }

Thanks to \eqref{utilissimo}, \eqref{crocifisso8} and \eqref{utilissimo2} together with Lemma \ref{arschloch}  (using also Remark \ref{archiloco} in order to compare the mildly varying $C^{j,\alpha}_{\rm loc}$ topologies) we are now able to say that
$d_t^{-j-\alpha}\tilde\eta_{t,j,k}, d_t^{-j-\alpha}\tilde\eta_t^{\diamond}$ and $d_t^{-j-\alpha}\tilde\eta_t^{\circ}$ converge in the topology of $C$\begin{small}$^{j,\beta}_{\rm loc}$\end{small}$(\C^m\times Y)$ for every $\beta<\alpha$ to limiting $2$-forms $\tilde\eta_{\infty,j,k}, \tilde\eta_\infty^\diamond$ and $\tilde\eta_\infty^\circ$ in $C^{j,\alpha}_{\rm loc}(\C^m \times Y)$, which are $O(r^{j+\alpha})$ at infinity, which are weakly closed (as a locally uniform limit of smooth closed forms) and of type $(1,1)$  with respect to $J_{\C^m} + J_{Y,z_\infty}$. Thanks to \eqref{crocifisso9} and \eqref{crocifisso4quater} we may assume that $\tilde{g}_t^\sharp \to g_{\C^m} + g_{Y,z_\infty}$ locally smoothly, where $g_{\C^m}$ is a constant K\"ahler metric on $\C^m$ (which equals $g_{\rm can}(0)$ when $j>0$ by \eqref{crocifisso9}, while for $j=0$ it equals $g_{\rm can}(0)$ plus the subsequential limit of the constant forms $\ti{\eta}^\ddagger_t$, which are uniformly bounded). As in \cite[Proposition 3.11]{HT2}, all of these limiting forms are $\ddbar$-exact on $\C^m\times Y$.

Thanks to \eqref{ottimo_a}, the functions $d_t^{-j-\alpha}\ti{A}^*_{t,i,p,k}$ converge in $C^{j+2+2k,\beta}_{\rm loc}(\C^m\times Y)$ to limiting functions $\ti{A}_{\infty,i,p,k}^*(z)$ from $\C^m$, while $\ti{G}_{i,p,k}$ converge locally smoothly to functions $\ti{G}_{\infty,i,p,k}(y)$ pulled back from $Y$, and since
\begin{equation}
\ti{\eta}^\circ_t=\sum_{i=2}^j\sum_{p=1}^{N_{i,k}}\ddbar\ti{\mathfrak{G}}_{t,k}(\ti{A}^*_{t,i,p,k},\ti{G}_{i,p,k})
\end{equation}
which can be expanded as in \eqref{krktk}, and it follows that
\begin{equation}\label{vanquish}
d_t^{-j-\alpha}\ti{\eta}_t^\circ\to \tilde\eta_\infty^\circ:=
\sum_{i=2}^j\sum_{p=1}^{N_{i,k}}\ddbar\ti{\mathfrak{G}}_{\infty,k}(\ti{A}^*_{\infty,i,p,k},\ti{G}_{\infty,i,p,k}),
\end{equation}
at least in the topology of $C^{0,\beta}_{\rm loc}(\C^m \times Y)$, where
\begin{equation}\label{futilis}\ti{\mathfrak{G}}_{\infty,k}(A,G)=\sum_{\ell=0}^{k}(-1)^\ell (\Delta^{\C^m})^\ell A\cdot (\Delta^Y)^{-\ell-1}G,\end{equation}
thanks to \eqref{inutilis}. In fact we will not use this explicit formula, but only the obvious fact that $\ti{\mathfrak{G}}_{\infty,k}(A,G)$ is bilinear, and in particular it vanishes when $A\equiv 0$.

Now passing to the limit in \eqref{killmenow} implies that the limit $\tr{\omega_{\C^m} + \omega_{Y,z_\infty}}{(\tilde\eta_{\infty,j,k}+\ti{\eta}^\diamond_\infty + \ti{\eta}^{\circ}_\infty)}$ of $d_t^{-j-\alpha}\tr{\ti{\omega}^\sharp_t}(\ti{\eta}_{t,j,k}+\ti{\eta}^\diamond_t+\ti{\eta}^\circ_t)$ is a polynomial in the $z$-variables of degree at most $j$, i.e.
\begin{equation}\label{quamolim}
\tr{\omega_{\C^m} + \omega_{Y,z_\infty}}{(\tilde\eta_{\infty,j,k}+\ti{\eta}^\diamond_\infty + \ti{\eta}^{\circ}_\infty)}=\sum_{|\gamma|\leq j}c_\gamma(y)z^\gamma=:\mathcal{F},
\end{equation}
for some functions $c_\gamma$ on $Y$.
Furthermore, thanks to \eqref{satanas} and \eqref{cacanew4}, we see that the $C^0_{\rm loc}$ limit of
\begin{equation}\label{totalshit}
d_t^{-j-\alpha}\left(c_t \frac{\ti{\omega}_{\rm can}^m\wedge(\ve_t^2\Theta_t^*\Psi_t^*\omega_F)^n}{(\tilde\omega_t^\sharp)^{m+n}} - 1\right),
\end{equation}
exists and is equal to the limit of
$d_t^{-j-\alpha}\tr{\tilde\omega_t^\sharp}{(\tilde\eta_{t,j,k}+\tilde\eta_t^\diamond+\tilde\eta_t^\circ)}$, which is $\mathcal{F}$ from \eqref{quamolim}. We can then apply the Selection Theorem \ref{ghost} (as mentioned earlier, the functions $\hat{A}^\sharp_{t,i,p,k}$ satisfy the bounds \eqref{linftyyy0} with $\alpha_0=\frac{\alpha(1+\alpha)}{j+\alpha}$ thanks to \eqref{prechecazzo}) and see that $\mathcal{F}$ is also equal to the limit of
\begin{equation}
\delta_t^{-j-\alpha}\Sigma_t^*\left(f_{t,0}+\sum_{i=2}^{j}\sum_{p=1}^{N_{i,k}} f_{t,i,p}G_{i,p,k}\right),
\end{equation}
as in \eqref{stronza}. Thus, for any function $G$ on $\C^m\times Y$ which is fiberwise $L^2$ orthogonal to the span of the functions $\{\ti{G}_{\infty,i,p,k}\}_{2\leq i\leq j,1\leq p\leq N_{i,k}}$ together with the constants, and for any $z\in\C^m$ we have
\begin{equation}
\int_{\{z\}\times Y}\mathcal{F}(z,y) G(z,y) \omega_{Y,z_\infty}^n(y)=0,
\end{equation}
which implies that we can write
\begin{equation}
\mathcal{F}(z,y)=g_{0}(z)+\sum_{i=2}^j\sum_{p=1}^{N_{i,k}}g_{i,p,k}(z)\ti{G}_{\infty,i,p,k}(y),
\end{equation}
for some functions $g_0,g_{i,p,k}$ on $\C^m$.
Since $\mathcal{F}$ is a polynomial in $z$ of degree at most $j$, by fiberwise $L^2$ projecting $\mathcal{F}$ onto each $\ti{G}_{\infty,i,p,k}$ and onto the constants we see that the coefficients $g_0(z),g_{i,p,k}(z)$ are also polynomials of degree at most $j$. Writing $g_0(z)=\sum_{|\gamma|\leq j}g_{0,\gamma}z^\gamma$ and $g_{i,p,k}(z)=\sum_{|\gamma|\leq j}g_{i,p,k,\gamma} z^\gamma$ (with $g_{0,\gamma},g_{i,p,k,\gamma}\in\mathbb{C}$), we see that
\begin{equation}\label{quamolim2}
\mathcal{F}(z,y)=\sum_{|\gamma|\leq j}\left(g_{0,\gamma}+\sum_{i=2}^j\sum_{p=1}^{N_{i,k}}g_{i,p,k,\gamma}\ti{G}_{\infty,i,p,k}(y)\right) z^\alpha,
\end{equation}
i.e. $\mathcal{F}$ is a linear combination of the functions $\ti{G}_{\infty,i,p,k}$ together with the constant $1$, with coefficients that are polynomials in $z$ of degree at most $j$. For convenience, we can rewrite this as
\begin{equation}\label{quamolim3}
\tr{\omega_{\C^m} + \omega_{Y,z_\infty}}{(\tilde\eta_{\infty,j,k}+\ti{\eta}^\diamond_\infty + \ti{\eta}^{\circ}_\infty)}=K_0(z)+\sum_q K_q(z)H_q(y),\end{equation}
where $K_0(z),K_q(z)$ are polynomials of degree at most $j$, and $H_q(y)$ are functions pulled back from the fiber $Y$ that lie in the span of the functions $\ti{G}_{\infty,i,p,k}, 2\leq i\leq j, 1\leq p\leq N_{i,k}$.

Let us now go back to the definition in \eqref{sechs}
\begin{equation}\label{tamiucunde}
d_t^{-j-\alpha}\ti{A}_{t,i,p,k}=d_t^{-j-\alpha}\ti{P}_{t,i,p,k}(\ti{\eta}_{t,i-1,k})
\end{equation}
where recall that
\begin{equation}
\tilde{P}_{t,i,p,k}(\alpha) = n({\rm pr}_B)_*(\ti{G}_{i,p,k}\alpha\wedge\Theta_t^*\Psi_t^*\omega_F^{n-1}) + \ve_t^2{\rm tr}^{\ti{\omega}_{\rm can}}({\rm pr}_B)_*(\ti{G}_{i,p,k}\alpha\wedge \Theta_t^*\Psi_t^*\omega_F^n).
\end{equation}
By definition, we can write
\begin{equation}\ti{\eta}_{t,i-1,k}=
\tilde\eta_{t,j,k}+\ti{\eta}^\diamond_t + \ti{\eta}^\circ_t+\ti{\eta}^\dagger_t+\ti{\eta}^\ddagger_t-\ti{\gamma}_{t,0}-\sum_{q=2}^{i-1}\ti{\gamma}_{t,q,k},\end{equation}
and since $\ti{P}_{t,i,p,k}$ annihilates any $(1,1)$ form from the base by \eqref{frombase}, we have
\begin{equation}\label{drek}
d_t^{-j-\alpha}\ti{A}_{t,i,p,k}=d_t^{-j-\alpha}\ti{P}_{t,i,p,k}(\tilde\eta_{t,j,k}+\ti{\eta}^\diamond_t + \ti{\eta}^\circ_t)
+d_t^{-j-\alpha}\ti{P}_{t,i,p,k}(\ti{\eta}^\dagger_t)-d_t^{-j-\alpha}\sum_{q=2}^{i-1}\ti{P}_{t,i,p,k}(\ti{\gamma}_{t,q,k}).\end{equation}
Now recall that
\begin{equation}\ti{\eta}^\dagger_t=\sum_{a=2}^j\sum_{b=1}^{N_{a,k}}\ddbar\ti{\mathfrak{G}}_{t,k}(\ti{A}^\sharp_{t,a,b,k},\ti{G}_{a,b,k}),\end{equation}
\begin{equation}\ti{\gamma}_{t,q,k}=\sum_{c=1}^{N_{q,k}}\ddbar\ti{\mathfrak{G}}_{t,k}(\ti{A}_{t,q,c,k},\ti{G}_{q,c,k}),\end{equation}
we can employ Lemma \ref{paraminatrix} (transferred to the tilde picture) and get
\begin{equation}\label{drek2}\begin{split}
\sum_{a=2}^j\sum_{b=1}^{N_{a,k}}\ti{P}_{t,i,p,k}(\ddbar(\ti{\mathfrak{G}}_{t,k}(\ti{A}^\sharp_{t,a,b,k},\ti{G}_{a,b,k}))) &= \sum_{a=2}^j\sum_{b=1}^{N_{a,k}}\ti{A}^\sharp_{t,a,b,k}\int_{\{z\}\times Y} \ti{G}_{i,p,k}\ti{G}_{a,b,k}\,\Theta_t^*\Psi_t^*\omega_F^n \\
&+  \sum_{a=2}^j\sum_{b=1}^{N_{a,k}}\sum_{\iota=0}^{j} e^{-(2k+2-\iota)\frac{t}{2}} \ti{\Phi}_{\iota,i,p,k}(\ti{G}_{a,b,k}) \circledast \D^{\iota}\ti{A}^\sharp_{t,a,b,k}\\
&=\ti{A}^\sharp_{t,i,p,k} +  \sum_{a=2}^j\sum_{b=1}^{N_{a,k}}\sum_{\iota=0}^{j}  e^{-(2k+2-\iota)\frac{t}{2}}\ti{\Phi}_{\iota,i,p,k}(\ti{G}_{a,b,k}) \circledast \D^{\iota}\ti{A}^\sharp_{t,a,b,k},
\end{split}\end{equation}
using here crucially that $\ti{A}^\sharp_{t,a,b,k}$ is a polynomial of degree at most $j$, and that the $\ti{G}_{i,p,k}$ are fiberwise orthonormal.
From \eqref{checazzo} we can in particular crudely bound
\begin{equation}\label{drek3}d_t^{-j-\alpha}e^{-(2k+2-\iota)\frac{t}{2}}\left\|\ti{\Phi}_{\iota,i,p,k}(\ti{G}_{a,b,k}) \circledast \D^{\iota}\ti{A}^\sharp_{t,a,b,k}\right\|_{L^\infty(\ti{B}_S,\ti{g}_t)}\leq
C d_t^{-j-\alpha}e^{-(2k+2-j)\frac{t}{2}}=o(1),\end{equation}
for any fixed $S$, since by assumption $j\leq k$. On the other hand, since $q<i$, we have
\begin{equation}\begin{split}
\sum_{c=1}^{N_{q,k}}\ti{P}_{t,i,p,k}(\ddbar(\ti{\mathfrak{G}}_{t,k}(\ti{A}_{t,q,c,k},\ti{G}_{q,c,k}))) &=  \sum_{c=1}^{N_{q,k}}\sum_{\iota=0}^{2k+2} e^{-(2k+2-\iota)\frac{t}{2}}\ti{\Phi}_{\iota,i,p,k}(\ti{G}_{q,c,k}) \circledast \D^{\iota}\ti{A}_{t,q,c,k}\\
&=\sum_{c=1}^{N_{q,k}}\sum_{\iota=0}^{j} e^{-(2k+2-\iota)\frac{t}{2}}\ti{\Phi}_{\iota,i,p,k}(\ti{G}_{q,c,k}) \circledast \D^{\iota}\ti{A}^\sharp_{t,q,c,k}\\
&+\sum_{c=1}^{N_{q,k}}\sum_{\iota=0}^{2k+2} e^{-(2k+2-\iota)\frac{t}{2}}\ti{\Phi}_{\iota,i,p,k}(\ti{G}_{q,c,k}) \circledast \D^{\iota}\ti{A}^*_{t,q,c,k},
\end{split}\end{equation}
and we can argue exactly as above for the terms with $\ti{A}^\sharp_{t,q,c,k}$, while for the terms with $\ti{A}^*_{t,q,c,k}$ we use \eqref{ottimo_a} which gives $d_t^{-j-\alpha}\|\D^\iota\ti{A}^*_{t,i,p,k}\|_{L^\infty(\ti{B}_S,\ti{g}_t)}\leq C$ for all $0\leq \iota\leq 2+j+2k$ and fixed $S$, and we see that
\begin{equation}\label{drek4}d_t^{-j-\alpha}\sum_{c=1}^{N_{q,k}}\ti{P}_{t,i,p,k}(\ddbar(\ti{\mathfrak{G}}_{t,k}(\ti{A}_{t,q,c,k},\ti{G}_{q,c,k}))) =
d_t^{-j-\alpha}\sum_{c=1}^{N_{q,k}} \ti{\Phi}_{2k+2,i,p,k}(\ti{G}_{q,c,k}) \circledast \D^{2k+2}\ti{A}^*_{t,q,c,k}+o(1),\end{equation}
locally uniformly, and combining \eqref{drek}, \eqref{drek2}, \eqref{drek3} and \eqref{drek4} we obtain
\begin{equation}\label{apery}\begin{split}
d_t^{-j-\alpha}\ti{A}_{t,i,p,k}&=d_t^{-j-\alpha}\ti{P}_{t,i,p,k}(\tilde\eta_{t,j,k}+\ti{\eta}^\diamond_t + \ti{\eta}^\circ_t)
+d_t^{-j-\alpha}\ti{A}^\sharp_{t,i,p,k}\\
&-d_t^{-j-\alpha}\sum_{q=2}^{i-1}\sum_{c=1}^{N_{q,k}} \ti{\Phi}_{2k+2,i,p,k}(\ti{G}_{q,c,k}) \circledast \D^{2k+2}\ti{A}^*_{t,q,c,k}+o(1).\end{split}\end{equation}

We now want to pass \eqref{apery} to the limit as $t\to\infty$ (in the $C^0_{\rm loc}$ topology say). As mentioned earlier, we have
$d_t^{-j-\alpha}\ti{A}^*_{t,i,p,k}\to \ti{A}^*_{\infty,i,p,k}$, $d_t^{-j-\alpha}\ti{\eta}_t^\circ \to \ti{\eta}_\infty^\circ, d_t^{-j-\alpha}\ti{\eta}_t^\diamond \to \ti{\eta}_\infty^\diamond, d_t^{-j-\alpha}\ti{\eta}_{t,j,k}\to \tilde\eta_{\infty,j,k}$ and $\ti{G}_{i,p,k}\to\ti{G}_{\infty,i,p,k}$, so that $d_t^{-j-\alpha}\ti{P}_{t,i,p,k}(\tilde\eta_{t,j,k}+\ti{\eta}^\diamond_t + \ti{\eta}^\circ_t)$ converges to
\begin{equation}\begin{split}
&\int_{\{z_\infty\}\times Y}\ti{G}_{\infty,i,p,k}
\tr{\omega_{Y,z_\infty}}{(\tilde\eta_{\infty,j,k}+\ti{\eta}^\diamond_\infty+\ti{\eta}_\infty^\circ)}\omega_{Y,z_\infty}^n+
{\rm tr}^{\omega_{\C^m}} ({\rm pr}_B)_*(\ti{G}_{\infty,i,p,k}(\tilde\eta_{\infty,j,k}+\ti{\eta}^\diamond_\infty+\ti{\eta}_\infty^\circ) \wedge \omega_{Y,z_\infty}^n)\\
&=\int_{\{z_\infty\}\times Y}\ti{G}_{\infty,i,p,k}
\tr{\omega_{Y,z_\infty}}{(\tilde\eta_{\infty,j,k}+\ti{\eta}^\diamond_\infty+\ti{\eta}_\infty^\circ)}\omega_{Y,z_\infty}^n+
\int_{\{z_\infty\}\times Y}\ti{G}_{\infty,i,p,k}{\rm tr}^{\omega_{\C^m}}(\tilde\eta_{\infty,j,k}+\ti{\eta}^\diamond_\infty+\ti{\eta}_\infty^\circ)\omega_{Y,z_\infty}^n\\
&=\int_{\{z_\infty\}\times Y}\ti{G}_{\infty,i,p,k}
\tr{\omega_{\C^m}+\omega_{Y,z_\infty}}{(\tilde\eta_{\infty,j,k}+\ti{\eta}^\diamond_\infty+\ti{\eta}_\infty^\circ)}\omega_{Y,z_\infty}^n,
\end{split}
\end{equation}
and so \eqref{apery} limits to
\begin{equation}\label{zetaof5}\begin{split}
\ti{A}^*_{\infty,i,p,k}&=\int_{\{z_\infty\}\times Y}\ti{G}_{\infty,i,p,k}
\tr{\omega_{\C^m}+\omega_{Y,z_\infty}}{(\tilde\eta_{\infty,j,k}+\ti{\eta}^\diamond_\infty+\ti{\eta}_\infty^\circ)}\omega_{Y,z_\infty}^n\\
&-\sum_{q=2}^{i-1}\sum_{c=1}^{N_{q,k}} \ti{\Phi}_{2k+2,i,p,k}(\ti{G}_{\infty,q,c,k}) \circledast \D^{2k+2}\ti{A}^*_{\infty,q,c,k}.
\end{split}
\end{equation}
We can then plug in \eqref{quamolim3} and we see that
\begin{equation}\label{aperi}
\ti{A}^*_{\infty,i,p,k}=Q_{i,p}(z)-\sum_{q=2}^{i-1}\sum_{c=1}^{N_{q,k}} \ti{\Phi}_{2k+2,i,p,k}(\ti{G}_{\infty,q,c,k}) \circledast \D^{2k+2}\ti{A}^*_{\infty,q,c,k},
\end{equation}
where $Q_{i,p}(z)$ is a polynomial on $\C^m$ of degree at most $j$, and using this we can show by induction on $i$ that $\ti{A}^*_{\infty,i,p,k}=0$ for all $i,p$. Indeed, in the base case of the induction $i=2$ the last term in \eqref{aperi} is not present, and so $\ti{A}^*_{\infty,2,p,k}$ is a polynomial of degree at most $j$, but since it also has vanishing $j$-jet at origin (recall that we have translated the $\mathbb{C}^m$ factor so that $\ti{x}_t=(0,\ti{y}_t)$), it must be identically zero. The induction step is then exactly the same.
Recalling \eqref{futilis} (or simply the remark after it), it then follows that $\tilde{\eta}_\infty^\circ = 0$.

It also follows that the contribution
\begin{equation}d_t^{-j-\alpha}\sum_{i=2}^j\sum_{p=1}^{N_{i,k}}\sum_{\iota = -2}^{2k}\ve_t^\iota\frac{|\D^{j+2+\iota}\ti{A}^*_{t,i,p,k}(\ti{x}_t) -  \P_{\ti{x}'_t\ti{x}_t}(\D^{j+2+\iota}\ti{A}^*_{t,i,p,k}(\ti{x}'_t))|_{\ti{g}_t(\ti{x}_t)}}{d^{\ti{g}_t}(\ti{x}_t,\ti{x}'_t)^\alpha}\end{equation}
to \eqref{zumteufel163} goes to zero.

Going back to \eqref{zetaof5} this means that
\begin{equation}\label{zetaof9}
\int_{\{z_\infty\}\times Y}\ti{G}_{\infty,i,p,k}
\tr{\omega_{\C^m}+\omega_{Y,z_\infty}}{(\tilde\eta_{\infty,j,k}+\ti{\eta}^\diamond_\infty)}\omega_{Y,z_\infty}^n=0.
\end{equation}

Next, as mentioned earlier we can write $\tilde\eta_{\infty,j,k} = i\partial\overline\partial \tilde\varphi_\infty$ with $\tilde\varphi_\infty = O(r^{j+2+\alpha})$ and $\underline{\tilde\varphi_\infty} = 0$, and $\ti{\eta}^\diamond_\infty=\ddbar \ti{A}_{\infty,0}$ where $\ti{A}_{\infty,0}$ is a function from the base which is $O(r^{j+2+\alpha})$.

We go back to \eqref{quamolim3}, which we can write as
\begin{equation}\tr{\omega_{\C^m} + \omega_{Y,z_\infty}}{(\tilde\eta_{\infty,j,k}+\ti{\eta}^\diamond_\infty)}=\Delta^{\omega_{\C^m} + \omega_{Y,z_\infty}}(\ti{\vp}_\infty+\ti{A}_{\infty,0})=K_0(z)+\sum_q K_q(z)H_q(y),\end{equation}
and letting $(\Delta^{\C^m})^{-1}K_0(z)$ denote any degree $j+2$ polynomial on $\C^m$ with Laplacian equal to $K_0(z)$, which clearly exists, we can write \begin{equation}\label{promisisti}
\Delta^{\omega_{\C^m} + \omega_{Y,z_\infty}}(\ti{\vp}_\infty+\ti{A}_{\infty,0}-(\Delta^{\C^m})^{-1}K_0)
=\sum_j K_j(z)H_j(y).
\end{equation}
We thus see that the function on the LHS belongs to the fiberwise span of the $\ti{G}_{\infty,i,p,k}$'s, and since by \eqref{zetaof9} it is also fiberwise $L^2$ orthogonal to such span, we conclude that
\begin{equation}\Delta^{\omega_{\C^m} + \omega_{Y,z_\infty}}(\ti{\vp}_\infty+\ti{A}_{\infty,0}-(\Delta^{\C^m})^{-1}K_0)=0.\end{equation}
So the function $\ti{\vp}_\infty+\ti{A}_{\infty,0}-(\Delta^{\C^m})^{-1}K_0$ is actually a harmonic function on $\C^m \times Y$, $O(r^{j+2+\alpha})$ by construction.
Thanks to \cite[Proposition 3.12]{HT2} this implies that $\ti{\vp}_\infty+\ti{A}_{\infty,0}-(\Delta^{\C^m})^{-1}K_0$ is a harmonic polynomial of degree at most $j+2$ on $\C^m$. Absorbing this polynomial into $(\Delta^{\C^m})^{-1}K_0(z)$ (which by construction was only unique modulo harmonic polynomials anyway) we obtain the identity
\begin{equation}\label{etseminieius}
\tilde{\varphi}_\infty+ \tilde{A}_{\infty,0}(z)= (\Delta^{\C^m})^{-1}K_0(z).
\end{equation}
Taking the fiber average of this identity immediately tells us that $\ti{A}_{\infty,0}$ is a polynomial of degree at most $j+2$.
This implies that the contribution
\begin{equation}d_t^{-j-\alpha}\frac{|\D^j\ti{\eta}^\diamond_t(\ti{x}_t) -  \P_{\ti{x}'_t\ti{x}_t}(\D^j\ti{\eta}^\diamond_t(\ti{x}'_t))|_{\ti{g}_t(\ti{x}_t)}}{d^{\ti{g}_t}(\ti{x}_t,\ti{x}'_t)^\alpha}\end{equation}
to \eqref{zumteufel163} also goes to zero.

Then going back to \eqref{etseminieius} we see also that $\tilde{\varphi}_\infty$ is equal to the pullback of a polynomial from the base, but since it also has fiberwise average zero, we conclude that $\tilde{\varphi}_\infty=0$, and hence $\ti{\eta}_{\infty,j,k}=0.$ This implies that the contribution
\begin{equation}d_t^{-j-\alpha}\frac{|\D^j\ti{\eta}_{t,j,k}(\ti{x}_t) -  \P_{\ti{x}'_t\ti{x}_t}(\D^j\ti{\eta}_{t,j,k}(\ti{x}'_t))|_{\ti{g}_t(\ti{x}_t)}}{d^{\ti{g}_t}(\ti{x}_t,\ti{x}'_t)^\alpha}\end{equation}
to \eqref{zumteufel163} also goes to zero, thus giving a contradiction.

\subsection{Subcase C: $\epsilon_t \to 0$. }

It follows from \eqref{utilissimo}, \eqref{crocifisso8}, using also Remark \ref{archiloco}, that $d_t^{-j-\alpha}\tilde\eta_{t,j,k}$ and $d_t^{-j-\alpha}\tilde\eta_t^{\diamond}$
converge in the topology of $C$\begin{small}$^{j,\beta}_{\rm loc}$\end{small}$(\C^m\times Y)$ for every $\beta<\alpha$ to limiting $2$-forms $\tilde\eta_{\infty,j,k}$ and $\tilde\eta_\infty^\diamond$ in $C^{j,\alpha}_{\rm loc}(\C^m \times Y)$, which are weakly closed (as a locally uniform limit of smooth closed forms) and of type $(1,1)$  with respect to $J_{\C^m} + J_{Y,z_\infty}$. We have $\tilde{g}_t^\sharp \to g_{\C^m}$ locally smoothly thanks to \eqref{crocifisso9}, \eqref{crocifisso4quater} and the fact that clearly $\ve_t^2\Theta_t^*\Psi_t^*\omega_F\to 0$ locally smoothly, where $g_{\C^m}$ is a constant K\"ahler metric on $\C^m$ (which as in subcase B equals $g_{\rm can}(0)$ when $j>0$, while for $j=0$ it equals $g_{\rm can}(0)$ plus the subsequential limit of the constant forms $\ti{\eta}^\ddagger_t$).
Similarly, we can pass $d_t^{-j-\alpha}\ti{\eta}^\circ_t$ to a limit $\ti{\eta}^\circ_\infty$ since \eqref{crocifisso3ter} give us a uniform $C^{j,\alpha}_{\rm loc}$ bound (with respect to a fixed metric rather than $\ti{g}_t$).

Since $\ti{\eta}^\diamond_t$ is the pullback of a form from the base, the same is true for $\ti{\eta}^\diamond_\infty$, which by \eqref{whoknows113} is $O(r^{j+\alpha})$ at infinity, and it is also $\de\db$-exact since it is weakly closed on $\C^m$.
From \eqref{zumteufel162} and \eqref{zumteufel163} we see that
\begin{equation}\label{topogigio2}\begin{split}
&d_t^{-j-\alpha}\Bigg(\sum_{i=2}^j\sum_{p=1}^{N_{i,k}}\sum_{\iota = -2}^{2k}\ve_t^{\iota}|\D^{j+2+\iota}\ti{A}^*_{t,i,p,k}(\ti{x}_t) -  \P_{\ti{x}'_t\ti{x}_t}(\D^{j+2+\iota}\ti{A}^*_{t,i,p,k}(\ti{x}'_t))|_{\ti{g}_t(\ti{x}_t)}\\
&+|\D^j\ti{\eta}^\diamond_t(\ti{x}_t) -  \P_{\ti{x}'_t\ti{x}_t}(\D^j\ti{\eta}^\diamond_t(\ti{x}'_t))|_{\ti{g}_t(\ti{x}_t)}
+|\D^j\ti{\eta}_{t,j,k}(\ti{x}_t) -  \P_{\ti{x}'_t\ti{x}_t}(\D^j\ti{\eta}_{t,j,k}(\ti{x}'_t))|_{\ti{g}_t(\ti{x}_t)}\Bigg)
=1.\end{split}
\end{equation}
However, recall that thanks to \eqref{utilissimo} we have for any fixed $R$
\begin{equation}d_t^{-j-\alpha}\|\D^\iota\ti{\eta}_{t,j,k}\|_{L^\infty(\ti{B}_{Rd_t^{-1}},\ti{g}_t)}\to 0,\quad 0\leq \iota\leq j,\end{equation}
and so the contribution of $d_t^{-j-\alpha}\ti{\eta}_{t,j,k}$ to \eqref{topogigio2} is negligible. Since $\ti{\omega}^\sharp_t\to \omega_{\C^m}$ locally smoothly, it also follows that the contribution of $d_t^{-j-\alpha}\ti{\eta}_{t,j,k}$ to \eqref{killmenow} is negligible too, which implies that
\begin{equation}\label{bubbi}
d_t^{-j-\alpha}[\D^j_{\mathbf{b\cdots b}} (\mathrm{tr}^{\ti{\omega}^\sharp_t}(\ti{\eta}^\circ_t+\ti{\eta}^\diamond_t))]_{C^\alpha_{\rm base}(\ti{B}_R,\ti{g}_t)}=o(1),
\end{equation}
for all $R$ fixed.

Furthermore, thanks to \eqref{ottimo}, the contributions of all the terms of the sum in \eqref{topogigio2} with $\iota>-2$ also go to zero, and so we get
\begin{equation}\label{topogigio3}\begin{split}
&d_t^{-j-\alpha}\Bigg(\sum_{i=2}^j\sum_{p=1}^{N_{i,k}}\ve_t^{-2}|\D^j\ti{A}^*_{t,i,p,k}(\ti{x}_t) -  \P_{\ti{x}'_t\ti{x}_t}(\D^j\ti{A}^*_{t,i,p,k}(\ti{x}'_t))|_{\ti{g}_t(\ti{x}_t)}
+|\D^j\ti{\eta}^\diamond_t(\ti{x}_t) -  \P_{\ti{x}'_t\ti{x}_t}(\D^j\ti{\eta}^\diamond_t(\ti{x}'_t))|_{\ti{g}_t(\ti{x}_t)}\Bigg)\\
&=1+o(1).
\end{split}
\end{equation}
We can write
\begin{equation}\label{extremus}d_t^{-j-\alpha}\D^j_{\mathbf{b\cdots b}}(\mathrm{tr}^{\ti{\omega}^\sharp_t}(\ti{\eta}^\circ_t+\ti{\eta}^\diamond_t))=
d_t^{-j-\alpha}\D^j_{\mathbf{b\cdots b}}(\mathrm{tr}^{\ti{\omega}^\sharp_t}\ti{\eta}^\diamond_t)+
(m+n)d_t^{-j-\alpha}\D^j_{\mathbf{b\cdots b}}\frac{\ti{\eta}^\circ_t\wedge(\ti{\omega}^\sharp_t)^{m+n-1}}{(\ti{\omega}^\sharp_t)^{m+n}}.\end{equation}
For the first term we claim that
\begin{equation}
d_t^{-j-\alpha}\D^j_{\mathbf{b\cdots b}}(\mathrm{tr}^{\ti{\omega}^\sharp_t}\ti{\eta}^\diamond_t)=
d_t^{-j-\alpha}\D^j_{\mathbf{b\cdots b}}(\mathrm{tr}^{{\omega}_{\C^m}}\ti{\eta}^\diamond_t)+o(1),
\end{equation}
uniformly on $\ti{B}_R$, which follows by observing that $\mathrm{tr}^{\ti{\omega}^\sharp_t}\ti{\eta}^\diamond_t=\mathrm{tr}^{(\ti{\omega}^\sharp_t)_{\mathbf{bb}}}\ti{\eta}^\diamond_t$ since $\ti{\eta}^\diamond_t$ is pulled back from the base, and using the bounds \eqref{crocifisso8} for $d_t^{-j-\alpha}\ti{\eta}^\diamond_t$ together with the aforementioned fact that $(\ti{\omega}^\sharp_t)_{\mathbf{bb}}-\omega_{\C^m}\to 0$ locally smoothly. Importantly, $\D^j_{\mathbf{b\cdots b}}(\mathrm{tr}^{\omega_{\C^m}}\ti{\eta}^\diamond_t)$ is the pullback of a function from the base.
Ignoring combinatorial constants, we then schematically expand \eqref{extremus} as
\begin{equation}\label{extrema}\begin{split}
d_t^{-j-\alpha}\D^j(\mathrm{tr}^{\omega_{\C^m}}\ti{\eta}^\diamond_t)+
d_t^{-j-\alpha}\sum_{p+q+r=j}(\D^p_{\mathbf{b\cdots b}}\ti{\eta}^\circ_t)\D^q_{\mathbf{b\cdots b}}((\ti{\omega}^\sharp_t)^{m+n-1})\D^r_{\mathbf{b\cdots b}}(((\ti{\omega}^\sharp_t)^{m+n})^{-1})+o(1),
\end{split}\end{equation}
and from
\eqref{crocifisso3bis}, \eqref{crocifisso5bis} and \eqref{kombinacja4} we see that all terms in the sum in \eqref{extrema} are uniformly $o(1)$ except when $p=j,q=r=0$, and so \eqref{extrema} equals
\begin{equation}d_t^{-j-\alpha}\left(\D^j (\mathrm{tr}^{\omega_{\C^m}}\ti{\eta}^\diamond_t)+(m+n)\frac{(\D^j_{\mathbf{b\cdots b}}\ti{\eta}^\circ_t)\wedge(\ti{\omega}^\sharp_t)^{m+n-1}}{(\ti{\omega}^\sharp_t)^{m+n}}\right)+o(1),\end{equation}
uniformly on $\ti{B}_R$. For the second term, we use \eqref{gimme} which gives
\begin{equation}d_t^{-j-\alpha}\D^j_{\mathbf{b\cdots b}}\ti{\eta}^\circ_t
=d_t^{-j-\alpha}\sum_{i=2}^j\sum_{p=1}^{N_{i,k}}(\ddbar(\Delta^{\Theta_t^*\Psi_t^*\omega_F|_{\{\cdot\}\times Y}})^{-1}\ti{G}_{i,p,k})_{\mathbf{ff}} \D^j\ti{A}^*_{t,i,p,k}+o(1),\end{equation}
uniformly on $\ti{B}_R$, and so
\begin{equation}\label{excellent}\begin{split}
&d_t^{-j-\alpha}\D^j_{\mathbf{b\cdots b}}(\mathrm{tr}^{\ti{\omega}^\sharp_t}(\ti{\eta}^\circ_t+\ti{\eta}^\diamond_t))\\
&=d_t^{-j-\alpha}\left(\D^j(\mathrm{tr}^{\omega_{\C^m}}\ti{\eta}^\diamond_t)+\sum_{i=2}^j\sum_{p=1}^{N_{i,k}}(\D^j\ti{A}^*_{t,i,p,k})\mathrm{tr}^{\ti{\omega}^\sharp_t}(\ddbar(\Delta^{\Theta_t^*\Psi_t^*\omega_F|_{\{\cdot\}\times Y}})^{-1}\ti{G}_{i,p,k})_{\mathbf{ff}} \right)+o(1)\\
&=d_t^{-j-\alpha}\left(\D^j (\mathrm{tr}^{\omega_{\C^m}}\ti{\eta}^\diamond_t)+\ve_t^{-2}\sum_{i=2}^j\sum_{p=1}^{N_{i,k}}(\D^j\ti{A}^*_{t,i,p,k})\ti{G}_{i,p,k}\right)+o(1),
\end{split}\end{equation}
uniformly on $\ti{B}_R$. Here we used that $d_t^{-j-\alpha}\|\D^j\ti{A}^*_{t,i,p,k}\|_{L^\infty(\ti{B}_R,\ti{g}_t)}\leq C\ve_t^2$ by \eqref{ottimo} and $\|\ti{\eta}^\dagger_t\|_{L^\infty(\ti{B}_R,\ti{g}_t)}=o(1)$ by \eqref{crocifisso4}, so that we can exchange $\mathrm{tr}^{\ti{\omega}^\sharp_t}(\ddbar)_{\mathbf{ff}}$ with $\mathrm{tr}^{\ve_t^2\Theta_t^*\Psi_t^*\omega_F}(\ddbar)_{\mathbf{ff}}$ with only an $o(1)$ error.

Now \eqref{bubbi} says that the quantity in \eqref{excellent} is asymptotically independent of the base directions on any fixed ball $\ti{B}_R$. Taking the fiberwise average of \eqref{excellent} thus shows that $d_t^{-j-\alpha}\D^j  (\mathrm{tr}^{\omega_{\C^m}}\ti{\eta}^\diamond_t)$ is approaching a constant locally uniformly, and in the limit we obtain
\begin{equation}\label{fine}
\nabla^{j,g_\C^m}\mathrm{tr}^{g_{\C^m}}\ti{\eta}^\diamond_\infty=({\rm const}).
\end{equation}
On the other hand, taking the fiberwise $L^2$ inner product of \eqref{excellent} with each of the $\ti{G}_{\ell,q,k}$'s (which are themselves becoming asymptotically constant in the base directions) shows that $\ve_t^{-2}d_t^{-j-\alpha}\D^j \ti{A}^*_{t,i,p,k}$ is also approaching a constant locally uniformly (for all $2\leq i\leq j, 1\leq p\leq N_{i,k}$), so in particular
\begin{equation}\label{bubbi2}
\sum_{i=2}^j\sum_{p=1}^{N_{i,k}}\ve_t^{-2}d_t^{-j-\alpha}|\D^j\ti{A}^*_{t,i,p,k}(\ti{x}_t) -  \P_{\ti{x}'_t\ti{x}_t}(\D^j\ti{A}^*_{t,i,p,k}(\ti{x}'_t))|_{\ti{g}_t(\ti{x}_t)}\to 0,
\end{equation}
so the contribution from $\D^j\ti{A}^*_{t,i,p,k}$ to \eqref{topogigio3} goes to zero too, which shows that $\ti{\eta}^\diamond_\infty$ is not annihilated by $[\nabla^{j,g_\C^m} \cdot]_{C^\alpha}$.

Thus $\ti{\eta}^\diamond_\infty$ satisfies $|\nabla^{j,g_{\C^m}}\ti{\eta}^\diamond_\infty|=O(r^{\alpha})$, it has a global $\de\db$-potential of class $C^{j+2,\alpha}_{\rm loc}(\C^m)$, and satisfies \eqref{fine}, hence it can be written as $\ti{\eta}^\diamond_\infty=\ddbar\vp$ for some smooth function $\vp$ on $\C^m$ with
\begin{equation}
\nabla^{j,g_\C^m}\Delta^{g_{\C^m}}\vp=({\rm const}).
\end{equation}
It follows that $\vp=\ell+h$ where $\ell$ is a real polynomial of degree $\leq j+2$ on $\C^m$ and $h$ is a harmonic function on $\C^m$ with $|\ddbar h|=O(r^{j+\alpha}),$ and Liouville's Theorem shows that the coefficients of $\ddbar h$ are polynomials of degree at most $j$, thus contradicting the
fact that $\ti{\eta}^\diamond_\infty$ is not annihilated by $[\nabla^{j,g_\C^m} \cdot]_{C^\alpha}$.

This finally concludes the proof of \eqref{noncolla}, and hence of Theorem \ref{shutupandsuffer}.

\section{Proof of the main theorems}\label{sectmain}

For the sake of brevity, in this section all norms and seminorms will be taken on an arbitrary ball $B'\Subset B$ (or on $B'\times Y$ for tensors on the total space) which we allow to shrink slightly whenever interpolation is used, and the generic uniform constant $C$ is allowed to depend on $d(B', \de B)$.

\begin{proof}[Proof of Theorem \ref{mthmA}]
Theorem \ref{mthmA} follows quite easily from Theorem \ref{shutupandsuffer}, as follows.  Recall that on $X$ we have the Ricci-flat K\"ahler metrics $\omega^\bullet_t=f^*\omega_B+e^{-t}\omega_X+\ddbar\vp_t$, which satisfy \eqref{ma_initial}
\begin{equation}\label{maalt}
(\omega^\bullet_t)^{m+n}=\ti{c}_t e^{-nt}\omega_X^{m+n}, \quad \ti{c}_t=e^{nt}\frac{\int_X(f^*\omega_B+e^{-t}\omega_X)^{m+n}}{\int_X\omega_X^{m+n}}, \quad
\sup_X\vp_t=0.
\end{equation}
On $X\setminus S$ we let $\omega_F=\omega_X+\ddbar\rho$ be the semi-Ricci-flat form defined in \eqref{pstwlrllk}, and we define a smooth function $G$ on $X\setminus S$ by
\begin{equation}
e^G=\frac{\omega_X^{m+n}}{f^*\omega_{B}^m\wedge\omega_F^n}.
\end{equation}
One then easily checks \cite{ST,To} that $G$ is pulled back from the base, where it equals
\begin{equation}
e^G=\frac{f_*(\omega_X^{m+n})}{\omega_B^m\int_{X_b}\omega_X^n},
\end{equation}
where of course $\int_{X_b}\omega_X^n$ is independent of $b$. Furthermore, $e^G$ it is integrable on $B\setminus f(S)$ (an is even in $L^p(\omega_B^m)$ for some $p>1$) and satisfies $\int_{B\setminus f(S)}e^G\omega_B^n=\frac{\int_X\omega_X^{m+n}}{\int_{X_b}\omega_X^n}.$
We can then solve the Monge-Amp\`ere equation \cite{ST,To}
\begin{equation}\label{mabase}
\omega_{\rm can}^m=(\omega_B+\ddbar \psi_\infty)^m=\frac{\int_B\omega_B^m\int_{X_b}\omega_X^n}{\int_X\omega_X^{m+n}} e^G\omega_B^m,
\end{equation}
with $\psi_\infty$ smooth on $X\setminus S$ (and globally continuous, which we will not need) and on $X\setminus S$ define $\omega^\natural_t=f^*\omega_{\rm can}+e^{-t}\omega_F$ and $\psi_t=\vp_t-\psi_\infty-e^{-t}\rho$, so that we have
$\omega^\bullet_t=\omega^\natural_t+\ddbar\psi_t$. Thus, combining \eqref{maalt} and \eqref{mabase}, we see that on $X\setminus S$ we have
\begin{equation}\label{marepetida}
(\omega^\bullet_t)^{m+n}=(\omega^\natural_t+\ddbar\psi_t)^{m+n}=c_te^{-nt}\omega_{\rm can}^m\wedge\omega_F^n,
\end{equation}
where
\begin{equation}
c_t=e^{nt}\frac{\int_X(f^*\omega_B+e^{-t}\omega_X)^{m+n}}{\int_B\omega_B^m\int_{X_b}\omega_X^n},
\end{equation}
which is indeed equal to a polynomial in $e^{-t}$ of degree at most $m$ with constant coefficient $\binom{m+n}{n}$. It is easy to see that given any $K\Subset X\setminus S$, there is $t_K$ such that $\omega^\natural_t$ is a K\"ahler metric on $K$ for all $t\geq t_K$, uniformly equivalent to $f^*\omega_B+e^{-t}\omega_X$.

To prove Theorem \ref{mthmA} we can assume that we are given an arbitrary coordinate unit ball compactly contained in $B\setminus f(S)$, over which $f$ is $C^\infty$ trivial. As usual, we simply denote by $B$ this ball, and its preimage is $B\times Y$ equipped with a complex structure $J$ as in Theorem \ref{shutupandsuffer}.
Thanks to \cite{To} we know that on $B\times Y$ we have
\begin{equation}
C^{-1}\omega^\natural_t\leq \omega^\bullet_t\leq C\omega^\natural_t,
\end{equation}
for all $t\geq 0$ (assuming without loss that $\omega^\natural_t$ is K\"ahler for all $t\geq 0$), and that $\ddbar\psi_t\to 0$ weakly as currents
(hence $\psi_t\to 0$ in $L^1$ by standard psh functions theory, since $\vp_t$ is normalized by $\sup_X\vp_t=0$).

We are thus in good shape to apply Theorem \ref{shutupandsuffer}. To prove that $\omega^\bullet_t$ is locally uniformly bounded in $C^{k}$ of a fixed metric (where $k\geq 0$ is arbitrary), we take $j=k$ in Theorem \ref{shutupandsuffer}, so that up to shrinking $B$ we have the decomposition \eqref{decomponiti}
\begin{equation}
\omega^\bullet_t = \omega_t^\natural + \gamma_{t,0} + \gamma_{t,2,k} + \cdots + \gamma_{t,k,k} + \eta_{t,k,k},
\end{equation}
where $\omega_t^\natural$ is clearly smoothly bounded, $\gamma_{t,0}$ has uniform $C^{k,\alpha}$ bounds by \eqref{vier3} and \eqref{vier2}, $\eta_{t,k,k}$ has even shrinking uniform $C^{k,\alpha}$ bounds by \eqref{eins} and \eqref{zwei} (hence a standard uniform $C^{k,\alpha}$ bound thanks to Lemma \ref{arschloch}), and the $\gamma_{t,i,k}$ have uniform $C^{k}$ bounds by the following argument: from the definition \eqref{fuenf} and \eqref{e:Gstructure} we can write for $2\leq i\leq k$
\begin{equation}\label{coot}
\gamma_{t,i,k}=\ddbar\sum_{p=1}^{N_{i,k}}\sum_{\iota=0}^{2k}\sum_{\ell=\lceil\frac{\iota}{2}\rceil}^{k}e^{-\ell t}\Phi_{\iota,\ell}(G_{i,p,k})\circledast\D^\iota A_{t,i,p,k},
\end{equation}
and so schematically for $0\leq r\leq k$
\begin{equation}\label{cornucopia}
\D^r\gamma_{t,i,k}=\sum_{p=1}^{N_{i,k}}\sum_{\iota=0}^{2k}\sum_{\ell=\lceil\frac{\iota}{2}\rceil}^{k}\sum_{s=0}^{r+1}\sum_{i_1+i_2=s+1}e^{-\ell t}(\D^{r+1-s}J)\circledast\D^{i_1}\Phi_{\iota,\ell}(G_{i,p,k})\circledast\D^{i_2+\iota} A_{t,i,p,k},
\end{equation}
and using a fixed metric $g_X$ we have clearly $|\D^{i_1}\Phi_{\iota,\ell}(G_{i,p,k})|_{g_X}\leq C$ and $|(\D^{r+1-s}J)|_{g_X}\leq C$, while from \eqref{drei2} we see that
$|\D^{i_2+\iota} A_{t,i,p,k}|=o(1)$ when $i_2+\iota\leq k+2$ and from \eqref{stroh}
$|\D^{i_2+\iota} A_{t,i,p,k}|=o(e^{(i_2+\iota-k-2)\frac{t}{2}})$ when $k+2<i_2+\iota\leq k+2+2k$, and so
\begin{equation}\label{pietrogamba}\begin{split}
|\D^r\gamma_{t,i,k}|_{g_X}&\leq o(1)+ o(1)\sum_{i_1+i_2=r+2}\sum_{\iota=k+3-i_2}^{2k}\sum_{\ell=\lceil\frac{\iota}{2}\rceil}^{k}e^{-\ell t}e^{(i_2+\iota-k-2)\frac{t}{2}}\\
&\leq o(1)+o(e^{(r-k)\frac{t}{2}})\sum_{\iota=k+3-i_2}^{2k}\sum_{\ell=\lceil\frac{\iota}{2}\rceil}^{k}e^{-\left(\ell-\frac{\iota}{2}\right) t}\leq o(1)+o(e^{(r-k)\frac{t}{2}})=o(1),
\end{split}\end{equation}
since $i_2\leq r+2$ and $r\leq k$.
\end{proof}

\begin{proof}[Proof of Theorem \ref{mthmB}]
From Theorem \ref{shutupandsuffer} applied with $j=2$ and $k\geq 2$ arbitrary we know that
\begin{equation}\omega_t^\bullet=\omega_{\rm can}+e^{-t}\omega_F+\gamma_{t,0}+\gamma_{t,2,k}+\eta_{t,2,k},\end{equation}
where we have that $\gamma_{t,0}=\ddbar\underline{\psi_t}$ goes to zero in the $C^k$ norm (by applying Theorem \ref{shutupandsuffer}, specificially \eqref{vier3}, with $j=k$), and $\gamma_{t,2,k}$ has the schematic structure
\begin{equation}\gamma_{t,2,k}=\ddbar\sum_{p=1}^{N_{2,k}}\sum_{\iota=0}^{2k}\sum_{\ell=\lceil\frac{\iota}{2}\rceil}^{k}e^{-\ell t}\Phi_{\iota,\ell}(G_{2,p,k})\circledast\D^\iota A_{t,2,p,k},\end{equation}
where $A_{t,2,p,k}$ are functions from the base which thanks to \eqref{drei2} and \eqref{stroh} satisfy
\begin{equation}\label{quack}
|\D^iA_{t,2,p,k}|\leq \begin{cases} Ce^{-(3+\alpha)\left(1-\frac{i}{4+\alpha}\right)\frac{t}{2}},\quad 0\leq i\leq 4,\\
o(e^{-(4-i)\frac{t}{2}}),\quad 5\leq i\leq 4+2k.
\end{cases}
\end{equation}
Interestingly, we will also need another interpolation, as in \eqref{610}, interpolating between $|A_{t,2,p,k}|\leq Ce^{-(3+\alpha)\frac{t}{2}}$ from \eqref{drei2} and $[\D^2A_{t,2,p,k}]_{C^\alpha}\leq Ce^{-t}$ from \eqref{sieben} gives
\begin{equation}\label{sberequack}
|\D^iA_{t,2,p,k}|\leq C e^{-t} \left(e^{\frac{-1-\alpha}{2}t}\right)^{1-\frac{i}{2+\alpha}},\quad 0\leq i\leq 2.
\end{equation}
Our goal is to clarify the structure of the term $\gamma_{t,2,k}$. This will take us some work, and the very first step is the claim that arguing as in \eqref{pietrogamba} and using \eqref{quack} and \eqref{sberequack} we will have
\begin{equation}\label{brutus}
|\gamma_{t,2,k}|_{g_X}=o(e^{-t}).
\end{equation}
To prove the claim \eqref{brutus}, we argue as in \eqref{cornucopia} and bound
\begin{equation}\label{cornelius}
|\gamma_{t,2,k}|_{g_X}\leq C\sum_{i_1+i_2=2}\sum_{p=1}^{N_{2,k}}\sum_{\iota=0}^{2k}\sum_{\ell=\lceil\frac{\iota}{2}\rceil}^{k}e^{-\ell t}|\D^{i_2+\iota} A_{t,2,p,k}|,
\end{equation}
and we bound the RHS of \eqref{cornelius} by $o(e^{-t})$ by considering the possible values of $i_2+\iota\in\{0,\dots,2k+2\}$: if $i_2+\iota=0,1,2$ then \eqref{sberequack} in particular gives $|\D^{i_2+\iota} A_{t,2,p,k}|=o(e^{-t})$, so good. If $i_2+\iota=3,4$ then necessarily $\iota\geq 1$ and so $\ell\geq 1$, while \eqref{quack} in particular gives $|\D^{i_2+\iota} A_{t,2,p,k}|=o(1)$ so the RHS of \eqref{cornelius} is again $o(e^{-t})$. And if $i_2+\iota\geq 5$ then we use \eqref{quack} exactly as in \eqref{pietrogamba} to bound the RHS of \eqref{cornelius} by
\begin{equation}\label{korn}
o(e^{-t})+o(1)\sum_{i_1+i_2=2}\sum_{p=1}^{N_{2,k}}\sum_{\iota=5-i_2}^{2k}\sum_{\ell=\lceil\frac{\iota}{2}\rceil}^{k}e^{-\ell t}e^{(i_2+\iota-4)\frac{t}{2}}
\leq o(e^{-t})+o(e^{-t})\sum_{\iota=3}^{2k}\sum_{\ell=\lceil\frac{\iota}{2}\rceil}^{k}e^{-\left(\ell-\frac{\iota}{2}\right) t}=o(e^{-t}),
\end{equation}
using $i_2\leq 2$, which concludes the proof of \eqref{brutus}.

On the other hand, thanks to \eqref{eins} and \eqref{zwei}, together with interpolation (Proposition \ref{l:higher-interpol}, taking here $t$ sufficiently large compared to the radius $R$ of the ball that we are working on, so that $\mathbb{B}^{g_t}(p,R)=B_R\times Y$ for all $p\in B_R\times Y$), $\eta_{t,2,k}$ satisfies
\begin{equation}\label{3satan}|\D^i\eta_{t,2,k}|_{g_t}\leq Ce^{\frac{i-2-\alpha}{2}t},\quad [\D^i\eta_{t,2,k}]_{C^\alpha(g_t)}\leq Ce^{\frac{i-2}{2}t},\quad 0\leq i\leq 2,
\end{equation}
so in particular
\begin{equation}\label{tuquoque}
|\eta_{t,2,k}|_{g_X}=o(e^{-t}), \quad |(\eta_{t,2,k})_{\mathbf{ff}}|_{g_X}=o(e^{-2t}).
\end{equation}

Next, we seek better estimates than \eqref{brutus} for the $\mathbf{ff}$ components of $\gamma_{t,2,k}$. We claim that we have
\begin{equation}\label{satanacchio}
(\gamma_{t,2,k})_{\mathbf{ff}}=\sum_{p=1}^{N_{2,k}}A_{t,2,p,k}\de_{\mathbf{f}}\db_{\mathbf{f}}(\Delta^{\omega_F|_{\{\cdot\}\times Y}})^{-1}G_{2,p,k}+ o(e^{-2t}),
\end{equation}
where the $o(e^{-2t})$ is in $L^\infty_{\rm loc}(g_X)$. Indeed from \eqref{coot} we can write
\begin{equation}
(\gamma_{t,2,k})_{\mathbf{ff}}=\sum_{p=1}^{N_{2,k}}\sum_{\iota=0}^{2k}\sum_{\ell=\lceil\frac{\iota}{2}\rceil}^{k}e^{-\ell t}\de_{\mathbf{f}}\db_{\mathbf{f}}\Phi_{\iota,\ell}(G_{2,p,k})\circledast\D^\iota A_{t,2,p,k},
\end{equation}
and we can estimate each term as follows. For $\iota\geq 4$ we have $|\D^\iota A_{t,2,p,k}|=o(e^{-(4-\iota)\frac{t}{2}})$ from \eqref{quack}, and so
\begin{equation}
e^{-\ell t}|\de_{\mathbf{f}}\db_{\mathbf{f}}\Phi_{\iota,\ell}(G_{2,p,k})\circledast\D^\iota A_{t,2,p,k}|_{g_X}\leq o(1)e^{-2t}e^{-\left(\ell-\frac{\iota}{2}\right)t}=o(e^{-2t}),
\end{equation}
since $\ell\geq \frac{\iota}{2}$. For $\iota=3$, we have $\ell\geq 2$ and $|\D^\iota A_{t,2,p,k}|=o(1)$ from \eqref{quack}, so the term is again $o(e^{-2t})$. For $\iota=1,2$, we have $\ell\geq 1$ and $|\D^\iota A_{t,2,p,k}|=o(e^{-t})$ from \eqref{sberequack}, so the term is again $o(e^{-2t})$. And for $\iota=0$, let us first look at the terms with $\ell\geq 1$. For these, we have $|A_{t,2,p,k}|=O(e^{-(3+\alpha)\frac{t}{2}})$, and so when multiplied by $e^{-\ell t},\ell\geq 1$, these terms are indeed $o(e^{-2t})$. So we are only left with the terms where $\iota=\ell=0$ which equal
\begin{equation}
\sum_{p=1}^{N_{2,k}}A_{t,2,p,k}\de_{\mathbf{f}}\db_{\mathbf{f}}(\Delta^{\omega_F|_{\{\cdot\}\times Y}})^{-1}G_{2,p,k},
\end{equation}
since $\Phi_{0,0}(G)=(\Delta^{\omega_F|_{\{\cdot\}\times Y}})^{-1}G$ by \eqref{sangennaro}, thus proving \eqref{satanacchio}. In particular, using the bound \eqref{quack} in \eqref{satanacchio} gives
\begin{equation}\label{brutus_better}
|(\gamma_{t,2,k})_{\mathbf{ff}}|_{g_X}\leq Ce^{-\frac{3+\alpha}{2}t}.
\end{equation}
Our next goal is to prove that
\begin{equation}\label{filimi}
|A_{t,2,p,k}|\leq Ce^{-2t},
\end{equation}
which improves upon \eqref{drei2}. To see this, recall from \eqref{marepetida} that
\begin{equation}\label{3sat}c_t e^{-nt}\omega_{\rm can}^m\wedge\omega_F^n=(\omega^\bullet_t)^{m+n}=(\omega_{\rm can}+\ddbar\underline{\psi_t}+e^{-t}\omega_F
+\gamma_{t,2,k}+\eta_{t,2,k})^{m+n},\end{equation}
with $c_t=\binom{m+n}{n}+O(e^{-t})$. Multiply this by $\frac{e^{nt}}{\binom{m+n}{n}}$ and using the above estimates \eqref{brutus}, \eqref{tuquoque}, \eqref{brutus_better} (which imply that $|e^t\gamma_{t,2,k}+e^t\eta_{t,2,k}|_{g_X}=o(1)$ and $|(e^t\gamma_{t,2,k}+e^t\eta_{t,2,k})_{\mathbf{ff}}|_{g_X}=o(e^{-\frac{t}{2}})$)
and using also that $ i\partial\overline\partial\underline{\psi_t}$ is small in $C^2$, we can expand it as
\begin{equation}\label{3sata}\begin{split}
\frac{c_t}{\binom{m+n}{n}}\omega_{\rm can}^m\wedge\omega_F^n&=(\omega_{\rm can}+\ddbar\underline{\psi_t})^m\wedge(\omega_F+e^t\gamma_{t,2,k}+e^t\eta_{t,2,k})_{\mathbf{ff}}^n\\
&+\frac{m}{n+1}e^{-t}
(\omega_{\rm can}+\ddbar\underline{\psi_t})^{m-1}\wedge(\omega_F+e^t\gamma_{t,2,k}+e^t\eta_{t,2,k})^{n+1}+O(e^{-2t})\\
&=(\omega_{\rm can}+\ddbar\underline{\psi_t})^m\wedge\omega_F^n+
ne^t(\omega_{\rm can}+\ddbar\underline{\psi_t})^m(\omega_F)_{\mathbf{ff}}^{n-1}(\gamma_{t,2,k})_{\mathbf{ff}}\\
&+\frac{m}{n+1}e^{-t}
(\omega_{\rm can}+\ddbar\underline{\psi_t})^{m-1}\wedge\omega_F^{n+1}+o(e^{-t}),
\end{split}\end{equation}
where the error terms $o(e^{-t})$ are in $L^\infty_{\rm loc}(g_X)$. Define a function $\mathcal{S}$ on $B\times Y$ by
\begin{equation}\label{semmes}
\frac{m}{n+1}\omega_{\rm can}^{m-1}\wedge\omega_F^{n+1}=\mathcal{S}\omega_{\rm can}^{m}\wedge\omega_F^{n},
\end{equation}
so that (using again that  $ i\partial\overline\partial\underline{\psi_t}$ is $o(1)$) \eqref{3sata} gives
\begin{equation}\label{e:satanas}
\omega_{\rm can}^m\wedge(\omega_F)_{\mathbf{ff}}^n\left(1+e^t\mathrm{tr}^{\omega_F|_{\{\cdot\}\times Y}}(\gamma_{t,2,k})_{\mathbf{ff}}+e^{-t}\mathcal{S}\right)(1+o(1)_{\rm base})
=\frac{c_t}{\binom{m+n}{n}}\omega_{\rm can}^m\wedge(\omega_F)_{\mathbf{ff}}^n +o(e^{-t}).
\end{equation}
Divide both sides by  $\omega_{\rm can}^m\wedge(\omega_F)_{\mathbf{ff}}^n$, obtaining an equality of functions, and after subtracting its fiberwise average and multiplying by $e^{-t}$ and dividing by $(1+o(1)_{\rm base})$ we obtain
\begin{equation}
\label{e:satanas2}
e^{-2t}(\mathcal{S} - \underline{\mathcal{S}}) + \mathrm{tr}^{\omega_F|_{\{\cdot\}\times Y}}(\gamma_{t,2,k})_{\mathbf{ff}} = o(e^{-2t}),
\end{equation}
using here that $\gamma_{t,2,k}$ is $\de\db$-exact so its fiberwise trace has zero integral.
Next, taking the fiberwise trace of \eqref{satanacchio} with respect to the fiberwise restriction of $\omega_F$ gives
\begin{equation}\label{satanacchio2}
\mathrm{tr}^{\omega_F|_{\{\cdot\}\times Y}}(\gamma_{t,2,k})_{\mathbf{ff}}=\sum_{p=1}^{N_{2,k}}A_{t,2,p,k}G_{2,p,k} + o(e^{-2t}).
\end{equation}
Subtituting \eqref{satanacchio2} into \eqref{e:satanas2} gives
\begin{equation}\label{freak}
\sum_{p=1}^{N_{2,k}}A_{t,2,p,k}G_{2,p,k}=-e^{-2t}(\mathcal{S} - \underline{\mathcal{S}})+o(e^{-2t}),
\end{equation}
and integrating this against $G_{2,p,k}$ on any fiber $\{z\}\times Y$ gives
\begin{equation}\label{scetate}
A_{t,2,p,k} =-e^{-2t}\int_{\{z\}\times Y}\mathcal{S} G_{2,p,k}\omega_F^n+o(e^{-2t}),
\end{equation}
which is valid as usual in $L^\infty$, and the desired \eqref{filimi} follows.

The next step is to use the improved bound in \eqref{filimi} to obtain better bounds for the derivatives of $A_{t,2,p,k}$, as follows.
Interpolating between \eqref{filimi} and $[\D^4 A_{t,2,p,k}]_{C^{\alpha}}\leq C$ from \eqref{sieben} as in \eqref{610} gives
\begin{equation}\label{kkq0}|\D^iA_{t,2,p,k}|\leq Ce^{-\frac{4-i+i\frac{\alpha}{4+\alpha}}{2}t},\quad 0\leq i\leq 4,\end{equation}
and letting $\gamma=\frac{\alpha}{4+\alpha}$ we obtain in particular
\begin{equation}\label{kkq}|\D^iA_{t,2,p,k}|\leq Ce^{-\frac{4-i+\gamma}{2}t},\quad 1\leq i\leq 4.\end{equation}
We then interpolate again as in \eqref{stroh} between
$[\D^{4+\iota} A_{t,2,p,k}]_{C^{\alpha}}\leq Ce^{\frac{\iota}{2}t}$ from \eqref{sieben} ($1\leq \iota\leq 2k$) and $|\D^4A_{t,2,p,k}|\leq Ce^{-\frac{\gamma}{2}t}$ from \eqref{kkq}, and set $4+\iota=i$, and we can improve this to
\begin{equation}\label{skd}|\D^iA_{t,2,p,k}|\leq C e^{-\frac{4-i+\gamma}{2}t},\quad 1\leq i\leq 4+2k.\end{equation}
Similarly, using interpolation we can bound
\begin{equation}\label{skd2}[\D^iA_{t,2,p,k}]_{C^\alpha}\leq Ce^{-\frac{4-i-\alpha+\gamma}{2}t},\quad 1\leq i\leq 4+2k.\end{equation}
Also, when $i=0$, we can interpolate between \eqref{filimi} and \eqref{kkq} with $i=1$ and obtain
\begin{equation}\label{skd3}[A_{t,2,p,k}]_{C^\alpha}\leq Ce^{-\frac{4-\alpha+\alpha\gamma}{2}t},\end{equation}
and for convenience we will redefine $\gamma$ to be $\alpha\gamma$ from now on.
We now use \eqref{skd}, \eqref{skd2}, \eqref{skd3} together with \eqref{filimi} to bound the error term in \eqref{satanacchio2} as follows. Recall that from \eqref{sangennaro} we have
\begin{equation}\label{drunkenstein}
\gamma_{t,2,k}=\ddbar\sum_{p=1}^{N_{2,k}}A_{t,2,p,k}(\Delta^{\omega_F|_{\{\cdot\}\times Y}})^{-1}G_{2,p,k}+\ddbar\sum_{p=1}^{N_{2,k}}\sum_{\iota=0}^{2k}\sum_{\ell=\max(1,\lceil\frac{\iota}{2}\rceil)}^{k}e^{-\ell t}\Phi_{\iota,\ell}(G_{2,p,k})\circledast\D^\iota A_{t,2,p,k},
\end{equation}
and if we define
\begin{equation}\label{herrnikolaus}
\gamma^\sharp_{t,2,k}:=\gamma_{t,2,k}-\ddbar\sum_{p=1}^{N_{2,k}}A_{t,2,p,k}(\Delta^{\omega_F|_{\{\cdot\}\times Y}})^{-1}G_{2,p,k},
\end{equation}
then the error term in \eqref{satanacchio2} is exactly equal to
\begin{equation}\label{herrwilking}\begin{split}
\mathrm{tr}^{\omega_F|_{\{\cdot\}\times Y}}(\gamma^\sharp_{t,2,k})_{\mathbf{ff}}&=\sum_{p=1}^{N_{2,k}}\sum_{\iota=0}^{2k}\sum_{\ell=\max(1,\lceil\frac{\iota}{2}\rceil)}^{k}e^{-\ell t}(\mathrm{tr}^{\omega_F|_{\{\cdot\}\times Y}}\de_{\mathbf{f}}\db_{\mathbf{f}}\Phi_{\iota,\ell}(G_{2,p,k}))\circledast\D^\iota A_{t,2,p,k}\\
&=\sum_{p=1}^{N_{2,k}}\sum_{\iota=1}^{2k}\sum_{\ell=\lceil\frac{\iota}{2}\rceil}^{k}e^{-\ell t}(\mathrm{tr}^{\omega_F|_{\{\cdot\}\times Y}}\de_{\mathbf{f}}\db_{\mathbf{f}}\Phi_{\iota,\ell}(G_{2,p,k}))\circledast\D^\iota A_{t,2,p,k}\\
&+\sum_{p=1}^{N_{2,k}}\sum_{\ell=1}^{k}e^{-\ell t}(\mathrm{tr}^{\omega_F|_{\{\cdot\}\times Y}}\de_{\mathbf{f}}\db_{\mathbf{f}}\Phi_{0,\ell}(G_{2,p,k})) A_{t,2,p,k}.
\end{split}
\end{equation}
For $0\leq r\leq 2$ we can expand schematically
\begin{equation}\label{krukk}\begin{split}
\D^r\mathrm{tr}^{\omega_F|_{\{\cdot\}\times Y}}(\gamma^\sharp_{t,2,k})_{\mathbf{ff}}&=\sum_{i_1+i_2=r}\sum_{p=1}^{N_{2,k}}\sum_{\iota=1}^{2k}\sum_{\ell=\lceil\frac{\iota}{2}\rceil}^{k}e^{-\ell t}\D^{i_1}(\mathrm{tr}^{\omega_F|_{\{\cdot\}\times Y}}\de_{\mathbf{f}}\db_{\mathbf{f}}\Phi_{\iota,\ell}(G_{2,p,k}))\circledast\D^{i_2+\iota} A_{t,2,p,k}\\
&+\sum_{i_1+i_2=r}\sum_{p=1}^{N_{2,k}}\sum_{\ell=1}^{k}e^{-\ell t}\D^{i_1}(\mathrm{tr}^{\omega_F|_{\{\cdot\}\times Y}}\de_{\mathbf{f}}\db_{\mathbf{f}}\Phi_{0,\ell}(G_{2,p,k}))\circledast \D^{i_2} A_{t,2,p,k},
\end{split}
\end{equation}
and using $|\D^{i_1}(\mathrm{tr}^{\omega_F|_{\{\cdot\}\times Y}}\de_{\mathbf{f}}\db_{\mathbf{f}}\Phi_{\iota,\ell}(G_{2,p,k}))|_{g_t}\leq Ce^{i_1\frac{t}{2}}$
with \eqref{filimi} and \eqref{skd} we can bound the $g_t$-norm of \eqref{krukk} by
\begin{equation}\label{saw2}
\begin{split}
|\D^r\mathrm{tr}^{\omega_F|_{\{\cdot\}\times Y}}(\gamma^\sharp_{t,2,k})_{\mathbf{ff}}|_{g_t}&\leq Ce^{-\frac{\gamma}{2}t}\sum_{i_1+i_2=r}\sum_{\iota=1}^{2k}\sum_{\ell=\lceil\frac{\iota}{2}\rceil}^{k}e^{-\ell t}e^{(i_1+i_2+\iota-4)\frac{t}{2}}
+C \sum_{i_1+i_2=r}\sum_{\ell=1}^{k}e^{-\ell t}e^{(i_1+i_2-4)\frac{t}{2}}\\
&=Ce^{(r-4-\gamma)\frac{t}{2}}\sum_{\iota=1}^{2k}\sum_{\ell=\lceil\frac{\iota}{2}\rceil}^{k}e^{-\left(\ell-\frac{\iota}{2}\right) t}
+Ce^{(r-4)\frac{t}{2}}\sum_{\ell=1}^{k}e^{-\ell t}\\
&\leq Ce^{(r-4-\gamma)\frac{t}{2}},
\end{split}
\end{equation}
which implies in particular that when we restrict to a fiber $\{\cdot\}\times Y$ and differentiate only vertically then
\begin{equation}\label{saw3}
|\D^r_{\mathbf{f}\cdots\mathbf{f}}\mathrm{tr}^{\omega_F|_{\{\cdot\}\times Y}}(\gamma^\sharp_{t,2,k})_{\mathbf{ff}}|_{g_X|_{\{\cdot\}\times Y}}\leq Ce^{-(4+\gamma)\frac{t}{2}},
\end{equation}
for $0\leq r\leq 2$. Similarly, we can bound the H\"older $C^\alpha(g_t)$ seminorm of \eqref{krukk} by using the estimate $[\D^{i_1}(\mathrm{tr}^{\omega_F|_{\{\cdot\}\times Y}}\de_{\mathbf{f}}\db_{\mathbf{f}}\Phi_{\iota,\ell}(G_{2,p,k}))]_{C^\alpha(g_t)}\leq Ce^{(i_1+\alpha)\frac{t}{2}}$ together with \eqref{filimi}, \eqref{skd}, \eqref{skd2}, \eqref{skd3} and get for $0\leq r\leq 2$
\begin{equation}\label{fructus22}
[\D^r\mathrm{tr}^{\omega_F|_{\{\cdot\}\times Y}}(\gamma^\sharp_{t,2,k})_{\mathbf{ff}}]_{C^\alpha(g_t)}\leq Ce^{(r-4-\gamma+\alpha)\frac{t}{2}},
\end{equation}
and so fiberwise this gives
\begin{equation}\label{saw4}
[\D^r_{\mathbf{f}\cdots\mathbf{f}}\mathrm{tr}^{\omega_F|_{\{\cdot\}\times Y}}(\gamma^\sharp_{t,2,k})_{\mathbf{ff}}]_{C^\alpha(\{\cdot\}\times Y,g_X)}\leq Ce^{-(4+\gamma)\frac{t}{2}}.
\end{equation}
As for the whole $\gamma_{t,2,k}$, we can schematically expand
\begin{equation}\label{kruk}
\D^r\gamma_{t,2,k}=\sum_{p=1}^{N_{2,k}}\sum_{\iota=0}^{2k}\sum_{\ell=\lceil\frac{\iota}{2}\rceil}^{k}\sum_{s=0}^{r+1}\sum_{i_1+i_2=s+1}e^{-\ell t}(\D^{r+1-s}J)\circledast\D^{i_1}\Phi_{\iota,\ell}(G_{2,p,k})\circledast\D^{i_2+\iota} A_{t,2,p,k},
\end{equation}
and using $|\D^{i_1}\Phi_{\iota,\ell}(G_{i,p,k})|_{g_t}\leq Ce^{i_1\frac{t}{2}}, |\D^\iota J|_{g_t}\leq Ce^{\iota\frac{t}{2}}$ and \eqref{filimi}, \eqref{skd} we can bound
\begin{equation}\label{saw}
|\D^r\gamma_{t,2,k}|_{g_t}\leq C\sum_{i_1+i_2=r+2}\sum_{\iota=0}^{2k}\sum_{\ell=\lceil\frac{\iota}{2}\rceil}^{k}e^{-\ell t}e^{(i_1+i_2+r+1-s+\iota-4)\frac{t}{2}}=e^{(r-2)\frac{t}{2}}\sum_{\iota=0}^{2k}\sum_{\ell=\lceil\frac{\iota}{2}\rceil}^{k}e^{-\left(\ell-\frac{\iota}{2}\right) t}\leq Ce^{(r-2)\frac{t}{2}},
\end{equation}
for $0\leq r\leq 2$,
and similarly
\begin{equation}\label{saw6}
[\D^r\gamma_{t,2,k}]_{C^\alpha(g_t)}\leq Ce^{(r+\alpha-2)\frac{t}{2}},
\end{equation}
and arguing along similar lines, we obtain
\begin{equation}\label{saw5}
|\D^r\gamma^\sharp_{t,2,k}|_{g_t}\leq Ce^{(r-2-\gamma)\frac{t}{2}},\quad [\D^r\gamma^\sharp_{t,2,k}]_{C^\alpha(g_t)}\leq Ce^{(r+\alpha-2-\gamma)\frac{t}{2}},
\end{equation}
for $0\leq r\leq 2$.
When we take only fiber derivatives, \eqref{saw} and \eqref{saw6} imply that
\begin{equation}\label{kinger}
|\D^r_{\mathbf{f}\cdots\mathbf{f}}\gamma_{t,2,k}|_{g_X}\leq Ce^{-t},\quad [\D^r_{\mathbf{f}\cdots\mathbf{f}}\gamma_{t,2,k}]_{C^\alpha(\{\cdot\}\times Y,g_X)}\leq Ce^{-t},\quad 0\leq r\leq 2,
\end{equation}
but we can actually do better than this by going back to \eqref{kruk}: for any term in \eqref{kruk} with $i_2+\iota>0$, we use the bounds in \eqref{skd},
\eqref{skd2} and we gain a factor of $e^{-\gamma t/2}$ over \eqref{kinger}, and for the terms with $i_2+\iota=0$ we have $r+1-s+i_1=r+2$, and
using now that $|(\D^{r+1-s}J)\circledast\D^{i_1}\Phi_{\iota,\ell}(G_{2,p,k})|_{g_X}\leq C$, we gain a factor of $e^{-(i_1+r+1-s)\frac{t}{2}}=e^{-(r+2)\frac{t}{2}}\leq e^{-\gamma t/2}$ over \eqref{kinger}. In conclusion, we obtain
\begin{equation}\label{zinger}
|\D^r_{\mathbf{f}\cdots\mathbf{f}}\gamma_{t,2,k}|_{g_X}\leq Ce^{-(2+\gamma)\frac{t}{2}},\quad [\D^r_{\mathbf{f}\cdots\mathbf{f}}\gamma_{t,2,k}]_{C^\alpha(\{\cdot\}\times Y,g_X)}\leq Ce^{-(2+\gamma)\frac{t}{2}},\quad 0\leq r\leq 2,
\end{equation}

We have thus derived estimates \eqref{saw3} and \eqref{saw4} for the error term in \eqref{satanacchio2} and its fiberwise derivatives. Our next task is to derive analogous estimates for the error term in \eqref{e:satanas2}, so we need to recall how this error term was obtained:
we first expanded \eqref{3sat} as \eqref{3sata} and derived \eqref{e:satanas}, and then we divided it by $\omega_{\rm can}^m\wedge\omega_F^n$, subtracted its fiber average, and multiplied by $e^{-t}$ to obtain \eqref{e:satanas2}. To estimate the error terms, we thus follow the same procedure and go back to \eqref{3sat}, and divide it by $\binom{m+n}{n}e^{-nt}$
to get
\begin{equation}\label{3satana}\begin{split}
&\frac{c_t}{\binom{m+n}{n}} \omega_{\rm can}^m\wedge\omega_F^n\\
&=(\omega_{\rm can}+\ddbar\underline{\psi_t})^m\wedge\left(\omega_F+e^t\gamma_{t,2,k}+e^t\eta_{t,2,k}\right)^{n}_{\mathbf{ff}}\\
&+\sum_{j=1}^m \frac{m!n!}{(m-j)!(n+j)!}\left(\omega_{\rm can}+\ddbar\underline{\psi_t}\right)^{m-j}\wedge\left(\omega_F+e^t\gamma_{t,2,k}+e^t\eta_{t,2,k}\right)^{n}\wedge\left(e^{-t}\omega_F+\gamma_{t,2,k}+\eta_{t,2,k}\right)^{j}\\
&=(\omega_{\rm can}+\ddbar\underline{\psi_t})^m\wedge\left((\omega_F)^n_{\mathbf{ff}}+ne^t(\omega_F)^{n-1}_{\mathbf{ff}}(\gamma_{t,2,k})_{\mathbf{ff}}\right)\\
&+ne^t(\omega_{\rm can}+\ddbar\underline{\psi_t})^m\wedge(\omega_F)^{n-1}_{\mathbf{ff}}(\eta_{t,2,k})_{\mathbf{ff}}+\sum_{k=2}^n\binom{n}{k} (\omega_{\rm can}+\ddbar\underline{\psi_t})^m\wedge (\omega_F)^{n-k}_{\mathbf{ff}}(e^t\gamma_{t,2,k}+e^t\eta_{t,2,k})^k_{\mathbf{ff}}\\
&+\frac{m}{n+1}e^{-t}\left(\omega_{\rm can}+\ddbar\underline{\psi_t}\right)^{m-1}\wedge\omega_F^{n+1}\\
&+\frac{m}{n+1}\sum_{p=1}^n \binom{n}{p}e^{-t}\left(\omega_{\rm can}+\ddbar\underline{\psi_t}\right)^{m-1}\wedge\omega_F^{n-p+1}\wedge(e^t\gamma_{t,2,k}+e^t\eta_{t,2,k})^p\\
&+\frac{m}{n+1}\left(\omega_{\rm can}+\ddbar\underline{\psi_t}\right)^{m-1}\wedge\left(\omega_F+e^t\gamma_{t,2,k}+e^t\eta_{t,2,k}\right)^{n}\wedge\left(\gamma_{t,2,k}+\eta_{t,2,k}\right)\\
&+\sum_{j=2}^m \frac{m!n!}{(m-j)!(n+j)!}\left(\omega_{\rm can}+\ddbar\underline{\psi_t}\right)^{m-j}\wedge\left(\omega_F+e^t\gamma_{t,2,k}+e^t\eta_{t,2,k}\right)^{n}\wedge\left(e^{-t}\omega_F+\gamma_{t,2,k}+\eta_{t,2,k}\right)^{j},
\end{split}\end{equation}
so the error terms in \eqref{3sata} (marked as $o(e^{-t})$ there) consist precisely of the terms in lines $-1,-2,-3,-5$ of \eqref{3satana}. To obtain the error terms in \eqref{e:satanas} we have to add to these the parts of lines $-4,-6$ which contain $\ddbar\underline{\psi_t}$. And to finally obtain the error terms in \eqref{e:satanas2} we need to divide by $\omega_{\rm can}^m\wedge\omega_F^n$, subtract the fiber average, and multiply by $e^{-t}$.

So first, let us show that, after dividing by $\omega_{\rm can}^m\wedge\omega_F^n$, the terms in lines $-1,-2,-3,-5$ of \eqref{3satana} satisfy
\begin{equation}\label{stronzetto}
|\D^r_{\mathbf{f}\cdots\mathbf{f}}(\ \cdot\ )|_{g_X}\leq Ce^{-(2+\gamma)\frac{t}{2}},
\quad [\D^r_{\mathbf{f}\cdots\mathbf{f}}(\ \cdot\ )]_{C^\alpha(\{\cdot\}\times Y,g_X)}\leq Ce^{-(2+\gamma)\frac{t}{2}},
\end{equation}
for $0\leq r\leq 2$. To do this, first recall that $\ddbar\underline{\psi_t}$ is $o(1)$ in $C^{2,\alpha}$ from \eqref{vier3} and \eqref{vier2}. Thanks to \eqref{3satan} we have in particular
\begin{equation}
|\D^i_{\mathbf{f}\cdots\mathbf{f}}\eta_{t,2,k}|_{g_X}\leq Ce^{-(2+\alpha)\frac{t}{2}},\quad [\D^i_{\mathbf{f}\cdots\mathbf{f}}\eta_{t,2,k}]_{C^\alpha(\{\cdot\}\times Y,g_X)}\leq Ce^{-(2+\alpha)\frac{t}{2}},\quad 0\leq i\leq 2,
\end{equation}
which together with \eqref{zinger} imply
\begin{equation}\label{sauron}
|\D^i_{\mathbf{f}\cdots\mathbf{f}}(\omega_F+e^t\gamma_{t,2,k}+e^t\eta_{t,2,k})|_{g_X}\leq C,\quad
|\D^i_{\mathbf{f}\cdots\mathbf{f}}(e^t\gamma_{t,2,k}+e^t\eta_{t,2,k})|_{g_X}\leq Ce^{-\gamma\frac{t}{2}},\quad 0\leq i\leq 2,
\end{equation}
\begin{equation}\label{saruman}
[\D^i_{\mathbf{f}\cdots\mathbf{f}}(\omega_F+e^t\gamma_{t,2,k}+e^t\eta_{t,2,k})]_{C^\alpha(\{\cdot\}\times Y,g_X)}\leq C,\quad
[\D^i_{\mathbf{f}\cdots\mathbf{f}}(e^t\gamma_{t,2,k}+e^t\eta_{t,2,k})]_{C^\alpha(\{\cdot\}\times Y,g_X)}\leq Ce^{-\gamma\frac{t}{2}},\quad 0\leq i\leq 2,
\end{equation}
while for the fiber components, from \eqref{3satan}, \eqref{saw} and \eqref{saw6} we obtain the better bounds
\begin{equation}\label{mordor}
|\D^i_{\mathbf{f}\cdots\mathbf{f}}(e^t\gamma_{t,2,k})_{\mathbf{ff}}|_{g_X}\leq Ce^{-t},\quad |\D^i_{\mathbf{f}\cdots\mathbf{f}}(e^t\eta_{t,2,k})_{\mathbf{ff}}|_{g_X}\leq Ce^{-(2+\alpha)\frac{t}{2}},\quad 0\leq i\leq 2,\end{equation}
\begin{equation}\label{gollum}
[\D^i_{\mathbf{f}\cdots\mathbf{f}}(e^t\gamma_{t,2,k})_{\mathbf{ff}}]_{C^\alpha(\{\cdot\}\times Y,g_X)}\leq Ce^{-t},\quad [\D^i_{\mathbf{f}\cdots\mathbf{f}}(e^t\eta_{t,2,k})_{\mathbf{ff}}]_{C^\alpha(\{\cdot\}\times Y,g_X)}\leq Ce^{-(2+\alpha)\frac{t}{2}},\quad 0\leq i\leq 2,\end{equation}
and using \eqref{sauron}, \eqref{saruman}, \eqref{mordor} and \eqref{gollum} we immediately see
that the terms in lines $-1,-2,-3,-5$ of \eqref{3satana} satisfy \eqref{stronzetto} for $0\leq r\leq 2$.

Next, we consider the parts of lines $-4,-6$ which contain $\ddbar\underline{\psi_t}$ (again after dividing by $\omega_{\rm can}^m\wedge\omega_F^n$). Explicitly, these are the terms
\begin{equation}\label{aperito}
\sum_{j=1}^m\binom{m}{j}\omega_{\rm can}^{m-j}\wedge(\ddbar\underline{\psi_t})^j\wedge\left((\omega_F)^n_{\mathbf{ff}}+ne^t(\omega_F)^{n-1}_{\mathbf{ff}}(\gamma_{t,2,k})_{\mathbf{ff}}\right),
\end{equation}
\begin{equation}\label{aperyto}
\frac{m}{n+1}e^{-t}\sum_{j=1}^{m-1}\binom{m-1}{j}\omega_{\rm can}^{m-1-j}\wedge(\ddbar\underline{\psi_t})^{j}\wedge\omega_F^{n+1},
\end{equation}
(as usual divided by $\omega_{\rm can}^m\wedge\omega_F^{n}$). The first piece of \eqref{aperito} equals
\begin{equation}
\sum_{j=1}^m\binom{m}{j}\frac{\omega_{\rm can}^{m-j}\wedge(\ddbar\underline{\psi_t})^j}{\omega_{\rm can}^m},
\end{equation}
which is a function pulled back from the base which goes to zero locally smoothly. On the other hand, thanks to \eqref{sauron}, \eqref{saruman}, \eqref{mordor}, \eqref{gollum}, the second piece of \eqref{aperito} as well as \eqref{aperyto} satisfy
\begin{equation}\label{stronzetto_depotenziato}
|\D^r_{\mathbf{f}\cdots\mathbf{f}}(\ \cdot\ )|_{g_X}=o(e^{-t}),
\quad [\D^r_{\mathbf{f}\cdots\mathbf{f}}(\ \cdot\ )]_{C^\alpha(\{\cdot\}\times Y,g_X)}=o(e^{-t}),
\end{equation}
for $0\leq r\leq 2$, which is weaker than \eqref{stronzetto}.

Combining the discussions of all these pieces, shows that the error terms in \eqref{e:satanas} is equal to the sum of $o(1)_{\rm from\ base}$ plus terms that satisfy \eqref{stronzetto_depotenziato}
for $0\leq r\leq 2$. Recall that to go from \eqref{e:satanas} to \eqref{e:satanas2} we need to divide by the reference volume form, subtract the fiber average (which kills the above $o(1)_{\rm from\ base}$ term) and multiply by $e^{-t}$, and hence the error terms in \eqref{e:satanas2} satisfy
\begin{equation}\label{satanhash_depotenziato}
|\D^r_{\mathbf{f}\cdots\mathbf{f}}(\ \cdot\ )|_{g_X}=o(e^{-2t}),
\quad [\D^r_{\mathbf{f}\cdots\mathbf{f}}(\ \cdot\ )]_{C^\alpha(\{\cdot\}\times Y,g_X)}=o(e^{-2t}),
\end{equation}
for $0\leq r\leq 2$, which is weaker than
\begin{equation}\label{satanhash}
|\D^r_{\mathbf{f}\cdots\mathbf{f}}(\ \cdot\ )|_{g_X}\leq Ce^{-(4+\gamma)\frac{t}{2}},
\quad [\D^r_{\mathbf{f}\cdots\mathbf{f}}(\ \cdot\ )]_{C^\alpha(\{\cdot\}\times Y,g_X)}\leq Ce^{-(4+\gamma)\frac{t}{2}},
\end{equation}
which is the bound satisfied by the error terms in \eqref{satanacchio2} thanks to \eqref{saw3} and \eqref{saw4}. Thus, we can finally conclude that in \eqref{freak}
\begin{equation}\label{feffer}
\sum_{p=1}^{N_{2,k}}A_{t,2,p,k}G_{2,p,k}=-e^{-2t}(\mathcal{S} - \underline{\mathcal{S}}) +(\mathrm{Err}),
\end{equation}
the error term (Err) satisfies \eqref{satanhash_depotenziato} for $0\leq r\leq 2$.
Going back to \eqref{herrnikolaus}, observe that
\begin{equation}\label{nykterstein}
\ddbar\sum_{p=1}^{N_{2,k}}A_{t,2,p,k}(\Delta^{\omega_F|_{\{\cdot\}\times Y}})^{-1}G_{2,p,k}=
\ddbar(\Delta^{\omega_F|_{\{\cdot\}\times Y}})^{-1}\left(\sum_{p=1}^{N_{2,k}}A_{t,2,p,k}G_{2,p,k}\right),
\end{equation}
and so using \eqref{saw5} we see that
\begin{equation}\label{stein}
\gamma_{t,2,k}=\ddbar(\Delta^{\omega_F|_{\{\cdot\}\times Y}})^{-1}\left(\sum_{p=1}^{N_{2,k}}A_{t,2,p,k}G_{2,p,k}\right)+(\mathrm{Err}'),
\end{equation}
where the $(1,1)$-form $(\mathrm{Err}')$ has fiber-fiber components that satisfy \eqref{satanhash} for $0\leq r\leq 2$.

When substituting \eqref{feffer} into \eqref{stein}, we need to understand the regularity of $\ddbar(\Delta^{\omega_F|_{\{\cdot\}\times Y}})^{-1}(\mathrm{Err})$ where (Err) are functions that satisfy \eqref{satanhash_depotenziato} for $0\leq r\leq 2$. When we restrict purely to a fiber $\{z\}\times Y$, the operator $(\Delta^{\omega_F|_{\{\cdot\}\times Y}})^{-1}$ increases regularity by $2$ derivatives in H\"older spaces, by standard Schauder theory, and so $(\Delta^{\omega_F|_{\{\cdot\}\times Y}})^{-1}(\mathrm{Err})$ satisfies \eqref{satanhash_depotenziato} for $0\leq r\leq 4$. It follows that the fiber-fiber components of $\ddbar(\Delta^{\omega_F|_{\{\cdot\}\times Y}})^{-1}(\mathrm{Err})$ satisfy \eqref{satanhash_depotenziato} for $0\leq r\leq 2$. Thus from \eqref{feffer} and \eqref{stein} we obtain that
\begin{equation}\label{herrschumi}\begin{split}
\gamma_{t,2,k}&=-e^{-2t}\ddbar(\Delta^{\omega_F|_{\{\cdot\}\times Y}})^{-1}\left(\mathcal{S} -\underline{\mathcal{S}}\right)+(\mathrm{err}),
\end{split}\end{equation}
where the $(1,1)$-form (err) has fiber-fiber components that satisfy \eqref{satanhash_depotenziato} for $0\leq r\leq 2$. In addition, (err) is also $O(1)$ in $C^2(g_t)$ because $\gamma_{t,2,k}$ is $O(1)$ in $C^2(g_t)$ thanks to \eqref{saw}, while the first term on the RHS of \eqref{herrschumi} also clearly $O(1)$ in $C^2(g_t)$. More precisely, we obtain
\begin{equation}\label{tepossino}
|\D^i({\rm err})|_{g_t}\leq Ce^{-\frac{2-i}{2}t}, \quad 0\leq i\leq 2.
\end{equation}

Next, we wish to have a more explicit understanding of the term $\mathcal{S} -\underline{\mathcal{S}}$. For this, we use the work of Schumacher and collaborators \cite{BCS,Schu}. For $1\leq \mu\leq m$ we let $W_\mu$ be the unique $(1,0)$ smooth vector field on $B\times Y$ which is the $\omega_F$-horizontal lift of the standard coordinate vector field $\de_{w^\mu}$ on $\C^m$, i.e. it satisfies $({\rm pr}_B)_*W_\mu=\de_{w^\mu}$ and $\omega_F(W_\mu, T^{0,1}X_z)=0$ for all $z\in B$ (where as usual $X_z=\{z\}\times Y$), cf. \cite[\S 3.1]{BCS}.
Then applying the fiberwise $\db$ operator to $W_\mu$ gives
\begin{equation}A_\mu=\db_{X_z} W_\mu,\end{equation}
where $A_\mu$ is the unique $T^{1,0}X_z$-valued $(0,1)$-form on $X_z$ harmonic with respect to the Ricci-flat metric $\omega_F|_{X_z}$ that represents the Kodaira-Spencer class $\kappa(\de_{w^\mu})\in H^1(X_z, T^{1,0}X_z)$. We now claim that we have the identity
\begin{equation}\label{zykel}
\mathcal{S}-\underline{\mathcal{S}}=(\Delta^{\omega_F|_{\{\cdot\}\times Y}})^{-1}\left(g_{\rm can}^{\mu\bar\nu}(\langle A_\mu, A_{\bar{\nu}}\rangle - \underline{\langle A_{\mu}, A_{\bar\nu}\rangle} )\right),
\end{equation}
where $\langle\cdot,\cdot\rangle$ is the fiberwise Ricci-flat inner product. This identity is proved as follows: first, a classical computation of Semmes \cite{Se} shows that our ``geodesic curvature'' quantity $\mathcal{S}$ defined in \eqref{semmes} is equal to
\begin{equation}
g_{\rm can}^{\mu\bar\nu}\omega_F(W_\mu,\ov{W_\nu}),
\end{equation}
and then \eqref{zykel} follows from this together with the fiberwise Laplacian formula of Schumacher (see \cite[Lemma 4.1]{BCS}) valid on the fiber $X_z$
\begin{equation}\label{BCS}
\Delta^{\omega_F|_{X_z}}\biggl[ \omega_F(W_\mu,\overline{W_\nu})\biggr] = \omega_{\rm WP}(\de_{w^\mu},\de_{\ov{w}^\nu})|_z -\langle A_\mu,\ov{A_\nu}\rangle,\end{equation}
where $\omega_{\rm WP}\geq 0$ is a Weil-Petersson form on $B$, so the term $\omega_{\rm WP}(\de_{w^\mu},\de_{\ov{w}^\nu})|_z$ is constant on $X_z$.

In conclusion, we have the decomposition
\begin{equation}
\omega_t^\bullet = f^*\omega_{\rm can} + e^{-t}\omega_F+\gamma_{t,0} - e^{-2t}\partial\overline\partial(\Delta^{\omega_F|_{\{\cdot\}\times Y}})^{-1} (\Delta^{\omega_F|_{\{\cdot\}\times Y}})^{-1}(g_{\rm can}^{\mu\bar\nu}(\langle A_\mu, A_{\bar{\nu}}\rangle - \underline{\langle A_{\mu}, A_{\bar\nu}\rangle} )) +\eta_{t,2,k}+ (\mathrm{err}),
\end{equation}
where (err) is $O(1)$ in $C^2(g_t)$, its fiber-fiber components satisfy \eqref{satanhash_depotenziato} for $0\leq r\leq 2$, and it satisfies \eqref{tepossino}. But the term $\eta_{t,2,k}$ satisfies even better estimates than these thanks to \eqref{3satan}, and so it can be absorbed in (err), while the term $\gamma_{t,0}$ is pulled back from the base and by \eqref{vier3} and \eqref{vier2} it is $o(1)$ in $C^{2,\alpha}$.

To complete the proof of Theorem \ref{mthmB}, we set ${\rm error}_1=\gamma_{t,0}, {\rm error}_2=({\rm err})$ and we thus need to verify that our estimates for (err) imply the stated bounds \eqref{simplicio}. For simplicity let us write (err)=$\eta$, which is a $\de\db$-exact $(1,1)$-form whose potential has fiberwise average zero. To prove \eqref{simplicio}, we start by converting $\D^j\eta$ to $\nabla^{z,j}\eta$ (at an arbitrary point $(z,y)$). By definition these are equal for $j=0,1,$ while for $j=2$ we have
$\D^2\eta=\nabla^{z,2}\eta+\mathbf{A}\circledast\eta$ by \eqref{komm}, and $|\mathbf{A}|_{g_t}\leq Ce^t$ (since by definition $\mathbf{A}_{\mathbf{fff}}^\bullet=0$).

On the other hand, from \eqref{tepossino} we know that
\begin{equation}
|\D^j\eta|_{g_t}\leq Ce^{-\frac{2-j}{2}t},\quad 0\leq j\leq 2,
\end{equation}
and so
\begin{equation}\label{donkey1}
|\eta|_{g_t}\leq Ce^{-t}, \quad |\nabla^{z}\eta|_{g_t}\leq Ce^{-\frac{t}{2}},
\end{equation}
and
\begin{equation}
|\nabla^{z,2}\eta|_{g_t}\leq |\D^2\eta|_{g_t}+C|\mathbf{A}|_{g_t} |\eta|_{g_t}\leq C+Ce^t e^{-t}\leq C.
\end{equation}
We then work in local product coordinates, and convert
\begin{equation}\label{prwzt1}
\nabla^z\eta=\de\eta+\Gamma^z\circledast\eta,
\end{equation}
with $|\Gamma^z|_{g_t}\leq Ce^{\frac{t}{2}}$ along the fiber over $z$ (since $\Gamma^z$ are the Christoffel symbols of the product metric $g_{z,t}$) and so
\begin{equation}\label{donkey2}
|\de \eta|_{g_t}\leq C e^{-\frac{t}{2}},
\end{equation}
and also
\begin{equation}\label{prwzt2}
\nabla^{z,2}\eta=\de^2\eta+\de\Gamma^z\circledast\eta+\Gamma^z\circledast\de\eta+\Gamma^z\circledast\Gamma^z\circledast\eta,
\end{equation}
and using $|\de\Gamma^z|_{g_t}\leq Ce^t$ (since $\de_{\mathbf{f}}(\Gamma^z)_{\mathbf{ff}}^{\mathbf{b}}=0$) we obtain
\begin{equation}\label{donkey3}
|\de^2 \eta|_{g_t}\leq C.
\end{equation}
Combining \eqref{donkey1}, \eqref{donkey2} and \eqref{donkey3} gives (in a fixed metric)
\begin{equation}\label{donkeyshow}
|(\de^j \eta)\{\ell\}|\leq Ce^{\frac{-2-\ell+j}{2}t},\quad 0\leq j\leq 2,
\end{equation}
which for $\ell < j+2$ agrees with the statement of \eqref{simplicio}. To also obtain the slightly better estimate for $\ell = j+2$ stated in \eqref{simplicio}, recall that from the fact that the fiber-fiber components of (err) satisfy \eqref{satanhash_depotenziato} for $0\leq r\leq 2$, we see that
\begin{equation}\label{monkey2}
|\nabla^{z,j}_{\mathbf{f}\cdots\mathbf{f}}(\eta_{\mathbf{ff}}|_{\{z\}\times Y})|_{g_X}=o(e^{-2t}),\quad 0\leq j\leq 2,
\end{equation}
and using \eqref{monkey2} together with \eqref{prwzt1}, \eqref{prwzt2} and $|\Gamma^z|_{g_X}\leq C, |\de\Gamma^z|_{g_X}\leq C$, gives
\begin{equation}
|\de^j(\eta_{\mathbf{ff}}|_{\{z\}\times Y})|_{g_X}=o(e^{-2t}),\quad 0\leq j\leq 2,
\end{equation}
which is the improvement over \eqref{donkeyshow} for $\ell = j+2$ claimed in \eqref{simplicio}.
\end{proof}

\end{document}